\documentclass[reqno,10pt]{amsart}
\usepackage{amssymb,amsmath,amsthm
,}
\usepackage{a4wide}
\usepackage{color}
\usepackage{esint}
\usepackage{calrsfs}
\newtheorem{theorem}{Theorem}

\newtheorem{corollary}[theorem]{Corollary}
\newtheorem{definition}[theorem]{Definition}
\newtheorem{lemma}[theorem]{Lemma}
\newtheorem{proposition}[theorem]{Proposition}
\newtheorem{remark}[theorem]{Remark}
\let\a=\alpha
\let\e=\varepsilon
\let\d=\delta

\let\pt=\partial
\let\t=\vartheta
\let\O=\Omega

\let\o=\omega
\let\g=\gamma
\let\th=\theta
\let\l=\lambda

\let\f=\varphi

\let \z=\zeta

\newcommand{\R}{\mathbb{R}}
\renewcommand{\P}{\mathbf{P}}
\renewcommand{\S}{\mathbf{S}}

\newcommand{\be}{\begin{equation}}
\newcommand{\bm}{\begin{multline}}
\newcommand{\ee}{\end{equation}}
\newcommand{\dd}{\mathrm{d}}

\numberwithin{equation}{subsection}
\numberwithin{theorem}{section}

\begin{document}
\date{}
\title[Stationary solutions to the Boltzmann equation\dots]{ {\large {%
Stationary solutions to the Boltzmann equation in the Hydrodynamic limit}}}
\author{R. Esposito}
\thanks{(R.E.) International Research Center M\&MOCS, Univ. dell'Aquila,
Cisterna di Latina, (LT) 04012 Italy}
\author{Y. Guo}
\thanks{(Y.G.) Division of Applied Mathematics, Brown University,
Providence, RI 02812, U.S.A.}
\author{C. Kim}
\thanks{(C.K.) Department of Mathematics, University of Wisconsin, Madison,
53706-1325 WI, U.S.A.}
\author{R. Marra}
\thanks{(R.M.) Dipartimento di Fisica and Unit\`a INFN, Universit\`a di Roma
Tor Vergata, 00133 Roma, Italy}

\begin{abstract}
Despite its conceptual and practical importance, the rigorous derivation of
the steady incompressible Navier-Stokes-Fourier system from the Boltzmann
theory has been {an} outstanding {open problem} for general domains in 3D.
We settle this open question in {the} affirmative, in the presence of a
small external field and a small boundary temperature variation for the
diffuse boundary condition. 
We employ a recent quantitative $L^{2}-L^{\infty }$ approach with new 
$L^{{6}}$ estimates for the hydrodynamic part $\mathbf{P}f$ of the
distribution function. Our results also imply the validity of Fourier law in the hydrodynamical
limit, and our method {leads to {asymptotical} stability of steady Boltzmann
solutions as well as the derivation of the {unsteady} Navier-Stokes-Fourier
system}.
\end{abstract}

\maketitle
\tableofcontents

\section{Introduction}

\subsection{Background}

The hydrodynamic limit of the Boltzmann equation has been the subject of
many studies since the pioneering work by Hilbert, who introduced {his
famous expansion in the Knudsen number $\e$} {in \cite{H1,H2}}, realizing
the first example {of} the program he {proposed} {in the sixth of his famous
questions \cite{H00}}. Mathematical results on the closeness of the Hilbert
expansion {of the Bolzmann equation to the solutions of the compressible
Euler equations} for small Knudsen number $\e$, were obtained by Caflisch 
\cite{Ca}{,} and {Lachowicz} \cite{Lac}, while Nishida \cite{Ni}, Asano and
Ukai \cite{AU} proved this by different methods. More recently, the
convergence in the presence of singularities for the Euler equations have
been obtained in \cite{SHY} and \cite{HWY}. The relativistic Euler limit has
been studied in \cite{SS}.

On a longer time scale $\e^{-1}$, where diffusion effects become
significant, the problem can be faced only in the low Mach numbers regime
(Mach number of order $\e$ or smaller) {due to the lack of scaling
invariance of the compressible Navier-Stokes equations.} {Hence} the
Boltzmann solution has been proved to be close to the incompressible
Navier-Stokes-Fourier system. Mathematical results were given, among the
others, in \cite{DEL, BU,Guo06,GJ,GJJ} for smooth solutions. For weak
solutions {(renormalized solutions)}, after several partial steps   \cite{BGL91,BGL93,BGL98,BGL00, LM, MS},  the the full result for
the convergence of the renormalized solutions has been obtained by Golse and
Saint-Raymond \cite{GSR}. For more references, see \cite{Saint}.

Much less is known about the {steady} solutions. {It is worth to notice
that, even for fixed Knudsen numbers, the analog of DiPerna-Lions'
renormalized solutions {\cite{DL}} is not available for the steady case},
due to lack of $L^{1}$ and entropy estimates. {In \cite{Gui1,Gui2}, steady
solutions were constructed in convex domains near Maxwellians, and their
positivity was left open. The only other results} are for special,
essentially one {dimensional geometry} (see \cite{AN} for results at fixed
Knudsen numbers and \cite{ELM1,ELM2,AEMN1,AEMN2} for results at small
Knudsen numbers in {certain} {special 2D geometry}). In a recent paper \cite%
{EGKM}, via a new $L^{2}-L^{\infty }$ framework, we have constructed the {%
steady} solution to the Boltzmann equation close {to} Maxwellians, in 3D
general domains, for a gas in contact with a boundary with a prescribed
temperature profile modeled by the diffuse reflection boundary condition.
The question about positivity of this steady solution was resolved as a
consequence of their dynamical stability. As pointed in \cite{Golse},
despite the importance of steady Navier-Stokes-Fourier equations in
applications, it has been an outstanding open problem to derive them from
the steady Boltzmann theory.

The goal of our paper is to employ the $L^{2}-L^{\infty }$ framework
developed in \cite{EGKM} to study the hydrodynamical limit of the solutions
to the {steady} Boltzmann equation, in the low Mach numbers regime, in a
general domain with boundary where a temperature profile is specified. {In
order to perform the hydrodynamic limit, uniform in $\e$ estimates are
required. To achieve this we need to show higer integrability than $L^2$ for
the slow modes of the solution.}

\vspace{8pt}

{Let $\O $ be a bounded open region of $\mathbb{R}^d$ {for either $d=2$ or $%
d=3$.} We consider the Boltzmann equation for the distribution density $%
F(t,x,v)$ with $t\in \mathbb{R}_{+}:=[0, \infty)$, $x\in \O $, $v\in \mathbb{%
R}^{3}$. In the diffusive regime, the time evolution of the gas, subject to
the action of a field $\vec{G}$, is described by the following \textit{%
rescaled} Boltzmann equation: 
\begin{equation}
\pt_t F+\e^{-1}v\cdot\nabla_x F+\vec{G}\cdot \nabla_v F=\e^{-2}Q(F,F),
\label{basic2}
\end{equation}%
}where the Boltzmann collision operator is defined as 
\begin{eqnarray*}
Q(F,H)(v)&:=&\int_{\mathbb{R}^{3}}\int_{{\mathbb{S}^{2}} }B(v-u,\o %
)[F(v^{\prime })H(u^{\prime })-F(v)H(u)] \mathrm{d}\o \mathrm{d} u \\
&:=& Q_{+}(F,H){(v)}-Q_{-}(F,H){(v)},
\end{eqnarray*}%
with $v^{\prime } = v-[(v-u)\cdot \o ]\o , \ u^{\prime }={v} + [(v-u)\cdot 
\o ]\o {.}$ {Here,} $B$ is chosen as \textit{the hard spheres cross section}
throughout this paper, 
\begin{equation}  \label{hard_shpere}
B(V,\o )=|V\cdot \o |.
\end{equation}

The interaction {of gas} with the boundary $\pt\O $ is given by the diffuse
reflection boundary condition, defined as follows: Let 
\begin{equation}
M_{\rho, u,T}:=\frac{\rho}{(2\pi T)^{\frac 3 2}}\exp\Big[-\frac {|v-u|^2}{2 T%
}\Big],  \notag
\end{equation}
be the {local} Maxwellian with density $\rho$, mean velocity $u$, and
temperature $T$. {For} a prescribed function ${T_{w}}$ on $\pt\O $, {we
define} 
\begin{equation}  \label{Tw}
M_w= \sqrt{\frac{2\pi}{{T_{w}}}}M_{1,0,{T_{w}}}.
\end{equation}
We impose {\textit{the diffuse reflection boundary condition}} as 
\begin{equation}
F = \mathcal{P}^w_\g(F),{\ \ \ \ \text{on} \ \ \gamma_{-},}  \label{bc0}
\end{equation}
where%
\begin{equation}
\mathcal{P}^w_{\gamma } {F}(x,v){\ \ := \ }M_w(x,v)\int_{n(x)\cdot {u}>0} {F}%
(x,{u})\{n(x)\cdot {u}\}\mathrm{d} {u}.  \label{pgamma}
\end{equation}%
Here, we denote by $n(x)$ the {outward} normal to $\pt\O $ at {$%
x\in\partial\Omega$} and we decompose the phase boundary $\gamma:=
\partial\Omega \times \mathbb{R}^{3}$ as 
\begin{equation}  \label{phase_bdry}
\begin{split}
\gamma _{\pm} & \ \ := \ \ \{(x,v)\in \partial \Omega \times \mathbb{R}^{3}{:%
} \ n(x)\cdot v\gtrless 0\}, \\
\gamma _{0} & \ \ := \ \ \{(x,v)\in \partial \Omega \times \mathbb{R}^{3}{:}
\ n(x)\cdot v=0\}.
\end{split}%
\end{equation}
{We remind that the boundary condition (\ref{bc0}), (\ref{pgamma}) ensures
the zero net mass flow at the boundary: 
\begin{equation}
\int_{\mathbb{R}^3} F(x,v) \{ n(x) \cdot v\} \mathrm{d} v =0, \ \ \ \text{%
for any } x\in \pt\O .  \notag
\end{equation}%
}

{The rescaled Boltzmann equation} (\ref{basic2}) is studied under the
assumption of low Mach numbers, meaning that the average velocity is small
compared to the sound speed. This can be achieved by looking for solutions 
\begin{equation}
F-\mu=\mathfrak{M}\sqrt{\mu} f,  \label{lmn}
\end{equation}
with {\textit{the global Maxwellian}} 
\begin{equation}  \label{g_Maxwellian}
{\mu(v)=M_{1,0,1} = \frac{1}{(2\pi)^{3/2}} e^{- \frac{|v|^{2}}{2}} . }
\end{equation}
{Here, the number} $\mathfrak{M}$ {is} proportional to the Mach number. The
case {of} $\mathfrak{M}=\e$ corresponds to the incompressible
Navier-Stokes-Fourier limit (INSF) that will be discussed in this paper. The
case {of} $\mathfrak{M}\ll \e$ corresponding to the incompressible
Stokes-Fourier {limit}, is simpler and the results of this paper also cover this case
which will not be discussed explicitly.

The condition (\ref{lmn}), once assumed initially, needs to be checked at
later times. By multiplying (\ref{basic2}) by $v$ and integrating on
velocities, we see that the change of mean velocity is proportional to $\vec{%
G}$. Thus, we need to assume $\vec{G}=\mathfrak{M}\Phi $ with {a} bounded ${%
\Phi}$. Moreover, to make (\ref{lmn}) compatible with the boundary
conditions, we need to assume that $T_{w}=1+\mathfrak{M}\t_w $. {In
particular, for the INSF case, we have } 
\begin{equation}  \label{G_T_NSF}
{\Vec{G}=\e\Phi, \ \ \ T_{w}=1+\e\t _{w}.}
\end{equation}

\subsection{Notation and preliminary definitions}

Let $\Theta_w$ be any fixed smooth function on $\O $ such that $\Theta_w|_{%
\pt\O }=\t_w$ and 
\begin{equation}
\|\Theta_w\|_{W^{1,\infty}(\O )}\lesssim \|\t_w\|_{W^{1,\infty}(\pt\O )} .
\label{Thetaw}
\end{equation}
Let 
\begin{equation}
f_w= \sqrt{\mu}[\Theta_w(|v|^2-3)/2+\rho_w], \quad \rho_w =-\Theta_w+ |\O %
|^{-1}\int_{\O } \Theta_w,  \label{fw}
\end{equation}
where $\mu$ is the standard Maxwellian in (\ref{g_Maxwellian}). The average
of $\Theta_w$ is added so that $\iint_{\O \times \mathbb{R}^3} f_w=0$. We
look for a solution in the form {%
\begin{equation}
F \ = \ \mu + \e \sqrt{\mu} \big(f_w + f).  \label{exprf}
\end{equation}%
} Note that there is no loss of generality in assuming that the zero-mass
condition 
\begin{equation}
\iint_{\O \times\mathbb{R}^3}f\sqrt{\mu}=0,  \label{zeromass1}
\end{equation}
so that $\iint_{\O \times\mathbb{R}^3}F=\iint_{\O \times\mathbb{R}^3}\mu=|\O %
|$.

Our aim is to show that we can construct $f$ such that $F$ solves (\ref%
{basic2}) and (\ref{bc0}) both in the steady and unsteady case. Moreover, as 
$\e\to 0$, $f$ converges to some suitable sense to $f_1$ given by 
\begin{equation}
f_1:=[\rho+u\cdot v+\frac{|v|^2-3}{2} \th ]\sqrt{\mu} ,  \label{f1}
\end{equation}
{where $(\rho, u, \th )$ represents the density, velocity, and temperature
fluctuations. The density and the temperature fluctuations satisfy the
Boussinesq relation} 
\begin{equation}
\nabla_x(\rho+\th )=0,  \label{div}
\end{equation}
and {the velocity and the temperature fluctuations satisf{y} the
Incompressible Navier-Stokes Fourier System (INSF)} 
\begin{equation}
\begin{split}  \label{INSF_dy}
\pt_t u+ u\cdot\nabla_x u +\nabla_x p= \mathfrak{v}\Delta u +\Phi, \ \
\nabla_x\cdot u= 0 \ \ \ &\text{in } \ \Omega, \\
\pt_t \th +u\cdot\nabla_x (\th +\Theta_w)= \kappa\Delta(\th +\Theta_w) \ \ \
&\text{in } \ \Omega, \\
u(x,0)=u_0(x), \ \ \th (x,0)= \th _0(x) \ \ \ &\text{in } \ \Omega, \\
u(x) =0, \quad \th (x) =0\ \ \ &\text{on} \ \partial\Omega,
\end{split}%
\end{equation}
where $\mathfrak{v}$ is the viscosity and $\kappa$ is the heat conductivity and $p$ the pressure.
We have used the function $\Theta_w$ in our definition to impose the null
boundary data for $\th $. 

We recall the definition of the linearized collision operator: 
\begin{equation}  \label{def_L}
Lf=-\frac{1}{\sqrt{\mu}}[Q(\mu,\sqrt{\mu} f)+ Q(\sqrt{\mu} ,f \mu)],
\end{equation}
and the nonlinear collision operator: 
\begin{equation}  \label{Gamma}
\Gamma(f,g)=\frac{1}{\sqrt{\mu} }[Q(\sqrt{\mu} f,\sqrt{\mu} g)+Q(\sqrt{\mu}
g,\sqrt{\mu} f)].
\end{equation}
The null space of $L$, $\text{Null} \/ L$ is a five-dimensional subspace of $%
L^2(\mathbb{R}^3)$ spanned by 
\begin{equation*}
\Big\{\sqrt{\mu}, \ v \sqrt{\mu}, \ \frac{|v|^{2}-3}{2} \sqrt{\mu} \Big\}.
\end{equation*}
We denote the orthogonal projection of $f$ onto $\text{Null}\/L$ as 
\begin{equation}  \label{Pabc}
\P f \ = \ a \sqrt{\mu} + v\cdot b \sqrt{\mu} + c \frac{|v|^{2}-3}{2} \sqrt{%
\mu},
\end{equation}
and $(\mathbf{I}-\mathbf{P})$ the projection on the orthogonal complement of 
$\text{Null}\/L$. The inverse operator $L^{-1}$ is defined as follows: {if $%
\P g=0$,} $L^{-1}g$ is the unique solution of $L(L^{-1} g) = g,$ and $%
\mathbf{P}(L^{-1}g)=0$.

Note that the functions $f_1$ and $f_w$ are in $\text{Null} \/ L$. \vspace{%
4pt}

It is well-known that, (see \cite{CIP}) 
\begin{equation}
L f= \nu f -Kf,  \notag
\end{equation}
where {the collision frequency is defined as} 
\begin{equation}
\nu(v)=\frac{1}{\sqrt{\mu} }Q_-(\sqrt{\mu} f,\mu)=\int_{\mathbb{R}^3} \int_{%
\mathbb{S}^2} |(v-u)\cdot \o | \sqrt{\mu}(u)\mathrm{d}\o \mathrm{d} u. 
\notag
\end{equation}
{For the hard sphere cross section (\ref{hard_shpere}),} there are positive
numbers $C_0$ and $C_1$ such that, for $\langle v\rangle:= \sqrt{1+ |v|^{2}}$%
, 
\begin{equation}
C_0 \langle v\rangle \le \nu(v)\le C_1 \langle v\rangle.  \label{nu0}
\end{equation}
Moreover the compact operator on $L^2(\mathbb{R}_v^3)$, $K$ is defined as 
\begin{equation}
K f= \frac{1}{\sqrt{\mu}}[Q_+(\mu,\sqrt{\mu} f)+Q_+(\sqrt{\mu} f,\mu)
-Q_-(\mu,\sqrt{\mu} f)]=\int_{\mathbb{R}^3} [\mathbf{k}_1(v,u)-\mathbf{k}%
_2(v,u)]f(u) \mathrm{d} u.  \notag
\end{equation}
The operator $L$ is symmetric on the  dense subspace $D_L=\{f\in L^2(\mathbb{R}_v^3)\,|\,\nu^{\frac 1 2} f \in  L^2(\mathbb{R}_v^3)\}$ : $%
(f,Lg)_{2}=(g,Lf)_{2}$ where $(\cdot\, , \,\cdot)_{2}$ is the $L^{2}(\mathbb{%
R}^3_v)$ inner product. 

The following spectral inequality holds for $L$: 
\begin{equation}
(f,Lf)_{2}\gtrsim \| \nu^{1/2}(\mathbf{I}-\mathbf{P}) f\|_{L^2(\mathbb{R}%
_v^3)}^2.  \label{spectL}
\end{equation}

\vspace{8pt}

\subsection{Boundary Conditions}

From the definition of $\Theta_w$, we have 
\begin{equation}
M_{1+ \e \rho_w,0, 1+ \e \Theta_w}\Big|_{\gamma_{-}} =\mathcal{P}^w_\gamma
(M_{1+ \e \rho_w, 0, 1+ \e \Theta_w}).  \notag
\end{equation}
Moreover, by expanding $M_{1+ \e \rho_w,0, 1+ \e \Theta_w} $ in $\e$, we get 
\begin{equation}  \label{mu_e}
M_{1+ \e \rho_w,0, 1+ \e \Theta_w} \ = \ \mu+ \e f_{w} \sqrt{\mu} + \e^2 \f_{%
\e} ,
\end{equation}
where 
\begin{equation}
{|\f_{\e} |\leq O(\| \t_w\|^2_{L^{\infty}(\partial\Omega)} )\langle
v\rangle^{4} \mu(v)}.  \label{phie}
\end{equation}
Therefore, on $\gamma_{-}$ 
\begin{equation}
\mu+ \e f_w \sqrt{\mu}+\e^2 \f_{\e} \sqrt{\mu} \ = \ \mathcal{P}^w_\gamma
(\mu + \e f_w \sqrt{\mu} +\e^2 \f_{\e} \sqrt{\mu}).  \label{bcmueps}
\end{equation}
On the other hand, from (\ref{bc0}) and (\ref{exprf}), on $\gamma_{-}$, 
\begin{equation}
\mu+ \e (f_w+f) \sqrt{\mu} \ = \ \mathcal{P}^w_\gamma[\mu+ \e (f_w+f) \sqrt{%
\mu}] .  \notag
\end{equation}
Subtracting above two equations, we obtain the boundary condition for $f$: 
\begin{equation}
f|_{\gamma_{-}} = \sqrt{\mu} ^{-1}\mathcal{P}^w_\gamma (\sqrt{\mu} f) +\e r ,
\notag
\end{equation}
with 
\begin{equation}
r= {\sqrt{\mu}}^{-1} \mathcal{P}^w_\gamma \big(\sqrt{\mu}\f_{\e}%
\big)-\f_{\e}, \ \ |r|_{L^{\infty} (\partial\O \times\mathbb{R}^{3})}\lesssim \e %
|\vartheta_{w}|_{L^{\infty} (\partial\Omega)}.  \label{r}
\end{equation}
Furthermore we can write 
\begin{equation}
{\sqrt{\mu}^{-1}}\mathcal{P}^w_\g ( \sqrt{\mu}f) \ = \ {P}_\g f+\e \mathcal{Q%
}f,  \notag
\end{equation}
with 
\begin{eqnarray}
{P}_\g f(x,v)&:=&\sqrt{2\pi}\sqrt{\mu(v)} \int_{n(x)\cdot u>0} f(u) \sqrt{%
\mu(u)} \{n(x) \cdot u\} \mathrm{d} u,  \label{P_gamma} \\
\mathcal{Q}f &:=& \e^{-1}\big[ {\sqrt{\mu}}^{-1} \mathcal{P}^w_\gamma(%
\sqrt{\mu} f) -{P}_\g f\big].  \label{defQ}
\end{eqnarray}
Note that the boundary operator $\mathcal{Q}$ is bounded uniformly in $\e$ {%
and, for $0 \leq \beta < \frac{1}{4}$, 
\begin{equation}
| e^{\beta|v|^{2}}\mathcal{Q}f |_{L^\infty(\pt \O \times\mathbb{R}%
^3)}\lesssim |\t _{w} |_{L^\infty(\pt\O )}.  \label{boundQ}
\end{equation}%
}This follows by expanding $M_w$ in (\ref{Tw}) with $T_{w}= 1+ \e \t_{w}$ in 
$\e$ to obtain 
\begin{equation}  \label{exp_Mw}
M_{w}(x,v) \ = \ \sqrt{2\pi} \mu(v) + \e \t_{w}\sqrt{2\pi} \big(\frac{|v|^{2}%
}{2} - 2\big) \mu(v) + \e^{2} O(|\t_{w}|^{2}) \langle v\rangle^{4} \mu(v).
\end{equation}
Hence the boundary condition for $f$ becomes 
\begin{equation}
\begin{split}
f \ = \ {P}_\g f +\e [\mathcal{Q}f +r], \ \ \ \text{on} \ \ \gamma_{-}
\label{bcR}
\end{split}%
\end{equation}
with $\mathcal{Q}$ in (\ref{defQ}) and $r$ in (\ref{r}).

{From $\int_{n\cdot v<0} \mu \{n\cdot v\} \mathrm{d} v=-1=\int_{n\cdot v<0}
M_w\{n\cdot v\} \mathrm{d} v$ and (\ref{r}) and (\ref{defQ}), it follows
that 
\begin{equation}  \label{Q_zero}
\int_{ n(x) \cdot v<0} \mathcal{Q} f \sqrt{\mu}\{ n(x) \cdot v\} \mathrm{d}
v \ = \ 0 \ = \ \int_{ n(x) \cdot v<0} r \sqrt{\mu} \{ n(x)\cdot v\} \mathrm{%
d} v , \quad\text{for any } x\in \pt\O .
\end{equation}%
}

\vspace{8pt}

\textit{Notation.} We use $\| \cdot \|_p$ and $\| \cdot \|_{L^{p}}$ for both
of the $L^p(\bar{\Omega }\times\mathbb{R}^3)$ norm and the $L^p(\bar{\Omega}%
) $ norm, and $( \, \cdot\,,\, \cdot \, )$ for the $L^2(\bar{\Omega}\times%
\mathbb{R}^3)$ inner product or $L^2( \mathbb{R}^3)$ inner product, where $%
\bar{\Omega} : = \Omega \cup \partial\Omega$. We subscript this to denote
the variables, thus $\| \cdot \|_{L^{p}_{y}}$ means $L^{p}(\{y \in Y\})$. We
denote $\|\cdot \|_{\nu}\equiv \|\nu^{1/2}\cdot \|_2$ and $\|f\|_{H^k}=
\|f\|_{2}+ \sum_{i=1}^{k}\|\nabla_x^i f \|_{2}$. We also denote $\| \cdot
\|_{L^{p}L^{q}}:= \| \cdot \|_{L^{p}(L^{q})} : = \big\| \| \cdot \|_{L^{q}}%
\big\|_{L^{p}}$. For the phase boundary integration, we define $\mathrm{d}%
\gamma = |n(x)\cdot v|\mathrm{d} S(x)\mathrm{d} v$ where $\mathrm{d} S(x) $
is the surface measure and define $|f|_p^p = \int_{\gamma} |f(x,v)|^p 
\mathrm{d}\gamma$ and the corresponding space as $L^p(\partial\Omega\times%
\mathbb{R}^3;\mathrm{d}\gamma)=L^p(\partial\Omega\times\mathbb{R}^3)$.
Further $|f|_{p,\pm}= |f \mathbf{1}_{\gamma_{\pm}}|_p$. We also use $|f|_p^p
= \int_{\partial\Omega} |f(x)|^p \mathrm{d} S(x)$. Denote $%
f_{\pm}=f_{\gamma_{\pm}}$. $X \lesssim Y$ is equivalent to $X \le C Y$,
where $C$ is a constant not depending on $X$ and $Y$. We subscript this to
denote dependence on parameters, thus $X \lesssim_\alpha Y$ means $X \leq
C_{\alpha} Y$. The notation $X \ll_{a} Y$ is equivalent to $X\leq C_{a} Y$,
where $C_{a}>0$ is sufficiently small.

\subsection{Main Results}

We first focus on the steady case. The steady solution to the Boltzmann
equation is obtained with the same procedure discussed before for the
unsteady case: 
\begin{equation}  \label{exprfs}
F_s =\mu +\e\sqrt{\mu}[f_w+ f_{s}],
\end{equation}
The unknown $f_s$ has to satisfy the following equation 
\begin{equation}  \label{equation_R_s}
v \cdot \nabla_{x} f_{s} + \e^{2} \frac{1}{\sqrt{\mu}} \Phi \cdot \nabla_{v} %
\Big[ \sqrt{\mu} f_{s} \Big] + \e^{-1}Lf_{s} = L_{1} f_{s} +
\Gamma(f_{s},f_{s})+ A_{s},
\end{equation}
where 
\begin{equation}
L_1 f_s= \frac 1{\sqrt{\mu}}[Q(\sqrt{\mu}f_w,\sqrt{\mu}f_s)+Q(\sqrt{\mu}f_s,%
\sqrt{\mu}f_w)],  \label{L1def}
\end{equation}
\begin{equation}
A_s= {\e\Phi\cdot v\sqrt{\mu}} -v\cdot \nabla_x f_w-\e^2\frac 1{\sqrt{\mu}}
\Phi\cdot \nabla_v \Big[ \sqrt{\mu} f_{w} \Big]+ {\Gamma(f_w,f_w)},
\label{A}
\end{equation}
with boundary conditions 
\begin{equation}
\begin{split}
f_s \ = \ {P}_\g f_s +\e [\mathcal{Q}f_s +r], \ \ \ \text{on} \ \ \gamma_{-}
\label{bcR1}
\end{split}%
\end{equation}
with $\mathcal{Q}$ in (\ref{defQ}) and $r$ in (\ref{r}).

Note that, by (\ref{fw}), 
\begin{equation}
\P A_s= {\e\Phi\cdot v\sqrt{\mu}}, \quad \iint_{\O \times \mathbb{R}^3} A_s 
\sqrt{\mu}=0.  \label{PAintA}
\end{equation}

\begin{theorem}
\label{mainth} {Assume $\Omega$ is an open bounded subset of $\mathbb{R}^{3}$
with $C^{3}$ boundary $\partial\Omega$. We also assume the hard sphere cross
section (\ref{hard_shpere}).}

If $\Phi = \Phi(x)\in 
C^{1}(\Omega)$, $\vartheta_{w} \in W^{1,\infty}(\partial\Omega)$ and 
\begin{equation}
\| {\t}_{w} \| _{H^{1/2}(\pt\O )}+\| \Phi \| _{L^{2}(\O )}\ll1,
\label{small_field}
\end{equation}
then, for $0<\e\ll 1$, there is a unique positive solution $F_{s} \ {\geq 0}$%
, given by (\ref{exprf}) with $f_{s} $ satisfying (\ref{equation_R_s}) and
the boundary condition (\ref{bcR}). 

Moreover, 
\begin{equation}  \label{energy_steady1}
\| f_{s} \|_{2} +\|\P f_s \|_6+\e^{-1} \|(\mathbf{I} - \mathbf{P}) f_{s}
\|_{\nu} + \e^{-\frac 1 2}|(1-P_\gamma)f_s |_{2,+}+ \e^{\frac 1 2} \| w
f_{s} \|_{\infty} \ll 1,
\end{equation}
where $w(v)= e^{\beta|v|^{2}}$ with $0< \beta \ll 1$. Finally, as $\e\to 0$, 
$f_s $ converges weakly to $f_{1,s} =[u_s\cdot v+\th _s(|v|^2-5)/2]\sqrt{\mu}
$ with $(p_{s},u_s,\th _s)$ the unique solution to the steady INSF with
Dirichlet boundary conditions and subject to the external field $\Phi$: 
\begin{equation}
\begin{split}  \label{INSF_st}
u_s\cdot\nabla_x u_s +\nabla_x p_s=\mathfrak{v}\Delta u_s +\Phi, \ \ \
\nabla_x\cdot u_s=0 \ \ \ & \text{in} \ \ \Omega, \\
u_s\cdot\nabla_x (\th _s+\Theta_w)=\kappa\Delta(\th _s +\Theta_w) \ \ \ & 
\text{in} \ \ \Omega, \\
u_s(x)=0, \quad \th _s(x)=0 \ \ \ & \text{on} \ \ \partial\Omega.
\end{split}%
\end{equation}
\end{theorem}

\noindent\textbf{Remark.} 
\textit{In particular Theorem \ref{mainth} implies the existence of solutions to the stationary INSF boundary value problem for small force and boundary temperature in the sense of (\ref{small_field}).}

\textit{Note that, if $\Phi=\nabla_x U$ is a potential field, $u_s\equiv0$, 
$p_s\equiv U$ is solution to the above system. Therefore, in order to have a
stationary solution with non vanishing velocity field, we may assume that $%
\Phi$ is not a potential field, such that $\nabla_x\cdot\Phi=0$. (See \cite%
{Golse})}

\bigskip

It is important to note that the key difficulty in this paper is to control
the `strong' nonlinear terms $\Gamma (f_{s},f_{s})$. \ The hard spheres
cross section is used to control the term $\e v\cdot \Phi f$ coming from the
external field.

We use the quantitative $L^{2}-L^{\infty }$ approach developed in \cite{EGKM}%
, in the presence of $\varepsilon$. We start with the energy estimates to
get 
\begin{equation*}
\frac{1}{\varepsilon }\| (\mathbf{I}-\mathbf{P})f_{s}\| _{\nu }\lesssim \|
\Gamma (f_{s},f_{s})\| _{2}+1.
\end{equation*}
The missing $\mathbf{P}R_{s}$ can be estimated by the coercivity estimates
in \cite{EGKM}, with carefully chosen proper test functions in the weak
formulation, such that (Proposition \ref{linearl2}): 
\begin{equation*}
\| \mathbf{P}f_{s}\| _{2}\lesssim \frac{1}{\varepsilon }\| (\mathbf{\ I}-%
\mathbf{P})f_{s}\| _{\nu }+\| \Gamma (f_{s},f_{s})\| _{2}+1.
\end{equation*}

We split 
\begin{equation*}
|\Gamma (f_{s},f_{s})|\ {\le }\ |\Gamma (\mathbf{P}f_{s},f_{s})|+|%
\Gamma ((\mathbf{I}-\mathbf{P})f_{s},f_{s})|\leq |\Gamma (\mathbf{P}f_{s},%
\mathbf{P}f_{s})|+|\Gamma (\mathbf{P}f_{s},(\mathbf{I}-\mathbf{P}%
)f_{s})|+|\Gamma ((\mathbf{I}-\mathbf{P})f_{s},f_{s})|.
\end{equation*}%
Since we expect ${\varepsilon }^{-1}\Vert (\mathbf{I}-\mathbf{P})f_{s}\Vert
_{\nu }\lesssim 1$, the last two parts of the nonlinear term are estimated
as 
\begin{equation*}
\|\Gamma (\mathbf{P}f_{s},(\mathbf{I}-\mathbf{P})f_{s})\|_{2}+{\Vert \Gamma
((\mathbf{I}-\mathbf{P})f_{s},f_{s})\Vert _{2}}\lesssim \lbrack \e^{-1}\Vert
(\mathbf{I}-\mathbf{P})_{s}\Vert _{2}][\e\Vert f_{s}\Vert _{\infty }].
\end{equation*}%
For the first main contribution, \ if we have 
\begin{equation}
\Vert \mathbf{P}f_{s}\Vert _{L^{6}}\lesssim 1,
\end{equation}%
then $\Vert \Gamma (\mathbf{P}f_{s},\mathbf{P}f_{s})\Vert _{2}\lesssim \Vert 
\mathbf{P}f_{s}\Vert _{3}\Vert \mathbf{P}f_{s}\Vert _{6}\lesssim 1.$ Thanks
to this $L^{6}$ bound, we can control the other two terms by establishing 
\begin{equation*}
\|f_{s}\|_{\infty }\lesssim \frac{1}{\sqrt{\varepsilon }}\Vert \mathbf{P}%
f_{s}\Vert _{L^{6}}+\frac{1}{\sqrt{\varepsilon }}[\e^{-1}\Vert (\mathbf{I}-%
\mathbf{P})_{s}\Vert _{2}]\lesssim \frac{1}{\sqrt{\varepsilon }}.
\end{equation*}

We now sketch the main idea for establishing such a crucial $L^{6}$ estimate
for $\mathbf{P}f$ as stated in Proposition \ref{linearl2}. For simplicity,
we consider a model problem of 
\begin{equation*}
v\cdot \nabla _{x}f=g\in L^{2},\text{ \ \ \ \ \ }\iint fdxdv=0,
\end{equation*}%
and look for estimate for $a(x)=\int_{\mathbf{R}^{3}}f\sqrt{\mu }dv.$ In 
\cite{EGKM}, we developed a quantitative method to estimate $L^{2}$ norm of $%
a(x),$ by carefully choosing a test function $\psi =(|v|^{2}-\beta
_{a})v\cdot \nabla _{x}\phi _{a}\sqrt{\mu }$ with $-\Delta _{x}\phi
_{a}=a(x) $, and the Neumann boundary condition $\frac{\partial \phi _{a}}{%
\partial n}=0.$ This process is similar in spirit to the energy estimates in
elliptic problems. The main contribution $\int -\Delta _{x}\phi
_{a}a=\|a\|_{L^{2}}^{2}$ results from the Green's identity of $\int v\cdot
\nabla _{x}f\psi dxdv,$ which leads to $L^{2}$ control of $a.$ The new
observation is that the key requirement for choosing $\phi _{a}$ is $\int
\int \nabla _{x}\phi _{a}$ $g<+\infty $, or 
\begin{equation*}
\nabla _{x}\phi _{a}\in L^{2}.
\end{equation*}%
By Sobolev embedding in 3D, it suffices to require $-\Delta \phi _{a}\in L^{%
\frac{6}{5}}$. Therefore, we may choose 
\begin{equation*}
-\Delta \phi _{a}=a^{5}(x)-|\O|^{-1}\int_{\O} a^{5},\text{ }
\end{equation*}%
to obtain a main contribution $\|a\|_{L^{6}}^{6}=\int -\Delta _{x}\phi
_{a}a. $ To close such a $L^{6}$ estimate, we need a $L^{6}$ bound of $(%
\mathbf{I}-\mathbf{P})f,$ which is controled by interpolation between a $%
L^{2}$ bound of $\varepsilon ^{-1}(\mathbf{I}-\mathbf{P})f$ \ (from energy
estimate) and $L^{\infty }$ bound for $\varepsilon ^{1/2}f.$ We also need to
control $L^{4/3}$ norm of $\nabla _{x}\phi _{a}$ at the boundary, which is
luckily bounded by $\|a\|^5_{L^{6}}$ exactly via the trace theorem. Such a $%
L^{6}$ estimate seems natural in terms of scalings of Sobolev spaces in 3D
and should lead to more applications in the future.

We remark that such a $L^{6}$ estimate is very different from the celebrated
Averaging Lemma for $v\cdot \nabla _{x}f=g\in L^{2}$, which states a gain of 
$L^{3}$ integrability ($H^{1/2}$ regularity) for \textit{any} velocity
average of $f$ in the whole space without boundary. Our $L^{6}$ estimates
are stronger than $L^{3}$, but only work with an additional assumption $(\mathbf{I}-\mathbf{P})f\in
L^{6}.$

\medskip

We remark that the convergence results provided by Theorem \ref{mainth} does
not give any indication on the rate of converge in $\e$ of the solution to
its limit. To discuss this we can be inspired by the Hilbert expansion.

Let us denote $F=\mu +\e\sqrt{\mu}g$ and $g_1=\sqrt{\mu}(\rho_s+u_s\cdot v+(%
\th _s+\Theta_w)(|v|^2-3)/2)$, 
\begin{equation}  \label{f2}
\begin{split}
g_2 \ := \ & \frac 1 2\sum_{i,j=1}^3 \mathcal{A}_{ij}[\pt_{x_i}u_{j,s}+\pt%
_{x_j}u_{i,s}]+\sum_{i=1}^3 \mathcal{B}_i \pt_{x_i}(\th _s+\Theta_w) \\
&-L^{-1}[\Gamma(f_1,f_1)] +\frac{|v|^2-3}{2} \th _2\sqrt{\mu},
\end{split}%
\end{equation}
where $\mathcal{A}_{ij}$ and $\mathcal{B}_i$ are given by 
\begin{equation}
\mathcal{A}_{ij}= L^{-1}\big( \sqrt{\mu}(v_iv_j-\frac{|v|^2}3\d_{i,j})\big)%
,\quad \mathcal{B}_i= L^{-1} \big(\sqrt{\mu} v_i(|v|^2-5)\big) ,  \notag
\end{equation}
and $\th _2=p_s-\fint p_s-(\th _s+\Theta_w)\rho_s$.

We define 
\begin{equation*}
R_{s}=\e^{-\frac{1}{2}}\{f_{s}+f_{w}-(g_{1}+\e g_{2})\},
\end{equation*}%
so that 
\begin{equation}
F_{s}=\mu +\e\sqrt{\mu }(g_{1}+\e g_{2}+\e^{\frac{1}{2}}R_{s}).
\label{expanR}
\end{equation}%
Then $R_{s}$ satisfies the equation 
\begin{equation}
v\cdot \nabla _{x}R_{s}+\e^{2}\frac{1}{\sqrt{\mu }}\Phi \cdot \nabla _{v}%
\Big[\sqrt{\mu }R_{s}\Big]+\e^{-1}LR_{s}=L_{1}R_{s}+\e^{1/2}\Gamma
(R_{s},R_{s})+\e^{1/2}\overline{A}_{s},  \label{equation_R_sR}
\end{equation}%
with the boundary condition 
\begin{equation}\label{bcRR}
R\ =\ {P}_{\g}R+\e\mathcal{Q}R+\e^{\frac{1}{2}}r,\ \ \ \text{on}\ \ \gamma
_{-}.
\end{equation}%
Here $\overline{A}_{s}$ is given by 
\begin{equation}
\overline{A}_{s}\ =-(\mathbf{I}-\mathbf{P})[v\cdot \nabla _{x}g_{2}]-2\Gamma
(g_{1},g_{2})  \label{Ast} \\
-\e\big\{\Phi \cdot \frac{1}{\sqrt{\mu }}\nabla _{v}\big[\sqrt{\mu }(g_{1}+%
\e g_{2})\big]-\Gamma (g_{2},g_{2})\big\}.
\end{equation}%
Equation (\ref{equation_R_sR}) is similar to the one for $f_{s}$ with the
extra factor $\sqrt{\e}$ in front of the non linear term $\Gamma
(R_{s},R_{s})$. In fact, using the same arguments employed to prove Theorems %
\ref{mainth} we can control the error between solutions of the
Navier-Stokes-Fourier approximation and the Boltzmann equation:

\begin{theorem}
\label{mainth2} 

If $\Phi = \Phi(x)\in H^{2} (\Omega) \cap C^{1}(\Omega)$, $\vartheta_{w} \in
H^{7/2}(\Omega)$ and 
\begin{equation}
\| {\t}_{w} \| _{H^{1+ }(\pt\O )}+\| \Phi \| _{L^{ \frac{3}{2}+}(\O )}\ll1,
\label{small_field1}
\end{equation}
then, for $0<\e\ll 1$, there exist a unique $R_{s}$ satisfying (\ref%
{equation_R_sR}) and the boundary condition (\ref{bcRR}). 

Moreover, 
\begin{equation}  \label{energy_steady}
\begin{split}
\| R_{s} \|_{2} + \e^{-1} \|(\mathbf{I} - \mathbf{P}) R_{s} \|_{\nu} \ll 1 ,
\ \ \e^{\frac 1 2} \| w R_{s} \|_{\infty} \ll 1,
\end{split}%
\end{equation}
where $w(v)= e^{\beta|v|^{2}}$ with $0< \beta \ll 1$.
\end{theorem}

Since the nonlinear interaction of $R_{s}$ now is much weaker with an extra
power of $\varepsilon ^{1/2},$ the proof of Theorem \ref{mainth2} follows
the same lines of the proof of Theorem \ref{mainth} and hence it will be
omitted.

We remark that $\sqrt{\varepsilon }R_{s}$ is of higher order in $L^{p}$ for $%
2\leq p<6$. On the other hand, $\sqrt{\varepsilon }R_{s}$ is small, but not
infinitesimal in $\e$ in $L^{\infty }$, so that it is possible that $%
\{F_s-(\mu+\e\sqrt{\mu} g_1)\}/\e=O(1)$ in $L^\infty$.

We note that the chosen power $\sqrt{\e}$ is forced from the fact that
higher powers of $\e$ make the boundary term too singular. In fact, if we
use $\a$ instead of $\frac 1 2$ in the definition of $R_s$, the boundary
condition for $R$ becomes 
\begin{equation}
\begin{split}
R \ = \ {P}_\g R +\e \mathcal{Q}R +\e^{1-\a}r, \ \ \ \text{on} \ \ \gamma_{-}
\label{bcRR0}
\end{split}%
\end{equation}
with 
\begin{equation}
r= \frac{1}{\sqrt{\mu}} \mathcal{P}^w_\gamma \big(\sqrt{\mu}[f_2-\f_{\e}]%
\big)-[f_2-\f_{\e}].  \label{r0}
\end{equation}
Since in the energy inequality we need to compensate a factor $\e^{-1}$ in
front of $\| \e^{1-\a}r\|_{L^2(\gamma_-)}^2$, the choice $\a=\frac 1 2$ is
the best we can do. 
In conclusion, the rate of convergence of the solution to its hydrodynamic
limit is at least $O(\sqrt{\e})$ in $L^2$.

More accurate estimates of the errors would require the truncation of the
expansion to higher order terms and boundary layer analysis, but there are
serious difficulties in performing such a program. Although we do not follow
this strategy, let us shortly indicate the main difficulties.

The usual approach is based on the representation of the solution by means
of an Hilbert-like expansion in the bulk, suitably corrected at the boundary
to satisfy the boundary conditions \cite{ELM1,ELM2,AEMN1,AEMN2}:%
\begin{equation}
F=\mu +\varepsilon \sqrt{\mu} \lbrack f_{1}+\varepsilon
f_{2}+\cdots+\varepsilon ^{k} f_{k+1}+\varepsilon f_{1}^{B}+\varepsilon ^{2}
f_{2}^{B}+\cdots+\varepsilon ^{k+1} f_{k+1}^{B}+\varepsilon ^{k}R]{.}
\label{layer}
\end{equation}%
{Here, the functions} $f_{k}$ are corrections in the bulk, while $f_{k}^{B}$
are boundary layer corrections which solve Milne-like problems, and $%
R=R^{\varepsilon }$ denotes the remainder. 
The corrections at the boundary are computed by means of a boundary layer
expansion which, in a general domain, presents some issues hard to deal
with. The usual strategy is to solve the $k$-th term of the boundary layer
expansion by looking at it in terms of the rescaled distance from the
boundary (see e.g. \cite{So1}). Using of such a variable, the problem looks
like a half-space linear problem (Milne problem) \cite{BCN} with a
correction due to the geometry which {can be interpreted as} an external
field of the order of the Knudsen number. {The field{,} due to the $k$-th
term of the boundary layer expansion{,} is usually included as source term
in the equation for the $(k+1)$-th term \cite{So1}}, but the lack of
regularity makes this hard to control.

This strategy has been used in \cite{BLP} in the much simpler case of the
neutron transport equations, but recently in \cite{GW} it has been proved
that the result in \cite{BLP} breaks down exactly because of the lack of
regularity. Therefore, the geometric field, even if of small size, has to be
included in the equation for the $k$-term of the expansion, as in {\cite%
{ELM3,AEMN1}} for the case of the gravity and \cite{GW} for the geometrical
field in the neutron transport equation in a disk: 
\begin{equation*}
F=\mu +\varepsilon\sqrt{\mu} \lbrack f_{1}+\varepsilon f_{2}+ \cdots +
\varepsilon ^{k} f_{k+1}+\varepsilon f_{1,\varepsilon }^{B}+\varepsilon ^{2}
f_{2,\varepsilon }^{B}+ \cdots +\varepsilon ^{k+1} f_{k+1,\varepsilon
}^{B}+\varepsilon ^{k}R],
\end{equation*}%
where $f_{k,\varepsilon }^{B}$ depends on $\varepsilon$. Unfortunately, this
strategy fails even for a general 2D domain because the analysis of the
derivatives' singularities presents severe difficulties (see {\cite%
{GKTT,GKTT2}} for the analysis at $\e\approx 1$). The only significant
exception is the paper \cite{GW} where this expansion is completely proved
in the case of the Boltzmann equation in a disk.

\medskip

\vspace{8pt}

Next we investigate the stability properties of the stationary solution. To
discuss this, we study the unsteady problem. The solution to (\ref{basic2})
is written as 
\begin{equation}  \label{time_pert}
F(t) = \mu + \e\sqrt{\mu}f,\quad f=f_{w} + f_s +\tilde f.
\end{equation}
The aim is to show that $\tilde f\to\tilde f_1=\sqrt{\mu}(\tilde \rho
+\tilde u\cdot v+\tilde \vartheta(\frac{|v|^2-3}2))$, with $\nabla_x[\tilde
\rho +\tilde \vartheta]=0$ and $(\tilde{u}, \tilde{\vartheta}, \tilde{p})$
satisfying 
\begin{equation}
\begin{split}  \label{tilde_u_theta}
\partial_{t} \tilde{u} + \tilde{u} \cdot \nabla_{x} \tilde{u} + \tilde{u}
\cdot \nabla_{x}u_{s} + u_{s} \cdot\nabla_{x} \tilde{u} + \nabla_{x} \tilde{p%
} = \mathfrak{v}\Delta \tilde{u}, \ \ \nabla_{x} \cdot \tilde{u} =0 \ \ \ &%
\text{in} \ \Omega, \\
\partial_{t} \tilde{\vartheta} + \tilde{u} \cdot \nabla_{x} \tilde{\vartheta}
+ \tilde{u} \cdot \nabla_{x} \vartheta_{s} + u_{s} \cdot \nabla_{x} \tilde{%
\vartheta} = \kappa\Delta \tilde{\vartheta} \ \ \ &\text{in} \ \Omega, \\
\tilde{u} =0, \ \ \tilde{\vartheta}=0 \ \ \ &\text{on} \ \partial\Omega.
\end{split}%
\end{equation}
The equation of $\tilde{f}$ is given by 
\begin{equation}  \label{tild_R}
\partial_{t} \tilde{f} + \e^{-1} v\cdot\nabla_{x} \tilde{f} + \e \Phi \cdot
\nabla_{v} \tilde{f} + \e^{-2} L \tilde{f} = \ \e^{-1} L_{ {f_w}+ f_{s}} 
\tilde{f} + \e^{-1}\Gamma(\tilde{f}, \tilde{f}) + \e \frac{\Phi \cdot v}{2}%
\tilde{f}.
\end{equation}
Here we have used the notation $L_{\phi} \psi := - [ \Gamma( \phi, \psi) +
\Gamma( \psi, \phi) ]$. Note that, due to symmetry, for all $\psi_{1},
\psi_{2} \in L^{2}$, 
\begin{equation}  \label{sym_L1}
(L_{ \phi }\psi_{1}, \psi_{2} ) =(L_{ \phi }\psi_{1}, (\mathbf{I} - \mathbf{P%
}) \psi_{2}).
\end{equation}
The boundary condition of $\tilde{f}$ is given by 
\begin{equation}  \label{bdry_tildeR}
\begin{split}
\tilde{f} |_{\gamma_{-} } = P_{\gamma} \tilde{f} +\e \mathcal{Q} \tilde{f}.
\end{split}%
\end{equation}
We define the energy and the dissipation as 
\begin{eqnarray}
\mathcal{E}_{\lambda }[\tilde{f}](t) &:=&\sup_{0\leq s\leq t}\| e^{\lambda s}%
\tilde{f}(s)\| _{2}^{2}+\sup_{0\leq s\leq t}\| e^{\lambda s}\partial _{t}%
\tilde{f} (s)\| _{2}^{2},  \label{cale} \\
\mathcal{D}_{\lambda }[\tilde{f}](t) &:=& \frac{1}{ \e^{2}}\int_{0}^{t}\|
e^{\lambda s}(\mathbf{I}-\mathbf{P}) \tilde{f}\| _{\nu }^{2}+\frac{1}{ \e^{2}%
}\int_{0}^{t}\| e^{\lambda s} (\mathbf{I}-\mathbf{P}) \partial _{t}\tilde{f}%
\| _{\nu }^{2}  \notag \\
&&+ \frac 1 {\e}\int_{0}^{t} | e^{\lambda s} (1-P_\gamma)\tilde{f} |
_{2,\gamma }^{2}+\frac 1 {\e}\int_{0}^{t} | e^{\lambda s}(1-P_\gamma) \tilde{%
f}_t | _{2,\gamma }^{2}  \label{cald} \\
&&+\int_{0}^{t}| e^{\lambda s} \tilde{f} | _{2,\gamma }^{2}+\int_{0}^{t}|
e^{\lambda s} \tilde{f}_t | _{2,\gamma }^{2} + \int_{0}^{t}\| e^{\lambda s}%
\mathbf{P}\tilde{f}\| _{2}^{2}+\int_{0}^{t}\| e^{\lambda s}\mathbf{P}%
\partial _{t}\tilde{f}\| _{2}^{2} .  \notag
\end{eqnarray}

\begin{theorem}
\label{energy_nonlinear} We assume the same hypotheses of Theorem \ref%
{mainth}. Suppose $F_{0}= F_{s} + \e \sqrt{\mu} \tilde{f}_0 \geq 0$, 
\begin{equation}  \label{initial_assupm_11}
\mathcal{E}_0[\tilde{f}] (0) + \e^{3/2} \| w \partial_{t}\tilde{f}_{0}
\|_{\infty} + \big\| \int_{\mathbb{R}^{3}} | \tilde{f}_{0}(x,v)| \langle
v\rangle^{2} \sqrt{\mu} \mathrm{d} v\big\|_{L^{6}(\Omega)} \ll 1,
\end{equation}
where $w(v)= e^{\beta|v|^{2}}$ with $0< \beta \ll 1$.

Then there exists a unique global solution $F\geq 0$ given by (\ref%
{time_pert}) with $\tilde{f}$ solving (\ref{tild_R}) and the boundary
condition (\ref{bc0}). 

Moreover, for some $0 < \lambda \ll 1$, 
\begin{equation}  \label{energy_estimate}
\begin{split}
\mathcal{E}_{\lambda}[\tilde{f}](\infty) + \mathcal{D}_{\lambda}[\tilde{f}%
](\infty) + \sup_{0 \leq t\leq \infty} \e^{1/2} \| w \tilde{f}(t)
\|_{\infty} + \sup_{0 \leq t\leq \infty} \e^{3/2} \| w \partial_{t}\tilde{f}%
(t) \|_{\infty} \ll 1 ,
\end{split}%
\end{equation}
Finally, as $\e\to 0$, $\tilde{f}$ converges weakly to $f_{1}=[\tilde{u}%
\cdot v+\tilde{\t}(|v|^2-5)/2]\sqrt{\mu}$ with $(\tilde{p},\tilde{u},\tilde{%
\t})$ is a solution to (\ref{tilde_u_theta}).
\end{theorem}

\begin{remark}
The initial assumption (\ref{initial_assupm_11}) in particular requires $\|%
\tilde{f}_t(0)\|_{L^2_{x,v}}\ll 1$. This is a very sharp condition of $%
\tilde{f}_0$, because, from equation (\ref{tild_R}) 
\begin{equation}  \label{tild_R1}
\partial_{t} \tilde{f}(0) =- \e^{-1} v\cdot\nabla_{x} \tilde{f}_0 - \e \Phi
\cdot \nabla_{v} \tilde{f}_0 - \e^{-2} L \tilde{f}_0 + \ \e^{-1} L_{{f_w}+
f_{s}} \tilde{f}_0 + \e^{-1}\Gamma(\tilde{f}_0, \tilde{f}_0) + \e \frac{\Phi
\cdot v}{2}\tilde{f}_0.
\end{equation}
To compensate the diverging factors one has to choose $\tilde f_0$ properly.
An example of such a choice is the following: let $\tilde f_{1,0}= (\tilde
\rho_0+\tilde u_0\cdot v+\tilde \th _0(|v|^2-3)/2)\sqrt{\mu}$ and assume
that $\nabla_x\cdot \tilde u_0=0$, $\nabla_x(\tilde\rho_0+\tilde\th _0)=0$,
so that $\P (v\cdot \nabla_x f_{1,0})=0$ and set 
\begin{equation*}
\tilde f_0=\tilde f_{0,1}-\e L^{-1}[v\cdot \nabla_x \tilde f_{1,0}- L_{ {f_w}%
+ f_{s}} \tilde{f}_0 -\Gamma(\tilde{f}_0, \tilde{f}_0)]+\e^2 h,
\end{equation*}
for some $L^2_{x,v}$ function $h$. Then, clearly, the diverging factors are
compensated and $f_t(0)$ is bounded in $L^2_{x,v}$. Thus, the smallness
condition is fulfilled by assuming $\tilde \rho$, $\tilde u$, $\tilde \th $
and $L h$ sufficiently small. Note that the initial data for the
hydrodynamic quantities are small but not depending on $\e$. Thus our result
implies the exponential stability of the constructed solution to the INSF
system.

We remark also that such an asymptotical stability implies non-negativity of
steady solution $F_{s}$ (Section 3.7). 
\end{remark}

We start with the energy estimates, as the steady case, to get 
\begin{equation*}
\|\tilde{f}(t)\|_{2}^{2}+\frac{1}{\varepsilon ^{2}}\int_{0}^{t}\Vert (%
\mathbf{I}-\mathbf{P})\tilde{f}\Vert _{\nu }^{2}\lesssim \int_{0}^{t}\Vert
\Gamma (\tilde{f},\tilde{f})\Vert _{2}^{2}+1.
\end{equation*}%
The missing $\mathbf{P}\tilde{f}$ can be estimated by the coercivity
estimates in \cite{EGKM}, with carefully chosen proper test functions in the
weak formulation together with the local conservation laws (Lemma \ref{dabc}%
): 
\begin{equation*}
\int_{0}^{t}\Vert \mathbf{P}\tilde{f}\Vert _{2}^{2}\lesssim \frac{1}{%
\varepsilon ^{2}}\int_{0}^{t}\Vert (\mathbf{I}-\mathbf{P})\tilde{f}\Vert
_{\nu }^{2}+\int_{0}^{t}\Vert \Gamma (\tilde{f},\tilde{f})\Vert _{2}^{2}+1.
\end{equation*}
We estimate the main nonlinear contribution as 
\begin{equation*}
\Vert \Gamma (\mathbf{P}\tilde{f},\mathbf{P}\tilde{f})\Vert
_{L_{t,x,v}^{2}}\lesssim \Vert \mathbf{P}\tilde{f}\Vert _{L_{t}^{\infty
}L_{x,v}^{6}}\cdot \Vert \mathbf{P}\tilde{f}\Vert _{L_{t}^{2}L_{x,v}^{3}}.
\end{equation*}%
The most important ingredient is to control $\Vert \mathbf{P}\tilde{f}\Vert
_{L_{t}^{\infty }L_{x,v}^{6}}$ as in the steady case, in the presence of the
term $\varepsilon f_{t}$ (Proposition \ref{p6time}). For further control of $%
\varepsilon f_{t}$, we repeat the energy estimate for $f_{t}$ and estimate
the nonlinear term  
\begin{equation*}
\Vert \Gamma (\mathbf{P}\tilde{f},\mathbf{P}\tilde{f}_{t})\Vert
_{L_{t,x,v}^{2}}\lesssim \Vert \mathbf{P}\tilde{f}\Vert _{L_{t}^{\infty
}L_{x,v}^{6}}\Vert \mathbf{P}\tilde{f}_{t}\Vert _{L_{t}^{2}L_{x,v}^{3}}
\end{equation*}%
with the same norm $\Vert \mathbf{P}\tilde{f}\Vert _{L_{t}^{\infty
}L_{x,v}^{6}}.$

To close the estimate, it suffices to control both $\Vert \mathbf{P}\tilde{f}%
\Vert _{L_{t}^{2}L_{x,v}^{3}}$ and $\Vert \mathbf{P}\tilde{f}_{t}\Vert
_{L_{t}^{2}L_{x,v}^{3}}$ by:%
\begin{eqnarray*}
\Vert \mathbf{P}\tilde{f}\Vert _{L_{t}^{2}L_{x,v}^{3}} &\lesssim &\ \Vert
\Gamma (\tilde{f},\tilde{f})\Vert _{L_{t,x,v}^{2}}{+\ \Vert \Gamma (f_{s},%
\tilde{f})\Vert _{L_{t,x,v}^{2}}}+1, \\
\Vert \mathbf{P}\tilde{f}_{t}\Vert _{L_{t}^{2}L_{x,v}^{3}} &\lesssim &\Vert
\Gamma (\tilde{f},\tilde{f}_{t})\Vert _{L_{t,x,v}^{2}}+\Vert \Gamma (\tilde{f%
}_{t},\tilde{f})\Vert _{L_{t,x,v}^{2}}+1.
\end{eqnarray*}%
We now illustrate the estimate for $\Vert \mathbf{P}\tilde{f}\Vert
_{L_{t}^{2}L_{x,v}^{3}}.$ In the absence of the external field and the
boundary, $\varepsilon ^{2}\Phi \equiv 0$ and $\Omega ={\mathbb{R}}^{3}$,
such gain of integrability is well-known from the Averaging Lemma \cite{GLPS, GP}
and the sharp Sobolev embedding $H^{1/2}\subset L^{3}$ (See also the case
for a convex bounded domain with $\varepsilon ^{2}\Phi \equiv 0$ in \cite%
{GLPS}). We need to extend this estimate properly to case of the bounded
domain $\Omega $ in the presence of the external field \ $\varepsilon
^{2}\Phi \neq 0$. We first consider an extension of $\tilde{f}$ to the whole
space, denoted by $\bar{f}$, such that $\bar{f}\in L^{2}$ and 
\begin{equation*}
\varepsilon \bar{f}_{t}+v\cdot \nabla _{x}\bar{f}+{\e^{2}\Phi \cdot \nabla
_{v}\bar{f}}\in L^{2}.
\end{equation*}%
This would require that $\bar{f}$ is continuous along all exterior
trajectories, matching with given incoming and outgoing data of $f$ on the
boundary. For a general domain $\Omega $ with $\varepsilon ^{2}\Phi \neq 0$,
the exterior trajectories can be complicated and they can connect the
outgoing set $\gamma _{+}$ and incoming set $\gamma _{-}$, arbitrarily near
the grazing set $\gamma _{0}$. It is not clear that an extension $\bar{f}$
would satisfy both $\bar{f}\in L^{2}$ and $\varepsilon \bar{f}_{t}+v\cdot
\nabla _{x}\bar{f}+{\e^{2}\Phi \cdot \nabla _{v}\bar{f}}\in L^{2}$, due to a
possible discontinuity of $\bar{f}$ \cite{Kim}.

We circumvent this difficulty via an extension lemma, Lemma \ref%
{extension_dyn}, which asserts that, for the function cutoff from the
grazing set $\gamma _{0}$, 
\begin{equation}
f_{\delta }\ {\sim }\ \mathbf{1}_{\{|v|<\frac{1}{\delta }\}}\mathbf{1}%
_{\{|n(x)\cdot v|>\delta \ \text{or}\ \text{dist}(x,\partial \Omega )>\delta
\}}\tilde{f},\ \ \ \text{for}\ \delta \ll 1,  \label{rextend}
\end{equation}%
such an extension $\overline{f_{\delta }}$ is indeed possible. Here, $\text{%
dist}(x,\partial \Omega ):=\inf_{y\in \partial \Omega }|x-y|$. Luckily, $%
\mathbf{P}f_{\delta }\thicksim \mathbf{P}\tilde{f}$ thanks to the estimate ${%
\varepsilon }^{-1}\Vert (\mathbf{I}-\mathbf{P})\tilde{f}\Vert _{2}\sim 1$.
In the presence of external field $\varepsilon ^{2}\Phi \neq 0$, we modify
the proof of averaging lemma for $\mathbf{P}\overline{f_{\delta }}$
(Proposition \ref{flfs}) via a careful splitting to show its effect is small
for our purpose. 
Similarly to the steady case, as in \cite{EGKM}, we may bootstrap such $L^{6}
$ estimates to an improved $L^{\infty }$ estimate for $\Omega \in \mathbf{R}%
^{3}$ 
\begin{equation*}
\Vert \tilde{f}\Vert _{L^{\infty }}\lesssim \frac{1}{\sqrt{\varepsilon }}%
\Vert \mathbf{P}\tilde{f}\Vert _{L_{t}^{\infty }L_{x,v}^{6}}+\frac{1}{%
\varepsilon ^{3/2}}\Vert (\mathbf{I}-\mathbf{P})\tilde{f}\Vert
_{L_{t}^{\infty }L_{x,v}^{2}}+1\lesssim \frac{1}{\sqrt{\varepsilon }}.
\end{equation*}%

\bigskip

As for the stationary case, Theorem \ref{energy_nonlinear} does not provide
the rate of convergence. Proceeding as in the steady case we define: 
\begin{equation*}
\tilde f_2 \ := \frac 1 2\sum_{i,j=1}^3 \mathcal{A}_{ij}[\pt_{x_i}\tilde
u_{j,s}+\pt_{x_j}\tilde u_{i,s}]+\sum_{i=1}^3 \mathcal{B}_i \pt_{x_i}\tilde 
\th \newline
-L^{-1}[\Gamma(\tilde f_1,\tilde f_1)] +\frac{|v|^2-3}{2} \tilde \th _2\sqrt{%
\mu},
\end{equation*}
with $\tilde\th _2=\tilde p-\fint \tilde p- \tilde\th \rho$ and 
\begin{equation}
\tilde R= \e^{-\frac 1 2}[F-\mu-\e\tilde f_1-\e^2\tilde f_2].
\end{equation}
Then $\tilde R$ has to solve 
\begin{equation}  \label{tild_RR}
\begin{split}
& \partial_{t} \tilde{R} + \e^{-1} v\cdot\nabla_{x} \tilde{R} + \e \Phi
\cdot \nabla_{v} \tilde{R} + \e^{-2} L \tilde{R} \\
= & \ \e^{-1} L_{1} \tilde{R} + \e^{-1} L_{\e^{1/2} R_{s}} \tilde{R} + \e%
^{-1}L_{R_{s}} (\tilde{f}_{1}+ \e \tilde{f}_{2}) + \e^{-1/2} \Gamma(\tilde{R}%
, \tilde{R}) + \e \frac{\Phi \cdot v}{2}\tilde{R}+ \e^{-1/2} \tilde{A},
\end{split}%
\end{equation}
where 
\begin{multline}  \label{Astaa}
\tilde A\ = -(\mathbf{I}-\mathbf{P}) [v\cdot \nabla_x (f_2 +\tilde
f_2)]-2\Gamma(g_1+\tilde f_1,g_2+ \tilde f_2) \\
-\e\big\{\pt_t \tilde f_2+\Phi\cdot\frac 1 {\sqrt{\mu}}\nabla_v \big[\sqrt{%
\mu}(g_1 +\tilde f_1+\e (g_2+\tilde f_2)) \big]-\Gamma(g_2+\tilde
f_2,g_2+\tilde f_2)\big\}-\overline{A}_s,
\end{multline}
and has to satisfy the boundary condition 
\begin{equation}  \label{bdry_tildeRR}
\begin{split}
\tilde{R} |_{\gamma_{-} } = P_{\gamma} \tilde{R} +\e \mathcal{Q} \tilde{R} + %
\e^{1/2}\tilde{r},
\end{split}%
\end{equation}
where $\tilde{r}:= \e^{-1} [\mu^{- \frac{1}{2}} \mathcal{P}^w_{\gamma} (%
\tilde{f}_{1} \sqrt{\mu}) - \tilde{f}_{1} ] + [ \mu^{- \frac{1}{2}} \mathcal{%
P}_{\gamma}^{w} (\tilde{f}_{2} \sqrt{\mu}) - \tilde{f}_{2} ]$.

We establish the error estimates between the solutions of the
Navier-Stokes-Fourier approximation and the Boltzmann equation as follows:

\begin{theorem}
\label{energy_nonlinear1} 
Suppose 
\begin{equation}  \label{initial_assupm_111}
\mathcal{E}_0[\tilde{R}] (0) + \e^{3/2} \| w \partial_{t}\tilde{R}_{0}
\|_{\infty} + \big\| \int_{\mathbb{R}^{3}} | \tilde{R}_{0}(x,v)| \langle
v\rangle^{2} \sqrt{\mu} \mathrm{d} v\big\|_{L^{3}(\Omega)} + \e^{\frac 1 2}
\| w \tilde{R}_{0} \|_{\infty} \ll 1,
\end{equation}
where $w(v)= e^{\beta|v|^{2}}$ with $0< \beta \ll 1$.

Then for $\e$ sufficiently small there exists a unique global solution 
$\tilde{R}$ solving (\ref{tild_RR}) and the boundary condition (\ref%
{bdry_tildeRR}). 

Moreover, for some $0 < \lambda \ll 1$, 
\begin{equation}  \label{energy_estimate1}
\begin{split}
\mathcal{E}_{\lambda}[\tilde{R}](\infty) + \mathcal{D}_{\lambda}[\tilde{R}%
](\infty) + \sup_{0 \leq t\leq \infty} \e^{3/2} \| w \partial_{t}\tilde{R}%
(t) \|_{\infty} +\sup_{0 \leq t\leq \infty} \e^{\frac 1 2} \| w \tilde{R}(t)
\|_{\infty} \ll 1 ,
\end{split}%
\end{equation}
where $\mathcal{E}_\lambda[\tilde{R}] (T)$ and $\mathcal{D}_\lambda[\tilde{R}%
] (T)$ are defined in (\ref{cale}) and (\ref{cald}) with $\tilde f$ replaced
by $\tilde{R}$. 
\end{theorem}

Since the nonlinearity for $\tilde{R}$ and $\tilde{R}_{t}$ is weaker with an
extra power of $\varepsilon ^{1/2},$ the proof of this theorem also follows
along the same lines of Theorem \ref{energy_nonlinear} and it will omitted.

\begin{theorem}
\label{remb}{\ Assume that $(u(t),\th(t))$ is a solution to the INSF initial boundary value problem for $t\in [0,T]$, $T>0$, such that
\begin{equation}
\sup_{0 \leq t \leq T}\| {u}(t) \|_{H^{4}(\Omega)} +\sup_{0 \leq t \leq T}
\| {\th }(t) \|_{H^{4}(\Omega)} \ {<} \ \infty.  \notag
\end{equation}
Then there is $\e(T)>0$ such that for $\e\le \e(T)$ there exists a
unique solution $F(t) = \mu + \e [ {f}_{1} + \e
{f}_{2} + \e^{1/2}{R}(t)] \sqrt{\mu}\geq 0$ on $t \in [0,T]$ such that 
\begin{equation}
\begin{split}
\mathcal{E}_{0}[R] (T) + \mathcal{D}_{0}[R] (T)+ \sup_{0 \leq t\leq T} \e%
^{3/2} \| w {R}_{t}(t) \|_{\infty}+ \sup_{0 \leq t\leq T} \e^{\frac 1 2} \|
w {R}(t) \|_{\infty} \lesssim 1 ,
\end{split}
\notag
\end{equation}
where $\mathcal{E}_0[R] (T)$ and $\mathcal{D}_0[R] (T)$ are defined in (\ref%
{cale}) and (\ref{cald}) with $\tilde f$ replaced by $R$. Moreover, $f_1$
and $f_2$ are defined in (\ref{f1f2}). }
\end{theorem}

The proof of this sharp local-in-time validity theorem is given in Section
3.9. We note that the interval of the validity is the same as the life-span
of the classical solutions to the NSF system. In particular, if $\O\subset \R^2$,
then $T$ can be arbitrary.

\section{Steady Problems}

\subsection{Preliminary and the linear theorem}

{Assume $\partial\Omega$ is $C^{3}$. Then for any $x_{0} \in \partial\Omega$%
, there exists $0<r_{0}, r_{1}\ll1$ and $C^{3}$ function $\eta:
\{x_{\parallel} = (x_{\parallel,1}, x_{\parallel,2}) \in \mathbb{R}^{2}
:|x_{\parallel}| < r_{1}\} \rightarrow \partial\Omega \cap B(x_{0}, r_{0})$
such that if $x \in \partial\Omega \cap B(x_{0}, r_{0})$ then there exists a
unique $x_{\parallel}\in \mathbb{R}^{2}$ with $|x_{\parallel}|< r_{1}$ which
satisfies $x= \eta(x_{\parallel})$. Here, we have used the notation $%
B(x_{0}, r_{0}) := \{x \in \mathbb{R}^{3}: |x-x_{0}| < r_{0}\}$. Without
loss of generality we assume that $|\partial_{x_{\parallel,i}}
\eta(x_{\parallel})| \neq 0$ for $i=1,2$. }

Assume $\text{dist}(x, \partial\Omega)\ll1$ and $x_{0} \in \partial\Omega$
such that $\text{dist}(x, x_{0}) = \text{dist} (x, \partial\Omega)$. Then
there exists $\eta$ which is a parametrization of $\partial\Omega$ around $%
x_{0}$. Clearly 
\begin{equation}  \label{mini_eta}
\nabla_{x_{\parallel}}| \eta(x_{\parallel})
-x|^{2}=(\partial_{x_{\parallel,1}} | \eta(x_{\parallel}) -x |^{2},
\partial_{x_{\parallel,2}} | \eta(x_{\parallel}) -x |^{2} )=0, \ \ \ \ \text{%
for some} \ \ x_{\parallel}.
\end{equation}
On the other hand, if $|\eta(x_{\parallel}) -x|\ll 1$, 
\begin{equation*}
\partial^{2}_{x_{\parallel,i}} | \eta(x_{\parallel}) - x|^{2} =
\partial_{x_{\parallel,i}} \big[ 2 \partial_{i} \eta(x_{\parallel}) \cdot(
\eta(x_{\parallel}) -x)\big] = O(|\eta(x_{\parallel} ) -x|) +2 |\partial_{i}
\eta(x_{\parallel})|^{2} \neq 0.
\end{equation*}
Then, by the implicit function theorem, there exists a unique $%
x_{\parallel}(x) \in C^{2}$ satisfying (\ref{mini_eta}). Moreover, 
\begin{eqnarray*}
\left[%
\begin{array}{cc}
\partial_{x_{i}} x_{\parallel,1} &  \\ 
\partial_{x_{i}} x_{\parallel,2} & 
\end{array}%
\right] = \left[%
\begin{array}{cc}
|\partial_{1} \eta |^{2} + \partial_{1}^{2} \eta \cdot ( \eta - x) & 
\partial_{1} \eta \cdot \partial_{2}\eta + \partial_{1}\partial_{2} \eta
\cdot ( \eta - x) \\ 
\partial_{1} \eta \cdot \partial_{2} \eta + \partial_{1} \partial_{2} \eta
\cdot ( \eta - x) & | \partial_{2} \eta| + \partial_{2}^{2} \eta \cdot (
\eta - x)%
\end{array}%
\right] ^{-1} \left[%
\begin{array}{cc}
-\partial_{1} \eta_{i} &  \\ 
-\partial_{2} \eta_{i} & 
\end{array}%
\right] ,
\end{eqnarray*}
where $\eta=\eta (x_{\parallel})$. Then we define $x_{\perp} \in C^{2}$ for $%
\text{dist}(x, \partial\Omega) \ll1$, 
\begin{equation}  \label{xperp}
x_{\perp} (x) := [x- \eta(x_{\parallel} (x))] \cdot n(x_{\parallel}(x)).
\end{equation}
Note that $\text{dist}(x, \partial\Omega) =|x_{\perp}(x)|$ if $\text{dist}%
(x, \partial\Omega) \ll 1$.

By the compactness of $\partial\Omega$, we conclude that if $\text{dist}%
(x,\partial\Omega)< 4r$ for some $0< r\ll_{\Omega} 1$ then there exists $%
(x_{\parallel}(x), x_{\perp}(x)) \in C^{2}$ such that $x=
\eta(x_{\parallel}(x)) + x_{\perp}(x) n(x_{\parallel}(x))$.

Finally we define the $C^{2}$ function $\xi : \mathbb{R}^{3} \rightarrow 
\mathbb{R}$ as 
\begin{equation}  \label{xi_def}
\xi(x) := x_{\perp}(x) \chi( \frac{|\text{dist}(x,\Omega)|^{2}}{4r^{2}} ) +r%
\big[ 1- \chi( \frac{|\text{dist}(x,\Omega)|^{2}}{r^{2}} )\big],
\end{equation}
where 
\begin{equation}  \label{chi}
\chi \in C^{\infty}_{c}(\mathbb{R}) \ \text{such that} \ 0 \leq \chi \leq 1,
\ \chi^{\prime} (x) \geq -4 \times \mathbf{1}_{\frac{1}{2} \leq |x| \leq 1}
\ \text{and} \ \chi(x)= 
\begin{cases}
1 & \text{if } |x| \leq \frac{1}{2}, \\ 
0 & \text{if } |x| \geq 1 .%
\end{cases}%
\end{equation}
Then $\Omega = \{x \in \mathbb{R}^{3}: \xi(x) <0\}$. If $|\xi(x)| \ll 1$
then $\xi(x)= x_{\perp}(x)$.

Moreover $n(x)\equiv \frac{\nabla \xi(x)}{|\nabla \xi(x)|}$ at the boundary $%
x \in\partial\Omega$. From now we define 
\begin{equation}  \label{n_inte}
n(x) \ := \ {\nabla \xi(x)} /{|\nabla \xi(x)|} \ \ \ \text{for} \ \ x \in 
\mathbb{R}^{3}.
\end{equation}

\vspace{8pt}

We use this new coordinate (\ref{xperp}) to extend $\Phi$ on the whole
space, and denote this extension by $\bar{\Phi}$, with $\| \bar{\Phi }
\|_{\infty} \leq \| {\Phi } \|_{\infty}$: For $0< \delta \ll1$, 
\begin{eqnarray*}
\bar{\Phi}(x) \ := \ \Phi(x) \mathbf{1}_{x \in \bar{\Omega}} + \Phi
(\eta(x_{\parallel} (x))) \chi( \frac{|\xi(x)|}{\delta} )\mathbf{1}_{x \in 
\mathbb{R}^{3} \backslash \bar{\Omega}}.
\end{eqnarray*}
Therefore without loss of generality we assume that $\Phi$ is defined on the
whole space $\mathbb{R}^{3}$.

\begin{definition}
Assume $\Phi= \Phi(x) \in C^{1}$. Consider the steady linear transport
equation 
\begin{equation}  \label{st_linear}
v\cdot \nabla_{x} f + \e^{2} \Phi \cdot \nabla_{x} f = g.
\end{equation}
The equations of the characteristics for (\ref{st_linear}) are 
\begin{equation}
\dot X=V, \quad \dot V=\e^2\Phi(X), \quad X(t;t;x,v)=x, \quad V(t;t;x,v)=v.
\label{char}
\end{equation}
{If $X(\tau;t,x,v) \in \Omega$ for all $\tau$ between $s$ and $t$, then} 
\begin{equation}  \label{intfor}
\begin{split}
& X(s;t;x,v)=x+v (s-t) +\e^2\int_t^s\int_t^\tau\Phi(X(\tau^{\prime};t;x,v))%
\mathrm{d} \tau^{\prime}\mathrm{d} \tau, \\
& V(s;t;x,v)= v+\e^2\int_t^s \Phi(\tau;s;x,v)) \mathrm{d} \tau .
\end{split}%
\end{equation}
Note that the ODE (\ref{char}) is autonomous since $\Phi$ is
time-independent.

Define 
\begin{equation}  \label{def_tb}
\begin{split}
t_{\mathbf{b}}(x,v)&:=\inf \{t\geq 0:X(-t;0;x,v)\notin \Omega \}, \\
x_{\mathbf{b}}(x,v)&:=X(-t_{\mathbf{b}}(x,v);0;x,v,0), \ \ \ v_{\mathbf{b}%
}(x,v):=V(-t_{\mathbf{b}}(x,v);0;x,v),
\end{split}%
\end{equation}
and 
\begin{equation}  \label{def_tf}
\begin{split}
t_{\mathbf{f}}(x,v)&:=\inf \{t\geq 0:X( t;0;x,v)\notin \Omega \}, \\
x_{\mathbf{f}}(x,v)&:=X( t_{\mathbf{f}}(x,v);0;x,v,0), \ \ \ v_{\mathbf{f}%
}(x,v) :=V( t_{\mathbf{f}}(x,v);0;x,v).
\end{split}%
\end{equation}
Clearly $(x_{\mathbf{b}}(x,v),v_{\mathbf{b}}(x,v))\in \g_-$ and $(x_{\mathbf{%
f}}(x,v),v_{\mathbf{f}}(x,v))\in \g_+$.
\end{definition}

\begin{lemma}
{\label{bdry_int}}For any open subset $\Omega\subset \mathbb{R}^{3}$, $B
\subset \partial\Omega$, and $f\in L^1(\Omega\times \mathbb{R}^3)$, 
\begin{eqnarray}
&&\iint_{\Omega \times \mathbb{R}^{3}} |f(x,v)| \mathbf{1}_{x_{\mathbf{b}%
}(x,v) \in B } \mathbf{1}_{t_{\mathbf{b}}(x,v) \leq \frac{1}{m} \ln \frac{1}{%
\e}} \mathrm{d} x \mathrm{d} v  \label{bdry_int2} \\
&=& \int_{B}\int_{n(y) \cdot u <0 } \int_{0}^{ \min\{ t_{\mathbf{f}}(y,u), 
\frac{1}{m} \ln \frac{1}{\e} \} } |f(X( s;0,y,u), V( s;0,y,u))|  \notag \\
&& \ \ \ \ \ \ \ \ \ \ \ \ \ \ \ \ \ \ \ \ \ \ \ \ \ \ \ \ \ \ \ \ \ \ \ \ \
\ \ \ \times \{ |n(y) \cdot u| + O(\e)(1+|u|) s \} \mathrm{d} s \mathrm{d} u 
\mathrm{d} S_{y} ,  \notag
\end{eqnarray}
and 
\begin{eqnarray}
&&\iint_{\Omega \times \mathbb{R}^{3}} |f(x,v)| \mathbf{1}_{x_{\mathbf{f}%
}(x,v) \in B } \mathbf{1}_{t_{\mathbf{f}}(x,v) \leq \frac{1}{m} \ln \frac{1}{%
\e}} \mathrm{d} x \mathrm{d} v  \label{bdry_int1} \\
&=& \int_{ B}\int_{n(y) \cdot u >0 } \int^{0}_{ -\min\{ t_{\mathbf{b}}(y,u), 
\frac{1}{m} \ln \frac{1}{\e} \} } | f(X( s;0,y,u), V( s;0,y,u))|  \notag \\
&& \ \ \ \ \ \ \ \ \ \ \ \ \ \ \ \ \ \ \ \ \ \ \ \ \ \ \ \ \ \ \ \ \ \ \ \ \
\ \ \ \ \ \ \ \ \ \times \{ |n(y) \cdot u| + O(\e)(1+|u|) |s| \} \mathrm{d}
s \mathrm{d} u \mathrm{d} S_{y} .  \notag
\end{eqnarray}
\end{lemma}

\begin{proof}
\textit{Step 1.} From (\ref{char}), for $\nabla \in \{\nabla_{x}, \nabla_{v}
\}$, 
\begin{equation}
\frac{d}{ds} \binom{ \nabla X}{\nabla V} =\mathbb{A} \binom{ \nabla X}{%
\nabla V}, \quad \mathbb{A}=\left( 
\begin{array}{c|c}
0_{3,3} & I_{3,3} \\ \hline
\e^{2} \nabla_{x} \Phi & 0_{3,3}%
\end{array}
\right) .
\end{equation}
Note $\binom{ \nabla X}{\nabla V}|_{s=t} = Id$. Since the matrix $\mathbb{A}$
is bounded, there exists $C_{\Phi } >0$ such that 
\begin{equation}  \label{free_DX}
\begin{split}
& | \partial_{x_{j}} X_{i} (s; t,x,v ) | \ \leq \ \ C_{\Phi}e^{C_{\Phi}|t-s|
}, \ \ \ \ \ \ \ \ \ \ \ \ | \partial_{v_{j}} X_{i} (s; t,x,v ) | \ \ \leq \
\ C_{\Phi} |t-s| e^{C_{\Phi}|t-s| }, \\
& | \partial_{x_{j}} V_{i} (s; t,x,v ) | \ \leq \ \ C_{\Phi}\e^{2} |t-s|
e^{C_{\Phi}|t-s| }, \ \ | \partial_{v_{j}} V_{i} (s; t,x,v ) | \ \ \leq \ \
C_{\Phi}e^{C_{\Phi}|t-s| }.
\end{split}%
\end{equation}

\noindent\textit{Step 2.} Assume $n_{3}(x_{\mathbf{b}}(y,u)) \neq 0$ so that
the boundary $\partial\Omega$ is locally a graph of $\eta (y_{1},y_{2})$: $%
x= (y_{1},y_{2},y_{3}) \in \partial\Omega$ iff $y_{3} = \eta (y_{1},y_{2})$.
By the definitions, 
\begin{eqnarray*}
X&:=& X(s; 0,y, u) = \left( 
\begin{array}{c}
y_{2} \\ 
y_{2} \\ 
\eta (y_{1} , y_{2})%
\end{array}
\right)+ u s + \e^{2} \int^{s}_{0}\int^{\tau}_{0} \Phi(X(\tau^{\prime}; 0,
y, u)) \mathrm{d} \tau^{\prime} \mathrm{d} \tau , \\
V&:=& V(s ; 0,y, u ) = u+ \e^{2}\int^{s}_{0} \Phi (X(\tau;0,y,u)) \mathrm{d}
\tau.
\end{eqnarray*}
From (\ref{free_DX}), 
\begin{eqnarray*}
\frac{\partial X}{\partial (y_{1}, y_{2})} &=& \left(%
\begin{array}{cc}
1 & 0 \\ 
0 & 1 \\ 
\partial_{1} \eta (y_{1}, y_{2}) & \partial_{2} \eta (y_{1}, y_{2})%
\end{array}%
\right) + \e^{2} \int^{s}_{0} \int^{\tau}_{0} \nabla_{x}\Phi
(X(\tau^{\prime};0,y,u)) \cdot \nabla_{x} X (\tau^{\prime};0,y,u) \mathrm{d}
\tau^{\prime} \mathrm{d} \tau \\
&=& \left(%
\begin{array}{cc}
1 & 0 \\ 
0 & 1 \\ 
\partial_{1} \eta ( y_{1}, y_{2}) & \partial_{2} \eta ( y_{1}, y_{2})%
\end{array}%
\right) + O( \e^{2}) s ^{2} e^{C_{\Phi} s} , \\
\frac{\partial X}{\partial s} &=& V = u + O(\e^{2}) s, \ \ \ \frac{\partial V%
}{\partial s} \ = \ \e^{2} \Phi(X), \\
\frac{\partial X}{\partial v} &=& s I_{3,3} + \e^{2} \int^{s}_{0}
\int^{\tau}_{0} \nabla_{x} \Phi (X( \tau^{\prime} ;0,y,u)) \cdot \nabla_{v}
X( \tau^{\prime} ;0,y,u) \mathrm{d} \tau^{\prime} \mathrm{d} \tau = s
I_{3,3} + O(\e^{2})s^{3} e^{C_{\Phi} s} , \\
\frac{\partial V}{\partial (y_{1}, y_{2})} &=& \e^{2} \int^{s}_{0} \nabla
_{x}\Phi ( X(\tau; 0, y, u)) \cdot \nabla_{x} X(\tau; 0, y, u) \mathrm{d}
\tau = O(\e^{2}) s e^{C_{\Phi } s} , \\
\frac{\partial V}{\partial v} &=& I_{3,3} + \e^{2} \int^{s}_{0} \nabla_{x}
\Phi (X(\tau;0,y,u)) \cdot \nabla_{v} X(\tau;0,y,u) \mathrm{d} \tau =
I_{3,3} + O(\e^{2}) s^{2}e^{C_{\Phi} s}.
\end{eqnarray*}
We consider 
\begin{eqnarray*}
\det\left( \frac{\partial (X,V)}{\partial (y_{1},y_{2}, s, v)} \right)
&=&\det \left(%
\begin{array}{c|c|c}
I_{2,2} + O(\e^{2}) s^{2} e^{C_{\Phi} s} & v + O(\e^{2})s & sI_{3,3} + O(\e%
^{2}) s^{3}e^{C_{\Phi} s} \\ 
\nabla \eta+ O(\e^{2}) s^{2} e^{C_{\Phi} s} &  &  \\ \hline
O(\e^{2})s e^{C_{\Phi} s} & \e^{2} \Phi(X) & I_{3,3} +O( \e^{2}) s^{2} e^{
C_{\Phi} s}%
\end{array}
\right).
\end{eqnarray*}
Recall the formula for the block matrix when a submatrix $D$ is invertible 
\begin{equation*}
\det \left(%
\begin{array}{c|c}
A & B \\ \hline
C & D%
\end{array}%
\right)= \det(D) \det (A-BD^{-1}C).
\end{equation*}
For $s \leq \frac{1}{m} \ln \frac{1}{\e}$ for $m \gg 1$, the submatrix $%
\frac{\partial V}{\partial v}$ is invertible and 
\begin{eqnarray*}
\det \Big( \frac{\partial V}{\partial v}\Big) = 1 + O(\e^{2}) s^{2}
e^{C_{\Phi} s} \neq 0 , \ \ \ \Big ( \frac{\partial V}{\partial v}\Big) %
^{-1} = I_{3,3} - \frac{O(\e^{2}) s^{2} e^{C_{\Phi} s}}{1+ 3O(\e^{2}) s^{2}
e^{C_{\Phi} s} }.
\end{eqnarray*}
Furthermore, for $s \leq \frac{1}{m} \ln \frac{1}{\e}$ for $m \gg 1$, we
have $s^{k} e^{C_{\Phi} s} \lesssim \e^{0+}$ and therefore 
\begin{eqnarray*}
&&\det \Big( \frac{\partial V}{\partial v}\Big) \det \left( \frac{\partial X%
}{\partial (x,s)} - \frac{\partial X }{\partial v} \Big ( \frac{\partial V}{%
\partial v}\Big) ^{-1} \frac{\partial V}{\partial (x,s)} \right) \\
&=& \{1+ O(\e) s\} \det \bigg( \bigg(%
\begin{array}{cc}
I_{2,2} + O(\e^{2} ) s^{2} e^{C_{\Phi} s} & u+ O(\e^{2} ) s \\ 
\nabla \eta + O(\e^{2} ) s^{2} e^{C_{\Phi} s} & 
\end{array}
\bigg) \\
&& \ \ \ \ \ \ \ \ \ \ \ \ \ \ \ \ \ \ \ \ \ \ - \bigg( s I_{3,3} + O(\e^{2}
) s^{2} e^{C_{\Phi} s} \bigg) \bigg( I_{3,3} - \frac{O(\e^{2}) s^{2}
e^{C_{\Phi} s}}{1+ 3O(\e^{2}) s^{2} e^{C_{\Phi} s} } \bigg) \bigg( O(\e^{2}
) s^{2} e^{C_{\Phi} s} \bigg) \bigg) \\
&=& \{1+ O(\e) s\} \det \left( \bigg( 
\begin{array}{cc}
I_{2,2} + O(\e^{2} ) s^{2} e^{C_{\Phi} s} & u+ O(\e^{2} ) s \\ 
\nabla \eta + O(\e^{2} ) s^{2} e^{C_{\Phi} s} & 
\end{array}
\bigg) + \bigg( O(\e^{2} ) s^{3} e^{C_{\Phi} s} \bigg) \right) \\
&=&\{1+ O(\e) s\} \det \bigg(%
\begin{array}{cc}
I_{2,2} + O(\e ) s & u + O(\e ) s \\ 
\nabla \eta + O(\e ) s & 
\end{array}
\bigg) \\
&=& u\cdot \big(- \partial_{1} \eta(y_{1}, y_{2}), - \partial_{2}
\eta(y_{1}, y_{2}), 1\big) + O(\e) s(1+ |u|) \\
&=& -u \cdot n(y) \sqrt{1+ (\partial_{1} \eta(y_{1}, y_{2}))^{2}
+(\partial_{2} \eta(y_{1}, y_{2}))^{2} } + O(\e) s(1+ |u|).
\end{eqnarray*}
Therefore, 
\begin{eqnarray*}
\det\left( \frac{\partial (X,V)}{\partial (s, y,u)} \right) = O(1) n(y)
\cdot u + O(\e) (1+ |u|) s .
\end{eqnarray*}
These prove (\ref{bdry_int2}). For (\ref{bdry_int1}), note that if $s \leq 
\frac{1}{m} \ln \frac{1}{\e}, \ n(y)\cdot u>0$ then $t_{\mathbf{b}}( X(s; 0,
y,u), V(s;0,y,u))= s \leq \frac{1}{m} \ln \frac{1}{\e}$, and if $s \leq 
\frac{1}{m} \ln \frac{1}{\e}, \ n(y)\cdot u<0$ then $t_{\mathbf{f}}( X(s; 0,
y,u), V(s;0,y,u))= s \leq \frac{1}{m} \ln \frac{1}{\e}$. This confirms (\ref%
{bdry_int1}).
\end{proof}

Next lemma extends the Ukai's Lemma (\cite{CIP}) to the case with external
fields.

\begin{lemma}
\label{trace_s}Assume $\Omega$ is an open bounded subset of $\mathbb{R}^{3}$
with $\partial\Omega$ is $C^{3}$. We define 
\begin{equation}  \label{non_grazing}
\gamma_{\pm}^{\delta} : = \{ (x,v) \in \gamma_{\pm} : | n(x)\cdot v | >
\delta, \ \ \delta\leq |v| \leq \frac{1}{\delta} \}.
\end{equation}
Then 
\begin{equation*}
| f\mathbf{1}_{\gamma_{\pm}^{\delta}} |_{1} \lesssim_{\delta, \Omega} \| f
\|_{1} + \| v\cdot \nabla_{x} f + \e^{2} \Phi \cdot \nabla_{v} f \|_{1} .
\end{equation*}

\end{lemma}

\begin{proof}
Let $f$ solve (\ref{st_linear}) in the sense of distributions. Then along
the trajectory for $(x,v) \in \gamma_{+}$, with $X(s)\equiv X(s;t;x,v)$ and $%
V(s)\equiv V(s;t;x,v)$,
\begin{eqnarray*}
| f(x,v)| &\lesssim& |f(X(s),V(s))| + \int^{t}_{s} | g(X(\tau), V(\tau) ) | 
\mathrm{d}\tau.
\end{eqnarray*}
Integrating over $s \in [ t - t_{\mathbf{b}}(x,v) , t ]$, we obtain 
\begin{equation}  \label{ukai_lower}
t_{\mathbf{b}}(x, v) |f(x,v)| \ \lesssim \ \int^{t}_{t-t_{\mathbf{b}}(x, v)}
|f(X(s),V(s))| \mathrm{d} s +t_{\mathbf{b}}(x, v) \int^{t}_{t-t_{\mathbf{b}%
}(x, v)} |g(X(\tau), V(\tau)) | \mathrm{d} \tau.
\end{equation}
On the other hand, for $(x,v) \in \gamma_{+}^{\delta}$, from $x_{\mathbf{b}%
}(x,v) = x- t_{\mathbf{b}}(x,v) v + O(\e^{2})(t_{\mathbf{b}}(x,v))^{2}$, 
\begin{equation}  \label{tb_expn}
t_{\mathbf{b}}(x,v) = |v|^{-1}{| x_{\mathbf{b}}(x,v)-x |} + O(\e^{2})
|v|^{-1} {(t_{\mathbf{b}}(x,v))^{2}} .
\end{equation}
We claim, for $(x,v) \in \gamma_{+}^{\delta}$, 
\begin{equation}  \label{lower_tb}
t_{\mathbf{b}}(x,v) \gtrsim_{\delta}1.
\end{equation}
For large $C\gg1$, we only need to consider $(x,v)\in \gamma_{+}^{\delta}$
such that $t_{\mathbf{b}}(x,v) \leq C$. From (\ref{tb_expn}), $t_{\mathbf{b}%
}(x,v) = \frac{| x_{\mathbf{b}}(x,v)-x |}{|v |} - t_{\mathbf{b}}(x,v)O(\e%
^{2}) \frac{ t_{\mathbf{b}}(x,v) }{|v |} \geq \frac{| x_{\mathbf{b}}(x,v)-x |%
}{|v |} - t_{\mathbf{b}}(x,v)O(\e^{2}) \frac{C }{\delta}$ so that 
\begin{equation*}
t_{\mathbf{b}}(x,v) \geq [{1+ O(\e^{2}) \frac{C}{\delta}}]^{-1} |v|^{-1}{|
x_{\mathbf{b}}(x,v)-x |} \gtrsim |v|^{-1}{|x_{\mathbf{b}}(x, v) -x |} .
\end{equation*}
From $|x_{\mathbf{b}} -x | \gtrsim |n(x) \cdot \frac{x-x_{\mathbf{b}}}{|x-x_{%
\mathbf{b}}|}|$ for $x_{\mathbf{b}} , x \in \partial\Omega$ (\cite{Guo08}), 
\begin{equation*}
t_{\mathbf{b}}(x, v) \ \gtrsim \ {|n(x) \cdot (x-x_{\mathbf{b}}(x,v))|} / {\
[|v\|x-x_{\mathbf{b}}(x,v)|]}.
\end{equation*}
On the other hand, for $(x,v) \in\gamma_{+}^{\delta}$ and $\e \ll 1$, 
\begin{eqnarray*}
|n(x) \cdot (x_{\mathbf{b}} - x) | = |n(x) \cdot [ t_{\mathbf{b}} v + O(\e%
^{2}) (t_{\mathbf{b}})^{2} ]| = t_{\mathbf{b}} |n(x)\cdot v| + O(\e^{2}) (t_{%
\mathbf{b}})^{2} \gtrsim t_{\mathbf{b}} |n(x)\cdot v| ,
\end{eqnarray*}
and $|x-x_{\mathbf{b}}| \leq t_{\mathbf{b}}|v| + O(\e^{2}) (t_{\mathbf{b}%
})^{2} \lesssim t_{\mathbf{b}} |v|$. Therefore, we conclude our claim (\ref%
{lower_tb}) by 
\begin{equation*}
t_{\mathbf{b}} \gtrsim {t_{\mathbf{b}} |n(x) \cdot v|}/ {[t_{\mathbf{b}}
|v|] } \gtrsim {|n(x)\cdot v|}{|v|} ^{-1}\gtrsim_{\delta} 1.
\end{equation*}

From (\ref{ukai_lower}) and (\ref{lower_tb}), 
\begin{equation*}
\mathbf{1}_{(x,v) \in \gamma_{+}^{\delta}}| f(x,v)| \lesssim_{\delta}
\int_{t-t_{\mathbf{b}}(x,v)}^{t} |f( X(s),V(s))| \mathrm{d} s + \int
^{t}_{t-t_{\mathbf{b}}(x,v)} |g( X(s), V(s))| \mathrm{d} s.
\end{equation*}
Integrating the above over $\gamma_{+}^{\delta}$, we deduce 
\begin{eqnarray*}
\int_{\gamma_{+}^{\delta}} |f(x,v)| |n(x)\cdot v| \mathrm{d} S_{x} \mathrm{d}
v &\lesssim&\int_{\gamma_{+}^{\delta}} \int^{t}_{t-t_{\mathbf{b}}(x,v)} |f(
X(s;t,x,v),V(s;t,x,v))| |n(x)\cdot v| \mathrm{d} s \mathrm{d} S_{x} \mathrm{d%
} v \\
& + &\int_{\gamma_{+}^{\delta}} \int_{t-t_{\mathbf{b}}(x,v)}^{t} |g(
X(s;t,x,v), V(s;t,x,v))| |n(x)\cdot v| \mathrm{d} s \mathrm{d} S_{x} \mathrm{%
d} v.
\end{eqnarray*}
We check that there exists $\e_{0}>0, m\gg 1$, and $\delta>0$ such that, for
all $0< \e < \e_{0}$, 
\begin{eqnarray*}
\gamma_{+} ^{\delta} \subset \{ (x,v) \in \gamma_{+} : t_{\mathbf{b}}(x, \pm
v) \leq m \ln \frac{1}{\e} \ \text{and} \ |v| \geq m \e^{2} \ln \frac{1}{\e}
\}.
\end{eqnarray*}
Clearly, $|v|> \delta\geq m \e^{2}_{0} \ln \frac{1}{\e_{0}} \geq m \e^{2}
\ln \frac{1}{\e}$. Since $\Omega$ is bounded, we have $|v|t_{\mathbf{b}}(x,
\pm v) \lesssim_{\Omega}1$ and $t_{\mathbf{b}}(x,\pm v)\lesssim \frac{1}{|v|}
\leq \frac{1}{\delta}\leq m \ln \frac{1}{\e_{0}}\leq m \ln \frac{1}{\e}$.
Then 
\begin{equation*}
O(\e) (1+ |v|) s \lesssim {O(\e)} (1+ \frac{1}{\delta}) m \ln \frac{1}{\e}
\lesssim_{\delta}o(1) |n(x) \cdot v|.
\end{equation*}
From (\ref{bdry_int1}), we conclude that 
\begin{eqnarray*}
\int_{\gamma_{+}^{\delta}} |f(x,v)| |n(x)\cdot v| \mathrm{d} S_{x} \mathrm{d}
v \ \lesssim \ \| f\|_{1} + \| g \|_{1} .
\end{eqnarray*}
The same arguments can be applied to bound $|f \mathbf{1}_{\gamma_{-}^{%
\delta}}|$.
\end{proof}

\begin{lemma}
{\label{green}}Let $\Phi \in C^{1}$. Assume that $f(x,v), \ g(x,v)\in
L^2(\Omega\times\mathbb{R}^3)$, $\{v \cdot \nabla_x+\e^2\Phi\cdot\nabla_v\}
f, \{v \cdot \nabla_x+\e^2\Phi\cdot\nabla_v\} g \in L^2(\Omega\times\mathbb{R%
}^3)$ and $f_{\gamma}, g_{\gamma}\in L^2(\partial\Omega\times\mathbb{R}^3)$.
Then 
\begin{eqnarray}
\iint_{\Omega\times\mathbb{R}^3}\{v \cdot \nabla_x f +\e^2\Phi\cdot\nabla_v
f \} g + \{v \cdot \nabla_x g+\e^2\Phi\cdot\nabla_v g\} f = \int_{\gamma_+}
f g - \int_{\gamma_-} f g .  \label{steadyGreen}
\end{eqnarray}
\end{lemma}

\begin{proof}
It is easy to check that the proof in Chapter 9 of \cite{CIP}, equation
(2.18), still holds in the presence of $C^{1}$ field.
\end{proof}

\vspace{8pt}

In the following sections, Section 2.2 and Section 2.3, we prove the next
linear estimate.

\begin{theorem}
\label{prop_linear_steady} For the steady case, we define a norm 
\begin{equation}  \label{norm_steady}
[\hskip-1pt [ f ]\hskip-1pt ] : = \e^{-1} \| (\mathbf{I} - \mathbf{P}) f
\|_{\nu} + \e^{-1/2} | (1- P_{\gamma}) f|_{2 } + | f|_{2 } + \| \mathbf{P} f
\|_{6} + \e^{1/2} \| w f \|_{\infty}.
\end{equation}

Suppose $\Phi \in L^{\infty}$, $g \in L^{2}(\Omega \times \mathbb{R}^{3})$,
and $r \in L^{2}(\gamma_{-})$ such that 
\begin{equation}
\iint_ {\Omega \times \mathbb{R}^3} g(x,v) \sqrt{\mu} \mathrm{d} x \mathrm{d}
v \ = \ 0 \ = \ \int_{\gamma_{-}} r(x,v) \sqrt{\mu} \mathrm{d} \gamma .
\label{constraint}
\end{equation}
Then, for sufficiently small $\e>0$, there exists a unique solution to 
\begin{equation}
v\cdot \nabla _{x}f+ {\ \e^{2}\frac 1 {\sqrt{\mu}} \Phi\cdot \nabla_v \left[%
\sqrt{\mu}f\right]}+\e^{-1}Lf= g ,\ \ \ f|_{\gamma_{-}} =P_{\gamma }f+ r ,
\label{linearf}
\end{equation}%
such that 
\begin{equation}
{\iint_{\Omega \times \mathbb{R}^{3}}f(x,v)\sqrt{\mu }\ \mathrm{d} x \mathrm{%
d} v=0},  \label{0mass_s}
\end{equation}
and 
\begin{equation}  \label{linear_steady0}
[\hskip-1pt [ f ]\hskip-1pt ] \lesssim \e^{-1/2} |r|_{2}+\| \nu^{-1/2} (%
\mathbf{I} - \mathbf{P}){g} \| _{2} + \e^{-1} \| \mathbf{P} g\|_{2} +\e%
^{\frac 1 2}|w r|_\infty+\e^{\frac 3 2}\| \langle v\rangle^{-1}w g\|_\infty.
\end{equation}
\end{theorem}

The proof is in the end of Section 2.3.

\subsection{$L^{\infty }$ Estimate}

The main goal of this section is to prove the following {proposition}.

\begin{proposition}
\label{point_s} Let $f$ satisfies, 
\begin{equation}  \label{linear_K_s}
\begin{split}
\big[ v \cdot \nabla_{x} + \e^{2} \Phi \cdot \nabla_{v} + \e^{-1}
C_{0}\langle v\rangle +\l \big] |f | \leq \e^{-1} K_{\beta} |f| + |g |, \\
\big|f |_{\gamma_{-}} \big|\leq P_{\gamma} |f| + |r|,
\end{split}%
\end{equation}
where $\l \ge 0$, for $0 < \beta < \frac{1}{4}$, $K_{\beta} |f| = \int_{%
\mathbb{R}^{3}} \mathbf{k}_{\beta} (v,u) |f(u) |\mathrm{d} u$ and 
\begin{equation}  \label{kbeta_s}
\mathbf{k}_{\beta}(v,u):= \big\{ |v-u|+ |v-u|^{-1} \big\}\exp\big[{- \beta
|v-u|^{2} - \beta \frac{[ |v|^{2} - |u|^{2} ]^{2} }{|v-u|^{2}}}\big].
\end{equation}
If $\P f \in L^6(\O \times \mathbb{R}^3)$ and $(\mathbf{I}-\mathbf{P}) f\in
L^2(\O \times \mathbb{R}^3)$, then, for $w(v) = e^{\beta^{\prime}|v|^{2}}$
with $0< \beta^{\prime} \ll \beta$, 
\begin{equation}  \label{point1_s}
\begin{split}
\e^{\frac 12} \| w f \|_{\infty} \lesssim& \ \e^{\frac 12} | w r |_{\infty}
+ \e^{\frac 32} \| \langle v\rangle^{-1} w g \|_{\infty} \\
&+ \| \P f \|_{L^{6}(\Omega \times \mathbb{R}^{3} )} + {\e^{-1}} \| (\mathbf{%
I} - \mathbf{P})f \|_{L^{2}(\Omega \times \mathbb{R}^{3})}.
\end{split}%
\end{equation}
\end{proposition}

We define the stochastic cycles for the steady case.

\begin{definition}
\label{cycle_dyn0} Define, for free variables $v_{k} \in\mathbb{R}^{3}$,
from (\ref{def_tb}) 
\begin{eqnarray*}
& {t}_{1} = t-t_{\mathbf{b}}(x,v) , \ \ {x}_{1} \ = \ X({t}_{1} ; t,x,v) =
x_{\mathbf{b}}(x,v) , & \\
& {t}_{2} = t_{1 }- t_{\mathbf{b}} (x_{1},v_{1}) , \ \ {x}_{2} = X( {t}_{2}
; {t}_{1} ,x_{1} ,v_{1} ) = x_{\mathbf{b}}(x_{1},v_{1}) ,& \\
& \vdots & \\
&{t}_{k+1} = {t}_{k} - t_{\mathbf{b}}(x_{k} ,v_{k} ) , \ \ {x}_{k+1} = X( {t}%
_{k+1} ; {t}_{k} ,x_{k} ,v_{k} ) = x_{\mathbf{b}}(x_{k},v_{k}) .&
\end{eqnarray*}

Set 
\begin{eqnarray*}
X_{\mathbf{cl}}(s;t,x,v) &:=& \sum_{k} \mathbf{1}_{[{t} _{k+1},{t} _{k})} (s)
X (s; {t}_{k} , x_{k} , v_{k} ),\\ V_{\mathbf{cl}}(s;t,x,v) &:=& \sum_{k} 
\mathbf{1}_{[{t}_{k+1},{t}_{k})} (s) V (s; {t}_{k} , x_{k} , v_{k} ).
\end{eqnarray*}

For $x\in\partial\Omega$, we define 
\begin{equation}
\mathcal{V}(x) : = \{ v \in \mathbb{R}^{3} : n(x) \cdot v>0 \}, \ \ \ 
\mathrm{d} \sigma(x, v) := \sqrt{2\pi} \mu(v) \{n(x) \cdot v\} \mathrm{d} v.
\label{tilde_w}
\end{equation}
For $j \in \mathbb{N}$, we denote 
\begin{equation}  \label{V_j}
\mathcal{V}_{j} : = \{ v_{j} \in \mathbb{R}^{3} : n(x_{j}) \cdot v_{j}>0 \},
\ \ \ \mathrm{d} \sigma_{j} := \sqrt{2\pi} \mu(v_{j}) \{n(x_{j}) \cdot
v_{j}\} \mathrm{d} v_{j}.
\end{equation}
\end{definition}

The following lemma is a generalized version of Lemma 23 of \cite{Guo08}.

\begin{lemma}[\protect\cite{Guo08}]
\label{small_lemma_s}Assume $\Phi = \Phi(x)\in C^{1}$. For sufficiently
large $T_{0}>0$, there exist constant $C_{1}, C_{2}>0$, independent of $%
T_{0} $, such that for $k= C_{1} T_{0}^{5/4}$, 
\begin{equation}
\sup_{(t, x,v) \in [0,T_{0}] \times \bar{\Omega} \times \mathbb{R}%
^{3}}\int_{\prod_{\ell =1}^{k-1} \mathcal{V}_{\ell}} \mathbf{1}_{ {t}_{k}
(t,x,v_{1},v_{2}, \cdots, v_{k-1})>0 } \Pi_{\ell=1}^{k-1} \mathrm{d}
\sigma_{\ell} < \Big\{ \frac{1}{2}\Big\}^{C_{2} T_{0}^{5/4}}.  \label{small1}
\end{equation}

\end{lemma}

\begin{proof}
For $0< \delta \ll1$, we define 
\begin{equation*}
\mathcal{V}_{\ell}^{\delta} : = \big\{ v_{\ell} \in \mathcal{V}_{\ell} : \
|v_{\ell} \cdot n(x_{\ell}) | > \delta \ \ \text{and} \ \ \delta<|v_{\ell}|
< \frac{1}{\delta} \big\}.
\end{equation*}
Clearly, $\int_{\mathcal{V}_{\ell} \backslash \mathcal{V}_{\ell}^{\delta}} 
\mathrm{d} \sigma_{\ell} \leq C \delta, $ where $C$ is independent of $\ell$%
. We claim that 
\begin{equation}  \label{tb_lower}
|t_{\ell} - t_{\ell+1}| \ \geq \ {\delta^{3}} / {C_{\Omega}}, \ \ \ \text{for%
} \ \ v_{\ell} \in \mathcal{V}_{\ell}^{\delta}.
\end{equation}
It suffices to prove, for $(x,v) \in \gamma^{\delta}_{-}$ and $0< \e \ll 1$, 
\begin{equation*}
t_{\mathbf{b}}(x,v) \ \gtrsim \ {|v|^{-2}}{|n(x) \cdot v|}.
\end{equation*}
Note that $\frac{|n(x) \cdot v|}{|v|^{2}} \leq \delta^{2}$. Therefore we
only need to consider the case of $t_{\mathbf{b}}(x,v)< \delta^{2}$.

From $|v|>\delta$ and $x_{\mathbf{b}} = x+ t_{\mathbf{b}} v + O(\e^{2})(t_{%
\mathbf{b}})^{2}$, 
\begin{equation*}
t_{\mathbf{b}} = {| x_{\mathbf{b}}-x |}{|v |^{-1}} + O(\e^{2}) {(t_{\mathbf{b%
}})^{2}}{|v |^{-1}} = {| x_{\mathbf{b}}-x |}{|v |^{-1}} + t_{\mathbf{b}} {\
O(\e^{2}) \delta } .
\end{equation*}
For fixed $\delta>0$ and $\e < \e_{0}\ll_{\delta} 1$, 
\begin{equation*}
t_{\mathbf{b}}(x, v) \gtrsim {|x_{\mathbf{b}}(x, v) -x |}{|v |^{-1}}.
\end{equation*}

From the fact $|x_{\mathbf{b}} -x | \gtrsim |n(x) \cdot \frac{x-x_{\mathbf{b}%
}}{|x-x_{\mathbf{b}}|}|$ for $x_{\mathbf{b}} , x \in \partial\Omega$ from 
\cite{Guo08}, we have 
\begin{equation*}
t_{\mathbf{b}}(x, v) \gtrsim {\big|n(x) \cdot [x-x_{\mathbf{b}}(x,v)]\big|%
^{1/2}}{|v|^{-1}}.
\end{equation*}
On the other hand, for $(x,v) \in\gamma_{-}^{\delta}$ and $\e \ll 1$ 
\begin{equation*}
|n(x) \cdot (x_{\mathbf{b}} - x) | = \big|n(x) \cdot [ t_{\mathbf{b}} v + O(%
\e^{2}) (t_{\mathbf{b}})^{2} ]\big| = t_{\mathbf{b}} |n(x)\cdot v| + O(\e%
^{2}) (t_{\mathbf{b}})^{2} \gtrsim t_{\mathbf{b}} |n(x)\cdot v|.
\end{equation*}
Therefore we prove our claim. The rest of proof of (\ref{small1}) is
identical to the proof of Lemma 23 on \cite{Guo08}.
\end{proof}

Now we are ready to prove the main result of this section:

\begin{proof}[\textbf{Proof of Proposition \protect\ref{point_s}}]
Define, for $w(v) = e^{\beta^{\prime} |v|^{2}}$, 
\begin{equation}  \label{h_dyn}
h(t,x,v) : = w(v) f(t,x,v).
\end{equation}
From Lemma 3 of \cite{Guo08}, there exists $\tilde{\beta}=\tilde{\beta}%
(\beta, \beta^{\prime})>0$ such that $\mathbf{k}_{\beta}(v,u) \frac{w(v)}{%
w(u)} \ \lesssim \ \mathbf{k}_{\tilde{\beta}} (v,u)$.

Then, from (\ref{linear_K_s}), 
\begin{equation}
\begin{split}  \label{h_eq_s}
\big[ v \cdot \nabla_{x} + \e^{2} \Phi \cdot \nabla_{v} + \e^{-1}
C_{0}\langle v\rangle + \frac{\e^{ {2}} \Phi \cdot \nabla_{v}w}{w} \big] | \e%
^{\frac 1 2} h| \leq \e^{-1} \int_{\mathbb{R}^{3}}\mathbf{k}_{\tilde{\beta}
}(v,u) | \e^{\frac 1 2} h(u) | \mathrm{d} u + \e^{\frac 1 2}|wg |.
\end{split}%
\end{equation}
Clearly $\e^{-1} C_{0}\langle v\rangle + \frac{\e^{ {2}} \Phi \cdot
\nabla_{v}w}{w} \sim \e^{-1} C_{0}\langle v\rangle$.

From (\ref{linear_K_s}), on $(x,v) \in \gamma_{-}$, 
\begin{equation}  \label{h_BC}
\begin{split}
\e^{\frac 1 2} |h( x,v)| &\leq \sqrt{2\pi} w(v) \sqrt{\mu(v)} \int_{n(x)
\cdot u>0}\e^{\frac 1 2}|h (x,u) |\frac{\sqrt{\mu(u)}}{w(u)} \{n(x) \cdot
u\} \mathrm{d} u + \e^{\frac 1 2} w(v) |{r} ( x,v) |. \\
& {\lesssim}\frac{1}{\tilde{w}(v)} \int_{n(x) \cdot u>0} \e^{\frac 1 2}| h (
x,u)| \tilde{w} (u) \mathrm{d} \sigma+ \e^{\frac 1 2} w(v) | {r} ( x,v)| ,
\end{split}%
\end{equation}
where we define 
\begin{equation} 
\tilde{w}(v) := \frac{1}{w(v)\sqrt{\mu(v)}}.  \notag
\end{equation}

\vspace{4pt}

We claim, for $t= T_{0}$ {defined as in Lemma \ref{small_lemma_s} (not
depending on $\e$)}, 
\begin{equation}  \label{claim1_s}
\begin{split}
|\e^{\frac 1 2}h ( x,v)| &\leq \ \big[ CT_{0}^{5/4}\Big\{\frac{4}{5} \Big\}%
^{C_{2} T_{0}^{5/4}} + o(1)C_{T_{0}}  \big] \| \e^{\frac 1 2} h
\|_{\infty} + C_{T_{0}}   \e^{\frac 1 2} \| w r \|_{\infty} +C_{T_{0}}   \e^{\frac 3 2}
\| \langle v\rangle^{-1} w g \|_{\infty} \\
& \ \ +  C_{T_{0}}   \Big[ \| \P f\|_{L^{6}(\Omega )} + \frac{1}{\e} \| (%
\mathbf{I} - \mathbf{P})f \|_{L^{2}(\Omega \times \mathbb{R}^{3})} \Big] .
\end{split}%
\end{equation}
Once (\ref{claim1_s}) holds, Proposition \ref{point_s} is a direct
consequence.

We first prove (\ref{claim1_s}). From (\ref{h_eq_s}), for ${t}_{1} (t,x,v) <
s \leq t$, 
\begin{eqnarray*}
&&\frac{d}{ds } \Big[ e^{- \int^{t}_{s} \frac{ C_{0} }{\e } \langle V( \tau
;t,x,v) \rangle \mathrm{d}\tau } \e h ( X _{\mathbf{cl}}(s;t,x,v), V _{%
\mathbf{cl}}(s;t,x,v)) \Big] \\
& \leq& e^{- \int^{t}_{s} \frac{ C_{0} }{\e } \langle V( \tau ;t,x,v)
\rangle \mathrm{d}\tau } \frac{1}{\e } \int_{\mathbb{R}^{3}} \mathbf{k}_{%
\tilde{\beta}} (V _{\mathbf{cl}}(s;t,x,v) ,v^{\prime}) | \e h( X _{\mathbf{cl%
}}(s;t,x,v), v^{\prime} )| \mathrm{d} v^{\prime} \\
&&+ e^{- \int^{t}_{s} \frac{ C_{0} }{\e } \langle V( \tau;t,x,v) \rangle 
\mathrm{d}\tau } |\e w g( X _{\mathbf{cl}}(s;t,x,v), V _{\mathbf{cl}}(
s;t,x,v) )|.
\end{eqnarray*}
Along the stochastic cycles, for $k= C_{1} T_{0}^{5/4}$, we deduce the
following estimate: 
\begin{eqnarray}
&& | \e^{\frac 1 2} h ( x,v) |  \notag \\
&\leq& \mathbf{1}_{\{ {t}_{1} (t,x,v)< 0\}} e^{ - \int^{t}_{0} \frac{ C_{0}
\langle V_{\mathbf{cl}} (\tau; t,x,v) \rangle }{\e } \mathrm{d}\tau} | \e%
^{\frac 1 2} h ( X_{\mathbf{cl}}(0; t,x,v ), V_{\mathbf{cl}}(0; t,x,v )) |
\label{h1in_s} \\
&+& \int_{ \max{\{0, {t}_{1} (t,x,v) \}}}^{t} \mathrm{d} s \ \frac{e^{-
\int^{t}_{s} \frac{ C_{0} \langle V_{\mathbf{cl}} ( \tau;t,x,v ) \rangle }{\e%
} \mathrm{d} \tau } }{\e }  \notag \\
&& \ \ \ \ \ \ \ \ \ \ \ \ \ \ \ \ \ \ \ \ \ \ \ \ \ \times \int_{\mathbb{R}%
^{3}} \mathrm{d} v^{\prime} \ \mathbf{k}_{\tilde{\beta} }( V_{\mathbf{cl}} (
s ; t,x,v) ,v^{\prime}) \big| \e^{\frac 1 2} h ( X_{\mathbf{cl}%
}(s;t,x,v),v^{\prime} )\big| \ \ \ \ \ \ \ \ \ \   \label{hk_s} \\
&+& \int_{ \max{\{0, {t}_{1} (t,x,v) \}}}^{t} \mathrm{d} s \ \frac{e^{-
\int^{t}_{s} \frac{ C_{0} \langle V_{\mathbf{cl}} ( \tau;t,x,v ) \rangle }{%
\e } \mathrm{d} \tau } }{\e } |\e^{\frac 3 2} w g ( X _{\mathbf{cl}%
}(s;t,x,v), V _{\mathbf{cl}}( s;t,x,v) )| \ \ \ \ \ \ \ \ \ \ \ \ 
\label{hnl_s} \\
&+&\mathbf{1}_{\{ {t}_{1}(t,x,v) \geq 0 \}} e^{ - \int^{t}_{ {t}_{1}
(t,x,v)} \frac{ C_{0} \langle V_{\mathbf{cl}} (\tau; t,x,v) \rangle }{\e } 
\mathrm{d}\tau} | \e^{\frac 1 2} w {r} ( x_{1} (x,v), v_{1} (x,v)) |
\label{h1ber_s} \\
& +& \mathbf{1}_{\{ {t}_{1} (t,x,v) \geq 0 \}} \frac{ e^{ - \int^{t}_{ {t}%
_{1} (t,x,v)} \frac{ C_{0} \langle V_{\mathbf{cl}} ( \tau; t,x,v) \rangle }{%
\e } \mathrm{d}\tau} }{\tilde{w} (v_{1})} \int_{\Pi_{j=1}^{k-1} \mathcal{V}_{j}}
H,  \notag
\end{eqnarray}
where $H$ is given by 
\begin{eqnarray}
&& \sum_{l =1}^{k-1} \mathbf{1}_{ {t}_{l+1} \leq 0< {t}_{l}} \big| \e^{\frac
1 2} h ( X_{\mathbf{cl}}(0 ; {t}_{l}, x_{l}, v_{l}) , V_{\mathbf{cl}} (0 ; 
{t}_{l}, x_{l}, v_{l}) ) \big|\notag\\
&& \ \ \ \ \ \ \ \ \ \ \ \ \ \ \  \  \ \  \times \Pi_{m=1}^{l-1} \frac{\tilde{w} (v_{m})}{\tilde{w} (V_{\mathbf{cl}} (t_{m+1};  v_{m}))}
 \mathrm{d} \Sigma_{l} (0)
  \label{h2in_s} \\
&+& \sum_{l =1}^{k-1} \int^{ {t}_{l}}_{\max \{ 0, {t}_{l+1} \}} \mathbf{1}_{ 
{t}_{l}>0} \frac{1}{\e } \int_{\mathbb{R}^{3}} \mathbf{k}_{\tilde{\beta} } (
V_{\mathbf{cl}} (\tau; {t}_{l}, x_{l}, v_{l}),u )\notag \\
&& \ \ \ \ \ \ \ \ \ \ \ \ \ \ \  \  \ \  \times\big| \e^{\frac 1 2} h (
X_{\mathbf{cl}}(\tau; {t}_{l}, x_{l}, v_{l}) , u) \big|
\Pi_{m=1}^{l-1} \frac{\tilde{w} (v_{m})}{\tilde{w} (V_{\mathbf{cl}} (t_{m+1};  v_{m}))}
 \mathrm{d} u \mathrm{%
d} \Sigma_{l} (\tau) \mathrm{d} \tau    \label{h2k_s} \\
&+& \sum_{l =1}^{k-1} \int^{ {t}_{l}}_{\max \{ 0, {t}_{l+1} \}} \mathbf{1}_{ 
{t}_{l}>0} \frac{1}{\e} \big| \e^{\frac 3 2} wg ( X_{\mathbf{cl}}( \tau; {t}%
_{l}, x_{l}, v_{l}) , V_{\mathbf{cl}} ( \tau; {t}_{l}, x_{l}, v_{l})) \big|\notag \\
&&  \ \ \ \ \ \ \ \ \ \ \ \ \ \ \  \  \ \  \times \Pi_{m=1}^{l-1} \frac{\tilde{w} (v_{m})}{\tilde{w} (V_{\mathbf{cl}} (t_{m+1};  v_{m}))}
\mathrm{d} \Sigma_{l} (\tau) \mathrm{d} \tau \ \ \ \ \ \ \ \ \ \ \ \ \ \ \ \
\ \ \ \ \   \label{h2g_s} \\
&+& \sum_{l=1}^{k-1} \mathbf{1}_{ {t}_{l}>0} | \e^{\frac 1 2} w(v_{l}) {r} (
x_{l+1} , v_{l}) |   
\Pi_{m=1}^{l-1} \frac{\tilde{w} (v_{m})}{\tilde{w} (V_{\mathbf{cl}} (t_{m+1};  v_{m}))}
\mathrm{d} \Sigma_{l} ( {t}_{l+1})  \label{h2ber_s} \\
&+& \mathbf{1}_{ {t}_{k} >0}| \e^{\frac 1 2} h ( x_{k}, v_{k-1})| 
\Pi_{m=1}^{k-2} \frac{\tilde{w} (v_{m})}{\tilde{w} (V_{\mathbf{cl}} (t_{m+1};  v_{m}))}
\mathrm{d}
\Sigma_{k-1} ( {t}_{k}),  \label{h2er_s}
\end{eqnarray}
where $V_{\mathbf{cl}} (t_{m+1};v_{m}):= V_{\mathbf{cl}} ((t_{m+1};t_{m},x_{m},v_{m})$ and $\mathrm{d} \Sigma_{k-1} ( {t}_{k})$ is evaluated at $s= {t}_{k}$ of 
\begin{equation}  \label{dsigma_unsteady_s}
\mathrm{d}\Sigma _{l} (s) := \{\prod_{j=l+1}^{k-1}\mathrm{d} \sigma _{j}\}\{
e^{ -\int^{ {t}_{l} }_{s} \frac{C_{0} \langle V_{\mathbf{cl}}( \tau ; {t}%
_{l} , x_{l}, v_{l}) \rangle }{\e } \mathrm{d} \tau } {w}(v_{l})\mathrm{d}%
\sigma _{l}\} \prod _{j=1}^{l-1}\{e^{ - \int^{ {t}_{j} }_{ {t}_{j+1} } \frac{
C_{0} \langle V_{\mathbf{cl}} ( \tau ; {t}_{j} , x_{j}, v_{j})\rangle }{\e } 
\mathrm{d} \tau }\mathrm{d}\sigma _{j}\}.
\end{equation}

Since 
\begin{equation*}
|V_{\mathbf{cl}}(t_{m+1};t_{m},x_{m},v_{m})-v_{m}|\leq \varepsilon
^{2}|t_{m+1}-t_{m}|\|\Phi \|_{\infty }\lesssim \varepsilon ^{2}T_{0}
\end{equation*}%
then for $\varepsilon ^{2}T_{0}\leq 1,$ we have  
\begin{equation}\label{w/wsteady}
\frac{\tilde{w}(v_{m})}{\tilde{w}(V_{\mathbf{cl}}(t_{m+1};t_{m},x_{m},v_{m}))}%
\leq 1 + O(\e).
\end{equation}

Directly, from our choice $k= C_{1} T_{0}^{5/4}$ 
\begin{equation*}
(\ref{h1in_s}) + (\ref{h2in_s}) \ \lesssim_{T_{0}} \  e^{- \frac{
C_{0} }{ \e } t } \|\e^{\frac 1 2} h \|_{\infty}, \ \ \ (\ref{h1ber_s}) + (%
\ref{h2ber_s}) \ \lesssim_{T_{0}}  \   \| \e^{\frac 1 2} w{r}
\|_{\infty},
\end{equation*}
and 
\begin{eqnarray*}
&&(\ref{hnl_s})+ (\ref{h2g_s}) \\
&\lesssim_{T_{0}}& \big\| \e^{\frac 3 2} \langle v \rangle^{-1} w g \big\|_{\infty}
\times \Big\{ \int^{t}_{0} \frac{ \langle V_{\mathbf{cl}} (s;t,x,v) \rangle 
}{\e } e^{ - \int^{t}_{s} \frac{C_{0} \langle V_{\mathbf{cl}} (\tau;t,x,v)
\rangle }{\e } \mathrm{d} \tau } \mathrm{d} s \\
&& \ \ \ \ \ \ \ \ \ \ \ \ \ \ \ \ \ \ \ \ \ \ \ \ \ + C_{1} T_{0}^{5/4}
\sup_{l} \int^{ {t}_{l}}_{0} \frac{ \langle V_{\mathbf{cl}} (\tau; {t}%
_{l},x_{l},v_{l}) \rangle }{\e } e^{ - \int^{ {t}_{l}}_{s} \frac{C_{0}
\langle V_{\mathbf{cl}} (\tau; {t}_{l},x_{l},v_{l}) \rangle }{\e } \mathrm{d}
\tau } \mathrm{d} \tau \Big\} \\
&\lesssim_{T_{0}} &   \big\| \e^{\frac 3 2} \langle v \rangle^{-1} w
g \big\|_{\infty} \times \int^{t}_{0} \frac{d}{ds} e^{ - \int^{t}_{s} \frac{%
C_{0} \langle V_{\mathbf{cl}} (\tau;t,x,v) \rangle }{\e } \mathrm{d} \tau } 
\mathrm{d} s \ \lesssim_{T_{0}} \   \big\| \e^{\frac 3 2} \langle v
\rangle^{-1} w g \big\|_{\infty},
\end{eqnarray*}
where we have used the fact that $\mathrm{d} \sigma_{j}$ is a probability
measure of $\mathcal{V}_{j}$.

Now we focus on (\ref{hk_s}) and (\ref{h2k_s}). For $N>1$, we can choose $%
m=m(N)\gg 1$ such that 
\begin{equation}  \label{k_m}
\mathbf{k}_{m} (v,u) : = \mathbf{1}_{|v-u|\geq \frac{1}{m}} \mathbf{1}_{|u|
\leq m} \mathbf{1}_{|v| \leq m} \mathbf{k}_{\tilde{\beta} } (v,u), \ \ \
\sup_{v} \int_{\mathbb{R}^{3}} | \mathbf{k}_{m} (v,u) - \mathbf{k}_{\tilde{%
\beta} } (v,u) | \mathrm{d} u \leq \frac{1}{N}.
\end{equation}
We split $\mathbf{k}_{\tilde{\beta} } (v,u) = [ \mathbf{k}_{\tilde{\beta} }
(v,u)- \mathbf{k}_{m} (v,u) ] + \mathbf{k}_{m} (v,u)$, and the first
difference would lead to a small contribution in (\ref{hk_s}) and (\ref%
{h2k_s}) as, for $N \gg_{T_{0}} 1$, 
\begin{eqnarray*}
&& \frac{k}{N} \| \e^{\frac 1 2} h \|_{\infty} = \frac{ C_{1} T_{0}^{5/4}}{N}
\| \e^{\frac 1 2} h \|_{\infty}.
\end{eqnarray*}
We further split the time integrations in (\ref{hk_s}) and (\ref{h2k_s}) as $%
[ {t}_{l} - \kappa \e , {t}_{l}]$ and $[\max\{ 0, {t}_{l+1} \}, {t}%
_{l}-\kappa \e ]$: 
\begin{eqnarray*}
(\ref{hk_s}) = {\underbrace{ \int^{t}_{ t- \kappa \e } }}+ \int^{ t- \kappa %
\e }_{\max\{ 0, {t}_{1} \}}, \ \ \ \ \ \ (\ref{h2k_s}) = \mathbf{1}_{\{ {t}%
_{1} \geq 0 \}} \int_{\Pi_{j=1}^{k-1} \mathcal{V}_{j}} \sum_{l=1}^{k-1} %
\Big\{ \ \ {\underbrace{\int^{ {t}_{l}}_{ {t}_{l} - \kappa \e } }}+ \int^{ {t%
}_{l} - \kappa \e }_{\max\{ 0, {t}_{l+1} \}} \Big\}.
\end{eqnarray*}
The small-in-time contributions of both (\ref{hk_s}) and (\ref{h2k_s}),
underbraced terms, are bounded by 
\begin{eqnarray*}
&&\kappa \e \frac{1}{\e } \sup_{v} \int_{|v^{\prime}| \leq N} \mathbf{k}_{m}
(v, v^{\prime}) \mathrm{d} v^{\prime} \| \e^{\frac 1 2} h \|_{\infty}
\lesssim \kappa \| \e^{\frac 1 2} h \|_{\infty}, \\
&&C_{1} T_{0}^{5/4} \kappa \e \frac{1}{\e } \sup_{v} \int_{|v^{\prime}| \leq
N} \mathbf{k}_{m} (v,v^{\prime} ) \mathrm{d} v^{\prime} \| \e^{\frac 1 2} h
\|_{\infty} \lesssim \kappa C_{1} T_{0}^{5/4} \| \e^{\frac 1 2} h
\|_{\infty}.
\end{eqnarray*}
For (\ref{h2er_s}), by Lemma \ref{small_lemma_s} and (\ref{w/wsteady}),
\begin{equation*}
\begin{split}
(\ref{h2er_s}) &\lesssim
\left\{1+ O(\e)\right\}^{C_{1} T_{0}^{5/4}}
  \sup_{(t,x,v) \in [0, T_{0}] \times \bar{\Omega}
\times \mathbb{R}^{3}} \int_{\prod_{j=1}^{k-1} \mathcal{V}_{j}} \mathbf{1}_{ 
{t}_{k} (t,x,v,v_{1}, v_{2}, \cdots, v_{k-1})>0} \Pi_{j=1}^{k-1} \mathrm{d}
\sigma_{j} \| \e^{\frac 1 2} h \|_{\infty}\\
& \lesssim \Big\{\frac{4}{5} \Big\}%
^{C_{2} T_{0}^{5/4}} \| \e h \|_{\infty}.
\end{split}
\end{equation*}

Overall, for $(t,x,v) \in [0, T_{0}] \times \bar{\Omega} \times \mathbb{R}%
^{3}$, 
\begin{eqnarray}
&& |\e^{\frac 1 2} h ( x,v)|  \label{est1_h_s} \\
& \lesssim & \int_{\max{\{0, {t}_{1} (x,v) \}}}^{t- \kappa \e } \ \frac{
e^{- \frac{C_{0}}{\e } (t-s)} }{\e } \int_{|v^{\prime}| \leq m} \underbrace{ %
\big| \e^{\frac 1 2} h ( X_{\mathbf{cl}}(s;t,x,v),v^{\prime} )\big|} \mathrm{%
d} v^{\prime} \mathrm{d} s  \notag \\
& +&\mathbf{1}_{\{ {t}_{1} \geq 0 \}} \frac{ e^{- \frac{C_{0}}{\e } (t- {t}%
_{1})} }{\tilde{w} (v)} \int_{\Pi_{j=1}^{k-1} \mathcal{V}_{j}} \sum_{\ell
=1}^{k-1} \int^{ {t}_{\ell} - \kappa \e }_{\max \{ 0, {t}_{\ell+1} \}} \frac{
\mathbf{1}_{ {t}_{\ell}>0} }{\e }  \notag \\
&& \ \ \ \ \ \ \ \ \ \ \ \ \ \ \ \ \ \ \ \ \ \ \ \ \times \int_{
|v^{\prime\prime}| \leq m} \underbrace{\big| \e^{\frac 1 2} h ( X_{\mathbf{cl%
}}( \tau; {t}_{\ell}, x_{\ell}, v_{\ell}) , v^{\prime\prime} ) \big| } 
\mathrm{d} v^{\prime\prime} \mathrm{d} \Sigma_{\ell} (\tau) \mathrm{d} \tau 
\notag \\
& +& CT_{0}^{5/4} \Big\{ e^{- \frac{C_{0}}{\e }t} \|\e^{\frac 1 2} h
\|_{\infty} +  C_{T_{0}}\| \e^{\frac 1 2} w r(s) \|_{\infty} + C_{T_{0}}\| \e^{\frac 3
2}\langle v\rangle^{-1} w g \|_{\infty} \Big\}  \notag \\
& +&o(1)CT_{0}^{5/4} \| \e^{\frac 1 2} h \|_{\infty} + \Big\{\frac{1}{2} %
\Big\}^{C_{2} T_{0}^{5/4}} \| \e^{\frac 1 2} h \|_{\infty}.  \notag
\end{eqnarray}

Note that the same estimate holds for the underbraced terms in (\ref%
{est1_h_s}). We plug these estimates into the underbraced terms of (\ref%
{est1_h_s}) to have a bound as 
\begin{equation*}
| \e^{\frac 1 2} h^{\ell+1}(t,x,v)| \leq \mathbf{I}_{1} + \mathbf{I}_{2} + 
\mathbf{I}_{3}.
\end{equation*}
Here, using $w(u)\lesssim_{m} 1$ for $|u| \leq m$, 
\begin{equation*}
\begin{split}
\mathbf{I}_{1} &\lesssim_{m} \int_{\max{\{0, {t}_{1} \}}}^{t-\kappa \e } 
\mathrm{d} s \ \frac{ e^{- \frac{C_{0}}{\e } (t-s)} }{\e }
\int_{|v^{\prime}| \leq m} \mathrm{d} v^{\prime} \int^{s-\kappa \e }_{ \max
\{ 0, {t}_{1}^{\prime} \}} \mathrm{d} s^{\prime} \frac{e^{- \frac{C_{0}
(s-s^{\prime})}{\e }}}{\e } \\
& \ \ \ \times \int_{|u| \leq m} \mathrm{d} u\big| \e^{\frac 1 2} h( X_{%
\mathbf{cl}} (s^{\prime} ; s, X_{\mathbf{cl}}(s; t,x,v), v^{\prime} ), u ) %
\big| \\
&+ \int_{\max{\{0, {t}_{1} \}}}^{t-\kappa \e } \mathrm{d} s \ \frac{ e^{- 
\frac{C_{0}}{\e } (t-s)} }{\e } \int_{|v^{\prime}| \leq m} \mathrm{d}
v^{\prime} \ \mathbf{1}_{\{ {t}_{1}^{\prime} \geq 0 \}} \frac{ e^{- \frac{%
C_{0}}{\e } (s- {t}_{1}^{\prime})} }{\tilde{w} (v)} \\
& \ \ \ \times \int_{\Pi_{j=1}^{k-1} \mathcal{V}_{j}^{\prime}}
\sum_{\ell^{\prime} =1}^{k-1} \int^{ {t}_{\ell^{\prime}}^{\prime}-\kappa \e %
}_{\max \{ 0, {t}_{\ell^{\prime}+1}^{\prime} \}} \mathbf{1}_{ {t}%
_{\ell^{\prime}}^{\prime}>0} \frac{ 1}{\e } {\big| \e^{\frac 1 2} h (\tau,
X_{\mathbf{cl}}( \tau; {t}_{\ell^{\prime}}^{\prime},
x_{\ell^{\prime}}^{\prime}, v_{\ell^{\prime}}^{\prime}) , u) \big| } \mathrm{%
d} u \mathrm{d} \Sigma_{\ell^{\prime}} (\tau) \mathrm{d} \tau,
\end{split}%
\end{equation*}
where $t_{\ell^{\prime}}^{\prime} := \tilde{t}_{\ell^{\prime}} ( s, X_{%
\mathbf{cl}}(s; t,x,v), v^{\prime} ) ,\ {x}_{\ell^{\prime}}^{\prime} := {x}%
_{\ell^{\prime}} ( X_{\mathbf{cl}}(s; t,x,v), v^{\prime} ) , \ {v}%
_{\ell^{\prime}}^{\prime} := {v}_{\ell^{\prime}} ( X_{\mathbf{cl}}(s;
t,x,v), v^{\prime} )$. Moreover 
\begin{equation}
\begin{split}
\mathbf{I}_{2}&\lesssim_{m}\mathbf{1}_{\{ {t}_{1} \geq 0 \}} \frac{ e^{- 
\frac{C_{0}}{\e } (t- {t}_{1})} }{\tilde{w} (v)} \int_{\Pi_{j=1}^{k-1} 
\mathcal{V}_{j}} \sum_{\ell =1}^{k-1} \int^{ {t}_{\ell} -\kappa \e }_{\max
\{ 0, {t}_{\ell+1} \}} \mathrm{d} \Sigma_{\ell} (\tau) \mathrm{d} \tau \ 
\mathbf{1}_{ {t}_{\ell}>0} \frac{1}{\e } \int_{ |v^{\prime\prime}| \leq m} 
\mathrm{d} v^{\prime\prime} \\
& \ \ \ \times \int_{\max{\{0, {t}_{1}^{\prime\prime} \}}}^{\tau-\kappa \e } 
\mathrm{d} s^{\prime\prime} \ \frac{ e^{- \frac{C_{0}}{\e^{2}}
(\tau-s^{\prime\prime})} }{\e } \int_{|u| \leq m} \mathrm{d} u \ {\ \big| \e%
^{\frac 1 2} h \big( X_{\mathbf{cl}} ( s^{\prime\prime} ; \tau, X_{\mathbf{cl%
}} (\tau; {t}_{\ell}, x_{\ell}, v_{\ell} ),v^{\prime\prime}),u \big)\big|} \\
&+ \mathbf{1}_{\{ {t}_{1} \geq 0 \}} \frac{ e^{- \frac{C_{0}}{\e } (t- {t}%
_{1})} }{\tilde{w} (v)} \int_{\Pi_{j=1}^{k-1} \mathcal{V}_{j}} \sum_{\ell
=1}^{k-1} \int^{ {t}_{\ell}-\kappa \e }_{\max \{ 0, {t}_{\ell+1} \}} \mathrm{%
d} \Sigma_{\ell} (\tau) \mathrm{d} \tau \ \mathbf{1}_{ {t}_{\ell}>0} \frac{1%
}{\e } \int_{ |v^{\prime\prime}| \leq m} \mathrm{d} v^{\prime\prime} \\
& \ \ \ \times \mathbf{1}_{ {t}_{1}^{\prime\prime} \geq 0 } \frac{e^{- \frac{%
C_{0}}{\e }(\tau - {t}_{1} ^{\prime\prime}) } }{\tilde{w} (v^{\prime\prime})}
\int_{\prod_{j=1}^{k-1} \mathcal{V}_{j}^{\prime\prime}}
\sum_{\ell^{\prime\prime} =1}^{k-1} \int^{ {t}_{\ell^{\prime\prime}}^{\prime%
\prime}-\kappa \e }_{\max\{ 0, {t}^{\prime\prime} _{\ell^{\prime\prime}+1}
\}} \mathbf{1}_{ {t}_{\ell^{\prime\prime}}^{\prime\prime} >0} \frac{1}{\e }
\\
& \ \ \ \times \int_{|u| \leq m} \big| \e^{\frac 1 2} h \big( %
\tau^{\prime\prime}, X_{\mathbf{cl}}( \tau^{\prime\prime} ; {t}%
^{\prime\prime}_{\ell^{\prime\prime}},
x_{\ell^{\prime\prime}}^{\prime\prime},
v_{\ell^{\prime\prime}}^{\prime\prime} ), u \big) \big| \mathrm{d} u \mathrm{%
d} \Sigma_{\ell^{\prime\prime}}^{\prime\prime} (\tau^{\prime\prime}) \mathrm{%
d} \tau^{\prime\prime},
\end{split}
\notag
\end{equation}
where ${t}^{\prime\prime}_{\ell^{\prime\prime}} : = {t}_{\ell^{\prime%
\prime}} ( \tau, X_{\mathbf{cl}}(\tau ; {t}_{\ell}, x_{\ell} , v_{\ell}),
v^{\prime\prime} ) ,\ {x}^{\prime\prime}_{\ell^{\prime\prime}} : = {x}%
_{\ell^{\prime\prime}} ( X_{\mathbf{cl}}(\tau ; {t}_{\ell}, x_{\ell} ,
v_{\ell}), v^{\prime\prime} ), \ {v}^{\prime\prime}_{\ell^{\prime\prime }} :
= {v}_{\ell^{\prime\prime}} ( X_{\mathbf{cl}}(\tau; {t}_{\ell}, x_{\ell} ,
v_{\ell}), v^{\prime\prime} ) $. Furthermore 
\begin{eqnarray*}
\mathbf{I}_{3} &\lesssim& CT_{0}^{5/2} \Big\{ e^{- \frac{C_{0}}{\e }t} \| \e%
^{\frac 1 2} h \|_{\infty} + C_{T_{0}} \| \e^{\frac 1 2} w r \|_{\infty} +  C_{T_{0}}\| \e%
^{\frac 3 2}\langle v\rangle^{-1} w g \|_{\infty} \Big\} \\
&& +o(1)CT_{0}^{5/2} \| \e^{\frac 1 2} h \|_{\infty} + T_{0}^{5/4}\Big\{%
\frac{4}{5} \Big\}^{C_{2} T_{0}^{5/4}} \| \e^{\frac 1 2} h \|_{\infty}.
\end{eqnarray*}
This bound of $\mathbf{I}_{3}$ is already included in the RHS of (\ref%
{claim1_s}).

Now we focus on $\mathbf{I}_{1}$ and $\mathbf{I}_{2}$. Consider the change
of variables 
\begin{equation}  \label{COV_s}
v^{\prime } \ \mapsto \ y : = X(s^{\prime} ; s, X_{\mathbf{cl}}
(s;t,x,v),v^{\prime}).
\end{equation}
By a direct computation and (\ref{free_DX}), for $\max\{0, t_{1}^{\prime} \}
\leq s^{\prime} \leq s- \kappa \e \leq T_{0}$, 
\begin{equation}
\begin{split}
\frac{\partial X_{i} ( s^{\prime} ; s )}{\partial v_{j}^{\prime}} &= - {\ (
s- s^{\prime} ) } \delta_{ij} + \int^{s^{\prime}}_{ s} \mathrm{d}
\tau^{\prime} \int^{\tau^{\prime}}_{ s} \mathrm{d} \tau^{\prime\prime} \e%
^{2} \sum_{m} \partial_{m} \Phi_{i} (X(\tau^{\prime\prime} ; s )) \frac{%
\partial X_{m}}{\partial v_{j}^{\prime}} (\tau^{\prime\prime} ; s) \\
&=- {\ ( s- s^{\prime} ) } \big[ \delta_{ij} + O( \e^{2}) \| \Phi \|_{C^{1}}
T_{0}^{2} e^{C_{\Phi} T_{0}} \big].
\end{split}
\notag
\end{equation}
By the lower bound of $|s-s^{\prime}| \geq \kappa \e$, 
\begin{eqnarray*}
\det \nabla_{v^{\prime}} X(s^{\prime} ; s) = |s-s^{\prime}|^{3} \det \big(%
\delta_{ij} + O( \e^{2}) \| \Phi \|_{C^{1}} T_{0}^{2} e^{C_{\Phi} T_{0}} %
\big)\gtrsim \kappa^{3} \e^{3}.
\end{eqnarray*}

Now integrating over time first 
\begin{eqnarray*}
&& \int_{\max{\{0, {t}_{1} \}}}^{t-\kappa \e } \mathrm{d} s \ \frac{ e^{- 
\frac{C_{0}}{\e } (t-s)} }{\e } \int_{|v^{\prime}| \leq m} \mathrm{d}
v^{\prime}\int^{s-\kappa \e }_{ \max \{ 0, {t}_{1}^{\prime} \}} \mathrm{d}
s^{\prime} \frac{e^{- \frac{C_{0} (s-s^{\prime})}{\e }}}{\e } \\
&& \ \ \ \times \int_{|u| \leq m} \mathrm{d} u\big| \e h( X_{\mathbf{cl}}
(s^{\prime} ; s, X_{\mathbf{cl}}(s; t,x,v), v^{\prime} ), u ) \big| \\
&\lesssim&\sup_{0 \leq s^{\prime} \leq s-\kappa \e \leq s \leq t-\kappa \e}
\int_{|v^{\prime}| \leq m} \mathrm{d} v^{\prime} \int_{|u| \leq m} \mathrm{d}
u \ |\e h(X_{\mathbf{cl}} (s^{\prime} ; s, X_{\mathbf{cl}}(s; t,x,v),
v^{\prime} ), u )|,
\end{eqnarray*}
and then from $|h(u)| = w(u) |f(u)| \lesssim_{m} |f(u)|$ for $|u | \leq m$
and decomposing 
\begin{eqnarray*}
&\lesssim&\sup_{0 \leq s^{\prime} \leq s-\kappa \e \leq s \leq t-\kappa \e} %
\e^{\frac 1 2}\int_{|v^{\prime}| \leq m} \int_{|u| \leq m} \ | f(X_{\mathbf{%
cl}} (s^{\prime} ; s, X_{\mathbf{cl}}(s; t,x,v), v^{\prime} ), u )| \mathrm{d%
} u \mathrm{d} v^{\prime} \\
&\lesssim&\sup_{0 \leq s^{\prime} \leq s-\kappa \e \leq s \leq t-\kappa \e} %
\e^{\frac 1 2}\int_{|v^{\prime}| \leq m} \int_{|u| \leq m} \ \P f(X_{\mathbf{%
cl}} (s^{\prime} ; s, X_{\mathbf{cl}}(s; t,x,v), v^{\prime} ) )| \langle u
\rangle^{2} \sqrt{\mu(u)}\mathrm{d} u \mathrm{d} v^{\prime} \\
& +& \sup_{0 \leq s^{\prime} \leq s-\kappa \e \leq s \leq t-\kappa \e} \e%
^{\frac 1 2}\int_{|v^{\prime}| \leq m} \int_{|u| \leq m} | (\mathbf{I}- 
\mathbf{P}) f(X_{\mathbf{cl}} (s^{\prime} ; s, X_{\mathbf{cl}}(s; t,x,v),
v^{\prime} ) )| \mathrm{d} u \mathrm{d} v^{\prime}.
\end{eqnarray*}
For $\P f-$contribution, 
\begin{eqnarray*}
&& \e^{\frac 1 2} \int_{v^{ \prime} } \int_{u} \big| \P f \big( X_{\mathbf{cl%
}}( s^{\prime} ; s, X_{\mathbf{cl}}(s;t,x,v) , v ^{ \prime} ) \big) \langle
u\rangle^{2} \sqrt{\mu(u)} \big| \mathrm{d} u \mathrm{d} v^{\prime} \\
&&\lesssim_{m} \e^{\frac 1 2} \Big[ \int_{v^{ \prime} } \big| \P f \big( X_{%
\mathbf{cl}}( s^{\prime} ; s, X_{\mathbf{cl}}(s;t,x,v) , v ^{ \prime} ) %
\big) \big|^{6} \mathrm{d} v^{ \prime} \Big]^{1/6} \lesssim_{m} \e^{\frac 1
2} \Big[ \int_{\Omega} \big| \P f \big( y \big) \big|^{6} \frac{1}{\kappa^{3}%
\e^{3}} \mathrm{d} y\Big]^{1/6} \\
&& \lesssim_{m} \| \P f \|_{L^{6}(\Omega )} .
\end{eqnarray*}
For $(\mathbf{I} - \mathbf{P})f$ contribution, 
\begin{eqnarray*}
&& \e^{\frac 1 2} \int_{v^{ \prime} } \int_{u} |(\mathbf{I} - \mathbf{P}) f %
\big( X_{\mathbf{cl}}( s^{\prime} ; s, X_{\mathbf{cl}}(s;t,x,v) , v ^{
\prime} ) ,u \big) | \mathrm{d} u \mathrm{d} v^{\prime} \\
&&\lesssim_{m} \e^{\frac 1 2} \Big[ \iint \big| (\mathbf{I} - \mathbf{P})f %
\big( X_{\mathbf{cl}}( s^{\prime} ; s, X_{\mathbf{cl}}(s;t,x,v) , v ^{
\prime} ), u \big) \big|^{2} \mathrm{d} v^{ \prime} \mathrm{d} u \Big]^{1/2}
\\
&&\lesssim_{m} \e^{\frac 1 2} \Big[ \iint_{\Omega \times \mathbb{R}^{3}} %
\big| (\mathbf{I} - \mathbf{P}) f (y,u) \big|^{2} \frac{1}{\kappa^{3}\e^{3}} 
\mathrm{d} y \mathrm{d} u\Big]^{1/2} \\
&& \lesssim_{m} \frac{1}{\e} \| (\mathbf{I} - \mathbf{P})f \|_{L^{2}(\Omega
\times \mathbb{R}^{3} )} .
\end{eqnarray*}

We have the similar change of variables for $v^{\prime}_{\ell^{\prime} }
\mapsto X_{\mathbf{cl}}(\tau;
t^{\prime}_{\ell^{\prime}},x^{\prime}_{\ell^{\prime}},
v^{\prime}_{\ell^{\prime}} )$, and $v^{\prime\prime}_{\ell^{\prime\prime}}
\mapsto X_{\mathbf{cl}}( -\tau^{\prime\prime} ; {t}_{\ell^{\prime\prime}}^{%
\prime\prime}, x_{\ell^{\prime\prime}}^{\prime\prime},
v_{\ell^{\prime\prime}}^{\prime\prime})$, and $v^{\prime\prime} \mapsto X_{%
\mathbf{cl}} ( s^{\prime\prime} ; \tau , X_{\mathbf{cl}}(\tau; {t}_{\ell},
x_{\ell}, v_{\ell} ), v^{\prime\prime} )$.

Following the same proof, we conclude 
\begin{equation}
\mathbf{I}_{1} + \mathbf{I}_{2} \ \lesssim \ T_{0}^{5/2} \big( \| \P f
\|_{L^{6}(\Omega\times \mathbb{R}^3 )}+ \frac{1}{\e} \| (\mathbf{I} - 
\mathbf{P}) f \|_{L^{2}(\Omega \times \mathbb{R}^{3})}\big).
\end{equation}
All together we prove our claims (\ref{claim1_s}).
\end{proof}

\subsection{Steady $L^{2}-$Coercivity and $L^6$ bound}

The main purpose of this section is to prove the following:

\begin{proposition}
\label{linearl2} Suppose all the assumptions of Proposition \ref%
{prop_linear_steady} hold. Then, for sufficiently small $\e>0$, there exists
a unique solution to (\ref{linearf}). 
Moreover, 
\begin{equation}
\begin{split}
& \|\P f\| _{{2}}+\e^{-1}\| (\mathbf{I}-\mathbf{P}) f\| _{\nu } + {\e%
^{-1/2}|(1-P_{\gamma })f| _{2,+} {+ | f|_{2} } } \\
&\lesssim \| \nu^{- \frac{1}{2}} (\mathbf{I} - \mathbf{P}) {g} \|_{2} + \e%
^{-1} \| \mathbf{P} g \|_{2} +\e^{-1/2}|r |_{2,-} .  \label{completestimate}
\end{split}%
\end{equation}
Furthermore 
\begin{equation}
\begin{split}
\| \mathbf{P} f\|_{6} \lesssim& \e^{-1}\| (\mathbf{I}-\mathbf{P} )f\| _{\nu
} +\e^{-\frac 1 2}|(1-P_{\gamma })f|_{2,+}+|r|_{2,-}+\| \frac{g}{\sqrt{\nu}}%
\| _{2}+o(1)\e^{\frac 1 2}\|w f\|_\infty \\
&+|\e^{\frac 1 2}w r|_\infty+\|\e^{\frac 3 2} \langle v\rangle^{-1}w
g\|_\infty.  \label{linear_steady}
\end{split}%
\end{equation}
\end{proposition}

As the first step of the proof of Proposition \ref{linearl2}, we consider
the following penalized problem: 
\begin{equation}  \label{linpro1}
\begin{split}
\mathcal{L}f:=(\l +\e^{-1}\nu-\frac 1 2\e^2 \Phi\cdot v)f +v\cdot \nabla_x f+%
\e^2\Phi\cdot \nabla_v f = g&\quad \text{in }\O \times\mathbb{R}^3, \\
f =P_\g f+r&\quad \text{on }\gamma_{-}.
\end{split}%
\end{equation}

\begin{lemma}
\label{calL} Assume that $g\in L^2(\O \times\mathbb{R}^3)$ and $r\in
L^2(\gamma_-)$ and satisfy (\ref{constraint}). Moreover, let $\Phi\in
L^\infty({\O })$ and $\l >0$. Then, if $\e>0$ is sufficiently small, the
solution to (\ref{linpro1}) exists and is unique. Moreover it satisfies the
bounds 
\begin{equation}
\e^{-1} \| f\|^2 _{\nu}+ |(1-P_\g)f|_{2,+}^2 \ \lesssim \ \e \| \frac{g}{%
\sqrt{\nu}}\|^2_{2}+ |r|^2_{2,-}.  \label{stimal2}
\end{equation}
\end{lemma}

We remark that Lemma \ref{calL} implies that, for $\e$ sufficiently small,
the operator $\mathcal{L}^{-1}$ is well-defined and bounded as a map from $%
L^2$ to $L^2$.

\begin{proof}
\textit{Step 1.} Denote $\varpi:= \l +\e^{-1}\nu-\frac 1 2\e^2 \Phi\cdot v$.
Since $\nu\ge \nu_0 \langle v\rangle$, with $\nu_0>0$, if $\Phi$ is such
that $\frac 1 2\e^2 \|\Phi\|_\infty|v|\le \frac 12 \e^{-1}\nu$, we have $%
\varpi\ge \frac 1 2\e^{-1}\nu_0 \langle v\rangle$.

For the existence, we first consider the following problem: 
\begin{equation}
\varpi f +v\cdot \nabla_x f+\e^2\Phi\cdot \nabla_v f=g\quad \text{in }\O %
\times\mathbb{R}^3, \quad f\big|_{\gamma_-}=r,  \label{freetanpres}
\end{equation}
with a prescribed positive function $\varpi(x,v)$ and prescribed $g,r$.

From (\ref{char}), for $- t_{\mathbf{b}}(x,v) < t < t_{\mathbf{f}}(x,v)$, 
\begin{eqnarray*}
&&\frac{d}{dt} \Big[ f(X(t;0,x,v), V(t;0,x,v)) e^{\int^{t}_{0} \varpi (
X(\tau;0,x,v), V(\tau;0,x,v) ) \mathrm{d} \tau } \Big] \\
&=& g(X(t;0,x,v), V(t;0,x,v)) e^{\int^{t}_{0} \varpi ( X(\tau;0,x,v),
V(\tau;0,x,v) ) \mathrm{d} \tau}.
\end{eqnarray*}
Then, for $(X(t),V(t)):= (X(t;0,x,v), V(t;0,x,v))$ and $\varpi (\tau) : =
\varpi ( X(\tau;0,x,v), V(\tau;0,x,v) )$, 
\begin{equation}  \label{solforA}
\begin{split}
f(x,v)= r(x_{\mathbf{b}}(x,v),v_{\mathbf{b}}(x,v)) e^{ -\int_{-t_{\mathbf{b}%
}(x,v)}^{0} \varpi(\tau) } +\int_{-t_{\mathbf{b}}(x,v) }^{0}g(X(s ),V(s ))
e^{-\int^{0}_{s} \varpi(\tau) } \mathrm{d} s .
\end{split}%
\end{equation}
This proves the existence.

Combining with $\int^{0}_{- \infty} \varpi( s ) e^{-\int^{0}_{s} \varpi(
\tau ) \mathrm{d} \tau } \mathrm{d} s = \int^{0}_{-\infty} \frac{d}{ds}
e^{-\int^{0}_{s} \varpi( \tau ) \mathrm{d} \tau } \mathrm{d} s = 1-
e^{-\int^{0}_{-\infty} \varpi( \tau ) \mathrm{d} \tau } \lesssim 1$ and $%
\varpi(s)\gtrsim \e^{-1} \langle V(s) \rangle$, 
\begin{equation*}
\| f\|_{\infty} + |f|_{\infty} \ \lesssim \ \| \frac{g}{\nu} \|_{\infty} +
|r|_{\infty}.
\end{equation*}
Similarly, we can prove 
\begin{equation*}
\| e^{\beta |v|^{2}} f \|_{\infty} + | e^{\beta |v|^{2}} f|_{\infty}
\lesssim \| e^{\beta |v|^{2}} \frac{g}{\nu} \|_{\infty}+ |e^{\beta |v|^{2}}
r|_{\infty}.
\end{equation*}

\vspace{8pt}

\noindent\textit{Step 2.} Next we consider the diffuse reflection boundary
conditions. This is done by introducing the sequence $f^\ell$ solving 
\begin{equation}
\varpi f^{\ell+1}+v\cdot \nabla _{x}f^{\ell+1}+\e^2\Phi\cdot\nabla_v
f^{\ell+1}=g,\text{ \ \ \ }f^{\ell+1}_{{-}}=\t P_\g f^\ell +r,  \notag
\end{equation}
with $f^0=0$, $\ell\ge 0$ integer and $\t\in [0,1)$.

By multiplying by $f^{\ell+1}$ and integrating, using the Green identity we
obtain 
\begin{equation}
\frac 1 2\iint_{\O \times\mathbb{R}^3} \varpi |f^{\ell+1}|^2+\frac 1
2|f^{\ell+1}|_{2,+}^2=(f^{\ell+1}, g) +\frac 1 2\int_{\g_-}|\t P_\g
f^{\ell}+\frac 1 2 r|^2  \notag
\end{equation}
From (\ref{constraint}) and the definition of $P_\g$, we have $\int_{\g_-} r
P_\g f^{\ell}=0$ and hence 
\begin{equation}
\frac 1 2\int_{\g_-}|\t P_\g f^{\ell}+r|^2\le \frac 1 2 \t^2 |P_\g
f^{\ell}|_{2,-}^2 +|r|_{2,-}^2.  \notag
\end{equation}
Therefore, from $\varpi\ge \frac{\nu_0}{\e} \langle v\rangle$, $|P_\g
f^{\ell}|_{2,-}\le |f^{\ell}|_{2,+}$ and 
\begin{equation*}
| (f^{\ell+1}, g)| = \big| \big( \frac{1}{\sqrt{\e}} \sqrt{\nu} f^{\ell+1} , 
\sqrt{\e}\frac{g}{\sqrt{\nu}} \big) \big| \lesssim o(1) \e^{-1} \| f^{\ell+1
} \|_{\nu}^{2} + \e \| \frac{g}{\sqrt{\nu}} \|_{2}^{2},
\end{equation*}
we find 
\begin{equation}
\frac 1 {8\e} \|f^{\ell+1}\|_\nu^2 +\frac 1 2|f^{\ell+1}|_{2,+}^2\le \e \| 
\frac{g}{\sqrt{\nu}}\|_2^2+ \frac 1 2 |r|_{2,-}^2+ \frac 12\t^2
|f^{\ell}|_{2,+}^2.  \label{ene100}
\end{equation}
By iteration, since $\t<1$, we conclude that 
\begin{equation}
\e^{-1} \|f^{\ell+1}\|_\nu^2 + |f^{\ell+1}|^2_{2,+} \ \lesssim_\t \ \e \| 
\frac{g}{\sqrt{\nu}}\|_2^2+ |r|_{2,-}^2.  \notag
\end{equation}
Let us look now, for $\ell\ge 1$ at the difference $f^{\ell+1}-f^{\ell}$. By
the Green's identity, we obtain 
\begin{equation}
\frac 1 {4\e} \| f^{\ell+1} - f^{\ell}\|_\nu^2 +\frac 1 2| f^{\ell+1} -
f^{\ell}|^2 _{2,+} \ \lesssim \ \frac 1 2\t^2| f^{\ell } -
f^{\ell-1}|_{2,+}^2.  \notag
\end{equation}
Again by iteration, we obtain that the sequence $\{ f^{\ell} \}$ is a Cauchy
sequence and has a limit $f_\t$ depending on $\t$. Moreover, taking the
limit $\ell\to \infty$in (\ref{ene100}), we have 
\begin{equation}
\frac 1 {8\e} \|f_\t\|_\nu^2 + (1-\t^2)|f_\t|^2_{2,+} \ \lesssim \ \e \| 
\frac{g}{\sqrt{\nu}}\|_2^2+ |r|_2^2,  \notag
\end{equation}
where we used the trace theorem, Lemma \ref{trace_s}, for the boundary
integration. Then we see that $f_\t$ satisfies the uniform-in-$\t$bounds $
\frac 1 {8\e} \|f_\t\|_\nu^2 \lesssim \e \| \frac{g}{\sqrt{\nu}}\|_2^2+
|r|_2^2$.

Thus we can take the weak $L^{2}$limit as $\t\to 1$to obtain the existence
of the solution $f$to the first line of (\ref{linpro1}). For the boundary
condition we use Lemma \ref{trace_s} to show the second line of (\ref%
{linpro1}).

Then the difference $f_\t-f$satisfies the bound 
\begin{equation}
\frac 1{8\e}\|f-f_\t\|^2_\nu +\frac 12 |f-f_\t|_{2,+}^2 \le
(1-\t)|f|_{2,+}^2\to 0 \quad \text{as } \t\to 1.  \notag
\end{equation}
Hence the convergence is strong.

\vspace{8pt}

\noindent\textit{Step 3.} We can prove (\ref{stimal2}) by applying the
Green's identity to (\ref{linpro1}): We establish an important positivity
property of $\mathcal{L}$. Using Lemma \ref{green} and the boundary
condition for $f$, we get 
\begin{equation}
(f,\mathcal{L}f)=\iint_{\O \times\mathbb{R}^3} (\l +\e^{-1}\nu-\frac 1 2\e^2
\Phi\cdot v)f^2+\frac 1 2\int_{\g_+} f^2=(f,g)+\frac 1 2 \int_{\g_-} (P_\g
f+ r)^2.  \notag
\end{equation}
Following \textit{Step 2}, 
\begin{equation}
\iint_{\O \times\mathbb{R}^3} (\l +\e^{-1}\nu-\frac 1 2\e^2 \Phi\cdot
v)f^2+\frac 1 2\int_{\g_+} |(1-P_\g) f|^2 \lesssim o(1)\e^{-1} \|f\|_\nu^2 +%
\e \|\frac{g}{\sqrt{\nu}}\|^2_{2} + |r|_{2,-}^2.  \label{enerineq25}
\end{equation}
If $\e \ll1 $ then $\frac 1 2\e^2\|\Phi\|_\infty\le \frac {\nu_0}4$, and 
\begin{equation}
\l \|f\|_2^2+\frac{\e^{-1}}2\|f\|_\nu^2+\frac 1 2 |(1-P_\g) f|_2^2 \lesssim %
\e \|\frac{g}{\sqrt{\nu}}\|_2^2+ |r|_{2,-}^2.  \label{energiasenzaK}
\end{equation}

The inequality (\ref{stimal2}) follows immediately from (\ref{energiasenzaK}%
). The uniqueness follows from (\ref{stimal2}) because, if there are two
solutions, their difference satisfies (\ref{linpro1}) with $g=0$ and $r=0$.
Hence it must vanish.
\end{proof}

\begin{lemma}
\label{compact1} For any $\lambda, \e >0$, the operator $K\mathcal{L}^{-1}$
is compact in $L^{2}$. Explicitly, if $g^{n} \in L^{2}$ and $\sup_{n} \|
g^{n}\|_{2} <\infty$ then there exist subsequence $n_{k}$ such that $K
f^{n_{k}} \rightarrow Kf$ in $L^{2}$, where $f^{n}$ solve 
\begin{equation*}
\lambda f^{n} + v\cdot \nabla_{x} f^{n} + \frac{1}{\e} \nu f^{n} +\e^{2}
\Phi \cdot \nabla_{v} f^{n} - \frac 1 2\e^{2} \Phi \cdot v f^{n} = g^{n}, \
\ f^{n}|_{\gamma_{-}} = P_{\gamma } f^{n}+r.
\end{equation*}
\end{lemma}

This lemma is proved in Appendix \ref{A3}

Next we prove the essential bound for $\mathbf{P}f$, where, for $\mathfrak{}\in [0,1]$,  $f$ solves 
\begin{equation}
\begin{split}
[\l +(1- \mathfrak{r} ) \e^{-1}\nu-\frac 1 2 \e^2\Phi\cdot v)] f+v\cdot
\nabla_x f +\e^2\Phi\cdot \nabla_v f +\e^{-1}\mathfrak{r} Lf=g, & \ \ \ 
\text{in} \ \Omega \times \mathbb{R}^{3} \\
f_-=P_\g f+r,& \ \ \ \text{on} \ \gamma_{-}.  \label{eq0}
\end{split}%
\end{equation}
{We denote 
\begin{equation}
\mathring{f} \ := \ f- <f>\sqrt{\mu}, \ \ \ <f> \ := \ {\ \Big({%
\iint_{\Omega \times \mathbb{R}^{3}}f\sqrt{\mu }\mathrm{d} x\mathrm{d} v}}%
\Big) \Big/ \Big({\ {\iint_{\Omega \times \mathbb{R}^{3}}{\mu }\mathrm{d} x%
\mathrm{d} v} } \Big).  \label{margin_f}
\end{equation}%
}

\begin{lemma}
\label{steady_abc}Assume (\ref{constraint}). Let $f$ be a solution to (\ref%
{eq0}) in the sense of distribution.

Then, for all $\lambda \geq 0$ sufficiently small and all $\mathfrak{r} \in
[0,1]$ sufficiently close to $1$, 
\begin{equation}
\| \mathbf{P} \mathring f\| _{2}^{2} \ \lesssim \ \e^{-2}\| (\mathbf{I}-%
\mathbf{P})f\| _{\nu }^{2} +|(1-P_{\gamma })f|_{2,+}^{2}+|r|_{2,-}^{2}+\| 
\frac{g}{\sqrt{\nu}}\| _{2}^{2}+\e^2\|\Phi\|_\infty|<f>|^2 ,  \label{P1-P2}
\end{equation}
and 
\begin{equation}
\lambda | <f>| \ \lesssim \ (1- \mathfrak{r} ) \e^{-1} \| f \|_{2} .
\label{rg}
\end{equation}
Moreover, for $0< \eta \ll 1$ 
\begin{equation}
\begin{split}
\| \mathbf{P} \mathring f\|_{6} \ \lesssim& \ \e^{-1}\| (\mathbf{I}-\mathbf{P%
})f\| _{\nu } +\e^{-\frac 1 2}|(1-P_{\gamma })f|_{2,+}+|r|_{2,-}+\| \frac{g}{%
\sqrt{\nu}}\| _{2} \\
&+|\e^{\frac 1 2}w r|_\infty+\|\e^{\frac 3 2} \langle v\rangle^{-1}w
g\|_\infty +\eta\{|<f>|+ \e^{\frac 1 2}\|w f\|_\infty\},  \label{P1-P6}
\end{split}%
\end{equation}
and, in particular, for $\l =0$ and $\mathfrak{r} =1$, (\ref{linear_steady})
is verified.
\end{lemma}

\begin{proof}
\textit{Step 1.} Set $\varpi_\mathfrak{r} = \l +(1-\mathfrak{r} )\e%
^{-1}\nu-\frac 1 2\e^2\Phi\cdot v $. By the Green's identity (\ref%
{steadyGreen}) and (\ref{eq0}), 
\begin{eqnarray}
&&\iint_{\Omega \times \mathbb{R}^{3}} \varpi_\mathfrak{r} f \psi - {v}\cdot
\nabla _{x}\psi f-\e^2 f\Phi\cdot\nabla_v{\psi} + \int_{\gamma_+ }\psi f
-\int_{\gamma_- }\psi f  \notag \\
&&= -\e^{-1}\mathfrak{r} \iint_{\Omega \times \mathbb{R}^{3}}\psi L(\mathbf{I%
}-\mathbf{P})f +\iint_{\Omega \times \mathbb{R}^{3}}\psi g.
\label{weakformulation}
\end{eqnarray}%
First we claim (\ref{rg}). From (\ref{weakformulation}) with {$\psi =\sqrt{%
\mu}$,} 
\begin{equation}
<f>\left[\l +(1-\mathfrak{r} ) \e^{-1} \iint_{\O \times \mathbb{R}^3} \nu%
\sqrt{\mu}\right] +(1-\mathfrak{r} ) \e^{-1}\iint_{\O \times \mathbb{R}^3}
\nu \mathring f\sqrt{\mu}=0,  \label{estava}
\end{equation}
where we have used (\ref{constraint}), $\iint_{\O \times \mathbb{R}^3}
\mathring f\sqrt{\mu} =0$, and 
\begin{equation}
\begin{split}
\int_{\mathbb{R}^3} \sqrt{\mu} [\Phi\cdot \nabla_v f -\frac 1 2 \Phi\cdot v
f] \mathrm{d} v =\int_{\mathbb{R}^3} \sqrt{\mu} \Phi\cdot \frac {\nabla_v (%
\sqrt{\mu} f )}{\sqrt{\mu}} \mathrm{d} v=0,\ \ \int_{\gamma_{-}} P_{\gamma}
f_{+} \sqrt{\mu} \mathrm{d} \gamma-\int_{\gamma_{+}} f \sqrt{\mu} \mathrm{d}
\gamma=0.
\end{split}
\notag
\end{equation}
Clearly, $\iint \nu \mathring f \sqrt{\mu} \lesssim \| \mathring f \|_{2}
\leq\| f \|_{2}$ and these prove (\ref{rg}).

\vspace{4pt}

Now we prove (\ref{P1-P2}). Denote $\mathring a=:a -<f>$ so that $\mathbf{P}%
\mathring f=\{\mathring a+v\cdot b+c[\frac{|v|^{2}}{2}-\frac{3}{2}]\}\sqrt{%
\mu }$. 

\vspace{8pt} \noindent \textit{Step 2}. \textit{Estimate of \ }$c$. {We
claim that, for sufficiently small $\e>0$, 
\begin{equation}
\| c\| _{2} \ \lesssim \ o(1) \|\P f\|_2 + |(1-P_{\gamma })f|_{2,+} +\e%
^{-2}\| (\mathbf{I}-\mathbf{P})f\| _{2} +\| \frac{g}{\sqrt{\nu}} \| _{2}
+|r|_{2,-} ,  \label{cestimate}
\end{equation}
} 
\begin{equation}
\begin{split}
\| c\| _{6} \ \lesssim& \ o(1) \{ \|\mathbf{P}f\|_6 + \e^{1/2} \| w f
\|_{\infty} \} +\e^{-1}\|(\mathbf{I}-\mathbf{P}) f\|_\nu+\|(\mathbf{I}-%
\mathbf{P}) f\|_6 \\
&+ \e^{-\frac 1 2}|(1-P_{\gamma })f|_{2,+} + \|\frac {g}{\sqrt{\nu}}%
\|_2+|r|_{2,-} + |\e^{\frac 1 2}w r|_\infty+\|\e^{\frac 3 2} \langle
v\rangle^{-1}w g\|_\infty .  \label{cestimate6}
\end{split}%
\end{equation}
For $k=2,6$ we choose the test functions 
\begin{equation}
\psi =\psi_{c,k}\equiv (|v|^{2}-\beta _{c})\sqrt{\mu }v\cdot \nabla _{x}\f %
_{c,k}(x), \ \ \text{where} \ -\Delta _{x}\f _{c,k}(x)=c^{k-1}(x),\ \ \f %
_{c,k}|_{\partial \Omega }=0,  \label{phic}
\end{equation}
and $\beta _{c}$ is a constant to be determined.

From the standard elliptic estimate, we have 
\begin{equation}  \label{c_2_2}
\| \f _{c,2}\| _{H^{2}}\lesssim \| c\| _{2}.
\end{equation}
With the choice (\ref{phic}), the right hand side of (\ref{weakformulation})
is bounded by 
\begin{equation}
\text{r.h.s.(\ref{weakformulation})}\lesssim\| c\| _{2}\big\{\e^{-1}%
\mathfrak{r} \| (\mathbf{I}-\mathbf{P})f\| _{2}+\| \frac{g}{\sqrt{\nu}}\|
_{2}\big\}.  \label{boundrhs2}
\end{equation}

For $k=6$ we use the Sobolev-Gagliardo-Nirenberg inequality: for $1\le p\le
N $ and a bounded $C^{1}$ domain $\Omega \subset \mathbb{R}^{N}$, and $u\in
W^{1,p}(\Omega)$, 
\begin{equation}
\left(\int_{\O } |u|^q\right)^{\frac 1 {p^*}}\le C(N,p,\O ) \| u \|_{W^{1,p}
(\O )} , \ \ \text{ for any } \ p\le q\le p^*=\frac{Np}{N-p},  \label{SGN}
\end{equation}
and $W^{1,p}(\O )$ is continuously embedded in $L^q({\O })$ (see \cite{LG},
page 312).

Here $N=3$ and we are interested in $p^*=2$ which means $p=\frac 6 5$. Thus
for any $q\in [\frac 6 5, 2]$, we have 
\begin{equation}
\|\nabla \f_{c,6}\|_q\lesssim \|\f_{c,6}\|_{W^{2,\frac 6 5}}.  \notag
\end{equation}
Hence 
\begin{equation}  \label{c_5_6}
\|\nabla \f_{c,6}\|_2\lesssim \|c^5\|_{\frac 6 5}=\|c\|_6^5.
\end{equation}
Therefore, the right hand side of (\ref{weakformulation}), for $k=6$ is
bounded by 
\begin{equation}
\text{r.h.s.(\ref{weakformulation})}\lesssim \|\nabla\f_{c,6}\|_2 \left(\e%
^{-1}\|(\mathbf{I}-\mathbf{P}) f\|_\nu + \|g\|_2\right)\le \|c\|_6^5 \left(\e%
^{-1}\|(\mathbf{I}-\mathbf{P}) f\|_\nu + \| \frac{g}{\sqrt{\nu}} \|_2\right).
\end{equation}
Thus, by Young inequality ($|xy|\le \eta |x|^p +C_{\eta,p,q}|y|^q$, $%
p^{-1}+q^{-1}=1$) , we have 
\begin{equation}
\text{r.h.s.(\ref{weakformulation})}\lesssim \eta \|c\|_6^6 +C_\eta \left(\e%
^{-1}\|(\mathbf{I}-\mathbf{P}) f\|_\nu + \| \frac{g}{\sqrt{\nu}}
\|_2\right)^6.  \label{boundrhs3}
\end{equation}

We have ${v}\cdot \nabla _{x}\psi _{c,k}=\sum_{i,j=1}^{d}(|v|^{2}-\beta _{c})%
\sqrt{\mu }v_{i}v_{j}\partial _{ij}\f _{c,k}(x)$, and 
\begin{equation*}
\e^2[\Phi\cdot \nabla_v \psi_{c,k}- \frac 1 2v\cdot \Phi]f =\e^2 \sqrt{\mu}
f \Phi\cdot\nabla_v \big(\frac{\psi_{c,k}}{\sqrt{\mu}}\big) =\e^2\sqrt{\mu}
f\sum_{i,j}\Phi_i[\delta_{ij}(|v|^2-\beta_c)+2 v_iv_j]\pt_j\f_{c,k}.  \notag
\end{equation*}
Then the left hand side of (\ref{weakformulation}) takes the form, for $%
i=1,\cdots ,d$, 
\begin{eqnarray}
&& \ \ \iint_{\partial \Omega \times \mathbb{R}^{3}}(n(x)\cdot
v)(|v|^{2}-\beta _{c})\sqrt{\mu }\sum_{i=1}^{d}v_{i}\partial _{i}\f _{c,k} f 
\mathrm{d} S_{x} \mathrm{d} v  \label{t1} \\
& &- \iint_{\Omega \times \mathbb{R}^{3}}[(\l +(1-\mathfrak{r} )\e^{-1}\nu]f
(|v|^{2}-\beta _{c})\sqrt{\mu }\sum_i v_i\pt_i\f _{c,k} \mathrm{d} x \mathrm{%
d} v  \label{t2} \\
&&-\iint_{\Omega \times \mathbb{R}^{3}}(|v|^{2}-\beta _{c})\sqrt{\mu }\Big\{%
\sum_{i,j=1}^{d}v_{i}v_{j}\partial _{ij}\f _{c,k}\Big\}f \mathrm{d} x 
\mathrm{d} v  \label{t3} \\
&&+\e^2\sqrt{\mu} \sum_{i,j}\iint_{\Omega \times \mathbb{R}%
^{3}}\Phi_i[\delta_{ij}(|v|^2-\beta_c)+2 v_iv_j]\pt_j\f_{c,k}f \mathrm{d} x 
\mathrm{d} v .  \label{t4}
\end{eqnarray}%
We decompose 
\begin{eqnarray}
f &=&\Big\{a+v\cdot b+c\left[\frac{|v|^{2}}{2}-\frac{3}{2}\right]\Big\}\sqrt{%
\mu }+(\mathbf{I}-\mathbf{P})f,\text{ \ \ \ \ \ on }\Omega \times \mathbb{R}%
^{3},  \label{insidesplit} \\
f_{\gamma } &=&P_{\gamma }f+\mathbf{1}_{\gamma _{+}}(1-P_{\gamma })f+\mathbf{%
1}_{\gamma _{-}}r,\text{ \ \ \ \ \ \ \ \ \ \ \ \ \ \ \ \ \ \ \ \ \ \ \ on }%
\gamma ,  \label{bsplit}
\end{eqnarray}
and substitute (\ref{insidesplit}), (\ref{bsplit}) into (\ref{t1})$-$(\ref%
{t4}). Note that the off-diagonal parts ($v_{i}v_{j}$ with $i \neq j$) and $%
b $ term vanish by oddness in $v$. Now we choose $\beta _{c}=5$ so that, 
\begin{equation}
\int (|v|^{2}-\beta _{c})v_{i}^{2}\mu (v)dv=0, \ \ \ \text{for} \ i=1,2,3.
\label{beta}
\end{equation}%
Note that, thanks to the choice of $\beta_c=5$, we eliminate the $a$
contribution in the bulk. Then (\ref{t3}) becomes%
\begin{eqnarray}
(\ref{t3})=&&-\sum_{i=1}^{d}\int_{\mathbb{R}^{3}}(|v|^{2}-\beta
_{c})v_{i}^{2}\Big(\frac{|v|^{2}}{2}-\frac{3}{2}\Big)\mu (v)\mathrm{d}
v\int_{\Omega }\partial _{ii}{\f _{c,k}(x)c(x)}\mathrm{d} x  \label{c^k} \\
&&-\sum_{i=1}^{d}\iint_{\Omega \times \mathbb{R}^{3}}(|v|^{2}-\beta
_{c})v_{i}\sqrt{\mu }({v}\cdot \nabla _{x})\partial _{i}\f _{c,k}(\mathbf{I}-%
\mathbf{P})f.  \label{c(1-P)}
\end{eqnarray}

From $\int_{\mathbb{R}^{3}}(|v|^{2}-\beta _{c})v_{i}^{2}(\frac{|v|^{2}}{2}-%
\frac{3}{2})\mu (v)dv=10\pi \sqrt{2\pi }$ and $-\Delta _{x}\f _{c,k}=c^{k-1}$
for $k=2,6$, 
\begin{equation}
(\ref{c^k}) =-10\pi \sqrt{2\pi }\int_{\Omega }\Delta _{x}\f _{c,k} c = 10\pi 
\sqrt{2\pi}\|c\|_k^k.  \label{0236}
\end{equation}
Moreover, for $k=2$, 
\begin{equation}
(\ref{c(1-P)})\le \|\nabla^{2} \f_{c,2}\|_2\|(\mathbf{I}-\mathbf{P})
f\|_2^2\le \eta \|c\|_2^2+C_\eta \|(\mathbf{I}-\mathbf{P}) f\|_2^2,
\label{0237}
\end{equation}
and, for $k=6$, 
\begin{equation}
(\ref{c(1-P)})\le \|\nabla^{2} \f_{c,2}\|_{\frac 6 5}\|(\mathbf{I}-\mathbf{P}%
) f\|_6\le \eta \|c\|_6^6+C_\eta \|(\mathbf{I}-\mathbf{P}) f\|_6^6.
\label{0237bis}
\end{equation}

Consider (\ref{t1}). Because of the choice of $\beta _{c}$ to have (\ref%
{beta}), there is no $P_{\gamma }f$ contribution at the boundary in (\ref{t1}%
). Then for $k=2$ we have 
\begin{equation}
\begin{split}
(\ref{t1}) \lesssim \| c\| _{2}\{|(1-P_{\gamma })f|_{2,+}+|r|_{2,-}\},
\label{btes}
\end{split}%
\end{equation}
where we used $|\nabla _{x}\f _{c}|_{2}\lesssim \| \f _{c}\|
_{H^{2}}\lesssim \| c\| _{2}$ for an elliptic estimate and the trace
estimate.

Now we consider $k=6$ case. 
By the assumption that $\Omega$ is a $C^1$ domain in $\mathbb{R}^N$ with $%
N=3 $, we can use the following trace estimate (see \cite{LG}, page 466):

\begin{equation}
\left(\int_{\pt \O } dS(x) |u|^{\frac{p(N-1)}{N-p}}\right)^{\frac{N-p}{p(N-1)%
}}\le C(N,P)\left( {\int_{\O } dx |u|^{p}}+\int_{\O } dx |\nabla
u|^{p}\right)^{\frac{1}{p}}.
\end{equation}
This is a consequence of the trace theorem $W^{1,p}(\O ) \rightarrow W^{1- 
\frac{1}{p}, p}(\partial\O ),$ and the Sobolev embedding in $N-1$
dimensional sub-manifold $(W^{1-\frac{1}{p}, p} (\partial\O ) \subset L^{ 
\frac{p (N-1)}{N-p} } (\Omega) $ for $\frac{N-p}{p (N-1)}=\frac{1}{p}- \frac{%
1- \frac{1}{p}}{N-1}$). In particular, with $p=\frac 65$ and $N=3$ we have $%
\frac{p(N-1)}{N-p}= 
\frac 4 3$. With $u=\nabla\f_{c,6}$, we have 
\begin{equation}  \label{4/3}
\| \nabla_{x} \f_{c} \|_{L^{4/3} (\partial\O )} \lesssim \| c \|^{5}_{L^{6}
(\Omega)}.
\end{equation}
On the other hand, by Holder inequality 
\begin{equation}
\begin{split}
| \mu^{1/4} (1-P_{\gamma })f |_{ 4,{+}} &\leq \e^{1/4} \big[\e^{-1/2} |
(1-P_{\gamma })f |_{2,+} \big]^{1/2} | {\mu}^{1/2} (1-P_{\gamma })f
|_{\infty,+}^{1/2} \\
&\lesssim \big[\e^{-1/2} | (1-P_{\gamma })f |_{2,+} \big]^{1/2} \big[\e^{
1/2} \| w f \|_{\infty} \big]^{1/2} .
\end{split}
\notag
\end{equation}
Therefore, by the Young inequality, we conclude 
\begin{equation}
\begin{split}  \label{t1_6}
(\ref{t1}) &\lesssim \ \big\{ | \mu^{1/4}(1- P_{\gamma}) f|_{4,+} + |
\mu^{1/4} r|_{ 4/3,-} \big\}| \nabla_{x} \varphi_{c} |_{4/3,+} \\
& \lesssim \ \Big\{ \big[\e^{-1/2} | (1-P_{\gamma })f |_{2,+} \big]^{1/2} %
\big[\e^{ 1/2} \| w f \|_{\infty} \big]^{1/2} + | \mu^{1/4} r|_{ 4/3,-} %
\Big\}\| c \|_{L^{6} }^{5} \\
& \leq \ \eta \|c\|_6^6+ \eta^{\prime}\big[ \e^{\frac 1 2} \| w f \|_{\infty}%
\big]^{6}+C_{\eta,\eta^{\prime}} \big[\e^{-\frac 1 2}|(1-P_\gamma)f|_{2,+}%
\big]^{6}.
\end{split}%
\end{equation}

Now we consider (\ref{t2}). Using (\ref{c_2_2}) for $k=2$ and (\ref{c_5_6})
for $k=6$ respectively, we conclude that 
\begin{equation}
\begin{split}  \label{penes}
(\ref{t2}) \lesssim & \ [(\l +(1-\mathfrak{r} )\e^{-1}] \times \big\{ \| 
\mathbf{P}f \|_{k} + \| (\mathbf{I}-\mathbf{P}) f \|_{2}\big\} \| c
\|_{k}^{k-1} \\
\lesssim & \ [(\l +(1-\mathfrak{r} )\e^{-1} ] \times \big\{ \| \mathbf{P} f
\|_{k}^{k} +\| (\mathbf{I} - \mathbf{P} ) f \|_{2}^{k} \big\} \\
\lesssim & \ ( \lambda + o(1) )\| \mathbf{P} f \|_{k}^{k} +\| (\mathbf{I} - 
\mathbf{P} ) f \|_{2}^{k},
\end{split}%
\end{equation}
where we have used $\mathfrak{r} =1+ o(1) \e$.

Moreover, since $\int_{\mathbb{R}^{3}}\mu[(|v|^{2}-\beta _{c})+2v_i^2]\left[%
\frac{|v|^{2}}{2}-\frac{3}{2}\right]=2\sqrt{2\pi }$, and $\int_{\mathbb{R}%
^{3}}\mu[(|v|^{2}-\beta _{c})+2v_i^2]=3-\beta_c+2=0$, the term (\ref{t4})
becomes 
\begin{equation}
2\e^2 \int_{\Omega }c\Phi\cdot\nabla_x\f_{c,k}+\e^2\sqrt{\mu}
\sum_{i,j}\iint_{\Omega \times \mathbb{R}^{3}}\Phi_i[(\delta_{i,j}-\frac 1
2v_iv_j)(|v|^2-\beta_c)+2 v_iv_j]\pt_j\f_{c,k}(\mathbf{I}-\mathbf{P})f. 
\notag  \label{ft}
\end{equation}
Using $\notag\big|\int_{\Omega }c\Phi\cdot\nabla_x\f_{c,k}\big| \le
\|c\|_k^k\|\Phi\|_\infty$, (\ref{t4}) is bounded by 
\begin{equation}
(\ref{t4})\lesssim \e^2 \big[\|c\|_k^k+ \|(\mathbf{I}-\mathbf{P})f\|_k^k\big]%
\|\Phi\|_\infty.  \label{0233}
\end{equation}

By collecting the estimates (\ref{boundrhs2}), (\ref{0236}), (\ref{0237}), (%
\ref{btes}), (\ref{penes}), (\ref{0233}), for sufficiently small $\e>0$ we prove (\ref{cestimate}).

Similarly, collecting the estimates (\ref{boundrhs3}), (\ref{0236}), (\ref%
{0237bis}), (\ref{t1_6}), (\ref{penes}), (\ref{0233}), for $\e$ sufficiently
small we obtain (\ref{cestimate6}).

\bigskip

\vspace{4pt} \noindent \textit{Step 3.} \textit{Estimate of \ }$b$. {We
claim that, for sufficiently small $\e>0$, 
\begin{equation}
\| b\| _{2}^{2}\ \lesssim o(1) \|\P f\|_2^2 + \ |(1-P_{\gamma })f|_{2,+}^{2}+%
\e^{-2}\| (\mathbf{I}-\mathbf{P})f\| _{2}^{2}+\| \frac{g}{\sqrt{\nu}}\|
_{2}^{2}+|r|_{2,-}^{2}.  \label{bestimate}
\end{equation}%
} 
\begin{multline}
\| b\| _{6}^{6}\ \lesssim o(1) \|\P f\|_6^6+\Big(\e^{-\frac 1
2}|(1-P_{\gamma })f|_{2,+}+\e^{-1}\|(\mathbf{I}-\mathbf{P}) f\|_\nu+\|(%
\mathbf{I}-\mathbf{P}) f\|_6 + \|\frac {g}{\sqrt{\nu}}\|_2+|r|_2 \\
+ |\e^{\frac 1 2}w r|_\infty+\|\e^{\frac 3 2} \langle v\rangle^{-1}w
g\|_\infty\Big)^6.  \label{bestimate6}
\end{multline}
For $k=2,6$ we shall establish the estimate of $b$ by estimating $(\partial
_{i}\partial _{j}\Delta ^{-1}b^{k-1}_{j})b_{i}$ for all $i,j=1,\dots ,d$,
and $(\partial _{j}\partial _{j}\Delta ^{-1}b^{k-1}_{i})b_{i}$ for $i\neq j$.

We fix $i,j$. To estimate $\partial _{i}\partial _{j}\Delta
^{-1}b_{j}^{k-1}b_{i}$ we choose as test function in (\ref{weakformulation}%
), For $k=2,6$ 
\begin{equation}
\psi =\psi _{b,k}^{i,j}\equiv (v_{i}^{2}-\beta _{b})\sqrt{\mu }\partial _{j}%
\f _{b,k}^{j},\quad i,j=1,\dots ,d,  \label{phibj}
\end{equation}%
where $\beta _{b}$ is a constant to be determined, and 
\begin{equation}
-\Delta _{x}\f _{b,k}^{j}(x)=b^{k-1}_{j}(x),\ \ \ \f _{b,k}^{j}|_{\partial
\Omega }=0.  \label{jb}
\end{equation}%
For $k=2$, from the standard elliptic estimate $\| \f _{b}^{j}\|
_{H^{2}}\lesssim \| b\| _{2,2}.$ Hence, for $k=2$ the right hand side of (%
\ref{weakformulation}) is now bounded by 
\begin{equation}
\text{r.h.s(\ref{weakformulation})} \le\| b\| _{2}\big\{\e^{-1}\| (\mathbf{I}%
-\mathbf{P})f\| _{2}+\| \frac{g}{\sqrt{\nu}}\| _{2}\big\}.  \label{rhsbij}
\end{equation}
With the same argument as before, the right hand side of (\ref%
{weakformulation}) for $k=6$ is bounded by 
\begin{equation}
\text{r.h.s.(\ref{weakformulation})}\lesssim \eta \|b\|_6^6 +C_\eta \left(\e%
^{-1}\|(\mathbf{I}-\mathbf{P}) f\|_\nu + \|g\|_2\right)^6.  \label{rhsbij6}
\end{equation}

Now substitute (\ref{bsplit}) and (\ref{insidesplit}) into the left hand
side of (\ref{weakformulation}). Note that $(v_{i}^{2}-\beta
_{b})\{n(x)\cdot {v}\}\mu $ is odd in $v$, therefore $P_{\gamma }f$
contribution to (\ref{weakformulation}) vanishes. Moreover, by (\ref%
{insidesplit}), the $a,c$ contributions to (\ref{weakformulation}) also
vanish by oddness. Finally, in the field term only the $\mathbf{P} f$ part
survives because $\nabla_v\frac{\psi}{\sqrt{\mu}}= 2 v_i \pt_j\f_b^j$ and,
by oddness, the $a$ and $c$ contributions disappear. Therefore the left hand
side of (\ref{weakformulation}) takes the form 
\begin{eqnarray}
&& \ \ \ \iint_{\partial \Omega \times \mathbb{R}^{3}}(n(x)\cdot {v}
)(v_{i}^{2}-\beta _{b})\sqrt{\mu }\partial _{j}\f _{b,k}^{j}[(1-P_{\gamma
})f+r]\mathbf{1}_{\gamma _{+}}  \label{tb1} \\
&&+ \sum_i\iint_{\Omega \times \mathbb{R}^{3}}[(\l +(1-\mathfrak{r} )\e%
^{-1}\nu]f (v_i^{2}-\beta _{b})\sqrt{\mu } v_i\pt_j\f _{b,k}^j  \label{tb2}
\\
&&-\iint_{\Omega \times \mathbb{R}^{3}}(v_{i}^{2}-\beta _{b})\sqrt{\mu }%
\{\sum_{l}v_{l}\partial _{lj}\f _{b,k}^{j}\}f  \label{tb3} \\
&&-\e^2\sum_k\int_{\mathbb{R}^{3}} 2v_i v_k\int_{\Omega}\Phi_ib_k \pt_j\f%
_{b,k}^j.  \label{tb4}
\end{eqnarray}
By (\ref{jb}) and the trace estimate, for $k=2$, $|\pt_j\f_b^j|_2\le \|\f%
_b^j\|_{H^2}\le \|b\|_2$, for any $\eta>0$, the term (\ref{tb1}) is bounded
by 
\begin{equation}
(\ref{tb1})\le \frac 1{4\eta}(|(1-P_\g)f|^2_{2,+}+|r|_2^2)+\eta \|b\|_2^2.
\label{2463}
\end{equation}
For $k=6$, by the same argument as before, the term (\ref{tb1}) is bounded
by 
\begin{multline}
(\ref{tb1})\le \eta \|b\|_6^6+ \eta^{\prime}\|\P f \|_{L^{6}(\Omega \times 
\mathbb{R}^{3} )}+C_{\eta,\eta^{\prime}} (\e^{-\frac 1
2}|(1-P_\gamma)f|_{2})^{6} \\
+\eta^{\prime}( | \e^{\frac 12} w r |_{\infty} + \e^{\frac 32} \| \langle
v\rangle^{-1} w g \| _{\infty} + {\e^{-1}} \| (\mathbf{I} - \mathbf{P})f
\|_{L^{2}(\Omega \times \mathbb{R}^{3})})^{6}.  \label{Pgamma6b}
\end{multline}

The term (\ref{tb2}) is bounded, as in (\ref{penes}).

The term (\ref{tb3}) equals 
\begin{equation}
-{{\sum_{l}}}\int {{(v_{i}^{2}-\beta _{b})v_{l}^{2}\mu \partial _{lj}\f %
_{b,k}^{j}(x)b_{l}}}-\int (v_{i}^{2}-\beta _{b})v_{l}\sqrt{\mu }\partial
_{lj}\f _{b,k}^{j}(x)(\mathbf{I}-\mathbf{P})f.  \label{blap}
\end{equation}
We can choose $\beta _{b}>0$ such that for all $i$, 
\begin{equation}
\int_{\mathbb{R}^{3}}[(v_{i})^{2}-\beta _{b}]\mu (v) \mathrm{d} v=\int_{%
\mathbb{R}}[v_{1}^{2}-\beta _{b}]e^{-\frac{|v_{1}|^{2}}{2}}\mathrm{d}
v_{1}=0.  \label{alpha}
\end{equation}%
Note that for such chosen $\beta _{b}$, and for $i\neq k$, by an explicit
computation 
\begin{eqnarray*}
\int (v_{i}^{2}-\beta _{b})v_{k}^{2}\mu \mathrm{d} v &=&\int
(v_{1}^{2}-\beta _{b})v_{2}^{2}\frac{1}{2\pi }e^{-\frac{|v_{1}|^{2}}{2}}e^{-%
\frac{|v_{2}|^{2}}{2}}e^{-\frac{|v_{3}|^{2}}{2}} \mathrm{d} v=0, \\
\int (v_{i}^{2}-\beta _{b})v_{i}^{2}\mu \mathrm{d} v &=&\int_{\mathbb{R}%
}[v_{1}^{4}-\beta _{b}v_{1}^{2}]e^{-\frac{|v_{1}|^{2}}{2}}\mathrm{d} v_{1}
\neq 0.
\end{eqnarray*}
The first term in (\ref{blap}) becomes 
\begin{eqnarray}
&&-\iint_{\Omega \times \mathbb{R}^{3}}(v_{i}^{2}-\beta _{b})v_{i}^{2}\mu 
\mathrm{d} v\partial _{ij}\f _{b,k}^{j}(x)b_{i}+\sum_{k\neq i}\underbrace{%
\int_{\mathbb{R}^{3}}(v_{i}^{2}-\beta _{b})v_{k}^{2}\mu }_{=0}\int_{\Omega
}\partial _{kj}\f _{b,k}^{j}(x)b_{k}  \notag \\
&&= 2\sqrt{2\pi }\int_{\Omega }(\partial _{i}\partial _{j}\Delta
^{-1}b^{k-1}_{j})b_{i} .  \label{333}
\end{eqnarray}
The second term in (\ref{blap}), for any $\eta>0$ is bounded by 
\begin{equation}
\text{second term in (\ref{blap})}\lesssim \eta\|b\|_k^k +\frac 1{4\eta}\|(%
\mathbf{I}-\mathbf{P})f\|_2^2.  \label{2469}
\end{equation}
The term (\ref{tb4}) is bounded by 
\begin{equation}
(\ref{tb4})\lesssim \e^2\|\Phi\|_\infty \|b\|_k^k.  \label{2466}
\end{equation}

Collecting the bounds (\ref{rhsbij}), (\ref{2463}), (\ref{333}), (\ref{2469}%
), (\ref{2466}), we have the following estimate for all $i,j$: 
\begin{equation}
\begin{split}
\left| \int_{\Omega }\partial _{i}\partial _{j}\Delta ^{-1}b_{j}b_{i}\right|
\lesssim& \ \Big[(\e^{-2}+\e^2\|\Phi\|_\infty)\| (\mathbf{I}-\mathbf{P})f\|
_{\nu}^{2}+ |(1-P_{\gamma })f|_{2,+}^{2}+\| \frac{g}{\sqrt{\nu}}%
\|_{2}^{2}+|r|_{2}^{2}\Big] \\
&+(\eta+\e^2\|\Phi\|_\infty )\| b\| _{2}^{2} +o(1) \{\|\mathbf{P}f\|_2^2 +
\| ( \mathbf{I} - \mathbf{P})f\|_\nu^2 \} .  \label{stiijjkiki}
\end{split}%
\end{equation}

\vspace{8pt} Collecting the estimates (\ref{rhsbij6}), (\ref{Pgamma6b}), (%
\ref{333}), (\ref{2469}), (\ref{2466}), for $\e$ sufficiently small we
obtain 
\begin{equation}
\begin{split}
\left| \int_{\Omega }\partial _{i}\partial _{j}\Delta
^{-1}b^5_{j}b_{i}\right| \lesssim& \ \eta (\|b\|_6^6+ \|\P f\|_6^6)+
C_{\eta.\eta^{\prime}}\Big[(\e^{-2}+\e^2\|\Phi\|_\infty)\| (\mathbf{I}-%
\mathbf{P})f\| _{\nu}^{2}+ \e^{-1}|(1-P_{\gamma })f|_{2,+}^{2} \\
+\| \frac{g}{\sqrt{\nu}}\|_{2}^{2}+|r|_{2}^{2}\Big]^3&+\eta^{\prime}( | \e%
^{\frac 12} w r |_{\infty} + \e^{\frac 32} \| \langle v\rangle^{-1} w g \|
_{\infty} + {\e^{-1}} \| (\mathbf{I} - \mathbf{P})f \|_{L^{2}(\Omega \times 
\mathbb{R}^{3})})^{6} .  \label{stiijjkiki6}
\end{split}%
\end{equation}

To estimate $\partial _{j}(\partial _{j}\Delta ^{-1}b^{k-1}_{i})b_{i}$ for $%
i\neq j $, we choose as test function in (\ref{weakformulation}) 
\begin{equation}
\psi =|v|^{2}v_{i}v_{j}\sqrt{\mu }\partial _{j}\f _{b,k}^{i}(x),\quad i\neq
j,  \label{phibij}
\end{equation}%
where $\f _{b,k}^{i}$ is given by (\ref{jb}). Clearly, the right hand side
of (\ref{weakformulation}) is again bounded by (\ref{rhsbij}) for $k=2$ and
and (\ref{rhsbij6}) for $k=6$. We substitute again (\ref{bsplit}) and (\ref%
{insidesplit}) into the left hand side of (\ref{weakformulation}). The $%
P_{\gamma }f$ contribution and $a,c$ contributions vanish again due to
oddness. With this choice of $\psi$, we have 
\begin{equation}
\int_{\mathbb{R}^3} \sqrt{\mu} |v|^{2}v_{i}v_{j}\sqrt{\mu } (a+b\cdot v+ c%
\frac{|v|^2-3}2)=0, \quad \text{ if } i\ne j.  \notag
\end{equation}
The contribution from the field is 
\begin{eqnarray*}
&&\e^2 \sum_\ell\iint_{\O \times\mathbb{R}^3}\pt_j\f_{b,k}^i \Phi_\ell \sqrt{%
\mu}[2v_\ell v_iv_j+|v|^2(v_i\delta_{j,\ell}+v_j\delta_{i,\ell})]f \\
&&\lesssim (\e^2\|\Phi\|_\infty +\eta)\|b\|_k^k+\e^2\|\Phi\|_\infty\frac
1{4\eta} \| (\mathbf{I}-\mathbf{P})f\| _{\nu}^{2}.
\end{eqnarray*}
The contribution from the term containing $\l +\e^{-1}(1-\mathfrak{r} )\nu$
is bounded again as (\ref{penes}). The boundary terms is bounded by (\ref%
{Pgamma6b})

Finally, the bulk term becomes 
\begin{multline}
-\iint_{\Omega \times \mathbb{R} ^{3}}|v|^{2}v_{i}v_{j}\sqrt{\mu }%
\{\sum_{\ell}v_{\ell}\partial _{\ell j}\f _{b}^{i}\}f =  \label{777} \\
-{\iint_{\Omega \times \mathbb{R}^{3}}|v|^{2}v_{i}^{2}v_{j}^{2}\mu \lbrack
\partial _{ij}\f _{b,k}^{i}b_{j}+\partial _{jj}\f_{b,k}^{i}(x)b_{i}]}
-\iint_{\Omega \times \mathbb{R}^{3}}|v|^{2}v_{i}v_{j}v_{\ell}\sqrt{\mu }
\partial _{\ell j}\f_{b}^{i}(x)[\mathbf{I}-\mathbf{P}]f.
\end{multline}

Note that the first term in (\ref{777}) is evaluated as $\int_{\Omega
}\{(\partial _{i}\partial _{j}\Delta ^{-1}b^{k-1}_{i})b_{j}+(\partial
_{j}\partial _{j}\Delta ^{-1}b^{k-1}_{i})b_{i}\}$, thus collecting the above
bounds we get a bound for $(\pt_j\pt_j\Delta^{-1}b^{k-1}_i)b_i$ which,
combined with (\ref{stiijjkiki}) for $k=2$, and with (\ref{stiijjkiki6}) for 
$k=6$, gives (\ref{bestimate}) and (\ref{bestimate6}).

\bigskip

\noindent \textit{Step 4. Estimate of \ }$a$. {We claim that, for $\e$
sufficiently small, 
\begin{equation}
\| \mathring a \| _{2}^{2}\lesssim \e^{-2} \| (\mathbf{I}-\mathbf{P})f\|
_{2}^{2}+|(1-P_{\gamma })f|_{2,+}^{2}+|r|_{2}^{2}+\| \frac{g}{\sqrt{\nu}}\|
_{2}^{2}+\e^2\|\Phi\|_\infty( \|\mathbf{P} \mathring f\|_2^2+|<f>|^2).
\label{aestimate}
\end{equation}%
} 
\begin{multline}
\| \mathring a \| _{6}^{6}\lesssim \eta(\|a\|_6^6+\|\P f\|_6^6 )+
C_{\eta,\eta^{\prime}}\Big(\|(\mathbf{I}-\mathbf{P}) f\|_6+\e^{-1} \| (%
\mathbf{I}-\mathbf{P})f\| _{2}+\e^{-\frac 12}|(1-P_{\gamma
})f|_{2,+}+|r|_{2}+\| \frac{g}{\sqrt{\nu}}\| _{2}\Big)^{6}+ \\
\eta^{\prime}\Big( | \e^{\frac 12} w r |_{\infty} + \e^{\frac 32} \| \langle
v\rangle^{-1} w g \| _{\infty}\Big)^{6} .  \label{aestimate6}
\end{multline}

We choose a test function%
\begin{equation}
\psi =\psi _{a,k}\equiv (|v|^{2}-\beta _{a})v\cdot \nabla _{x}\f _{a}\sqrt{
\mu }=\sum_{i=1}^{d}(|v|^{2}-\beta _{a})v_{i}\partial _{i}\f_{a,k}\sqrt{\mu }%
,  \label{phia}
\end{equation}
where 
\begin{equation}
-\Delta _{x}\f _{a,k}(x)=\mathring a(x)-\fint a^{k-1},\quad\frac{\partial }{%
\partial n}\f _{a,k}|_{\partial \Omega }=0, \quad {\fint\f_{a,k}=0}.  \notag
\end{equation}

For $k=2$ it follows from the elliptic estimate
that $\| \f _{{a_k}}\| _{H^{2}}\lesssim \| \mathring a\| _{2}.$ Since $\int_{%
\mathbb{R}^{3}}(\frac{|v|^{2}}{2}-\frac{3}{2})(v_{i})^{2}\mu (v) \mathrm{d}
v\neq 0$, we choose $\beta _{a}=10>0$ so that, for all $i$, 
\begin{equation}
\int_{\mathbb{R}^{3}}(|v|^{2}-\beta _{a})(\frac{|v|^{2}}{2}-\frac{3}{2}
)(v_{i})^{2}\mu (v)=0.  \label{betaalpha}
\end{equation}%

Plugging $\psi _{a}$ into (\ref{weakformulation}), we bound its right hand
side by 
\begin{equation}
\text{r.h.s (\ref{weakformulation})}\lesssim\| \mathring a\| _{2}\big\{\e%
^{-1}\| (\mathbf{I}-\mathbf{P})f\| _{2}+\| g\| _{2}\big\}.  \label{rweaka}
\end{equation}
For $k=6$ we have the bound 
\begin{equation}
\text{r.h.s.(\ref{weakformulation})}\lesssim \eta \|\mathring a\|_6^6
+C_\eta \left(\e^{-1}\|(\mathbf{I}-\mathbf{P}) f\|_\nu + \|g\|_2\right)^6.
\label{boundrhsa3}
\end{equation}%
By (\ref{bsplit}) and (\ref{insidesplit}), since the $c$ contribution
vanishes in (\ref{weakformulation}) due to our choice of $\beta _{a}$ and
the $b$ contribution vanishes in (\ref{weakformulation}) due to the oddness,
the left hand side of (\ref{weakformulation}) takes the form of 
\begin{eqnarray}
&&\sum_{i=1}^{d}\int_{\gamma }\{n\cdot {v}\}(|v|^{2}-\beta _{a})v_{i}\sqrt{%
\mu }\partial _{i}\f _{a,k}(x)[P_{\gamma }f+(I-P_{\gamma })f\mathbf{1}%
_{\gamma _{+}}+r\mathbf{1}_{\gamma _{+}}]\ \ \ \ \ \   \label{aboundary} \\
&&- \iint_{\Omega \times \mathbb{R}^{3}}[(\l +(1-\mathfrak{r} )\e^{-1}\nu]f
(|v|^{2}-\beta _{a})\sqrt{\mu }\sum_i v_i\pt_i\f _{a,k}  \label{888} \\
&&-\sum_{i,\ell=1}^{d}\iint_{\Omega \times \mathbb{R}^{3}}(|v|^{2}-\beta
_{a})v_{i}v_{\ell}\partial _{i\ell}\f _{a,k}(x)a(x)\mu (v)  \label{abulk} \\
&& -\e^2\sum_{i,\ell}\iint_{\Omega \times \mathbb{R}^{3}}\Phi_\ell[%
2v_iv_\ell+(|v|^2-\beta_a)\delta_{i,\ell}]\{a+c\frac{|v|^2-3}{2}\}\mu\pt_i\f%
_a  \label{fora} \\
&&-\sum_{i,\ell =1}^{d}\iint_{\Omega \times \mathbb{R}^{3}}(|v|^{2}-\beta
_{a})v_{i}v_{\ell}\partial _{i\ell}\f _{a,k}(x)(\mathbf{I}-\mathbf{P})f
\label{a2bulk}
\end{eqnarray}%
We make an orthogonal decomposition at the boundary, ${v}_{i}=({v}\cdot
n)n_{i}+({v}_{\perp })_{i}=v_{n}n_{i}+({v}_{\perp })_{i}$. The contribution
of $P_{\gamma }f=z_{\gamma }(x)\sqrt{\mu }$ in (\ref{aboundary}) is 
\begin{equation*}
\int_{\gamma }(|v|^{2}-\beta _{a}){v}\cdot \nabla _{x} {\f} _{a,k } {v}%
_{n}\mu z_{\gamma } = \int_{\gamma }(|v|^{2}-\beta _{a}){v}_{n}\frac{%
\partial {\f} _{a,k}}{\partial n}{v}_{n}\mu z_{\gamma } +\int_{\gamma
}(|v|^{2}-\beta _{a})v_{\bot }\cdot \nabla _{x} {\f} _{a,k}{v}_{n}\mu
z_{\gamma }.
\end{equation*}
The first term vanishes by the Neumann boundary condition, 
while the second term also vanishes due to the oddness of $({v}_{\bot })_{i}{%
v}_{n}$ for all $i$. Therefore, for $k=2$, (\ref{aboundary}) and (\ref%
{a2bulk}) are bounded by $\| \mathring a\| _{2}\big\{\| (\mathbf{I}-\mathbf{P%
})f\| _{2}+|(1-P_{\gamma })f|_{2,+}+|r|_{2}\big\}$. The term (\ref{888}) is
bounded, as before, by (\ref{penes}). %

The term (\ref{abulk}), for $\ell\neq i$ vanishes due to the oddness. Hence
we only have the $\ell=i$ contribution: 
\begin{equation}
\sum_{i=1}^{d}\iint_{\Omega \times \mathbb{R}^{3}}(|v|^{2}-\beta
_{a})(v_{i})^{2}\mu (\partial _{ii}\f_{a})a=\sum_{i=1}^{d}\iint_{\Omega
\times \mathbb{R}^{3}}(|v|^{2}-\beta _{a,2})(v_{i})^{2}\mu (\partial _{ii}\f %
_{a,2}) \mathring a=-5\| \mathring a\|_2^2,  \notag
\end{equation}
because $\int (|v|^2-10)v_i^2\mu\neq 0$ and $\sum_i\int_\O dx \pt_{ii}\f%
_{a,2}=\int_{\pt\O } \pt_n \f_{a,2}=0$. Finally, the term (\ref{fora}) is
bounded by%
\begin{equation}
\e^2 \| \Phi\|_\infty\| \mathring a\|_2(\| \mathring a\|_2+|<f>|+\|c\|_2). 
\notag
\end{equation}
Using $-\Delta _{x}\f _{a}= \mathring a$, (\ref{estava}), and (\ref%
{cestimate}) we obtain 
\begin{equation}
\| \mathring a\| _{2}^{2}\lesssim {\ \e^{-2} \| (\mathbf{I}-\mathbf{P})f\|
_{2}^{2} }+|(1-P_{\gamma })f|_{2,+}^{2}+|r|_{2}^{2}+\| \frac{g}{\sqrt{\nu}}%
\| _{2}^{2} +o(1) \{\|\mathbf{P}f\|_2^2 + \| ( \mathbf{I} - \mathbf{P}%
)f\|_\nu^2 \}.  \notag  \label{aestimate0}
\end{equation}
Since $\|\mathbf{P}f\|_2^2\le \|\mathbf{P}\mathring f\|_2^2+|<f>|^2$, we
conclude (\ref{aestimate}). Finally we conclude (\ref{P1-P2}). The case $k=6$
is handled in a similar way using the same estimates as for $b$ and $c$. The
only term we have to check is 
\begin{multline}
\sum_{i=1}^{d}\iint_{\Omega \times \mathbb{R}^{3}}(|v|^{2}-\beta
_{a})(v_{i})^{2}\mu (\partial _{ii}\f_{a,6})a=\sum_{i=1}^{d}\iint_{\Omega
\times \mathbb{R}^{3}}(|v|^{2}-\beta _{a})(v_{i})^{2}\mu (\partial _{ii}\f %
_{a,6}) \mathring a  \notag \\
\sum_{i=1}^{d}\iint_{\Omega \times \mathbb{R}^{3}}(|v|^{2}-\beta
_{a})(v_{i})^{2}\mu \Big( \mathring a^5-\fint \mathring a^5) \Big) \mathring
a =-5\| \mathring a\|_6^6,
\end{multline}
where the first equality is due to again to $\sum_i\int_\O dx \pt_{ii}\f%
_{a,6}=\int_{\pt\O } \pt_n \f_{a,6}=0$ and the second to $\fint \mathring
a=0 $. Since 
\begin{equation}
\| (\mathbf{I}-\mathbf{P} )f\| _{6 }^6\le [\e^{-2}\|(\mathbf{I}-\mathbf{P})
f\|_2^2][\e^2\|(\mathbf{I}-\mathbf{P}) f\|_\infty^4]\le C_\eta(\e^{-1}\|(%
\mathbf{I}-\mathbf{P}) f\|_2)^6 +\eta (\e^{\frac 1 2}\|w f\|_\infty)^6,
\end{equation}
we obtain (\ref{P1-P6}). %
\end{proof}

Now we are ready to prove the main result of this section:

\begin{proof}[\textbf{Proof of Proposition \protect\ref{linearl2}}]
\textit{Step 1.} We claim that for any $\lambda >0$ and $0<\e \ll1$, there
exists a (unique) solution to 
\begin{equation}  \label{linpro2}
\l f^{\lambda} + v\cdot \nabla_x f^{\lambda} +\e^2\Phi\cdot \nabla_v
f^{\lambda}-\frac 1 2 \e^{2} \Phi\cdot v f^{\lambda} +\e^{-1}Lf^{\lambda} =g
\quad \text{in }\O \times\mathbb{R}^3, \ \ f^{\lambda}\big|_{\gamma_-}=P_\g
f^{\lambda}+r.
\end{equation}
Moreover, 
\begin{equation}  \label{rg2}
<f^{\lambda}> =0, \ \ \| \mathbf{P} f^{\lambda} \|_{2}^{2} + \e^{-2}\| (%
\mathbf{I}-\mathbf{P}) f^{\lambda} \| ^2_{\nu}+\e^{-1}|(1-P_\g)f^{\lambda}
|^2_{2}\lesssim \| \nu^{- \frac{1}{2}} (\mathbf{I} - \mathbf{P}) g\|_{2}^{2}
+ {\ \e^{-2} } \|\mathbf{P} g \|_{2}^{2} +|r|^2_{2}.
\end{equation}

From (\ref{linpro1}), a solution to (\ref{linpro2}) is a fixed point of the
map 
\begin{equation}  \label{fix_f}
f^{\lambda} \ \mapsto \ \mathcal{L}^{-1}\big[ \e^{-1}Kf ^{\lambda}+g\big].
\end{equation}
Note that from Lemma \ref{calL}, the operator $\mathcal{L}^{-1}$ is
well-defined and bounded. Hence for any $f^{\lambda}\in L^{2}$ there is $%
h\in L^{2}$ such that $f^{\lambda}=\mathcal{L}^{-1}h$. Thus (\ref{fix_f}),
the fixed point problem for $f^{\lambda}$, is equivalent to the fixed point
problem for $h$: 
\begin{equation}  \label{fix_h}
h \ \mapsto \ \e^{-1}K \mathcal{L}^{-1}h +g.
\end{equation}

In view of the application of the Schaefer's fixed point Theorem (\cite{EV},
page 504), to show the existence of the fixed point, we need to show that $K 
\mathcal{L}^{-1}$ is compact, which is proven in Lemma \ref{compact1}, {\
and the following \textit{a priori} uniform bound: if $h^{\mathfrak{r} }$
solves 
\begin{equation}
h^{\mathfrak{r} } \ = \ \mathfrak{r} \e^{-1} K \mathcal{L}^{-1} h^{\mathfrak{%
r} } + g, \ \ \text{for some } \mathfrak{r} \in[1^{-},1],  \label{unif_h}
\end{equation}
then $\|h^{\mathfrak{r} }\|_2$ is bounded uniformly in $\mathfrak{r} $.}
Since $\mathcal{L}^{-1}$ is bounded, it suffices to show uniform bound of $%
f^{\mathfrak{r} } = \mathcal{L}^{-1} h^{\mathfrak{r} }$ solving (\ref{eq0})
with $f=f^{\mathfrak{r} }$.

By the Green's identity, 
\begin{equation}
\begin{split}
& \lambda \| f^{\mathfrak{r} } \|_{2}^{2}+\mathfrak{r} \e^{-1} \| (\mathbf{I}
- \mathbf{P}) f^{\mathfrak{r} } \|_{\nu}^{2} {- o(1) \| \mathbf{P} f^{%
\mathfrak{r} }\|_{2}^{2}} + |(1-P_\g) f ^{\mathfrak{r} }|_{2,+}^2 \\
& \lesssim o(1) \| f^{\mathfrak{r} }\|_{\nu}^{2} + \e \| \nu^{- \frac{1}{2}}
(\mathbf{I} - \mathbf{P}) g\|_{2}^{2} + \e^{-1} \|\mathbf{P} g \|_{2}^{2} +
|r|_{2,-}^2 + \e^2\|\Phi \|_{\infty} \| f^{\mathfrak{r} }\|_{\nu}^2 .
\label{enerineq256}
\end{split}
\notag
\end{equation}
From $\| f^{\mathfrak{r} } \|_{\nu} \lesssim \| \mathbf{P} f^{\mathfrak{r}
}\|_{\nu} + \| (\mathbf{I} - \mathbf{P}) f^{\mathfrak{r} } \|_{\nu} \lesssim
\| f^{\mathfrak{r} } \|_{2} +\| (\mathbf{I} - \mathbf{P}) f^{\mathfrak{r} }
\|_{\nu}$, we have, for $\mathfrak{r} \sim 1$ and $\e\ll 1$, 
\begin{equation}  \label{energy_f_theta}
\begin{split}
& \lambda \| f^{\mathfrak{r} } \|_{2}^{2} + \mathfrak{r} \e^{-1} \| (\mathbf{%
I} - \mathbf{P}) f^{\mathfrak{r} } \|_{\nu}^{2}+ |(1-P_\g) f ^{\mathfrak{r}
}|_{2,+}^2{- o(1) \| \mathbf{P} f^{\mathfrak{r} }\|_{2}^{2}} \\
& \lesssim \ \e \| \nu^{- \frac{1}{2}} (\mathbf{I} - \mathbf{P}) g\|_{2}^{2}
+ \e^{-1}\|\mathbf{P} g \|_{2}^{2} + |r|_{2,-}^{2} .
\end{split}%
\end{equation}
Therefore we obtain an uniform in $\mathfrak{r} $ bound on $\|f^\mathfrak{r}
\|_2$. Since $f^\mathfrak{r} =\mathcal{L}^{-1}h^\mathfrak{r} $, from (\ref%
{unif_h}), we have 
\begin{equation}
h^\mathfrak{r} =\mathfrak{r} \e^{-1} K f^{ {\mathfrak{r} }} {+}g,
\end{equation}
so, $\|h^\mathfrak{r} \|_2$ is also bounded uniformly in $\mathfrak{r} $.
Note that in this argument $\e$ is fixed. Therefore, by the Schaefer's fixed
point Theorem there is a fixed point $h^\l$ for (\ref{fix_h}) and in
consequence, a fixed point $f^\l=\mathcal{L}^{-1}h^\l$ for (\ref{fix_f}).
Thus, we conclude the existence of a solution $f^{\lambda}$ to (\ref{linpro2}%
).

Now we prove the first identity of (\ref{rg2}). Estimating $[f^{\lambda}-f^{%
\mathfrak{r} }]$ using the Green's identity, 
\begin{eqnarray*}
&&\lambda \| f^{\lambda}-f^{\mathfrak{r} } \|_{2}^{2} + \mathfrak{r} \e^{-1}
\| (\mathbf{I} - \mathbf{P}) [f^{\lambda}- f^{\mathfrak{r} }] \|_{\nu}^{2}+
|(1-P_\g) [f^{\lambda}-f ^{\mathfrak{r} } ]|_{2,+}^2 \\
& \lesssim& O(1-\mathfrak{r} ) \e^{-1} \big\{ \| f^{\mathfrak{r} }
\|_{2}^{2} + \| f^{\lambda}-f^{\mathfrak{r} } \|_{2}^{2} + \| (\mathbf{I} - 
\mathbf{P})f^{\mathfrak{r} } \|_{\nu}^{2} + \| (\mathbf{I} - \mathbf{P})
[f^{\lambda}-f^{\mathfrak{r} }] \|_{\nu}^{2} \big\} \ \rightarrow \ 0 , \ \ 
\mathfrak{r} \uparrow 1 .
\end{eqnarray*}
From the above estimate and (\ref{rg}), for fixed $\lambda>0$, 
\begin{equation*}
|<f^{\lambda} >| = \lim_{\mathfrak{r} \rightarrow 1} |<f^{\mathfrak{r} }>| \
\lesssim_{\lambda} \ \lim_{\mathfrak{r} \rightarrow 1} (1- \mathfrak{r} ) \e%
^{-1} \big\{ \| \frac{g}{\sqrt{\nu}}\|_{2} + |r|_{2,-} \big\} =0.
\end{equation*}

By Lemma \ref{steady_abc} to (\ref{linpro2}) and from the first identity of (%
\ref{rg2}), 
\begin{equation}
\|\mathbf{P} f^{\lambda} \| _{2}^{2}\lesssim \e^{-2}\| (\mathbf{I}-\mathbf{P}%
)f^{\lambda}\| _{2}^{2}+|(1-P_{\gamma })f^{\lambda}|_{2,+}^{2}+\| \frac{g}{%
\sqrt{\nu}}\| _{2}^{2}+|r|_{2,-}^{2}.  \label{Pestimate2}
\end{equation}
The second estimate of (\ref{rg2}) is direct consequence of (\ref{Pestimate2}%
) and (\ref{energy_f_theta}) with $\mathfrak{r} \uparrow 1$.

\vspace{8pt}

\noindent\textit{Step 2.} To show the existence of the solution to (\ref%
{linearf}) we take the limit as $\l \to 0$ for $f^\l$ solving (\ref{linpro2}%
). Using (\ref{rg}), the uniform-in-$\lambda$ estimate, we have $f^{\lambda}
\rightharpoonup f$ weakly in $L^{2}$ where $f$ solves the linear problem (%
\ref{linearf}) with the estimate (\ref{linear_steady}).

Moreover, since $< f^\l>=0$, then also $<f>=0$ and we conclude (\ref{0mass_s}%
).

The difference $[f-f^\l]$ satisfies 
\begin{equation}
\begin{split}
\l [f-f^\l]+v\cdot \nabla_x [f-f^\l]+ \e^{2}\Phi\cdot \nabla_v
[f-f^\l]-\frac 1 2( \e^{2}\Phi\cdot v )[f-f^\l]+\e^{-1}L[f-f^\l]=\l f,& 
\notag \\
\ [f-f^\l]\big|_{\gamma_{-}} =P_\g [f-f^\l].&
\end{split}%
\end{equation}
By the Green's identity and Lemma \ref{steady_abc}, 
\begin{equation}
\|f-f^\l\|^2_2 \ \lesssim \ \l \|f\|_2^2.  \notag
\end{equation}
Therefore $f^\l$ converges strongly to $f$. The uniqueness follows using the
same argument with $\l =0$. 

\vspace{8pt} {\ \noindent\textit{Step 3.} By Lemma \ref{trace_s}, 
\begin{eqnarray*}
| P_{\gamma} f|_{2,+}^{2} &\lesssim& | \mathbf{1}_{\gamma_{+}^{\delta}}
P_{\gamma} f |_{2,+}^{2} \lesssim | \mathbf{1}_{\gamma_{+}^{\delta}} f
|_{2,+}^{2} + | (1- P_{\gamma}) f |_{2,+}^{2} \\
&\lesssim& \| f \|_{2}^{2} + \e^{-1} \| L (\mathbf{I} - \mathbf{P}) f f
\|_{1} + \| \frac{g}{\sqrt{\nu}}\|_{2}^{2} + | (1- P_{\gamma}) f |_{2,+}^{2}
\\
&\lesssim& \| \mathbf{P}f \|_{2}^{2} + \e^{-1} \| (\mathbf{I} - \mathbf{P})
f \|_{\nu}^{2}+ \| \frac{g}{\sqrt{\nu}} \|_{2}^{2}+ | (1- P_{\gamma}) f
|_{2,+}^{2} .
\end{eqnarray*}
For the incoming part, using the boundary condition, 
\begin{equation*}
| f|_{2,-}^{2} \lesssim \ |P_{\gamma} f |_{2,+}^{2} + |r |_{2,-} ^{2}
\lesssim\| \mathbf{P}f \|_{2}^{2} + \e^{-1} \| (\mathbf{I} - \mathbf{P}) f
\|_{\nu}^{2}+ \| \frac{g}{\sqrt{\nu}} \|_{2}^{2} + | (1- P_{\gamma}) f
|_{2,+}^{2} + |r |_{2,-} ^{2} .
\end{equation*}
Combing (\ref{rg2}), (\ref{Pestimate2}), and the above estimates, we
conclude (\ref{linear_steady}).}
\end{proof}

\begin{proof}[\textbf{Proof of Theorem \protect\ref{prop_linear_steady}}]
We only need to prove (\ref{linear_steady0}). Using (\ref{point1_s}) in
Proposition \ref{point_s} to bound $\e^{\frac{1}{2}} \| w f \|_{\infty}$ in (%
\ref{linear_steady}), we conclude, for $\e$ sufficiently small, 
\begin{equation}
\| \mathbf{P} f\|_{6} \lesssim \e^{-1}\| (\mathbf{I}-\mathbf{P} )f\| _{\nu }
+\e^{-\frac 1 2}|(1-P_{\gamma })f|_{2,+}+\| \frac{g}{\sqrt{\nu}}\| _{2}+\|\e%
^{\frac 3 2} \langle v\rangle^{-1}w g\|_\infty +|\e^{\frac 1 2}w
r|_\infty+|r|_{2,-}.  \label{P1-P6ter}
\end{equation}

From (\ref{point1_s}), (\ref{completestimate}), and (\ref{P1-P6ter}) we
conclude (\ref{linear_steady0}).
\end{proof}

\vspace{8pt}

\subsection{Validity of the Steady Problem}

The main purpose of this section is to prove Theorem \ref{mainth}. We need
several estimates before the proof of the main Theorem.

{\ }

\begin{lemma}
\label{pl3l6} 

Recall {the expression of } $\P f$ in \ref{Pabc}. 
Then, for $w= e^{\beta|v|^{2}}$, $0 < \beta \ll 1$, 
\begin{equation}  \label{Gamma_t_s}
\begin{split}
& \| \nu^{- \frac{1}{2}} \Gamma_{\pm} (f,g)\|_{L_{x,v}^{2}} \\
\lesssim \ \ \, & \e^{1/2}\big\{ \e^{1/2}\| w g \|_{ {\infty} } \big[ \e %
^{-1} \| \nu^{- \frac{1}{2}} (\mathbf{I}-\mathbf{P}) f \|_{L_{x,v}^{2}} \big]%
+ \e^{1/2} \| w f \|_{ {\infty} } \big[ \e ^{-1} \| \nu^{- \frac{1}{2}} (%
\mathbf{I}-\mathbf{P}) g \|_{L_{x,v}^{2}} \big]\big\} \\
& +\| \P f\|_{L^{6}_{x,v}}\|\P g\|_{L^{3}_{x,v}}.
\end{split}%
\end{equation}
\end{lemma}

\begin{proof}
By the decomposition 
\begin{eqnarray}
&&\| \nu^{- \frac{1}{2}} \Gamma_{\pm} (f, g)\|_{L^{2}_{x,v}}
  \label{Gamma_decom_s} \\
&\lesssim& \| \nu^{- \frac{1}{2}} \Gamma_{\pm} ( |(\mathbf{I}-\mathbf{P})
f|, |g|) \|_{L^{2}_{x,v}}  
 +
\| \nu^{- \frac{1}{2}} \Gamma_{{\pm}} ( |f|,|(\mathbf{I}-\mathbf{P}) g |
) \|_{L^{2}_{x,v}}
 +  \| \nu^{- \frac{1}{2}} \Gamma_{\pm} (|\P f| ,|\P g|) \|_{L^{2}_{x,v}}.
 \notag
\end{eqnarray}

The first two terms of the RHS of (\ref{Gamma_decom_s}) are bounded by 
\begin{eqnarray*}
 \e \| wg \|_{L^{\infty}_{x,v}}
\| \nu^{-1/2} \Gamma_{\pm}( \e%
^{-1}|(\mathbf{I}-\mathbf{P})f|, w^{-1} )\|_{L^{2}_{x,v}} + \e \| wf \|_{L^{\infty}_{x,v}}
\| \nu^{-1/2} \Gamma_{\pm}( \e%
^{-1}|(\mathbf{I}-\mathbf{P})g|, w^{-1} )\|_{L^{2}_{x,v}}.\end{eqnarray*}
From $|v|^{2} + |u|^{2} = |v^{\prime}|^{2} + |u^{\prime}|^{2}$ and $%
\nu^{-1/2} |(v-u) \cdot \omega| \sqrt{\mu(u)} \lesssim \nu^{-1/2} [ |v| +
|u|] \sqrt{\mu(u)} \lesssim [1+ |v| + |u|]^{ \frac{1}{2}} \mu(u)^{\frac{1}{2}%
-}, $ 
\begin{eqnarray}
&& \int_{\mathbb{R}^{3}} \nu^{-1} | \Gamma_{\pm}( \e^{-1} |(\mathbf{I} -%
\mathbf{P})f|, w^{-1}) (v)|^{2} \mathrm{d} v \notag \\
&\lesssim& \int_{\mathbb{R}^{3}} \int_{\mathbb{R}^{3}} \int_{\mathbb{S}^{2}}
[1+ |v^{\prime}| + |u^{\prime}|] | \e^{-1} (\mathbf{I} -\mathbf{P})
f(v^{\prime}) |^{2} w(u^{\prime})^{-2} \mathrm{d} \omega \mathrm{d} u 
\mathrm{d} v  \label{I-P_linear} \\
&&+ \int_{\mathbb{R}^{3}} \int_{\mathbb{R}^{3}} \int_{\mathbb{S}^{2}} [1+
|v^{\prime}| + |u^{\prime}|] | \e^{-1} (\mathbf{I} -\mathbf{P}) f(u^{\prime}
) |^{2} w(v^{\prime})^{-2} \mathrm{d} \omega \mathrm{d} u \mathrm{d} v 
\notag \\
&&+ \int_{\mathbb{R}^{3}} \int_{\mathbb{R}^{3}} \int_{\mathbb{S}^{2}} [1+ |v
| + |u |] | \e^{-1}(\mathbf{I} -\mathbf{P}) f(v) |^{2} w(u )^{-2} \mathrm{d}
\omega \mathrm{d} u \mathrm{d} v  \notag \\
&&+ \int_{\mathbb{R}^{3}} \int_{\mathbb{R}^{3}} \int_{\mathbb{S}^{2}} [1+ |v
| + |u |] | \e^{-1}(\mathbf{I} -\mathbf{P}) f(u) |^{2} w(v)^{-2} \mathrm{d}
\omega \mathrm{d} u \mathrm{d} v.  \notag
\end{eqnarray}
Now by the change of variables $(v,u) \leftrightarrow (v^{\prime},
u^{\prime})$ for the first term, $(v,u) \leftrightarrow (u^{\prime},
v^{\prime})$ for the second term and $(v,u) \leftrightarrow (u,v)$ for the
last term, we bound all the above terms as 
\begin{eqnarray}
&& \int_{\mathbb{R}^{3}} \nu^{-1} | \Gamma_{\pm}( \e^{-1} |(\mathbf{I} -%
\mathbf{P})f|, w^{-1}) |^{2}
\label{Gamma_2} \\
&\lesssim& \int_{\mathbb{R}^{3}} \Big[ \iint_{\mathbb{R}^{3} \times \mathbb{S%
}^{2}} [1+ |v| + |u|] w(u)^{-1} \mathrm{d} \omega \mathrm{d} u \Big] | \e%
^{-1} (\mathbf{I} -\mathbf{P}) f (v)|^{2} \mathrm{d} v  \notag \\
& \lesssim& \int_{\mathbb{R}^{3}} \nu^{-1} | \e^{-1} (\mathbf{I} -\mathbf{P}%
) f (v)|^{2} \mathrm{d} v .  \notag
\end{eqnarray}
Similarly, 
\begin{eqnarray}
&& \int_{\mathbb{R}^{3}} \nu^{-1} | \Gamma_{\pm}( \e^{-1} |(\mathbf{I} -%
\mathbf{P})g|, w^{-1}) (v)|^{2} \mathrm{d} v\lesssim
\int_{\mathbb{R}^{3}} \nu^{-1}| \e^{-1} (\mathbf{I} -\mathbf{P})
g (v)|^{2} \mathrm{d} v .  \label{Gamma_3}
\end{eqnarray}
Therefore, the first two terms of the RHS of (\ref{Gamma_decom_s}) are
bounded by 
\begin{eqnarray*}
\e \| w g \|_{\infty} \| \e^{-1} (\mathbf{I} - \mathbf{P}) f \|_{\nu} + \e %
\| w f \|_{\infty} \| \e^{-1} (\mathbf{I} - \mathbf{P})g \|_{\nu} .
\end{eqnarray*}

Due to the strong decay in $v$ of $\mathbf{P}f$, we have $\big\|\frac{1}{%
\mu^{0+}}|\mathbf{P}f(x ,v)| \big\|_{L^{\infty}_{v}} \lesssim \|\mathbf{P}%
f(x) \|_{L^{p}_{v}} $ for any $1 \leq p \leq \infty$. The last term of (\ref%
{Gamma_decom_s}) is bounded as, for fixed $v$, by $\| \nu^{-1/2}\Gamma (
\mu^{0+} , \mu^{0+} ) \|_{L^{2}_{ v}}< \infty$, 
\begin{eqnarray*}
 \| \nu^{- \frac{1}{2}} \Gamma_{\pm} (\P f ,\P g) \|_{L^{2}_{x,v}} 
\lesssim
\| \nu^{-1/2} \Gamma ( \mu^{0+} , \mu^{0+}) \|_{L^{2}_{ v}} \big\| %
\| \mathbf{P} f (\cdot ) \|_{L^{6}_{v}} \| \P g (\cdot ) \|_{L^{3}_{v}} %
\big\|_{L^{2}_{x}}
\lesssim
 \| \mathbf{P} f \|_{L^{6}_{x,v} } \| \mathbf{P} g
\|_{L^{3}_{x,v} } .
\end{eqnarray*}

All together we prove (\ref{Gamma_t_s}).\end{proof}

\begin{lemma}
\label{rSQ_unst}
Recall $r_{s}, f_{w}, A_{s}, \mathcal{Q} $ in (\ref{r}), (\ref{fw}), (\ref{A}%
), (\ref{defQ}). We have 
\begin{eqnarray*}
&& | r _{s} |_{2,- } + |w r_s|_{\infty,-}\lesssim |\t_w|_\infty , \\
&& \| f_{w} \|_{L^{6}_{x} L^{2}_{v}} +\|wf_{w} \|_{\infty} \lesssim
|\t_w|_\infty, \\
&&\|(\mathbf{I}-\mathbf{P}) A_{s} \|_{L^{2} (\Omega \times \mathbb{R}^{3})}
\ \lesssim |\t_w |_{W^{1,\infty}(\pt \O )}+\e^2\|\Phi\|_\infty |\t_w|_2 , \\
&& {\|\P A_s\|_{L^{2} (\Omega \times \mathbb{R}^{3})}\le \e\|\Phi\|_2,} \\
&& \| w A_{s} \|_{L^{\infty} (\Omega \times \mathbb{R}^{3}) } \ \lesssim
|\t_w |_{W^{1,\infty}(\pt \O )}+\e^2\|\Phi\|_\infty |\t_w|_\infty+ {\e %
\|\Phi\|_\infty} + |\vartheta_{w}|^{2}_{\infty} , \\
&& | \mathcal{Q} f |_{2,-} \ \lesssim \ \| \vartheta_{w} \|_{L^{\infty}
(\partial\Omega)} \big[1+ \e \| \vartheta_{w} \|_{L^{\infty}
(\partial\Omega)} \big] \| \sqrt{\mu} f \|_{L^{2} (\gamma_{+})}, \\
&& \| w \mathcal{Q}f \|_{ \infty } \ \lesssim \ | \vartheta_{w} |_{\infty} %
\big[1+ \e | \vartheta_{w} |_{\infty} \big] \| \sqrt{\mu} f \|_{L^{\infty} (%
\bar{\Omega} \times \mathbb{R}^{3})}.
\end{eqnarray*}
\end{lemma}

\begin{proof}
From (\ref{r}) and (\ref{phie}) we have the first estimate. From (\ref{fw})
and (\ref{Thetaw}) we have the second estimate. From the (\ref{A}) we have
the third and fourth estimates. Finally,

From (\ref{zeromass1}), (\ref{defQ}) and (\ref{exp_Mw}) 
we have 
\begin{eqnarray*}
\mathcal{Q}f(x,v) &=& \sqrt{2\pi} \big( \frac{|v|^{2}}{2} -2 \big) \sqrt{%
\mu(v)} \vartheta_{w}(x) \int_{n (x) \cdot u>0} f(x,u) \sqrt{\mu(u)} \{n(x)
\cdot u\} \mathrm{d} u \\
& &+\e O( |\vartheta_{w}|^{2} )\langle v\rangle^{4} \sqrt{\mu(v)} \int_{n(x)
\cdot u>0} R(x,u) \sqrt{\mu(u)} \{n(x) \cdot u\} \mathrm{d} u.
\end{eqnarray*}
By the standard Sobolev embedding we prove the estimates.
\end{proof}

\vspace{4pt}

Now we are ready to prove the main theorem for the steady case:

\begin{proof}[\textbf{Proof of Theorem \protect\ref{mainth}}]
We prove Theorem \ref{mainth} by considering a sequence $f^{\ell}$, for $%
\ell\ge 0$, 
\begin{equation}
\begin{split}  \label{steady_ell}
v\cdot\nabla_x f^{\ell+1}+ \e^2 \frac{1}{\sqrt{\mu}}\Phi\cdot\nabla_v [\sqrt{%
\mu} f^{\ell+1} ] +\frac 1{\e}L f^{\ell+1} &= \Gamma(f^{\ell},f^{\ell})+ L_1
f^{\ell} +A_{s} , \\
f^{\ell+1} |_{\gamma_-}&=P_\g f^{\ell+1}+ \e \mathcal{Q} f^{\ell} +\e r
_{s}, \ \ \ R^{0} \equiv0,
\end{split}%
\end{equation}
where $L_{1}$ and $A_{s}$ are defined at (\ref{L1def}) and (\ref{A}), $%
\mathcal{Q}$ at (\ref{defQ}), and $r_{s}$ at (\ref{r}). 
Note that Theorem \ref{prop_linear_steady}, with (\ref{PAintA}), (\ref%
{Q_zero}), guarantees the solvability of such a linear problem (\ref%
{steady_ell}).

\vspace{4pt}

\noindent\textit{Step 1. } 
For $0< \eta_{0} \ll 1$, we assume that, 
\begin{equation}
\| \vartheta_{w} \|_{H^{1/2}(\partial\Omega)}^{2} + \| \Phi \|_{2}^{2} + \e( 
{\ \| \t_w \|^{2}_{W^{1,\infty} (\pt\Omega)}}+ \| \Phi \|_{\infty}^{2} ) <
c_{0}{\eta_{0}},
\end{equation}
for $0< c_{0} \ll1$, and the induction hypothesis 
\begin{equation}  \label{induc_hyp}
\sup_{0 \leq j \leq \ell} [\hskip-1pt [ f^{j} ]\hskip-1pt ]^{2} < \eta_{0},
\end{equation}
where the norm $[\hskip-1pt [ \ \cdot \ ]\hskip-1pt ]$ is defined in (\ref%
{norm_steady}).

We apply Theorem \ref{prop_linear_steady} for 
\begin{equation}
f= f^{\ell+1} , \ g=\Gamma(f^{\ell},f^{\ell}) + L_{1}f^{\ell} + A_{s} , \ r= %
\e \mathcal{Q} f^{\ell} + \e r_{s},  \notag
\end{equation}
to achieve the same upper bound as in (\ref{induc_hyp}) for $[\hskip-1pt [
f^{\ell+1} ]\hskip-1pt ]^{2}$ in the next two steps.

We estimate the right hand side of (\ref{linear_steady0}) for our case. From
(\ref{L1def}), $\mathbf{P} L_{1} f=0$ and 
\begin{equation}
| L_{1} f^{\ell}(x,v)| \lesssim |\Theta_{w}(x)| \big\{| \Gamma_{\pm} (
\langle v\rangle^{2} \sqrt{\mu}, \mathbf{P}f^{\ell} )| + | \Gamma_{\pm} (
\langle v\rangle^{2} \sqrt{\mu}, (\mathbf{I} - \mathbf{P})f^{\ell} )| \big\} %
.  \notag
\end{equation}
Then 
\begin{equation}  \label{small_L1}
\| \nu^{-1/2} (\mathbf{I} - \mathbf{P}) L_{1} f^{\ell} \|_{2}\lesssim \|
\Theta_{w} \|_{3} \| \mathbf{P} f^{\ell} \|_{6} + \e \| \Theta_{w}
\|_{\infty} [ \e^{-1} \| (\mathbf{I} - \mathbf{P}) f^{\ell} \|_{\nu}].
\end{equation}

Using (\ref{small_L1}), (\ref{Gamma_t_s}), (\ref{L1def}), (\ref{fw}) and
Lemma \ref{rSQ_unst}, we obtain 
\begin{eqnarray*}
\| \nu^{-\frac{1}{2}} (\mathbf{I} - \mathbf{P})g \|_{2} & \lesssim & [\hskip%
-1pt [ f^{\ell}]\hskip-1pt ]^{2} + (\| \Theta_{w} \|_{3} + \e \| \Theta_{w}
\|_{\infty} ) [\hskip-1pt [ f^{\ell} ]\hskip-1pt ] + \| \nabla_{x}\Theta_{w}
\|_{2} + \e^{2} \| \Phi \Theta_{w}\|_{2} + \| |\Theta_{w}|^{2} \|_{2}, \\
\e^{-1} \| \mathbf{P} g \|_{2} &\lesssim& \| \Phi \|_{2} , \\
\e^{3/2} \| \langle v\rangle^{-1} w g \|_{\infty} &\lesssim & \e^{1/2} [%
\hskip-1pt [ f^{\ell} ]\hskip-1pt ]^{2} + \e \| \Theta_{w} \|_{\infty} [%
\hskip-1pt [ f^{\ell} ]\hskip-1pt ] + \e^{5/2} \| \Phi \|_{\infty} + \e%
^{3/2} \| \nabla_{x} \Theta_{w} \|_{\infty} \\
&& + \e^{7/2} \| \Phi \|_{\infty} \| \Theta_{w} \|_{\infty} + \e^{3/2} \|
\Theta_{w} \|_{\infty} ^{2}.
\end{eqnarray*}

From Lemma \ref{rSQ_unst} 
\begin{eqnarray*}
\e^{-1/2}|r|_{2} &\lesssim& \e^{1/2} | \vartheta_{w}|_{\infty} [\hskip-1pt [
f^{\ell} ]\hskip-1pt ] + \e^{1/2} | \vartheta_{w} |_{\infty}, \\
\e^{1/2} |w r| _{\infty} &\lesssim& \e^{3/2} |\vartheta_{w}|_{\infty } [%
\hskip-1pt [ f^{\ell} ]\hskip-1pt ] + \e^{3/2} | \vartheta_{w} |_{\infty}.
\end{eqnarray*}

Finally applying Theorem \ref{prop_linear_steady}, we conclude that 
\begin{equation}  \label{uniform_linear}
\begin{split}
[\hskip-1pt [ f^{\ell+1} ]\hskip-1pt ]^{2} \lesssim& \big\{ (1+ \e)[\hskip%
-1pt [ f^{\ell } ]\hskip-1pt ]^{2} + (\| \Theta_{w} \|_{3} + \e \|
\Theta_{w} \|_{\infty} ) ^{2} + (\e^{2}+ \e^{3} )
|\vartheta_{w}|_{\infty}^{2} \big\}[\hskip-1pt [ f^{\ell } ]\hskip-1pt ]^{2}
+ c_{0} \eta_{0}.
\end{split}%
\end{equation}
By $\| \Theta_{w} \|_{3}\lesssim \| \Theta_{w} \|_{H^{1}(\O )} \lesssim |
\vartheta _{w} |_{H^{\frac{1}{2}} (\partial\Omega)} < c_{0} \eta_{0} \ll1 $
we prove that $[\hskip-1pt [ f^{\ell+1} ]\hskip-1pt ]^{2} < \eta_{0}$.

\vspace{4pt}

\noindent\textit{Step 2. } We repeat \textit{Step 1} for $f^{\ell+1} -
f^{\ell}$ to show that $R^{\ell}$ is Cauchy sequence in $L^{\infty} \cap
L^{2}$ for fixed $\e$. Now it is standard to conclude that the limiting $%
f^{\ell} \rightarrow f$ solves the equation. The uniqueness is standard.
(See \cite{EGKM} for the details)

\medskip

\noindent\textit{Step 3. } To prove the weak convergence of $f^\e$, we use
the argument of \cite{BGL91} where it is proved, in the unsteady case and
without boundary, that, if $f^\e$ converges weakly to a limit, then the
limit has to be in the null space of $L$ and its components have to solve
the INSF system. Adding a force field is straightforward. The boundary
condition issue requires a little more care. Let $g^\e=f_w+f^\e_s$. The
equation for $g^\e$ is 
\begin{equation}
v\cdot\nabla_x g^\e +\e^2 \frac{\Phi\cdot \nabla_v(g^\e\sqrt{\mu})}{\sqrt{\mu%
}}+\frac 1 \e L g^\e = \Gamma(g^\e,g^\e)+\e\Phi\cdot v\sqrt{\mu}.
\label{eqns}
\end{equation}

From the previous results we know that $\Vert (\mathbf{I}-\mathbf{P}%
)g^{\e}\Vert _{\nu }\rightarrow 0$ as $\e\rightarrow 0$. Moreover $\P g^{\e}$
is bounded in $L_{x}^{6}$ and hence weakly compact and $\langle v\rangle
^{-1}\Gamma (g^{\e},g^{\e})$ is bounded in $L_{x,v}^{2}$. Therefore 
\begin{equation*}
v\cdot \nabla _{x}({g^{\varepsilon }}{{\langle v\rangle }^{-1}})+\e%
^{2}\langle v\rangle ^{-1}{\mu }^{-\frac{1}{2}}\Phi \cdot \nabla _{v}({%
g^{\varepsilon }\sqrt{\mu }})\in L_{x,v}^{2}
\end{equation*}%
Passing to the (weak) limit as $\varepsilon \rightarrow 0$, up to subsequences,  $g^{\varepsilon
}\rightarrow g_{1}$ weakly and $\e^{2}\langle v\rangle ^{-1}{\mu }^{-\frac{1%
}{2}}\Phi \cdot \nabla _{v}({g^{\varepsilon }\sqrt{\mu }})\rightarrow 0$ in
the sense of distribution, so that, as distribution 
\begin{equation*}
v\cdot \nabla _{x}({g^{\varepsilon }}{{\langle v\rangle }^{-1}})+\e%
^{2}\langle v\rangle ^{-1}{\mu }^{-\frac{1}{2}}\Phi \cdot \nabla _{v}({%
g^{\varepsilon }\sqrt{\mu }})\rightarrow v\cdot \nabla _{x}({g_{1}}{{\langle
v\rangle }^{-1}})
\end{equation*}%
But $v\cdot \nabla _{x}({g^{\varepsilon }}{{\langle v\rangle }^{-1}})+\e%
^{2}\langle v\rangle ^{-1}{\mu }^{-\frac{1}{2}}\Phi \cdot \nabla _{v}({%
g^{\varepsilon }\sqrt{\mu }})$ has a weak limit in $L_{x,v}^{2}.$ By the
uniqueness of the distribution limit, we deduce that the limit $g_{1}=\P %
g_{1}$ is such that 
\begin{equation*}
v\cdot \nabla _{x}({g_{1}}{{\langle v\rangle }^{-1}})\in L_{x,v}^{2},\text{
\ \ \ }\|g_{1}\|_{L^{6}}\ll1.
\end{equation*}%
But $g_{1}=\{\rho +u\cdot v+\theta (|v|^{2}-3)/2\}\sqrt{\mu }$, and, from
the linear independence of $\frac{v}{\sqrt{\nu }}\{1,v,v\otimes
v,|v|^{2},v|v|^{2}\}\sqrt{\mu },$ we deduce that $\rho ,u,\theta \in
H_{x}^{1}$. The equation for the hydrodynamic fields are deduced as in \cite%
{BGL91} as follows: we apply $\P $ to equation (\ref{eqns}) and take the
weak limit to obtain that 
\begin{equation*}
\P (v\cdot \nabla _{x}g_{1})=0,
\end{equation*}%
which is equivalent to 
\begin{equation*}
\nabla _{x}(\rho +\th )=0,\quad \nabla _{x}\cdot u=0.
\end{equation*}%
Then we multiply equation (\ref{eqns}) by $\e^{-1}v\sqrt{\mu }$, integrate
on velocity and take the weak limit. We obtain 
\begin{equation*}
\lim_{\e\rightarrow 0}\e^{-1}\nabla _{x}\cdot (v\sqrt{\mu },v\cdot \nabla
_{x}g^{e})_{L_{v}^{2}}=\Phi .
\end{equation*}%
To compute the above limit we write 
\begin{equation*}
\e^{-1}\nabla _{x}\cdot (v\otimes v\sqrt{\mu },g^{e})_{L_{v}^{2}}=(L^{-1}(v%
\otimes v-\frac{|v|^{2}}{3}\mathbb{I}),\e^{-1}Lg^{e})_{L_{v}^{2}}+\nabla
_{x}p^{e},
\end{equation*}%
where $p^{e}=\e^{-1}(\frac{1}{3}|v|^{2}\sqrt{\mu },g^{e})_{L_{v}^{2}}$. In
order to compute the limit of the first term, we note that, from the
equation, the weak limit 
\begin{equation*}
\lim_{\e\rightarrow 0}\e^{-1}Lg^{e}=\Gamma (g_{1},g_{1})-v\cdot \nabla
_{x}g_{1}.
\end{equation*}%
Hence we obtain 
\begin{equation*}
\Phi =\nabla _{x}((v\otimes v-\frac{|v|^{2}}{3}\mathbb{I}),L^{-1}(\Gamma
(g_{1},g_{1})-v\cdot \nabla _{x}g_{1})_{L_{v}^{2}}+\nabla _{x}p,
\end{equation*}%
where $p=\lim_{e}p_{e}$. It is standard to compute that 
\begin{equation*}
\nabla _{x}\cdot (v\otimes v-\frac{|v|^{2}}{3}\mathbb{I}),L^{-1}(\Gamma
(g_{1},g_{1})-v\cdot \nabla _{x}g_{1})_{L_{v}^{2}}=u\otimes u-\mathfrak{v}%
\Delta u,
\end{equation*}%
and hence $u$ is a weak solution to the incompressible Navier-Stokes
equation (\ref{INSF_st})$_{1}$. Similar arguments can be used to obtain (\ref%
{INSF_st})$_{2}$. The conditions $\nabla _{x}(\rho +\th )$ and $\nabla
_{x}\cdot u$ have been already obtained.

We only need to check the boundary conditions. We return to $f^{\e}_s=
g^\e-f_w$. By the smoothness of $f_w$,

\begin{equation*}
v\cdot \nabla _{x}(f_{s}^{\e}\langle v\rangle ^{-1})+\e^{2}\langle v\rangle
^{-1}{\mu }^{-\frac{1}{2}}\Phi \cdot \nabla _{v}({f_{s}^{\varepsilon }\sqrt{%
\mu }})\rightarrow v\cdot \nabla _{x}(f_{1}\langle v\rangle ^{-1})\in
L_{x,v}^{2}\text{ weakly. }
\end{equation*}%
By Lemma \ref{trace_s}, ${f_{s}^{e}}\langle v\rangle ^{-1}$ has local trace
which weakly converges to $\langle v\rangle ^{-1}P_{\gamma }f_{1}$ because $%
(1-P_{\gamma })f_{s}^{\e}\langle v\rangle ^{-1}\rightarrow 0$. So $f_{1}$ has
a local trace $P_{\gamma }f_{1}$ on $\gamma $. But $f_{1}\in
H_{x}^{1}C_{v}^{\infty }$, so that, for each $v\in \mathbb{R}^{3}$ {the
components $\rho -\rho _{w}$, $u$ and $\th -\Theta _{w}$} of $\P f_{1}$ have
trace on the boundary. Hence $\P f_{1}=P_{\gamma }f_{1}$. Thus $u=0$ and $%
\theta ={\vartheta_{w}}$ on $\pt\O $. Recall $||g_{1}||_{L^{6}}\ll1,$ so is $%
||u||_{L^{6}}+||\theta ||_{L^{6}},$ hence all the weak limit points must coincide with the unique solution to the steady Navier-Stokes-Fourier solution.

\medskip

The positivity $F_{s}\geq 0$ is left for the unsteady case in Section 3.7
\end{proof}

\section{Unsteady Problems}

\subsection{Preliminary and the linear theorem}

\begin{definition}
Assume $\Phi = \Phi(x) \in C^{1}$. Consider a unsteady linear transport
equation 
\begin{equation}  \label{linear_dyn}
\e \partial_{t}f +v \cdot\nabla_{x} f + \e^{2} \Phi \cdot \nabla_{v} f \ = \
g .
\end{equation}
The equations of the characteristics for (\ref{linear_dyn}) are 
\begin{equation}  \label{eq_YW}
\dot{Y}= \e^{-1} W, \ \ \dot{W}= \e \Phi(Y), \ \ \ Y(t;t,x,v)=x, \
W(t;t,x,v) =v.
\end{equation}
By the uniqueness of ODE 
\begin{equation}  \label{YW}
\begin{split}
[Y(s;t,x,v), W(s;t,x,v)] & \ = \ [ X( t- \frac{t-s}{\e} ;t,x,v ),V(t- \frac{%
t-s}{\e} ;t,x,v) ] \\
& \ = \ [X( \e^{-1}{\ s} ; 0,x,v),V( \e^{-1}{\ s} ; 0,x,v) ],
\end{split}%
\end{equation}
where $(X,V)$ is defined in (\ref{char}).

Define 
\begin{equation}  \label{tb_dyn}
\begin{split}
\tilde{t}_{\mathbf{b}}(x,v)& \ := \ \sup \{t>0: Y(-s;0,x,v) \in \Omega \ 
\text{for all} \ 0<s<t \} \\
& \ = \ \e\sup \{\frac{t}{\e}>0: X(-\frac{s}{\e};0,x,v) \in \Omega \ \text{%
for all} \ 0<\frac{s}{\e}<\frac{t}{\e} \} \ = \ \e t_{\mathbf{b}}(x,v), \\
\tilde{t}_{\mathbf{f}}(x,v)& \ := \ \sup \{t>0: Y(s;0,x,v) \in \Omega \ 
\text{for all} \ 0<s<t \} \\
& \ = \ \e\sup \{\frac{t}{\e}>0: X( \frac{s}{\e};0,x,v) \in \Omega \ \text{%
for all} \ 0<\frac{s}{\e}<\frac{t}{\e} \} \ = \ \e t_{\mathbf{f}}(x,v).
\end{split}%
\end{equation}
Moreover 
\begin{equation}  \label{xv_dyn}
\begin{split}
\tilde{x}_{\mathbf{b}}(x,v) &= Y(- \tilde{t}_{\mathbf{b}} (x,v);0,x,v) = X(
- \frac{ \tilde{t}_{\mathbf{b}}(x,v)}{\e}; 0,x,v)= X(-t_{\mathbf{b}}(x,v);
0,x,v) = x_{\mathbf{b}}(x,v), \\
\tilde{x}_{\mathbf{f}}(x,v) &= Y(- \tilde{t}_{\mathbf{f}}(x,v);0,x,v) = X( - 
\frac{ \tilde{t}_{\mathbf{f}}(x,v)}{\e}; 0,x,v)= X(-t_{\mathbf{f}}(x,v);
0,x,v) = x_{\mathbf{f}}(x,v), \\
\tilde{v}_{\mathbf{b}}(x,v) &= W(- \tilde{t}_{\mathbf{b}}(x,v);0,x,v) = V( - 
\frac{ \tilde{t}_{\mathbf{b}}(x,v)}{\e}; 0,x,v)= V(-t_{\mathbf{b}}(x,v);
0,x,v) = v_{\mathbf{b}}(x,v), \\
\tilde{v}_{\mathbf{f}} (x,v) &= W(- \tilde{t}_{\mathbf{f}}(x,v);0,x,v) = V(
- \frac{ \tilde{t}_{\mathbf{f}}(x,v)}{\e}; 0,x,v)= V(-t_{\mathbf{f}}(x,v);
0,x,v) = v_{\mathbf{f}}(x,v).
\end{split}%
\end{equation}
\end{definition}

\begin{lemma}
\label{trace_dynamic} For 
$f \in L^{1} ([0, T] \times \Omega \times \mathbb{R}^{3})$, 
\begin{eqnarray}
\int^{T}_{0} \int_{\gamma_{+}^{\delta }} | f(t,x,v)|\mathrm{d} \gamma 
\mathrm{d} t &\lesssim & \e \iint_{\Omega \times \mathbb{R}^{3}} | f(0,x,v)
| \mathrm{d} v \mathrm{d} x + \e \int^{T}_{0} \iint_{\Omega \times \mathbb{R}%
^{3}} | f(t,x,v) | \mathrm{d} v \mathrm{d} x \mathrm{d} t  \label{trace_d} \\
& &+ \int^{T}_{0} \iint_{\Omega \times \mathbb{R}^{3}} \big| [ \e %
\partial_{t} f + v\cdot \nabla_{x} f + \e^{2} \Phi \cdot \nabla_{v} f
](t,x,v) \big| \mathrm{d} v \mathrm{d} x \mathrm{d} t .  \notag
\end{eqnarray}
\end{lemma}

\begin{proof}
First we claim 
\begin{eqnarray}
&& \int_{\gamma_{+}^{\delta} } \mathrm{d} S_x \mathrm{d} v |n(x) \cdot v|
\int_{0}^{T} \mathrm{d} t \ \int^{0}_{\max\{- \e t_{\mathbf{b}} (x, v), -t
\} }\mathrm{d} s \ | f( t+s, X( {\e}^{-1} s; 0,x,v),V( {\e}^{-1}s; 0,x,v) ) |
\notag \\
&\lesssim & \e\int_{0}^{T} \iint_{\Omega \times \mathbb{R}^{3}} |f(t,x,v) |%
\mathrm{d} x \mathrm{d} v \mathrm{d} t .  \label{int_bdry_int_2}
\end{eqnarray}
By the Fubini theorem and the change of variables $\tilde{s} = \e^{-1} s$, 
\begin{eqnarray*}
&& \int_{\gamma_{+}^{\delta} } \mathrm{d} S_x \mathrm{d} v |n(x) \cdot v|
\int^{T}_{0} \mathrm{d} t \int^{0}_{\max \{- \e t_{\mathbf{b}} (x, v), -t
\}} \mathrm{d} s \cdots \\
& =& \int_{\gamma_{+}^{\delta} } \mathrm{d} S_x \mathrm{d} v |n(x) \cdot v|
\int^{0}_{\max \{- \e t_{\mathbf{b}} (x, v), -T \}} \mathrm{d} s
\int^{T}_{-s} \mathrm{d} t \ |f( t+s, X( {\e}^{-1} s; 0,x,v),V( {\e}^{-1}s;
0,x,v) )| \\
& \leq& \e \int_{\gamma_{+}^{\delta} } \mathrm{d} S_x \mathrm{d} v |n(x)
\cdot v| \int^{0}_{\max \{- t_{\mathbf{b}} (x, v), - \e^{-1}T \}} \mathrm{d} 
\tilde{s} \int^{T}_{0}\mathrm{d} t \ |f(t, X(\tilde{s} ;0,x,v), V(\tilde{s}
;0,x,v))| \\
&\lesssim& \e \iint_{\Omega \times \mathbb{R}^{3}} \int^{T}_{0}f(t,x,v)%
\mathrm{d} t \mathrm{d} x \mathrm{d} v,
\end{eqnarray*}
where we have used Lemma \ref{bdry_int} and $t_{\mathbf{b}}(x,v)
\lesssim_{\delta} 1$ for $(x,v) \in\gamma_{+}^{\delta}$.

Now we recall that 
\begin{eqnarray*}
&& \frac{d}{ds} |f( t+s, X( \e^{-1}{\ s} ; 0,x,v),V( \e^{-1}{\ s} ; 0,x,v) )|
\\
&\lesssim& \big| \big[\partial_{t} f+ \e^{-1} V \cdot \nabla_{x} f + \e \Phi
( X ) \cdot \nabla_{v}f\big] ( t+s, X( \e^{-1}{\ s} ; 0,x,v),V( \e^{-1}{\ s}
; 0,x,v) )\big| .
\end{eqnarray*}
and 
\begin{eqnarray*}
|f(t ,x,v)| &\leq& | f( t+s, X( {\e}^{-1} s; 0,x,v),V( {\e}^{-1}s; 0,x,v) )|
\\
& +& \int^{t}_{s} \big|[\partial_{t} f + \e^{-1} V \cdot \nabla_{x} f + \e %
\Phi (X)\cdot \nabla_{v} f] (t+\tau, X(\e^{-1} \tau; 0,x,v) ,V(\e^{-1}
\tau;0,x,v) ) \big|\mathrm{d} \tau.
\end{eqnarray*}

For $(y,u) \in \gamma_{+}$ and for $s \in [ \max\{- \e t_{\mathbf{b}}
(y,-u), -t \},0]$ 
\begin{eqnarray*}
&&\min \{ \e t_{\mathbf{b}} (y,-u), t \} \times |f(t ,y,u)| \\
&=& \int^{0}_{\max\{- \e t_{\mathbf{b}} (y,-u), -t \} } |f( t+s, X( {\e}%
^{-1} s; 0,y,u),V( {\e}^{-1}s; 0,y,u) ) |\mathrm{d} s \\
& +&\int^{0}_{\max\{- \e t_{\mathbf{b}} (y,-u), -t \} } \int^{t}_{s} \big| %
[\partial_{t} f + \e^{-1} V \cdot \nabla_{x} f + \e \Phi \cdot \nabla_{v} f]
(t+\tau, X(\e^{-1} \tau; 0,y,u) ,V(\e^{-1} \tau;0,y,u) ) \big|\mathrm{d}
\tau \mathrm{d} s.
\end{eqnarray*}
If $(t,y,u) \in [ \e \delta_{1}, T] \times \gamma_{+} \backslash
\gamma_{+}^{\delta }$ then $t_{\mathbf{b}} (y,-u) \gtrsim_{\Omega} |n(y)
\cdot u|/|u|^{2} \gtrsim \delta^{3} $. We use (\ref{int_bdry_int_2}) to
bound 
\begin{eqnarray*}
&& \min \{ \e \delta^{3}, \e \delta_{1} \} \times \int_{\gamma_{+}^{\delta
}} \mathrm{d} S_x \mathrm{d} v |n(y) \cdot u| \int_{\e\delta_{1}}^{T} 
\mathrm{d} t \ |f(t,y,u)| \\
&\leq& \int_{\gamma_{+}^{\delta }} \mathrm{d} S_x \mathrm{d} v |n(y) \cdot
u| \int_{\e\delta_{1}}^{T} \mathrm{d} t \ \int^{0}_{\max\{- \e t_{\mathbf{b}%
} (y,-u), -t \} }|f( t+s, X( {\e}^{-1} s; 0,y,u),V( {\e}^{-1}s; 0,y,u) )| 
\mathrm{d} s \\
&&+ \int_{\gamma_{+}^{\delta }} \mathrm{d} S_x \mathrm{d} v |n(y) \cdot u|
\int_{\e\delta_{1}}^{T} \mathrm{d} t \int^{0}_{\max\{- \e t_{\mathbf{b}}
(y,-u), -t \} } \int^{t}_{s} \\
&& \ \ \ \ \ \ \ \ \ \ \ \ \ \ \ \times\big| [\partial_{t} f + \e^{-1} V
\cdot \nabla_{x} f + \e \Phi \cdot \nabla_{v} f] (t+\tau, X(\e^{-1} \tau;
0,y,u) ,V(\e^{-1} \tau;0,y,u) )\big| \mathrm{d} \tau \mathrm{d} s \\
&\lesssim& \e \int_{0}^{T} \iint_{\Omega \times \mathbb{R}^{3}} |f(t,x ,v) |%
\mathrm{d} v \mathrm{d} x \mathrm{d} t + \int_{0}^{T} \iint_{\Omega \times 
\mathbb{R}^{3}} \big| [ \e\partial_{t} + v \cdot \nabla_{x} + \e^{2} \Phi
\cdot \nabla_{v} ]f(t,x,v)\big| \mathrm{d} v \mathrm{d} x \mathrm{d} t.
\end{eqnarray*}

We choose $\e\delta_{1} \ll \e\delta^{3} \lesssim \e \times \inf
\inf_{\gamma_{+} \backslash \gamma_{+}^{\delta }} t_{\mathbf{b}} (y,-u)$.
Then $\e t_{\mathbf{b}} (y,-u) > t$ for all $(t,y,u) \in [0, \e\delta_{1}]
\times \gamma_{+} \backslash \gamma_{+}^{\delta }$ so that the backward
trajectory hits the initial plan first: 
\begin{eqnarray*}
|f(t,y,u) |&\leq& |f(0, X(\e^{-1} t; 0,y,u), V(\e^{-1} t; 0,y,u))| \\
&+&\int^{0}_{-t} \big|[\partial_{t} f + \e^{-1} V \cdot \nabla_{x} f + \e %
\Phi(X) \cdot \nabla_{v} f] (t+s, X(\e^{-1} s;0,y,u), V(\e^{-1} s;0,y,u))%
\big|.
\end{eqnarray*}
Then from (\ref{int_bdry_int_2}) and the change of variables 
\begin{eqnarray*}
&&\int_{0}^{\e \delta_{1}} \int_{\gamma_{+}} | f(t,y,u) | \\
&\leq& \int_{0}^{\e \delta_{1}} \mathrm{d} t \int_{\gamma_{+}} \mathrm{d}
S_{y} \mathrm{d} v \ |n(y) \cdot u| |f(0, X(\e^{-1} t; 0,y,u), V(\e^{-1} t;
0,y,u))| \\
& & + \int_{0}^{ \e\delta_{1}} \mathrm{d} t \int_{\gamma_{+}} \mathrm{d}
S_{y} \mathrm{d} v \ |n(y) \cdot u| \\
&& \ \ \ \times \int^{0}_{-t}\big| [\partial_{t} f + \e^{-1} V \cdot
\nabla_{x} f + \e \Phi(X) \cdot \nabla_{v} f] (t+s, X(\e^{-1} s;0,y,u), V(\e%
^{-1} s;0,y,u)) \big|\mathrm{d} s \\
&\leq& \e \int_{0}^{ \delta_{1}} \mathrm{d} \tilde{t} \int_{\gamma_{+}} 
\mathrm{d} S_{y} \mathrm{d} v \ |n(y) \cdot u| | f(0, X( \tilde{t}; 0,y,u),
V( \tilde{t}; 0,y,u))| \ \ \ \ \ \ \ \ \ \ \ \ \ \ ( \e^{-1} t= \tilde{t} )
\\
&&+ \int_{0}^{T} \iint_{\Omega \times \mathbb{R}^{3}} \big| [ \e\partial_{t}
+ v \cdot \nabla_{x} + \e^{2} \Phi \cdot \nabla_{v} ]f(t,x ,v)\big| \mathrm{d%
} v \mathrm{d} x \mathrm{d} t \ \ \ \ \ \ \ \ \ \ \ \ \ \ \ \ \ \ ( \e^{-1}
s = \tilde{s} ) \\
&\lesssim& \e \delta_{1} \iint_{\Omega \times \mathbb{R}^{3}} |f_{0} (x ,v)| 
\mathrm{d} v \mathrm{d} x+ \int_{0}^{T} \iint_{\Omega \times \mathbb{R}^{3}} %
\big| [ \e\partial_{t} + v \cdot \nabla_{x} + \e^{2} \Phi \cdot \nabla_{v}
]f(t,x,v)\big| \mathrm{d} v \mathrm{d} x \mathrm{d} t.
\end{eqnarray*}
\end{proof}

\begin{lemma}
\label{timedepgreen} Assume $\Phi \in C^{1}$. Assume that $f(t,x,v), \
g(t,x,v)\in L^2(\mathbb{R}_{+} \times \Omega\times\mathbb{R}^3 )$, $\{\pt_t +%
\e^{-1}v \cdot \nabla_x+\e\Phi\cdot\nabla_v\} f, \{\pt_t+\e^{-1}v \cdot
\nabla_x+\e\Phi\cdot\nabla_v\} g \in L^2(\mathbb{R}_{+} \times \Omega\times%
\mathbb{R}^3)$ and $f_{\gamma}, g_{\gamma}\in L^2(\mathbb{R}_{+} \times
\gamma)$. Then 
\begin{eqnarray}
&&\int_s^t \iint_{\Omega\times\mathbb{R}^3}\{\e\pt_t +v \cdot \nabla_xf +\e%
^2\Phi\cdot\nabla_v f \} g + \{\e\pt_t + v \cdot \nabla_xg+\e%
^2\Phi\cdot\nabla_v g\} f \ \mathrm{d} v \mathrm{d} x \mathrm{d} \tau  \notag
\\
&&=\e \iint_{\Omega\times\mathbb{R}^3}f(s,x,v)g(s,x,v) \mathrm{d} v \mathrm{d%
} x - \e\iint_{\Omega\times\mathbb{R}^3}f(t,x,v)g(t,x,v) \mathrm{d} v 
\mathrm{d} x  \notag \\
&&+ \int_s^t \left[\int_{\gamma_+} f g \mathrm{d} \gamma- \int_{\gamma_-} f
g \mathrm{d}\gamma\right]\mathrm{d} \tau .  \notag
\end{eqnarray}
\end{lemma}

\begin{proof}
The proof is from Chapter 9 of \cite{CIP} with the same modification as
Lemma \ref{green}.
\end{proof}

\subsection{Gain of Integrability: $L_{t}^{2}L_{x}^{3}$ Estimate}

The main goal of this section is the following:

\begin{proposition}
\label{prop_3} Assume $g \in L^{2} (\mathbb{R}_{+} \times \Omega \times 
\mathbb{R}^{3})$, $f_{0} \in L^{2} (\Omega \times \mathbb{R}^{3})$, and $%
f_{\gamma}\in L^{2} ( \mathbb{R}_{+}\times \gamma)$. Let $f \in L^{\infty}( 
\mathbb{R}_{+}; L^{2} (\Omega \times \mathbb{R}^{3}))$ solve (\ref%
{linear_dyn}) in the sense of distribution and satisfy $f(t,x,v) =
f_{\gamma}(t,x,v) $ on $\mathbb{R}_{+} \times \gamma$ and $f(0,x,v) = f_{0}
(x,v)$ on $\Omega \times \mathbb{R}^{3}.$ Recall $\mathbf{P}f$ in (\ref{Pabc}%
).

Then there exist $\mathbf{S}_{1}f(t,x)$, $\mathbf{S}_{2}f(t,x)$, and $%
\mathbf{S}_{3}f(t,x)$ satisfying 
\begin{equation}  \label{S1}
|a (t,x)| + |b(t,x)| + |c(t,x)| \leq \mathbf{S}_{1}f(t,x) + \mathbf{S}_{2}
f(t,x)+ \mathbf{S}_{3} f(t,x),
\end{equation}
where the precise form of $\mathbf{S}_{i} f$ is defined in (\ref{def_S2}).

Moreover, 
\begin{equation}
\begin{split}
& \| \mathbf{S}_{1} f \|_{L^{2}_{t}L^{3}_{x}} + \e^{-\frac{1}{2}} \| \mathbf{%
S}_{2} f \|_{L^{2}_{t}L^{\frac{12}{5}}_{x}} \\
&\lesssim \| g \|_{L^{2}_{t,x,v}} + \| f_{0} \|_{L^{2}_{x,v}} + \| [ v\cdot
\nabla_{x} + \e^{2} \Phi \cdot \nabla_{v}] f_{0} \|_{L^{2}_{x,v}} + \| f_{0}
\|_{L^{2}(\gamma)} + \| f_{\gamma} \|_{L^{2}(\mathbb{R}_{+} \times \gamma)},
\label{fl0}
\end{split}%
\end{equation}
and 
\begin{equation}
\| \mathbf{S}_{3}f \|_{L^{2}_{t,x} } \lesssim \| (\mathbf{I}- \mathbf{P})
f\|_{L^{2}_{t,x,v} } .  \label{ins2}
\end{equation}
\end{proposition}

\bigskip

We need several lemmas to prove Proposition \ref{prop_3}. First we define $%
f_{\delta}$ which represents either the interior or the non-grazing parts of 
$f$ near the boundary.

\begin{definition}
We define, for $(t,x,v) \in \mathbb{R} \times \bar{\Omega} \times \mathbb{R}%
^{3}$ and for $0 < \delta \ll 1$, 
\begin{equation}  \label{Z_dyn}
\begin{split}
&f_{\delta}(t,x,v) \\
& : = [1-\chi(\frac{n(x) \cdot v}{\delta}) \chi \big(\frac{ \xi(x) }{\delta}%
\big) ] \chi(\delta|v|) \big\{ \mathbf{1}_{ t\in[0, \infty)} f(t,x,v)+ 
\mathbf{1}_{t\in(-\infty,0 ]} \chi(t) f_{0}(x,v) \big\}.
\end{split}%
\end{equation}
Here $n(x)$ and $\chi$ are defined in (\ref{n_inte}) and (\ref{chi})
respectively. %
\end{definition}

Here we extend $f_{\delta}$ to the negative time so that we are able to take
the time-derivative. Clearly, 
\begin{eqnarray*}
\| f_{\delta} \|_{L^{2} (\mathbb{R} \times \Omega \times \mathbb{R}^{3})} &
\lesssim & \| f \|_{L^{2} (\mathbb{R}_{+} \times \Omega \times \mathbb{R}%
^{3})} + \| f_{0}\|_{L^{2} (\Omega \times \mathbb{R}^{3})}, \\
\| f_{\delta} \|_{L^{2} ( \mathbb{R} \times \gamma)} &\lesssim& \|
f_{\gamma} \|_{L^{2} ( \mathbb{R}_{+} \times \gamma)} + \| f_{0} \|_{L^{2}
(\gamma)}.
\end{eqnarray*}
Note that, at the boundary $(x,v) \in \gamma:=\partial\Omega \times \mathbb{R%
}^{3}$, 
\begin{equation}  \label{Z_support_dyn}
f_{\delta}(t,x,v)|_{\gamma}\equiv 0, \ \ \text{for} \ |n(x) \cdot v| \leq
\delta \ \text{ or } \ |v|\geq \frac{1}{\delta} .
\end{equation}

\begin{lemma}
\label{extension_dyn} Assume the same hypothesis of Proposition \ref{prop_3}%
. Then there exists $\bar{f}(t,x,v) \in L^{2}( \mathbb{R} \times \mathbb{R}%
^{3} \times \mathbb{R}^{3})$, an extension of $f_{\delta}$ in (\ref{Z_dyn}),
such that 
\begin{equation}
\bar{f} |_{\Omega \times \mathbb{R}^{3}}\equiv f_{\delta} \ \text{ and } \ 
\bar{f} |_{\gamma}\equiv f_{ \delta} |_{\gamma} \ \text{ and } \ \bar{f }
|_{t=0} = f_{\delta} |_{t=0}.  \notag
\end{equation}
Moreover, in the sense of distributions on $\mathbb{R} \times \mathbb{R}^{3}
\times \mathbb{R}^{3}$, 
\begin{equation}  \label{eq_barf_dyn}
\{\e \partial_{t} + v\cdot \nabla_{x} + \e^{2} \Phi \cdot \nabla_{v} \} \bar{%
f} = h[f , g] = h_{1} + h_{2} + h_{3}+ h_{4},
\end{equation}
where 
\begin{eqnarray}
h_{1} (t,x,v)&=& \mathbf{1}_{(x,v) \in \Omega \times \mathbb{R}^{3}} [1-\chi(%
\frac{n(x) \cdot v}{\delta}) \chi \big( \frac{\xi(x)}{\delta}\big)]
\chi(\delta|v|)  \label{acca1} \\
&& \times \big[ \mathbf{1}_{t \in [0,\infty)} g(t,x,v) + \mathbf{1}_{t \in (
- \infty, 0 ]} \chi(t)\{ \e \frac{ \chi^{\prime} (t)}{\chi(t)} + v \cdot
\nabla_{x} + \e^{2} \Phi \cdot \nabla_{v} \} f_{0} (x,v)\big],  \notag \\
h_{2} (t,x,v) &=&\mathbf{1}_{(x,v) \in \Omega \times \mathbb{R}^{3}} \big[%
\mathbf{1}_{t \in [0,\infty)} f(t,x,v) + \mathbf{1}_{t \in (- \infty, 0 ]}
\chi(t) f_{0} (x,v) \big]  \label{acca2} \\
&& \times \{v\cdot \nabla_{x} + \e^{2} \Phi \cdot \nabla_{v}\} \Big( [1-\chi(%
\frac{n(x) \cdot v}{\delta}) \chi \big( \frac{\xi(x)}{\delta}\big) ]
\chi(\delta|v|)\Big),  \notag \\
h_{3} (t,x,v) &=& \mathbf{1}_{{\ (x,v) \in [\Omega_{\tilde{C} \delta^{4}}
\backslash \bar{\Omega}]\times \mathbb{R}^{3}}} \ \frac{1}{\tilde{C}
\delta^{4}}v \cdot \nabla_{x} \xi(x) \chi^{\prime} \big( \frac{\xi(x)}{%
\tilde{C} \delta^{4}} \big)  \label{acca3} \\
&& \times \big[ f_{\delta}(t- \e t_{\mathbf{b}}^{*} (x,v), x_{\mathbf{b}%
}^{*}(x,v), v_{\mathbf{b}}^{*}(x,v) ) \mathbf{1}_{x_{\mathbf{b}}^{*}(x,v)
\in \partial\Omega}  \notag \\
&& \ \ \ + f_{\delta}(t + \e t_{\mathbf{f}}^{*} (x,v), x_{\mathbf{f}%
}^{*}(x,v), v_{\mathbf{f}}^{*}(x,v) ) \mathbf{1}_{x_{\mathbf{f}}^{*}(x,v)
\in \partial\Omega} \big] ,  \notag \\
h_{4} (t,x,v) &=& \mathbf{1}_{{\ (x,v) \in [\Omega_{\tilde{C} \delta^{4}}
\backslash \bar{\Omega}]\times \mathbb{R}^{3}}} f_{\delta}( t- \e t_{\mathbf{%
b}}^{*} (x,v), x_{\mathbf{b}}^{*}(x,v), v_{\mathbf{b}}^{*}(x,v)) \chi \big( 
\frac{\xi(x)}{\tilde{C} \delta^{4}} \big) \chi^{\prime}(t_{\mathbf{b}%
}^{*}(x,v)) \mathbf{1}_{ x_{\mathbf{b}}^{*}(x,v) \in \partial\Omega}  \notag
\\
&&+ \mathbf{1}_{{\ (x,v) \in [\Omega_{\tilde{C} \delta^{4}} \backslash \bar{%
\Omega}]\times \mathbb{R}^{3}}} f_{\delta}( t+ \e t_{\mathbf{f}}^{*} (x,v) ,
x_{\mathbf{f}}^{*}(x,v), v_{\mathbf{f}}^{*}(x,v)) \chi \big( \frac{\xi(x)}{%
\tilde{C} \delta^{4}} \big) \chi^{\prime}(t_{\mathbf{f}}^{*}(x,v)) \mathbf{1}%
_{ x_{\mathbf{f}}^{*}(x,v) \in \partial\Omega},  \notag \\
&&  \label{acca4}
\end{eqnarray}
where 
\begin{equation}  \label{shell}
\Omega_{\tilde{C}\delta^{4}} \ : = \ \big\{ x \in \mathbb{R}^{3}: \ \xi(x) < 
\tilde{C}\delta ^{4} \big\},
\end{equation}
and, 
for $(x,v) \in\Omega_{\tilde{C}\delta^{4}}\backslash \bar{\Omega}$, with $%
\bar\Omega=\Omega\cup\pt\O $, 
\begin{eqnarray}
&&t_{\mathbf{b}}^{*}(x,v) \ := \ \inf\{ s>0: 0 <\xi(X(s;0,x,v)) < \tilde{C}
\delta^{4}\ \ \text{for all } 0 <\tau< s \}, \ \ t_{\mathbf{f}}^{*}(x,v) \
:= \ t_{\mathbf{b}}^{*}(x,-v) ,  \notag \\
&&(x_{\mathbf{b}}^{*} (x,v),v_{\mathbf{b}}^{*} (x,v) )\ : = (X(- t_{\mathbf{b%
}}^{*}(x,v);0,x,v), V(- t_{\mathbf{b}}^{*}(x,v);0,x,v)),  \label{tbstar} \\
&&(x_{\mathbf{f}}^{*} (x,v) , v_{\mathbf{f}}^{*} (x,v)) \ := (X( t_{\mathbf{f%
}}^{*}(x,v);0,x,v) , V( t_{\mathbf{f}}^{*}(x,v);0,x,v)).  \notag
\end{eqnarray}

Moreover, 
\begin{eqnarray}
&& \| h_{1} \|_{ L^{2}( \mathbb{R} \times \mathbb{R}^{3} \times \mathbb{R}%
^{3})} \lesssim \| g\|_{L^{2}( \mathbb{R}_{+} \times \Omega \times \mathbb{R}%
^{3})} + \e \| f_{0} \|_{L^{2} (\Omega \times \mathbb{R}^{3})} + \| [ v\cdot
\nabla_{x} + \e^{2} \Phi \cdot \nabla_{v} ] f_{0} \|_{L^{2 } (\Omega \times 
\mathbb{R}^{3})} ,  \notag \\
&& \| h_{2} \|_{L^{2} ( \mathbb{R} \times \mathbb{R}^{3} \times \mathbb{R}%
^{3})} \lesssim_{\delta} \|f \|_{L^{2}( \mathbb{R}_{+} \times \Omega \times 
\mathbb{R}^{3})} + \| f_{0} \|_{L^{2} (\Omega \times \mathbb{R}^{3})},
\label{force_Z_dyn} \\
&& \| h_{3} \|_{L^{2} ( \mathbb{R} \times \mathbb{R}^{3} \times \mathbb{R}%
^{3})} +\| h_{4} \|_{L^{2} ( \mathbb{R} \times\mathbb{R}^{3} \times \mathbb{R%
}^{3})} \lesssim_{\delta} \| f_{\gamma} \|_{L^{2} ( \mathbb{R}_{+} \times
\gamma)} + \| f_{0} \|_{L^{2} (\gamma)} .  \notag
\end{eqnarray}
\end{lemma}

\begin{proof}
The proof of this lemma is given in Appendix \ref{A1}.
\end{proof}

The next step to prove a version of the velocity averaging lemma in the
presence of a small external force. The presence of the external force
requires significant modifications to the original argument in \cite%
{GLPS,Saint}.

\begin{lemma}
\label{flfs} Let $f\in L^{2}(\mathbb{R} \times \mathbb{R}^{3} \times \mathbb{%
R}^{3}) $ be a function solving the transport equation 
\begin{equation}
\e \pt_t f+v\cdot \nabla_x f+\e^2\Phi\cdot \nabla_v f +f=q,
\label{bastransp}
\end{equation}
in the sense of distributions with $q\in L_{t,x,v}^2$. Let $\psi(v)$ be a
smooth function, which vanishes very fast as $|v| \rightarrow \infty$.

Then we can decompose $f$ as 
\begin{equation}  \label{decom_ls}
f= f_{la} + f_{sm},
\end{equation}
such that 
\begin{equation}
\Big\| \int_{\mathbb{R}^{3}} \psi f_{la} \mathrm{d} v \Big\|%
_{L^2_tL^3_x}\lesssim_{\psi} \|w^{-1}q\|_{L^2_{t,x,v}},  \label{fl}
\end{equation}
for $w= e^{\beta |v|^{2}}$ with $0< \beta\ll 1$, and 
\begin{equation}
\Big\| \int_{\mathbb{R}^{3}} \psi f_{sm} \mathrm{d} v \Big\|_{L^2_tL^{\frac{%
12}5}_x}\lesssim_{\psi, \|\Phi\|_\infty} \e^{\frac 1 2}
\|w^{-1}q\|_{L^2_{t,x,v}} .  \label{fs1}
\end{equation}
\end{lemma}

\begin{proof}
We only prove the case $\beta=0$, because the growing factor $w$ can be
absorbed in $\psi$ by redefining $f_{la}$.

\textit{Step 1. } We define a ``large'' part as 
\begin{equation}  \label{f_la_def}
f_{la} (t,x,v) : = \int^{\infty}_{0} e^{-\tau} q(t- \e \tau, x- \tau v, v) 
\mathrm{d} \tau.
\end{equation}
Then, clearly $f_{la}$ solves the transport equation without an external
force, in the sense of distributions, 
\begin{equation}
\e \pt_t f_{la}+v\cdot \nabla_x f_{la}+f_{la}=q.  \label{bastranspl}
\end{equation}
Via the change of variables $(t- \e \tau, x- \tau v, v) \mapsto (t,x,v)$ for
fixed $\tau$ and the Minkowski inequality 
\begin{equation}  \label{l2_fla}
\|f_{la} \|_{L^{2}_{t,x,v}} \leq \int^{\infty}_{0} e^{-\tau} \| q(t-\e \tau,
x-\tau v,v) \|_{L^{2}_{t,x,v}} \mathrm{d} \tau=\| q \|_{L^{2}_{t,x,v}}
\int^{\infty}_{0} e^{-\tau}\mathrm{d} \tau \lesssim \| q \|_{L^{2}_{t,x,v}} .
\end{equation}
Therefore, $f_{la}$ is in $L^{2}$. By the standard velocity averaging lemma 
\cite{GLPS,Saint}, $\int_{\mathbb{R}^{3}} \psi f_{la} \mathrm{d} v$ is in $%
L^2_tH^{\frac 1 2}_x$ and satisfies the bound (\ref{fl}) by the Sobolev
embedding $L^2(\mathbb{R};H^{\frac 12}(\mathbb{R}^3))\subset L^2(\mathbb{R}%
;L^3(\mathbb{R}^3))$.

\textit{Step 2. } Now we define a ``small'' part as 
\begin{equation}  \label{f_sm_def}
f_{sm}(t,x,v) : = f(t,x,v) - f_{la} (t,x,v).
\end{equation}
Clearly, 
\begin{equation}
\e \pt_t f_{sm}+v\cdot \nabla_x f_{sm} +f_{sm}=-\e^2\Phi\cdot \nabla_v f.
\label{bastransps}
\end{equation}
We use the shorthand notation 
\begin{equation}
\z=-\e^2\Phi f,  \label{notation_z}
\end{equation}
so that (\ref{bastransps}) becomes 
\begin{equation}
\e \pt_t f_{sm}+v\cdot \nabla_x f_{sm}+f_{sm}=\nabla_v \z.
\label{bastransps1}
\end{equation}
By Fourier transforming, 
\begin{equation}
\hat f_{sm} =\frac{\nabla_v \hat \z}{i ( \e \tau + v\cdot \xi)+1 }, \ \ \ |
\hat f_{sm}| \lesssim\frac{|\nabla_v \hat \z | }{| \e \tau + v\cdot \xi|+1 }.
\label{bastransps2}
\end{equation}

Recall $\chi$ is (\ref{chi}).  
Then, for $\a>0$, 
\begin{equation}
\int_{\mathbb{R}^3} \hat f_{sm} \psi= \int_{\mathbb{R}^3} \hat f_{sm} \psi
\chi\left(\frac{\e\tau+\xi\cdot v}{\a}\right)+\int_{\mathbb{R}^3} \hat
f_{sm} \psi \left[1-\chi\left(\frac{\e\tau+\xi\cdot v}{\a}\right)\right]:=%
\mathcal{T}_1+\mathcal{T}_2.  \notag
\end{equation}
Since in $\mathcal{T}_1$ we have only the contribution for $|\e %
\tau+\xi\cdot v|\le \a$, by integrating first on the velocity $v_{\xi^\perp}$
orthogonal to $\xi$ and then on the one dimensional variable $v_\xi=\frac{%
v\cdot \xi}{|\xi|}$, with the change of variables $
\frac{|\xi|v_\xi}{\a}$ we estimate, 
\begin{equation}
|\mathcal{T}_1|\le \frac{\a^{\frac 1 2}}{|\xi|^{\frac 1 2}}\|\hat
f_{sm}\|_{L^{2}_v}.  \notag
\end{equation}
As for $\mathcal{T}_2$, we use (\ref{bastransps2}) to obtain 
\begin{eqnarray*}
&& \mathcal{T}_2 = \int_{\mathbb{R}^3}dv \frac{ \nabla_v \hat \z }{ i(\e %
\tau + v\cdot \xi)+1 } \psi \left[1-\chi\left(\frac{\e\tau+\xi\cdot v}{\a}%
\right)\right]  \notag \\
&&= i\int_{\mathbb{R}^3}dv \frac{ \xi \hat \z }{[ i( \e \tau + v\cdot \xi)+1
]^2} \psi \left[1-\chi\left(\frac{\e\tau+\xi\cdot v}{\a}\right)\right]+\int_{%
\mathbb{R}^3}dv \frac{ \xi \hat \z }{ \a[ i( \e \tau + v\cdot \xi )+1]} \psi
\chi^{\prime}\left(\frac{\e\tau+\xi\cdot v}{\a}\right)  \notag \\
&& \ \ +\int_{\mathbb{R}^3}dv \frac{ \hat \z }{ i( \e \tau + v\cdot \xi )+1 }
\nabla_v\psi \left[1-\chi\left(\frac{\e\tau+\xi\cdot v}{\a}\right)\right] \\
&&:=\mathcal{K}_1+\mathcal{K}_2+\mathcal{K}_3.
\end{eqnarray*}
We have 
\begin{eqnarray*}
|\mathcal{K}_1 | &\le& |\xi|\left(\int_{\mathbb{R}^3}dv \frac{\psi(1-\chi)}{%
[ | \e \tau + v\cdot \xi | +1]^4}\right)^{\frac 1 2} \|\hat \z%
\|_{L^{2}_v}\le|\xi|\left(\int_{| \e\tau +\xi\cdot v|\ge \a}dv \frac{1}{[ | %
\e \tau + v\cdot \xi | +1]^4}\right)^{\frac 1 2} \|\hat \z\|_{L^{2}_v} \\
&\le& \frac{|\xi|^{\frac 12}}{\a^{\frac 3 2}}\|\hat \z\|_{L^{2}_v}. \\
|\mathcal{K}_2|&\le& \frac{|\xi|}{\a^2}\left(\int_{\mathbb{R}^3}dv \left[%
\chi^{\prime}\left(\frac{\e \tau + v\cdot \xi}{\a}\right)\right]%
^2\right)^{\frac 1 2}\|\hat \z\|_{L^{2}_v}\le \frac{|\xi|}{\a^2} \left(\frac{%
\a}{|\xi|}\right)^{\frac 1 2}\|\hat \z\|_{L^{2}_v} = \frac{|\xi|^{\frac 1 2}%
}{\a^{\frac 3 2}} \|\hat \z\|_{L^{2}_v}. \\
|\mathcal{K}_3|&\le& \left(\int_{|\e \tau+\xi\cdot v|>\a} dv\frac{|\nabla_v
\psi|^{2}}{1+|\e \tau +\xi\cdot v|^2}\right)^{\frac 1 2}\|\hat \z%
\|_{L^{2}_v}\le \frac{1}{\a^{\frac 1 2}|\xi|^{\frac 1 2}} \|\hat \z%
\|_{L^{2}_v}.
\end{eqnarray*}
Now we choose $\alpha$ such that 
\begin{equation*}
\frac{\a^{\frac 1 2}}{|\xi|^{\frac 1 2}}\|\hat f_{sm}\|_{L^{2}_v}=\frac{%
|\xi|^{\frac 12}}{\a^{\frac 3 2}}\|\hat \z\|_{L^{2}_v},
\end{equation*}
which means 
\begin{equation*}
\a=|\xi|^{\frac 1 2} \|\hat \z\|_{L^{2}_v}^{\frac 1 2}\|\hat
f_{sm}\|_{L^{2}_v}^{-\frac 1 2}.
\end{equation*}
With this choice, we have that, 
\begin{equation}
|\mathcal{T}_1|+|\mathcal{K}_1|+|\mathcal{K}_2| \le 3 |\xi|^{-\frac 1 4}
\|\hat \z\|_{L^{2}_v}^{\frac 1 4}\|\hat f_{sm}\|_{L^{2}_v}^{\frac 3 4}, \ \ |%
\mathcal{K}_3|\le |\xi|^{-\frac 3 4} \|\hat \z\|_{L^{2}_v}^{\frac 3 4}\|\hat
f_{sm}\|_{L^{2}_v}^{\frac 1 4}.  \notag
\end{equation}
Therefore, from (\ref{notation_z}) 
\begin{eqnarray*}
\big\||\xi|^{\frac{1}{4}} ( |\mathcal{T}_1|+|\mathcal{K}_1|+|\mathcal{K}_2| )%
\big\|_{L^{2}_{\tau, \xi }}&\lesssim& \big\|\|\hat \z\|_{L^{2}_v}^{\frac 1
4}\|\hat f_{sm}\|_{L^{2}_v}^{\frac 3 4}\big\|_{L^2_{\tau, \xi}}\le \e^{\frac
1 2}\|\Phi\|_\infty^{\frac 1 4} \|f\|_{L^2_{t,x,v}}^{\frac 1
4}\|f_{sm}\|^{\frac 3 4}_{L^2_{t,x,v}}, \\
\big\||\xi|^{\frac{3}{4}} |\mathcal{K}_3| \big\|_{L^{2}_{\tau, \xi
}}&\lesssim& \big\|\|\hat \z\|_{L^{2}_v}^{\frac 3 4}\|\hat
f_{sm}\|_{L^{2}_v}^{\frac 1 4}\big\|_{L^2_{\tau, \xi}}\le \e^{\frac 3
2}\|\Phi\|_\infty^{\frac{3}{4}} \|f\|_{L^2_{t,x,v}}^{\frac 3
4}\|f_{sm}\|^{\frac 1 4}_{L^2_{t,x,v}}.
\end{eqnarray*}
We also note that $\| \mathcal{K}_{3} \|_{L^{2}_{\tau, \xi}}\lesssim_{\psi}
\| \hat{\zeta} \|_{L^{2}_{\tau, \xi, v}}\lesssim \e^{2} \| \Phi \|_{\infty}
\| f \|_{L^{2}_{t,x,v}}$.

Denote the anti-Fourier transforming by $\mathcal{F}^{-1}_{\tau, \xi}$. 
By the Sobolev embedding $\dot H^{\frac 1 4} (\mathbb{R}^{3})\subset L^{%
\frac{12}5}(\mathbb{R}^{3})$ and $\dot H^{\frac 3 4}\subset L^{4}$, we have 
\begin{equation*}
\begin{split}
\| \mathcal{F}^{-1}_{\tau, \xi} ( \mathcal{T}_1 + \mathcal{K}_1 + \mathcal{K}%
_2 ) \|_{L^{2}_{t} L^{\frac{12}{5}}_{x}} &\lesssim \| \mathcal{F}%
^{-1}_{\tau, \xi} ( \mathcal{T}_1 + \mathcal{K}_1 + \mathcal{K}_2 )
\|_{L^{2}_{t}\dot{H}_{x}^{\frac{1}{4}}}\lesssim \e^{\frac 1 2}
\|f\|_{L^2_{t,x,v}}^{\frac 1 4}\|f_{sm}\|^{\frac 3 4}_{L^2_{t,x,v}}, \\
\| \mathcal{F}^{-1}_{\tau, \xi} \mathcal{K}_3 \|_{L^{2}_{t} L^{4}_{x}}
&\lesssim \| \mathcal{F}^{-1}_{\tau, \xi} \mathcal{K}_3 \|_{L^{2}_{t}\dot{H}%
_{x}^{\frac{1}{4}}}\lesssim \e^{\frac 3 2} \|f\|_{L^2_{t,x,v}}^{\frac 3
4}\|f_{sm}\|^{\frac 1 4}_{L^2_{t,x,v}}.
\end{split}%
\end{equation*}
By an interpolation ($L^{\frac{12}{5}}_{x} \subset L^{2}_{x} \cap L^{4}_{x}$%
) 
\begin{eqnarray*}
\| \mathcal{F}^{-1}_{\tau, \xi} \mathcal{K}_{3} \|_{L^{2}_{t} L^{\frac{12}{5}%
}_{x}} \leq \| \mathcal{F}^{-1}_{\tau, \xi} \mathcal{K}_{3}
\|_{L^{2}_{t,x}}^{\frac{2}{3}} \|\mathcal{F}^{-1}_{\tau, \xi} \mathcal{K}%
_{3} \|^{\frac{1}{3}}_{L^{2}_{t}L^{4}_{x}}\lesssim \e^{\frac{11}{6}} \| f
\|_{L^{2}_{t,x,v}}^{\frac{11}{12}} \| f_{sm} \|_{L^{2}_{t,x,v}}^{\frac{1}{12}%
}.
\end{eqnarray*}

Clearly from (\ref{f_sm_def}) and (\ref{l2_fla}), $\| f_{sm}
\|_{L^{2}_{t,x,v}} \leq \| f \|_{L^{2}_{t,x,v}} + \| f_{la}
\|_{L^{2}_{t,x,v}}\lesssim \| f \|_{L^{2}_{t,x,v}} +\| q \|_{L^{2}_{t,x,v}}$%
. On the other hand, from (\ref{bastransp}), 
\begin{equation*}
f(t,x,v) = \int^{\infty}_{0} e^{-\tau} q(t- \e \tau, X(t-\e \tau; t,x,v),
V(t-\e\tau;t,x,v)).
\end{equation*}
Following the argument to prove (\ref{l2_fla}), we obtain $\| f
\|_{L^{2}_{t,x,v}} \leq \| q \|_{L^{2}_{t,x,v}}.$ Therefore we conclude 
\begin{equation*}
\| \mathcal{F}^{-1}_{\tau, \xi} ( \mathcal{T}_1 + \mathcal{K}_1 + \mathcal{K}%
_2 + \mathcal{K}_{3} )\|_{L^{2}_{t} L^{\frac{12}{5}}_{x}} \lesssim \e^{\frac{%
1}{2}} ( \| f \|_{L^{2}_{t,x,v}} + \| q \|_{L^{2}_{t,x,v}}),
\end{equation*}
and hence the estimate (\ref{fs1}). %
\end{proof}

Now we are ready to prove the main result of this section:

\begin{proof}[\textbf{Proof of Proposition \protect\ref{prop_3}}]
Recall (\ref{Pabc}). We use temporary notations 
\begin{equation}  \label{zeta}
\begin{split}
[\zeta_{0}(v) , \zeta_{1}(v), \zeta_{2}(v) , \zeta_{3}(v), \zeta_{4}(v)] \
:= [\sqrt{\mu}, v_{1} \sqrt{\mu}, v_{2 }\sqrt{\mu}, v_{3} \sqrt{\mu}, {\
\frac 2 3} \frac{|v|^{2}-3}{2}\sqrt{\mu}].
\end{split}%
\end{equation}
From (\ref{Z_dyn}),  
\begin{eqnarray*}
&&\int_{\mathbb{R}^{3}} f_{\delta} (t,x,v) \zeta_{i} (v) \mathrm{d} v \\
&=& \int_{\mathbb{R}^{3}} \big[ 1- \chi( \frac{n(x) \cdot v}{\delta} ) \chi( 
\frac{\xi(x)}{\delta}) \big] \chi( \delta |v|) \big\{ \mathbf{1}_{t \geq 0}
f(t,x,v) + \mathbf{1}_{t \leq 0} \chi(t) f_{0} (x,v) \big\} \zeta_{i} (v) 
\mathrm{d} v \\
&=& \mathbf{1}_{t \geq 0} \int_{\mathbb{R}^{3}} \big[ 1- \chi( \frac{n(x)
\cdot v}{\delta} ) \chi( \frac{\xi(x)}{\delta}) \big] \chi( \delta |v|)%
\Big\{ \sum_{j=0}^{4} a_{j}(t,x) \zeta_{j}(v) + (\mathbf{I} - \mathbf{P})
f(t,x,v)\Big\} \zeta_{i}(v) \mathrm{d} v \\
&&+ \mathbf{1}_{t \leq 0} \int_{\mathbb{R}^{3}} \big[ 1- \chi( \frac{n(x)
\cdot v}{\delta} ) \chi( \frac{\xi(x)}{\delta}) \big] \chi( \delta |v|)
\chi(t) f_{0} (x,v) \zeta_{i} (v) \mathrm{d} v \\
&=& \mathbf{1}_{t\geq 0} \Big\{a_{i} (t,x) + O(\delta) \sum_{j=0}^{4} |a_{j}
(t,x)| + O_{\delta}(1) \int_{\mathbb{R}^{3}} | (\mathbf{I} - \mathbf{P})
f(t,x,v) | \zeta_{i} (v) \mathrm{d} v\Big\} \\
&&+ \mathbf{1} _{ t \leq 0} \chi(t) \int_{\mathbb{R}^{3}} f_{0} (x,v)
\zeta_{i} (v) \mathrm{d} v .
\end{eqnarray*}

Therefore 
\begin{eqnarray*}
\sum_{i=0}^{4}\mathbf{1}_{t \geq 0} |a_{i} (t,x)| &\leq& \sum_{i=0}^{4}\Big|%
\int_{\mathbb{R}^{3}} f_{\delta} (t,x,v) \zeta_{i} (v) \mathrm{d} v\Big| + 
\mathbf{1} _{ t \leq 0} \chi(t) \int_{\mathbb{R}^{3}} | f_{0} (x,v)|
\sum_{i=0}^{4}|\zeta_{i} (v)| \mathrm{d} v \\
&&+ \mathbf{1}_{t \geq 0} \Big\{ O(\delta) \sum_{j=0}^{4} |a_{j} (t,x)| +
O_{\delta} (1) \int_{\mathbb{R}^{3}} | (\mathbf{I} - \mathbf{P}) f(t,x,v) |
\sum_{i=0}^{4} |\zeta_{i} (v)| \mathrm{d} v \Big\}.
\end{eqnarray*}
Hence for all $i=0,1,2,3,4$, 
\begin{equation}  \label{bound_a_i}
\begin{split}
|a_{i} (t,x)| \leq& \ 4 \int_{\mathbb{R}^{3}} |f_{\delta} (t,x,v) | \langle
v\rangle^{2} \sqrt{\mu(v)} \mathrm{d} v + 4\chi(t)\mathbf{1} _{ t \leq 0}
\int_{\mathbb{R}^{3}} |f_{0} (x,v)| \langle v \rangle^{2} \sqrt{\mu(v)} 
\mathrm{d} v \\
& + 4 \int_{\mathbb{R}^{3}} | (\mathbf{I} - \mathbf{P}) f(t,x,v)| \langle
v\rangle^{2} \sqrt{\mu(v)} \mathrm{d} v.
\end{split}%
\end{equation}

Now we focus on the part of $f_{\delta}$ in (\ref{bound_a_i}). From Lemma %
\ref{extension_dyn}, there is an extension $\overline{f_{\delta}}$ defined
by (\ref{bar_Z_dyn}) such that 
\begin{equation*}
\int_{\mathbb{R}^{3}} |f_{\delta} (t,x,v) | \langle v\rangle^{2} \sqrt{\mu(v)%
} \mathrm{d} v \leq \int_{\mathbb{R}^{3}} | \overline{f_{\delta}}(t,x,v) |
\langle v\rangle^{2} \sqrt{\mu(v)} \mathrm{d} v.
\end{equation*}
and, solves 
\begin{equation}
\e \pt_t \overline{f_{\delta}}+v\cdot \nabla_x \overline{f_{\delta}}+\e%
^2\Phi\cdot \nabla_v \overline{f_{\delta}} + \overline{f_{\delta}}=h+%
\overline{f_{\delta}},  \label{bastransp0}
\end{equation}
in the sense of distributions with $h=\sum_{i=1}^4 h_i$ where $h_i$ are
defined in Lemma \ref{extension_dyn}.

Now we apply Lemma \ref{flfs} to (\ref{bastransp0}) with $f= \overline{%
f_{\delta}}$ and $q= h+ \overline{f_{\delta}}$. Then we can decompose 
\begin{equation}  \label{decom_f_delta}
\overline{f_{\delta}} = \overline{f_{\delta}}_{,la} + \overline{f_{\delta}}%
_{,sm},
\end{equation}
where $\overline{f_{\delta}}_{,la}$ is defined, as (\ref{f_la_def}), 
\begin{equation}  \label{f_delta_la}
\overline{f_{\delta}}_{,la} (t,x,v) = \int^{\infty}_{0} e^{-\tau} (h+ 
\overline{f_{\delta}}) (t- \e \tau, x- \tau v, v) \mathrm{d} \tau,
\end{equation}
and $\overline{f_{\delta}}_{,sm}:=\overline{f_{\delta}} - \overline{%
f_{\delta}}_{,la}$ as in (\ref{f_sm_def}).

Using (\ref{fl}) and (\ref{fs1}) with $\psi=\zeta_{i}$ in (\ref{zeta}), we
deduce that 
\begin{equation}  \label{average_fdelta}
\begin{split}
\Big\|\int_{\mathbb{R}^{3}}\zeta_{i} (v) \overline{f_{\delta}}_{,la} \mathrm{%
d} v\Big\|_{L^{2}_{t} L^{3}_{x}} &\lesssim \| w^{-1} h \|_{L^{2}_{t,x,v}} +
\| w^{-1} \overline{f_{\delta}} \|_{L^{2}_{t,x,v}}, \\
\Big\|\int_{\mathbb{R}^{3}}\zeta_{i} (v) \overline{f_{\delta}}_{,sm} \mathrm{%
d} v\Big\|_{L^{2}_{t} L^{\frac{12}{5}}_{x}} &\lesssim \e^{\frac{1}{2}}( \|
w^{-1} h \|_{L^{2}_{t,x,v}} + \| w^{-1} \overline{f_{\delta}}
\|_{L^{2}_{t,x,v}}).
\end{split}%
\end{equation}
Note that from Lemma \ref{extension_dyn} 
\begin{equation}
\begin{split}  \label{forcing_fdelta}
& \| w^{-1} h \|_{L^{2}_{t,x,v}} + \| w^{-1} \overline{f_{\delta}}
\|_{L^{2}_{t,x,v}} \\
&\lesssim \| g \|_{L^{2}_{t,x,v}} + \| f_{0} \|_{L^{2}_{x,v}} + \| [ v\cdot
\nabla_{x} + \e^{2} \Phi \cdot \nabla_{v}] f_{0} \|_{L^{2}_{x,v}} + \| f_{0}
\|_{L^{2}(\gamma)} + \| f_{\gamma} \|_{L^{2}(\mathbb{R}_{+} \times \gamma)}.
\end{split}%
\end{equation}

Finally we set 
\begin{equation}  \label{def_S2}
\begin{split}
\mathbf{S}_{1} f(t,x) &: = \int_{\mathbb{R}^{3}} | \overline{f_{\delta}}%
_{,la} | \langle v\rangle^{2} \sqrt{\mu(v)} \mathrm{d} v, \\
\mathbf{S}_{2} f(t,x) &: = \int_{\mathbb{R}^{3}} | \overline{f_{\delta}}%
_{,sm} | \langle v\rangle^{2} \sqrt{\mu(v)} \mathrm{d} v, \\
\mathbf{S}_{3} f(t,x) &: = 4 \int_{\mathbb{R}^{3}} | (\mathbf{I} - \mathbf{P}%
) f(t,x,v)| \langle v\rangle^{2} \sqrt{\mu(v)} \mathrm{d} v.
\end{split}%
\end{equation}
Then by (\ref{average_fdelta}) and (\ref{forcing_fdelta}) we conclude (\ref%
{fl0})$-$%
(\ref{ins2}).
\end{proof}

\bigskip

\subsection{Unsteady $L^{2}-$Coercivity Estimate}

The main purpose of this section is to prove the following:

\begin{proposition}
\label{dlinearl2}Suppose $\Phi = \Phi(x) \in C^{1}, g \in L^{2}(\mathbb{R}%
_{+} \times \Omega \times \mathbb{R}^{3})$, and $r \in L^{2}(\mathbb{R}_{+}
\times \gamma_{-})$ such that, for all $t>0$, 
\begin{equation}
\iint_{\Omega \times \mathbb{R}^{3}}g(t,x,v)\sqrt{\mu }\mathrm{d} v \mathrm{d%
} x=0=\int_{\gamma _{-}}r (t,x,v)\sqrt{\mu }\mathrm{d}\gamma .
\label{dlinearcondition}
\end{equation}
Then, for any sufficiently small $\e$, there exists a unique solution to the
problem 
\begin{equation}
\e\partial _{t}f+v\cdot \nabla _{x}f+\frac 1 {\sqrt{\mu}}\e%
^2\Phi\cdot\nabla_v(\sqrt{\mu} f)+\e^{-1}Lf=\ g,  \label{dlinear}
\end{equation}%
with $f|_{t=0}=f_{0}$ and $f_{-}=P_{\gamma }f+r$ on $\mathbb{R}_{+} \times
\gamma _{-}$ such that 
\begin{equation}
\iint_{\Omega \times \mathbb{R}^{3}}f(t,x,v)\sqrt{\mu }\mathrm{d} x\mathrm{d}
v=0, \ \ \ \text{for all} \ t\geq0.  \label{dlinearcondition1}
\end{equation}
Moreover, there is $0<\lambda \ll 1$ such that for $0 \leq s \leq t$, 
\begin{eqnarray}
&&\| e^{\lambda t}f(t)\|_2^2+ \e^{-2}\int_s^t \| e^{\lambda \tau}(\mathbf{I}-%
\mathbf{P}) f (\tau)\|_\nu^2 \mathrm{d} \tau + \int^{t}_{s} \| e^{\lambda
\tau}\mathbf{P} f (\tau)\|_{2}^{2} \mathrm{d} \tau  \notag \\
&& {\ + \e^{-1} {\ \int_s^t | e^{\lambda \tau} (1-P_\gamma) f |^2_{ 2 } }} +{%
\ {\ \int_s^t | e^{\lambda \tau} f |^2_{ 2 } }}  \label{completes_dyn} \\
& \lesssim & \ \|e^{\lambda s} f(s)\|_2^2 + \e^{-1}\int_s^t |e^{\lambda
\tau}r |^2_{2,-} {\ +\int_s^t \| \nu^{- \frac{1}{2}}e^{\lambda \tau} {(%
\mathbf{I} - \mathbf{P}) g } \|_{2}^{2} + \e^{-2} \int_s^t \| e^{\lambda
\tau} {\ \mathbf{P} g } \|_{2}^{2} . \notag }
\end{eqnarray}%
\end{proposition}

In order to prove the proposition we need the following:

\begin{lemma}
\label{dabc}Assume that $g$ and $r$ satisfy (\ref{dlinearcondition}) and $f$
satisfies (\ref{dlinear}), and (\ref{dlinearcondition1}). Then there exists
a function $G(t)$ such that, for all $0\le s\leq t$, $G(s)\lesssim
\|f(s)\|_{2}^{2}$ and 
\begin{equation*}
\int_{s}^{t}\|\mathbf{P}f(\tau)\|_{\nu }^{2} \lesssim
G(t)-G(s)+\int_{s}^{t}\| {\frac{g(\tau)}{\sqrt{\nu}}} \|_{2}^{2}
+|r(\tau)|_{2,-}^{2} + \e^{-2}\int_{s}^{t}\|(\mathbf{I}-\mathbf{P}%
)f(\tau)\|_{\nu }^{2} +\int_{s}^{t}|(1-P_{\gamma })f(\tau)|_{2,+}^{2} .
\end{equation*}
\end{lemma}

\begin{proof}
The key of the proof is to use the same choices of test functions (with
extra dependence on time) of (\ref{phic}), (\ref{phibj}), (\ref{phibij}) and
(\ref{phia}) and estimate the new contribution $\int_{s}^{t}\iint_{\Omega
\times \mathbb{R}^{3}}\partial _{t}\psi f$ in the time dependent weak
formulation 
\begin{eqnarray}
&& \int_{s}^{t}\int_{\gamma _{+}}\psi f -\int_{s}^{t}\int_{\gamma _{-}}\psi
f -\int_{s}^{t}\iint_{\Omega \times \mathbb{R}^{3}}\big\{v\cdot \nabla
_{x}\psi+\e^2\sqrt{\mu}\Phi\cdot\nabla_v\frac{\psi}{\sqrt{\mu}}\big\} f -\e%
\int_{s}^{t}\iint_{\Omega \times \mathbb{R}^{3}}\partial _{t}\psi f  \notag
\\
&&=-\e\iint_{\Omega \times \mathbb{R}^{3}}\psi f(t)+\e\iint_{\Omega \times 
\mathbb{R}^{3}}\psi f(s) +\e^{-1}\int_{s}^{t}\iint_{\Omega \times \mathbb{R}%
^{3}}[-\psi L(\mathbf{I}-\mathbf{P})f+\psi g].  \label{timeweak}
\end{eqnarray}
We note that, with such choices $G(t)=-\iint_{\Omega \times \mathbb{R}%
^{3}}\psi f(t)$, and $|G(t)|\lesssim \| f(t)\| _{2}^{2}$.\ Without loss of
generality we give the proof for $s=0$.

\textit{Remark.} We note that (\ref{dlinearcondition}), (\ref{dlinear}) and (%
\ref{dlinearcondition1}) are all invariant under a standard $t$%
-mollification for all $t>0$. The estimates in \textit{Step 1} to \textit{%
Step 3} {below} are obtained via a $t$-mollification so that all the
functions are smooth in $t$. For the notational simplicity we do not write
explicitly the parameter of the regularization. 

\vspace{4pt}

\noindent \textit{Step 1.} \textit{Estimate of }$\nabla _{x}\Delta
_{N}^{-1}\partial _{t}a=\nabla _{x}\partial _{t}\f _{a}$. In the weak
formulation (with time integration over $[t,t+\d ]$), if we choose the test
function $\psi =\f \sqrt{\mu }$ with $\f (x)$ dependent only of $x$, then we
get (note that $Lf$ and $g$, {integrated} against $\f (x)\sqrt{\mu }$ are
zero) 
\begin{equation*}
\e\int_{\Omega }[a(t+\d )-a(t)]\f (x)=\int_{t}^{t+\d }\int_{\Omega }(b\cdot
\nabla _{x})\f (x)+\int_{t}^{t+\d }\int_{\gamma _{-}}r\f \sqrt{\mu },
\end{equation*}%
where we have used the splitting (\ref{bsplit}) and (\ref{insidesplit}).
Taking difference quotient, we obtain for all $t$ 
\begin{equation*}
\e\int_{\Omega }\f \partial _{t}a=\int_{\Omega }(b\cdot \nabla _{x})\f %
+\int_{\gamma _{-}}r\f \sqrt{\mu }.
\end{equation*}%
Notice that, for $\f =1$, from (\ref{dlinearcondition}), the right hand side
of the above equation is zero. Hence, for all $t>0$, $\int_{\Omega }\partial
_{t}a(t)dx=0$. On the other hand, for all $\f (x)\in H^{1}(\Omega )$, we
have, by the trace theorem $|\f |_{2}\lesssim \| \f \| _{H^{1}}$, $\e\left\|
\int_{\Omega }\f(x)\partial _{t}a \mathrm{d} x\right\| \lesssim |r|_{2,-}|\f %
|_{2}+ \| b \| _{2}\| \f \| _{H^{1}}\lesssim \{ \| b(t) \|_{2}+|r|_{2,-}\}
\| \f \| _{H^{1}}$. Therefore we conclude that, for all $t>0$, 
\begin{equation*}
\e\| \partial _{t}a(t)\| _{(H^{1})^{\ast }}\ \lesssim \ \| b(t)\|
_{2}+|r|_{2},
\end{equation*}
where $(H^{1})^{\ast }\equiv (H^{1}(\Omega ))^{\ast }$ is the dual space of $%
H^{1}(\Omega )$ with respect to the dual pair $\langle A,B\rangle
=\int_{\Omega }A(x)B(x)\mathrm{d} x$, for $A\in H^{1}$ and $B\in
(H^{1})^{\ast }$.

On the other hand, $\f _{a}$ in (\ref{phia}) is the solution of $- \Delta
\partial_{t} \f _{a} = \partial_{t} a, \ \frac{\partial }{\partial n}%
\partial_{t} \f _{a} =0$ at $\partial\Omega$ with $\int_{\Omega }\partial
_{t}a(t,x)dx=0$ for all $t>0$. From the standard elliptic theory, 
\begin{equation*}
\e\| \nabla _{x}\partial _{t}\f _{a}\| _{2}=\e\| \Delta _{N}^{-1}\partial
_{t}a(t)\| _{H^{1}} \lesssim \e\| \partial _{t}a(t)\| _{(H^{1})^{\ast
}}\lesssim \{\| b(t)\| _{2}+|r|_{2}\}.
\end{equation*}%
Therefore, we conclude, for almost all $t>0$, 
\begin{equation}
\| \nabla _{x}\partial _{t}\f _{a}(t)\| _{2}\lesssim\e^{-1} \{\| b(t)\|
_{2}+|r|_{2}\}.  \label{-1a}
\end{equation}

\vspace{4pt}

\noindent \textit{Step 2.} \textit{Estimate of }$\nabla _{x}\Delta
^{-1}\partial _{t}b^{j}=\nabla _{x}\partial _{t}\f _{b}^{i}$. In (\ref%
{timeweak}), we choose a test function $\psi =\f
(x)v_{i}\sqrt{\mu }$. Since 
$\int v_{i}v_{j}\mu (v)\mathrm{d} v=\int v_{i}v_{j}(\frac{|v|^{2}}{2}-\frac{3%
}{2})\mu (v)\mathrm{d} v=\delta _{ij}$, we get 
\begin{eqnarray*}
\e \int_{\Omega }[b_{i}(t+\d )-b_{i}(t)]\f &=&-\int_{t}^{t+\d }\int_{\gamma
}f\varphi v_{i}\sqrt{\mu }+\int_{t}^{t+\d }\int_{\Omega }\partial _{i}\f %
\lbrack a+c]-\e^2\int_t^{t+\d}\int_\O \Phi_i\f a \\
& +&\e^{-1}\int_{t}^{t+\d }\iint_{\Omega \times \mathbb{R}%
^{3}}\sum_{j=1}^{d}v_{j}v_{i}\sqrt{\mu }\partial _{j}\f (\mathbf{I}-\mathbf{P%
})f+\int_{t}^{t+\d }\iint_{\Omega \times \mathbb{R}^{3}}\f v_{i}g\sqrt{\mu }.
\end{eqnarray*}%
Taking difference quotient, we obtain 
\begin{eqnarray*}
\e\int_{\Omega }\partial _{t}b_{i}(t)\f &=&-\int_{\gamma }f(t)v_{i}\f \sqrt{%
\mu }+\int_{\Omega }\partial _{i}\f \lbrack a(t)+c(t)]-\e^2\int_\O \Phi_i\f a
\\
& +&\e^{-1}\iint_{\Omega \times \mathbb{R}^{3}}\sum_{j=1}^{d}v_{j}v_{i}\sqrt{%
\mu }\partial _{j}\f (\mathbf{I}-\mathbf{P})f(t)+\iint_{\Omega \times 
\mathbb{R}^{3}}\f v_{i}g(t)\sqrt{\mu }.
\end{eqnarray*}%
For fixed $t>0$, we choose $\f = \partial _{t}\f _{b}^{i}$ in (\ref{jb})
solving $-\Delta \partial _{t}\f _{b}^{i}=\partial _{t}b_{i}(t), \ \partial
_{t}\f _{b}^{i}|_{\partial \Omega }=0. $ The boundary terms vanish because
of the Dirichlet boundary condition on $\partial _{t}\f _{b}^{i}$. Then we
have, for $t\geq 0$, 
\begin{eqnarray*}
&&\e\int_{\Omega }|\nabla _{x}\Delta ^{-1}\partial _{t}b_{i}(t)|^{2} =\e%
\int_{\Omega }|\nabla _{x} \partial _{t}\f _{b}^{i}|^{2} =-\e\int_{\Omega
}\Delta \partial _{t}\f _{b}^{i} \partial _{t}\f _{b}^{i} \\
&\lesssim &4\eta\e \{\| \nabla _{x} \partial _{t}\f _{b}^{i}\| _{2}^{2}+\|
\partial _{t}\f _{b}^{i}\| _{2}^{2}\}+\frac{1}{4\eta\e}[\| a(t)\|
_{2}^{2}+\| c(t)\| _{2}^{2}+\| \frac{g}{\sqrt{\nu}}\|_2^2]+\frac 1{4\eta\e%
^3}\| (\mathbf{I}-\mathbf{P})f(t)\| _{2}^{2} \\
&& +\frac{\e^3}{4\eta}\|\Phi\|_\infty^2\|a\|_2^2 \\
&\lesssim &8\eta\e \| \nabla _{x} \partial _{t}\f _{b}^{i} \| _{2}^{2}+\frac
1{4\eta\e}[\| a(t)\| _{2}^{2}+\| c(t)\| _{2}^{2}+\| \frac{g }{\sqrt{\nu}}
\|_2^2] +\frac 1{4\eta\e^3} \| (\mathbf{I}-\mathbf{P})f(t)\| _{2}^{2} \\
&&+\| \frac{g }{\sqrt{\nu}} \| _{2}^{2} +\frac{\e^3}{4\eta}%
\|\Phi\|_\infty^2\|a\|_2^2,
\end{eqnarray*}%
where we have used the Poincar\'e inequality. Hence, for all $t>0$ 
\begin{equation}
\| \nabla _{x} \partial _{t} \varphi_{b}^{i} \| _{2} \lesssim \e^{-1}\big\{%
\| a(t)\| _{2} +\| c(t)\| _{2}+\|\frac{g}{\sqrt{\nu}}\|_2\big\}+ \e^{-2}\| (%
\mathbf{I}-\mathbf{P})f(t)\| _{2}.  \label{-1b}
\end{equation}

\vspace{4pt}

\noindent \textit{Step 3.} \textit{Estimate of }$\nabla _{x}\Delta
^{-1}\partial _{t}c=\nabla _{x}\partial _{t}\f _{c}$. In the weak
formulation, we choose a test function $\f (x)(\frac{|v|^{2}}{2}-\frac{3}{2})%
\sqrt{\mu }$. Since $\int \mu (v)(\frac{|v|^{2}}{2}-\frac{3}{2})=0$, $\int
\mu (v)v_{i}v_{j}(\frac{|v|^{2}-3}{2})=\delta _{ij}$, $\int \mu (v)(\frac{%
|v|^{2}}{2}-\frac{3}{2})^{2} \neq 0$, 
\begin{eqnarray*}
&& \frac 3 2\e\int_{\Omega }\f (x)c(t+\d ,x)\mathrm{d} x-\frac 3 2\e%
\int_{\Omega }\f (x)c(t,x) \mathrm{d} x = \int_{t}^{t+\d }\int_{\Omega
}b\cdot \nabla _{x}\f -\int_{t}^{t+\d }\int_{\gamma }(\frac{|v|^{2}}{2}-%
\frac{3}{2})\sqrt{\mu }\f f \\
&& +\e^{-1}\int_{t}^{t+\d }\iint_{\Omega \times \mathbb{R}^{3}}(\mathbf{I}-%
\mathbf{P})f(\frac{|v|^{2}}{2}-\frac{3}{2})\sqrt{\mu }(v\cdot \nabla _{x})\f
\\
&& +\int_{t}^{t+\d }\iint_{\Omega \times \mathbb{R}^{3}}\f g\Big(\frac{%
|v|^{2}}{2}-\frac{3}{2}\Big)\sqrt{\mu }-\e^2\int_t^{t+\d}\int_\O \f\Phi\cdot
b.
\end{eqnarray*}%
Taking difference quotient, we obtain 
\begin{eqnarray*}
&&\e\int_{\Omega }\f (x)\partial _{t}c(t,x)\mathrm{d} x=\frac{2}{3}%
\int_{\Omega }b(t)\cdot \nabla _{x}\f -\frac 2 3\int_{\gamma }(\frac{|v|^{2}%
}{2}-\frac{3}{2})\sqrt{\mu }\varphi f(t) \\
&&\ \ +\frac 2 3\iint_{\Omega \times \mathbb{R}^{3}} \e^{-1}(\mathbf{I}-%
\mathbf{P})f(t)(\frac{|v|^{2}}{2}-\frac{3}{2})\sqrt{\mu }(v\cdot \nabla _{x})%
\f +\f g(t)(\frac{|v|^{2}}{2}-\frac{3}{2})\sqrt{\mu }-\frac 2 3\e^2\int_\O \f%
\Phi\cdot b.
\end{eqnarray*}%
Note that $\partial_{t} \varphi_{c}$ in (\ref{phic}) is the solution of $%
-\Delta \partial_{t} \varphi_{c}=\partial _{t}c(t)$, $\partial_{t}
\varphi_{c}|_{\partial \Omega }=0$. 

The boundary terms vanish because of the Dirichlet boundary condition on $%
\partial_{t} \varphi_{c}$. We follow the same procedure of estimates $\nabla
_{x}\Delta ^{-1}\partial _{t}a$ and $\nabla _{x}\Delta ^{-1}\partial _{t}b$
to have 
\begin{eqnarray*}
&&\e\| \nabla _{x}\Delta ^{-1}\partial _{t}c(t)\| _{2}^{2}=\e\int_{\Omega
}|\nabla _{x} \partial_{t} \varphi_{c}(x)|^{2}dx=\e\int_{\Omega }
\partial_{t} \varphi_{c}(x)\partial _{t}c(t,x)dx \\
&\lesssim &4 \eta\e \{\| \nabla _{x} \partial_{t} \varphi_{c} \| _{2}^{2}+\|
\partial_{t} \varphi_{c}\| _{2}^{2}\}+\frac 1 {4\eta\e}\| b(t)\|
_{2}^{2}+\frac 1 {4\eta\e^3}\| (\mathbf{I}-\mathbf{P})f(t)\| _{2}^{2}+\| 
\frac{g}{\sqrt{\nu}}\| _{2}^{2}+\frac{\e^3}{4\eta}\|\Phi\|_\infty^2\|b\|_2^2
\\
&\lesssim &8\eta\e \| \nabla _{x} \partial_{t} \varphi_{c}\| _{2}^{2}+\frac
1{4\eta\e}\| b(t)\| _{2}^{2} +\frac 1{4\eta\e^3}\| (\mathbf{I}-\mathbf{P}%
)f(t)\| _{2}^{2}+\| \frac{g}{\sqrt{\nu}}\| _{2}^{2}+\frac{\e^3}{4\eta}%
\|\Phi\|_\infty^2\|b\|_2^2,
\end{eqnarray*}%
where we have used the Poincar\'e inequality. Finally we have, for all $t>0$%
, 
\begin{equation}
\| \nabla _{x}\Delta ^{-1}\partial _{t}c(t)\| _{2}\lesssim \e^{-1}\big\{ \|
b(t)\| _{2} +\| \frac{g}{\sqrt{\nu}}\| _{2}\big\}+\e^{-2}\| (\mathbf{I}-%
\mathbf{P})f(t)\| _{2}.  \label{-1c}
\end{equation}

\vspace{4pt}

\noindent\textit{Step 4.} \textit{Estimate of }$a,b,c$\textit{\
contributions in} (\ref{timeweak}). To estimate $c$ contribution in (\ref%
{timeweak}), we plug (\ref{phic}) into (\ref{timeweak}) to have from (\ref%
{insidesplit}) 
\begin{equation*}
\begin{split}
&\int_{0}^{t}\iint_{\Omega \times \mathbb{R}^{3}}(|v|^{2}-\beta _{c})v_{i}%
\sqrt{\mu }\partial _{t}\partial _{i}\f _{c}f \\
=&\sum_{j=1}^{d}\int_{0}^{t}\iint_{\Omega \times \mathbb{R}%
^{3}}(|v|^{2}-\beta _{c})v_{i}v_{j}\mu (v)\partial _{t}\partial _{i}\f %
_{c}b_{j} +\int_{0}^{t}\iint_{\Omega \times \mathbb{R}^{3}}(|v|^{2}-\beta
_{c})v_{i}\sqrt{\mu }\partial _{t}\partial _{i}\f _{c}(\mathbf{I}-\mathbf{P}%
)f.
\end{split}%
\end{equation*}%
The second line has non-zero contribution only for $j=i$ which leads to zero
by the definition of $\beta _{c}$ in (\ref{beta}). We thus have from (\ref%
{-1c}), {for $\e $ small,} 
\begin{eqnarray}
&&\e\left\| \int_{0}^{t}\iint_{\Omega \times \mathbb{R}^{3}}(|v|^{2}-\beta
_{c})v_{i}\sqrt{\mu }\partial _{t}\partial _{i}\f _{c}f\right\|  \notag \\
&\lesssim& \e\int_{0}^{t}\Big\{\e^{-1}[\| b\| _{2}+\| \frac{g}{\sqrt{\nu}}
\| _{2}]+\e^{-2}\| (\mathbf{I}-\mathbf{P})f\| _{2} \Big\}\| (\mathbf{I}-%
\mathbf{P})f\| _{2}  \notag \\
& \lesssim& \int_{0}^{t} \eta\| b\| _{2}^{2}+ \e^{-1}\| (\mathbf{I}-\mathbf{P%
})f\| _{2}^{2}+\| \frac{g}{\sqrt{\nu}} \| _{2}^{2}.  \notag
\end{eqnarray}%
{Combining with (\ref{cestimate}), we conclude, for $\eta $ small}, 
\begin{eqnarray}
\int_{{0}}^{t}\| c(s)\| _{2}^{2}\mathrm{d} s &\lesssim &G(t)-G(0) +\int_{{0}%
}^{t}\Big\{\e^{-2}\| (\mathbf{I}-\mathbf{P})f(s)\| _{\nu }^{2}+\| \frac{g}{%
\sqrt{\nu}} \| _{2}^{2}+|(1-P_{\gamma })f(s)|_{2,+}^{2}  \notag
\label{timec} \\
&& \ \ \ \ \ \ \ \ \ \ \ \ \ \ \ \ \ \ \ \ \ \ \ +|r(s)|_{2}^{2}+\eta\|
b(s)\| _{2}^{2}+\e^2\|\Phi\|_\infty \big[\|a(s)\|_2^2+\|b(s)\|_2^2\big]\Big\}%
\mathrm{d} s.  \notag
\end{eqnarray}%
To estimate $b$ in (\ref{timeweak}), by (\ref{insidesplit}), we plug (\ref%
{phibj}) into (\ref{timeweak}) to get: 
\begin{equation}
\int_{0}^{t}\iint (v_{i}^{2}-\beta _{b})\sqrt{\mu }\partial _{t}\partial _{j}%
\f _{b}^{j}f={{\int_{0}^{t}\iint (v_{i}^{2}-\beta _{b}){\mu }\partial
_{t}\partial _{j}\f _{b}^{j}\{\frac{|v|^{2}}{2}-\frac{3}{2}\}c}}%
+(v_{i}^{2}-\beta _{b})\sqrt{\mu }\partial _{t}\partial _{j}\f_{b}^{j}(%
\mathbf{I}-\mathbf{P})f, \ \ \   \notag
\end{equation}%
where we have used (\ref{alpha}) to remove the $a$ contribution. We thus
have from (\ref{-1b}), 
\begin{eqnarray*}
&&\left\| \int_{0}^{t}\iint_{\Omega \times \mathbb{R}^{3}}(v_{i}^{2}-\beta
_{b})\sqrt{\mu }\partial _{t}\partial _{j}\f _{b}^{j}f\right\|  \notag \\
&\lesssim &\int_{0}^{t}\Big\{\| a\| _{2}+\| c\| _{2}+\| (\mathbf{I}-\mathbf{P%
})f\| _{2}+\| \frac{g}{\sqrt{\nu}} \| _{2}\Big\}\Big\{\| c\| _{2}+\| (%
\mathbf{I}-\mathbf{P})f\| _{2}\Big\}  \notag \\
&\lesssim &\int_{0}^{t}\Big\{\| (\mathbf{I}-\mathbf{P})f\| _{2}^{2}+\| c\|
_{2}^{2}+\| \frac{g}{\sqrt{\nu}} \| _{2}^{2}+{\e \| a\| _{2}^{2}\Big\}}.
\end{eqnarray*}%
Next we plug (\ref{phibij}) into (\ref{timeweak}) and from (\ref{-1b}), 
\begin{equation}
\begin{split}
& \e\int_{0}^{t}\int_{\Omega \times \mathbb{R}^{3}}|v|^{2}v_{i}v_{j}\sqrt{%
\mu }\partial _{t}\partial _{j}\f _{b}^{i}f=\e\int_{0}^{t}\int_{\Omega
\times \mathbb{R}^{3}}|v|^{2}v_{i}v_{j}\sqrt{\mu }\partial _{t}\partial _{j}%
\f _{b}^{i}(\mathbf{I}-\mathbf{P})f \\
\lesssim & \e\int_{0}^{t}\{\e^{-1}(\| a\| _{2}+\| c\| _{2}+\| \frac{g}{\sqrt{%
\nu}} \| _{2})+\e^{-2}\| (\mathbf{I}-\mathbf{P})f\| _{2}\}\| (\mathbf{I}-%
\mathbf{P})f\| _{2} \\
\lesssim & \int_{0}^{t}\Big\{\e^{-1}\| (\mathbf{I}-\mathbf{P})f\| _{\nu
}^{2}+\| \frac{g}{\sqrt{\nu}} \| _{2}^{2}+{\eta \big[\| a\| _{2}^{2}+\| c\|
_{2}^{2}}\big]\Big\}.
\end{split}
\label{bt2}
\end{equation}%
{Combining this with (\ref{bestimate}), we conclude from (\ref{timeweak})
that} 
\begin{equation}
\begin{split}  \label{timeb}
& \int_{0}^{t}\| b(s)\| _{2}^{2}\mathrm{d} s\lesssim G(t)-G(0) \\
& \ \ \ \ \ \ \ \ \ \ \ \ \ \ \ \ \ +\int_{0}^{t}\Big\{\e^{-2}\| (\mathbf{I}-%
\mathbf{P})f(s)\| _{\nu }^{2}+\| \frac{g}{\sqrt{\nu}} \|
_{2}^{2}+|(1-P_{\gamma })f(s)|_{2,+}^{2} \\
& \ \ \ \ \ \ \ \ \ \ \ \ \ \ \ \ \ \ \ \ \ \ +|r(s)|_{2}^{2}+\e%
^2\|\Phi\|_\infty(\|a\|_2^2+\|c\|_2^2)+\eta \| a\| _{2}^{2}+\| c\|
_{2}^{2}(1+\eta )\Big\}\mathrm{d} s.
\end{split}%
\end{equation}
Finally in order to estimate $a$ contribution in (\ref{timeweak}) we plug (%
\ref{phia}) for into (\ref{timeweak}). We estimate%
\begin{eqnarray}
&&\e\int_{0}^{t}\int_{\Omega \times \mathbb{R}^{3}}(|v|^{2}-\beta
_{a})v_{i}\mu \partial _{t}\partial _{i}\f _{a}f  \notag \\
&=&\e\int_{0}^{t}\int_{\Omega \times \mathbb{R}^{3}}(|v|^{2}-\beta
_{a})(v_{i})^{2}\mu \partial _{t}\partial _{i}\f _{a}b_{i}+\e%
\int_{0}^{t}\int_{\Omega \times \mathbb{R}^{3}}(|v|^{2}-\beta _{a})v_{i}\mu
\partial _{t}\partial _{i}\f _{a}(\mathbf{I}-\mathbf{P})f  \notag \\
&\lesssim &\e\int_{0}^{t}\e^{-1}\{\| b(t)\| _{2,\Omega }+|r|_{2}\}\{\| b\|
_{2}+\| (\mathbf{I}-\mathbf{P})f\| _{2}\}.  \notag
\end{eqnarray}%
Combining this with (\ref{aestimate}), we conclude 
\begin{equation}
\begin{split}  \label{timea}
\int_{0}^{t}\| a(s)\| _{2}^{2}\mathrm{d} s \ \lesssim \ &G(t)-G(0)
+\int_{0}^{t}\Big\{\e^{-2}\| (\mathbf{I}-\mathbf{P})f(s)\| _{\nu }^{2}
+|(1-P_{\gamma })f(s)|_{2,+}^{2} \\
& \ \ \ \ \ \ \ \ \ \ \ \ \ \ \ \ \ \ \ \ +\| \frac{g}{\sqrt{\nu}}(s) \|
_{2}^{2}+|r(s)|_{2}^{2}+{\| b\| _{2}^{2}}+\e^2\|\Phi\|_\infty \|\mathbf{P}
\mathring f\|_2^2\Big\}\mathrm{d} s.
\end{split}%
\end{equation}%
From (\ref{timec}), (\ref{timeb}) and (\ref{timea}), we prove the lemma for $%
\e$ sufficiently small, by choosing $\eta $ small.
\end{proof}

Now we are ready to prove the main result of this section:

\begin{proof}[\textbf{Proof of Proposition \protect\ref{dlinearl2}}]
Define the approximating sequence with $f^{0}\equiv f_{0}$, (with $\tilde
\nu=\nu-\frac 1 2\e^3\Phi\cdot v$): 
\begin{equation}
\partial _{t}f^{\ell +1}+\e^{-1}v\cdot \nabla _{x}f^{\ell +1}+\e%
\Phi\cdot\nabla_v f^{\ell+1}+\frac 1 {\e^2}\tilde\nu f^{\ell +1}-\frac 1 {\e%
^2}Kf^{\ell }=\e^{-1}g,\text{ \ \ }f^{n+1}|_{t=0}=f_{0},
\label{daproximatenew}
\end{equation}%
and $f_{-}^{\ell +1}=(1-\frac{\e}{j})P_{\gamma }f^{\ell }+r$.

\vskip.3cm \noindent \textit{Step 1.} \textit{Fix }$j$\textit{, }$f^{\ell
}\rightarrow f^{j}$\textit{\ as }$\ell \rightarrow \infty $\textit{.} {%
Notice that, 
\begin{eqnarray*}
|(Kf^{\ell },f^{\ell +1})| &\lesssim& \iint_{\mathbb{R}^{3} \times \mathbb{R}%
^{3}} |\mathbf{k}(v,u)|^{1/2} |f^{\ell}(u)| |\mathbf{k}(v,u)|^{1/2}
|f^{\ell+1} (v) | \mathrm{d} v \mathrm{d} u \\
&\lesssim&\sqrt{ \int_{u} |f^{\ell}(u)|^{2}\int_{v} |\mathbf{k}(v,u)| } 
\sqrt{ \int_{v} |f^{\ell+1}(v)|^{2} \int_{u} |\mathbf{k}(v,u)|} \ \lesssim \
\|f^{\ell}\| _{2}^{2} + \|f^{\ell+1}\| _{2}^{2},
\end{eqnarray*}
where we used $\sup_{u}\int_{v}| \mathbf{k}(v,u)|+\sup_{v}\int_{u}| \mathbf{k%
}(v,u)| <+\infty$.} 

Note that, since $\int_{\g_-} r\sqrt{\mu}=0$, we have 
\begin{equation}
|(1-\frac{\e}{j})P_{\gamma }f^{\ell }+r|_{{2,-}}^{2}=|(1-\frac{\e}{j}%
)P_{\gamma }f^{\ell }|_{{2,-}}^{2}+|r|_{2}^{2}.  \label{Cj}
\end{equation}%
By Green's identity (\ref{timeweak}) with $f^{\ell +1}$ in (\ref%
{daproximatenew}), 
\begin{eqnarray}
&&\| f^{\ell +1}(t)\| _{2}^{2}+\e^{-2}\int_{0}^{t}\| f^{\ell +1} \| _{ \nu
}^{2}+\e^{-1}\int_{0}^{t}|f^{\ell +1} |_{2,+}^{2}  \notag \\
&\leq & \e^{-1}\left[(1-\frac{\e}{j})^{2} \right] \int_{0}^{t}|P_{\gamma
}f^{\ell }|_{2,-}^{2} +\e^{-1}\int_{0}^{t}|r|_{2}^{2}+C\e^{-2}\int_{0}^{t}%
\max_{1\leq i\leq \ell +1}\| f^{i} \| _{2}^{2} +\int_{0}^{t}\| {\ \frac{g}{%
\sqrt{\nu}}} \| _{2 }^{2} +\| f_{0}\| _{2}^{2}  \notag \\
&\leq &\e^{-1}\left[ (1-\frac{\e}{j})^{2}\right] \int_{0}^{t}|f^{\ell }
|_{2,+}^{2} +\e^{-1}\int_{0}^{t}|r|_{2}^{2} +C\e^{-2}\int_{0}^{t}\max_{1\leq
i\leq \ell +1}\| f^{i} \| _{2}^{2} +\int_{0}^{t}\| {\frac{g}{\sqrt{\nu}} }\|
_{2 }^{2} +\| f_{0}\| _{2}^{2} .  \notag
\end{eqnarray}%

Set $\eta=(1-\frac \e j)^2<1$. Now use this inequality to bound for $\e%
^{-1}\int _0^t|f^\ell|^2_{2,-}$ and iterate:%
\begin{eqnarray*}
&&\| f^{\ell +1}(t)\| _{2}^{2}+\e^{-2}\int_{0}^{t}\| f^{\ell +1} \| _{ \nu
}^{2}+\e^{-1}\int_{0}^{t}|f^{\ell +1} |_{2,+}^{2}  \notag \\
&\leq &\eta\Big\{ \e^{-1}\eta \int_{0}^{t}|f^{\ell-1 } |_{2,+}^{2} +\e%
^{-1}\int_{0}^{t}|r|_{2}^{2}+C\e^{-2}\int_{0}^{t}\max_{1\leq i\leq \ell}\|
f^{i} \| _{2}^{2} +\int_{0}^{t}\| {\ \frac{g}{\sqrt{\nu}}} \| _{2 }^{2}ds+\|
f_{0}\| _{2}^{2} \Big\} \\
&&+\e^{-1}\int_{0}^{t}|r|_{2}^{2}+C\e^{-2}\int_{0}^{t}\max_{1\leq i\leq \ell
+1}\| f^{i} \| _{2}^{2} +\int_{0}^{t}\| {\frac{g}{\sqrt{\nu}}} \| _{\nu
}^{2} +\| f_{0}\| _{2}^{2}  \notag \\
&=& \eta^2\e^{-1}\int_0^t |f^{\ell-1 } |_{2,+}^{2} +(1+\eta)\Big\{\e%
^{-1}\int_{0}^{t}|r|_{2}^{2}+C\e^{-2}\int_{0}^{t}\max_{1\leq i\leq \ell
+1}\| f^{i} \| _{2}^{2} +\int_{0}^{t}\| {\frac{g}{\sqrt{\nu}}} \| _{2
}^{2}+\| f_{0}\| _{2}^{2}\Big\} \\
&& \ \ \ \ \ \ \ \ \ \ \ \ \ \ \ \ \ \ \ \ \ \ \ \ \ \ \ \ \ \ \ \ \ \ \ \ \
\ \ \ \ \ \ \ \vdots \\
& \lesssim& \eta^{\ell+1}\e^{-1}\int_0^t |f^0|_{2,{\ + }}^2 + \frac{%
(1-\eta)^{\ell+1}}{1-\eta} \Big\{\e^{-1}\int_{0}^{t}|r|_{2}^{2} + \e%
^{-2}\int_{0}^{t}\max_{1\leq i\leq \ell +1}\| f^{i} \| _{2}^{2}
+\int_{0}^{t}\| {\frac{g}{\sqrt{\nu}}} \| _{2}^{2} +\| f_{0}\| _{2}^{2}\Big\}%
.
\end{eqnarray*}

We therefore have, from $f^{0}\equiv f_{0}$, 
\begin{equation*}
\max_{1\leq i\leq \ell +1}\| f^{i}(t)\| _{2}^{2}\lesssim _{\eta ,j} \Big\{\e%
^{-1}\int_{0}^{t}|r|_{2}^{2}+\e^{-2}\int_{0}^{t}\max_{1\leq i\leq \ell +1}\|
f^{i} \| _{2}^{2} +\int_{0}^{t}\| g \| _{\nu }^{2} +\| f_{0}\| _{2}^{2}+\e%
^{-1}t\| f_{0}\| _{\nu }^{2}+t|f_{0}|_{2,+}^{2} \Big\}.
\end{equation*}
By Gronwall's lemma, we have, for fixed $t>0$, 
\begin{equation*}
\max_{1\leq i\leq \ell +1}\| f^{i}(t)\| _{2}^{2}\lesssim _{\eta,\e %
,j,t}\left\{ \int_{0}^{t}|r|_{2}^{2}+\int_{0}^{t}\| g(s)\| _{\nu }^{2}%
\mathrm{d} s+\| f_{0}\| _{2}^{2}+t\| f_{0}\| _{\nu
}^{2}+t|f_{0}|_{2,+}^{2}\right\}.
\end{equation*}%
This in turns leads to 
\begin{eqnarray*}
&&\max_{1\leq i\leq \ell +1}\left\{ \| f^{i}(t)\| _{2}^{2}+\int_{0}^{t}\|
f^{i} \| _{\nu }^{2}+\int_{0}^{t}|f^{i} |_{2, {+}}^{2} \right\} \\
&& \lesssim _{\eta,\e ,j,t}\left\{ \int_{0}^{t}|r|_{2}^{2}+\int_{0}^{t}\| g
\| _{ \nu }^{2} +\| f_{0}\| _{2}^{2}+t\| f_{0}\| _{\nu
}^{2}+t|f_{0}|_{2,+}^{2}\right\} .
\end{eqnarray*}
Upon taking the difference, we have 
\begin{equation}
\partial _{t}[f^{\ell +1}-f^{\ell }]+\e^{-1}v\cdot \nabla _{x}[f^{\ell +1}-f^{\ell
}]+ \e \Phi\cdot\nabla_v [f^{\ell +1}-f^{\ell }] + \e^{-2}\tilde\nu \lbrack f^{\ell +1}-f^{\ell }]=\e^{-2}K[f^{\ell
}-f^{\ell -1}],\text{ \ \ \ }  \label{dndifference}
\end{equation}
with \ $[f^{\ell +1}-f^{\ell }](0)\equiv 0$ and $f_{-}^{\ell +1}-f_{-}^{\ell
}=(1-\frac{\e}{j})P_{\gamma }[f^{\ell }-f^{\ell -1}]$.

Applying previous
iteration to $f^{\ell +1}-f^{\ell }$ yields 
\begin{eqnarray}
&&\|f^{\ell+1}(t)-f^{\ell}(t)\|_{2}+\varepsilon
^{-2}\int_{0}^{t}\|f^{\ell+1}(s)-f^{\ell}(s)\|_{\nu }^{2} \dd s+\varepsilon
^{-1}\int_{0}^{t}|f^{\ell+1}(s)-f^{\ell}(s)|_{ 2,{+}}^{2}\dd s \notag\\
&\leq &\eta\varepsilon ^{-1}\int_{0}^{t}|f^{\ell}(s)-f^{\ell-1}(s)|_{2,{+}}^{2} \dd s+C_{K}\int_{0}^{t}\|f^{\ell}(s)-f^{\ell-1}(s)\|_{2}^{2}\dd s
\label{f_ell_Cauchy}
\\
&\leq &\eta\Big\{\varepsilon ^{-1}\int_{0}^{t}|f^{\ell}(s)-f^{\ell-1}(s)|_{2,{+}}^{2}\dd s+\sup_{0\leq s\leq T}\|f^{\ell}(s)-f^{\ell-1}(s)\|_{2}^{2}\Big\},\notag
\end{eqnarray}%
for $TC_{K}<\eta<1.$ This implies that $f^{\ell}$ is Cauchy with respect
to the norm    
\begin{equation*}
\varepsilon ^{-1}\int_{0}^{T}|f^{\ell}(s)|_{2,{+}}^{2}\dd s+\sup_{0\leq
s\leq T}\|f^{\ell}(s)\|_{2}^{2},
\end{equation*}
in $[0,T].$ Repeating the arguemnt for $[0,T],[T,2T]....$ we deduce that for
finite $t,$ there exists a (unique) limit function $f^{\ell}\rightarrow f^{j}$
such that
\begin{equation}  \label{dlinearj}
\begin{split}
\partial_{t}{f^{j}}+\e^{-1}v\cdot \nabla _{x} {f^{j}} +\frac{\e}{\sqrt{\mu}}
\Phi\cdot\nabla_v (\sqrt{\mu} {f^{j}})+\e^{-2}L {f^{j}}=g,\text{ \ \ } {%
f^{j} }(0)=f_{0}, \\
{f^{j}}|_{\gamma_{-}}=(1-\frac{\e}{j})P_{\gamma }{f^{j}}+r.
\end{split}%
\end{equation}

\vspace{4pt} \noindent \textit{Step 2.} \textit{Let} $j\rightarrow \infty $.
Upon using Green's identity and the boundary condition and (\ref{Cj}), 
\begin{eqnarray*}
&&\| f^{j}(t)\| _{2}^{2}+\e^{-2}\int_{0}^{t}\| (\mathbf{I}-\mathbf{P})f^{j}
\| _{\nu }^{2} +\e^{-1}\int_{0}^{t}|(1-P_{\gamma })f^{j} |_{2,+}^{2}  \notag
\\
&\leq &\e^{-1}\int_{0}^{t}|r|_{2}^{2} {\ + \frac{\e}{j} \e^{-1} \int^{t}_{0}
|P_{\gamma} f^{j}|_{2,+}^{2}} {+\int_{0}^{t}\| \nu^{- \frac{1}{2}} {(\mathbf{%
I} - \mathbf{P}) g} \| _{2 }^{2} + \e^{-2}\int_{0}^{t}\| {\ \mathbf{P} g} \|
_{2 }^{2} } \\
&&+\frac \e 2\int_0^t \int_{\O \times\mathbb{R}^3}|\Phi\|v\|f^j|^2+ \|
f_{0}\| _{2}^{2} .
\end{eqnarray*}
{By the trace theorem, Lemma \ref{trace_s}, 
\begin{equation*}
\begin{split}
& \int^{t}_{0} | P_{\gamma} f^{j}|_{2,+}^{2} \lesssim \ \int^{t}_{0} | 
\mathbf{1}_{\gamma_{+}^{\delta}} f^{j}|_{2,+}^{2} + \int^{t}_{0} |
(1-P_{\gamma} ) f^{j}|_{2,+}^{2} \\
\lesssim & \e \| f^{j} (0) \|_{2}^{2} + \e \int^{t}_{0} \| f^{j} \|_{2}^{2}
+ \e \int^{t}_{0} \iint_{\Omega \times \mathbb{R}^{3}} \big| [ \partial_{t} +%
\e^{-1} v \cdot \nabla_{x} + \e \Phi \cdot \nabla_{v}] f^{j} f^{j} \big| +
\int^{t}_{0} | (1-P_{\gamma} ) f^{j}|_{2,+}^{2}. \\
\lesssim & \ \e \| f^{j} (0) \|_{2}^{2} + \e \int^{t}_{0} \| f^{j}
\|_{2}^{2} + \int_{0}^{t} \iint_{\Omega \times \mathbb{R}^{3}} \e \frac{%
|\Phi \cdot v|}{2} |f^{j}|^{2}+ \e^{-2} {\nu} |( \mathbf{I} - \mathbf{P}%
)f^{j}|^{2} + |g f^{j}| + \int^{t}_{0} | (1-P_{\gamma} ) f^{j}|_{2,+}^{2}.
\end{split}%
\end{equation*}
From the boundary condition in (\ref{dlinearj}), $\int^{t}_{0}
|f^{j}|_{2,-}^{2} \lesssim \int^{t}_{0} |f^{j}|_{2,+}^{2} + \int^{t}_{0}
|r|_{2,+} ^{2}$. Finally from 
\begin{eqnarray*}
\e\int_0^t \int_{\O \times\mathbb{R}^3} \frac{|\Phi \cdot v|}{2} |f^j|^2
&\lesssim& \e\|\Phi \|_\infty \|(\mathbf{I}-\mathbf{P}) f^j\|_\nu^2+ \e%
\|\Phi\|_\infty \int_0^t \|\P f^j(s)\|_2^2 \mathrm{d} s, \\
\int^{t}_{0} \iint_{\Omega \times \mathbb{R}^{3}} | g f^{j}| & \lesssim &
\int^{t}_{0} \| \nu^{-\frac{1}{2}} (\mathbf{I} - \mathbf{P}){g} \|_{2}^{2} + %
\e^{2} \int^{t}_{0} \| \mathbf{P} {g} \|_{2}^{2} \\
&& + o(1) \big[ \int_{0}^{t} \|\P f^j(s)\|_2^2 \mathrm{d} s + \e^{-2}
\int_{0}^{t} \| (\mathbf{I} - \P ) f^j(s)\|_\nu^2 \mathrm{d} s \big],
\end{eqnarray*}
we get 
\begin{equation}
\begin{split}
& \| f^{j}(t)\| _{2}^{2}+\e^{-2} \int_{0}^{t}\| (\mathbf{I}-\mathbf{P})f^{j}
\| _{\nu }^{2} + \int_{0}^{t}| f^{j} |_{2 }^{2} \\
& \lesssim \e^{-1}\int_{0}^{t}|r|_{2}^{2} +\int_{0}^{t}\| \nu^{- \frac{1}{2}%
} (\mathbf{I} - \mathbf{P}) {g} \| _{2}^{2} +\e^{-2} \int_{0}^{t}\| \mathbf{P%
} {g} \| _{2}^{2} + \| f_{0}\| _{2}^{2}.  \label{jgreen1}
\end{split}%
\end{equation}%
}

Since $\| \mathbf{P}f^{j}(s)\| _{2}^{2}\lesssim $ $\| f^{j}(s)\| _{2}^{2}$,
integrating (\ref{jgreen1}) from $0$ to $t$, we have 
\begin{equation}
\int_{0}^{t}\|\mathbf{P}f^{j} \|_{2}^{2} \lesssim _{t}\e^{-1}%
\int_{0}^{t}|r|_{2}^{2}+\int_{0}^{t}\|\frac{g}{\sqrt{\nu}}\|_{2}^{2} + \|
f_{0}\|_{2}^{2}.  \notag
\end{equation}
Thus, we conclude that, for $j\gg 1$ and $0 < \e \ll 1$, 
\begin{equation}
\|f^{j}(t)\|_{2}^{2}+\e^{-2}\int_{0}^{t}\|f^{j}(s)\|_{\nu }^{2} +
\int_{0}^{t}|f^{j} |_{2 }^{2} \lesssim _{\e,t} \e^{-1}%
\int_{0}^{t}|r|_{2}^{2}+\int_{0}^{t}\|\frac{g}{\sqrt{\nu}}\|_{2}^{2} + \|
f_{0}\|_{2}^{2}.  \label{jbound}
\end{equation}%
By taking a weak limit, we obtain a weak solution $f$ to (\ref{dlinear})
with the same bound (\ref{jbound}). Taking difference, we have%
\begin{eqnarray*}
&&\partial _{t}[f^{j}-f]+\e^{-1}v\cdot \nabla _{x}[f^{j}-f]+\e\frac{1}{\sqrt{%
\mu}}\Phi\cdot\nabla_v( \sqrt{\mu}[f^{j}-f])+\e^{-2}L[f^{j}-f]=0,  \notag \\
&& [f^{j}-f]_{-}=P_{\gamma }[f^{j}-f]+\frac{\e}{j}P_{\gamma }f^{j},\quad
[f^{j}-f](0)=0.  \notag
\end{eqnarray*}
Applying (\ref{jbound}) with $r=\frac{\e}{j}P_{\gamma }f^{j}$ we obtain 
\begin{equation*}
\| f^{j}(t)-f(t)\| _{2}^{2}+\e^{-1}\int_{0}^{t}\| f^{j}(s)-f(s)\| _{\nu
}^{2}ds+ \int_{0}^{t}|f^{j}(s)-f(s)|_{2}^{2}ds\lesssim _{t}\frac{1}{j}
\int_{0}^{t}|P_{\gamma }f^{j}|^{2} \rightarrow 0.
\end{equation*}%
We thus construct $f$ as a $L^{2}$ solution to (\ref{dlinear}). %

\vspace{4pt}

\noindent \textit{Step 3.} \textit{Final estimate. }To conclude our
proposition, let $y(t)\equiv e^{\lambda t}f(t)$. We multiply (\ref{dlinear})
by $e^{\lambda t}$, so that $y$ satisfies 
\begin{equation}
\partial _{t}y+\e^{-1}v\cdot \nabla _{x}y+\frac{1}{\sqrt{\mu}}\e%
\Phi\cdot\nabla_v(\sqrt{\mu} y)+\e^{-2}Ly=\lambda y+e^{\lambda t}g,\text{ \
\ \ \ }y|_{\gamma _{-}}=P_{\gamma }y_{+}+e^{\lambda t}r.  \label{lineary}
\end{equation}%
By the Green's identity, 
\begin{equation}
\begin{split}
& \frac 1 2\| y(t)\| _{2}^{2}+\e^{-2}\int_{0}^{t}\| (\mathbf{I}-\mathbf{P}
)y(s)\| _{\nu }^{2}+\e^{-1}\int_{0}^{t}|(1-P_{\gamma })y(s)|_{2,+}^2 \\
\leq & \lambda \int_{0}^{t}\| y(s)\| _{2}^{2}+\| y(0)\| _{2}^{2}+\e%
^{-1}\int_{0}^{t}e^{\lambda s}|r|_{2}^{2} +\int_{0}^{t}e^{\lambda s}\|
\nu^{- \frac{1}{2}} (\mathbf{I} - \mathbf{P })g \| _{2}^{2} \\
& \ \ + \e^{-2}\int_{0}^{t}e^{\lambda s}\| g \| _{2}^{2} +\frac 1 2\e%
\int_0^t \int_{\O \times\mathbb{R}^3}\e^{\l s}|\Phi| | v| |y|^2.  \notag
\end{split}%
\end{equation}
From (\ref{dlinear}) we know that $\iint_{\Omega \times \mathbb{R}^{3}}y%
\sqrt{\mu }=\iint_{\Omega \times \mathbb{R}^{3}}(\lambda y+e^{\lambda t}g)%
\sqrt{\mu }=0$, $\int_{\gamma _{-}}e^{\lambda t}r\sqrt{\mu }d\gamma =0$.
Applying Lemma \ref{dabc} to (\ref{lineary}), we deduce 
\begin{eqnarray*}
\int_{0}^{t}\| \mathbf{P}y(s)\| _{\nu }^{2}\mathrm{d} s &\lesssim& G(t)-G(0)+%
\e^{-2}\int_{0}^{t}\| (\mathbf{I}-\mathbf{P})y(s)\| _{\nu }^{2}\mathrm{d}
s+\int_{0}^{t}e^{\lambda s}\| g\| _{2}^{2}\mathrm{d} s  \notag \\
& +&\lambda \int_{0}^{t}\| y\| _{2}^{2}\mathrm{d} s+\e^{-1}\int_{0}^{t}%
\{|(1-P_{\gamma })y(s)|_{2,+}^{2}+e^{\lambda s}|r|_{2}^{2}\}\mathrm{d} s,
\end{eqnarray*}
where $G(t)\lesssim \e  \| y(t)\| _{2}^{2}$.

{Using the trace theorem and the boundary condition, as \textit{Step 2}, we
obtain 
\begin{eqnarray*}
\int^{t}_{0} |P_{\gamma} y|_{2,+}^{2} &\lesssim& \| y(0) \|_{2}^{2} +
\int^{t}_{0} \| \frac{e^{\lambda s} g}{\sqrt{\nu}} \|_{2}^{2} + \int^{t}_{0}
\| \mathbf{P}y \|_{2}^{2} + \e^{-2} \int^{t}_{0} \| (\mathbf{I} -\mathbf{P}%
)y \|_{\nu}^{2} + \int^{t}_{0 }|(1-P_{\gamma}) y|_{2,+}^{2}, \\
\int^{t}_{0} | y|_{2,-}^{2} &\lesssim& \int^{t}_{0} | y|_{2,+}^{2} +
\int^{t}_{0} | e^{\lambda \tau} r|_{2,+}^{2}.
\end{eqnarray*}%
}

All together, for $0 < \lambda \ll 1$ and $0 < \e \ll 1$, we conclude (\ref%
{completes_dyn}).
\end{proof}

\subsection{$L^{\infty }$ Estimate}

The main goal of this section is to prove the following:

\begin{proposition}
\label{point_dyn} Let $f$ satisfies 
\begin{equation}  \label{linear_K}
\begin{split}
\big[ \e \partial_{t} + v \cdot \nabla_{x} + \e^{2} \Phi \cdot \nabla_{v} + %
\e^{-1} C_{0}\langle v\rangle \big] |f | \leq \e^{-1} K_{\beta} |f| + |g |,
\\
\big|f |_{\gamma_{-}} \big|\leq P_{\gamma}|f| + |r|,\ \ \big|f|_{t=0}\big| %
\leq| f_{0}|.
\end{split}%
\end{equation}
Then, for $w(v) = e^{\beta^{\prime}|v|^{2}}$ with $0< \beta^{\prime} \ll
\beta$, 
\begin{equation}  \label{point1}
\begin{split}
\| \e^{\frac 1 2} w f(t) \|_{\infty} \lesssim& \ \| \e^{\frac 1 2} w f_{0}
\|_{\infty} + \sup_{0 \leq s \leq \infty}\| \e^{\frac 1 2} w r(s)
\|_{\infty} + \e^{\frac 3 2} \sup_{0 \leq s \leq \infty} \| \langle
v\rangle^{-1} w g(s)\|_{\infty} \\
&+ \sup_{0\leq s \leq t}\| \P f(s)\|_{L^{6}(\Omega )} + {\e}^{-1} \sup_{0
\leq s \leq t} \| (\mathbf{I} - \mathbf{P})f(s)\|_{L^{2}(\Omega \times 
\mathbb{R}^{3})},
\end{split}%
\end{equation}
and 
\begin{equation}  \label{point2}
\begin{split}
\| \e^{\frac 1 2} w f(t) \|_{\infty} \lesssim& \ \| \e^{\frac 1 2} w f_{0}
\|_{\infty} + \sup_{0 \leq s \leq \infty}\| \e^{\frac 1 2} w r(s)
\|_{\infty} + \e^{\frac 3 2} \sup_{0 \leq s \leq \infty} \| \langle
v\rangle^{-1} w g(s)\|_{\infty} \\
& + {\e}^{-1} \sup_{0 \leq s \leq t}\| f(s)\|_{L^{2}(\Omega \times \mathbb{R}%
^{3} )} .
\end{split}%
\end{equation}
\end{proposition}

We define the stochastic cycles for the unsteady case. Note that from (\ref%
{xv_dyn}), $\tilde{x}_{\mathbf{b}} (x,v) = x_{\mathbf{b}}(x,v)$.

\begin{definition}
\label{cycle_dyn} Define, for free variables $v_{k} \in\mathbb{R}^{3}$, from
(\ref{xv_dyn}) 
\begin{eqnarray*}
\tilde{t}_{1} &=& t-\tilde{t}_{\mathbf{b}}(x,v) = t-\e t_{\mathbf{b}}(x,v),
\\
\tilde{x}_{1} &=& Y(\tilde{t}_{1} ; t,x,v) =\tilde{x}_{\mathbf{b}}(x,v) = x_{%
\mathbf{b}}(x,v) = x_{1} , \\
\tilde{t}_{2} &=& t- \tilde{t}_{\mathbf{b}}(x,v) - \tilde{t}_{\mathbf{b}}
(x_{1} ,v_{1} ) = t-\e t_{\mathbf{b}}(x,v) - \e t_{\mathbf{b}}(x_{1} ,v_{1}
), \\
\tilde{x}_{2} & =& Y( \tilde{t}_{2} ; \tilde{t}_{1} ,x_{1} ,v_{1} ) =\tilde{x%
}_{\mathbf{b}} (x_{1},v_{1}) = x_{\mathbf{b}}(x_{1},v_{1}) =x_{2} , \\
&\vdots& \\
\tilde{t}_{k+1} &=& \tilde{t}_{k} - \tilde{t}_{\mathbf{b}}(x_{k} ,v_{k} ) =%
\tilde{t}_{k} - \e t_{\mathbf{b}} (x_{k} ,v_{k} ), \\
\tilde{x}_{k+1} &=& Y( \tilde{t}_{k+1} ; \tilde{t}_{k} ,x_{k} ,v_{k} ) = 
\tilde{x}_{\mathbf{b}} (x_{k},v_{k}) = x_{\mathbf{b}}(x_{k},v_{k})= x_{k+1}.
\end{eqnarray*}
and 
\begin{eqnarray*}
t-\tilde{t} _{1} &=& \e t_{\mathbf{b}}(x,v) = \e (t-t_{1}), \\
t- \tilde{t} _{2} &=&\e t_{\mathbf{b}}(x,v) + \e t_{\mathbf{b}%
}(x_{1},v_{1})= \e (t-t_{2}), \\
&\vdots& \\
t- \tilde{t} _{k } &=& \e (t-t_{k}).
\end{eqnarray*}
Set 
\begin{eqnarray*}
Y_{\mathbf{cl}}(s;t,x,v) &:=& \sum_{k} \mathbf{1}_{[\tilde{t} _{k+1},\tilde{t}
_{k})} (s) Y (s; \tilde{t}_{k} , x_{k} , v_{k} ),\\
W_{\mathbf{cl}%
}(s;t,x,v) &:=& \sum_{k} \mathbf{1}_{[\tilde{t}_{k+1},\tilde{t}_{k})} (s)
W(s; \tilde{t}_{k} , x_{k} , v_{k} ).
\end{eqnarray*}
Clearly 
\begin{equation}  \label{trj_st_dyn}
[Y_{\mathbf{cl}}(s;t,x,v), W _{\mathbf{cl}}(s;t,x,v)]= [X_{\mathbf{cl}}(t- 
\frac{t-s}{\e};t,x,v), V_{\mathbf{cl}}( t- \frac{t-s}{\e};t,x,v)].
\end{equation}
\end{definition}

The following lemma is a generalized version of Lemma 23 of \cite{Guo08}.

\begin{lemma}[\protect\cite{Guo08}]
\label{small_lemma} Assume $\Phi = \Phi(x) \in C^{1}$. For sufficiently
large $T_{0}>0$, there exist constant $C_{1}, C_{2}>0$, independent of $%
T_{0} $, such that for $k= C_{1} T_{0}^{5/4}$, 
\begin{equation}
\sup_{(t, x,v) \in [0,\e T_{0}] \times \bar{\Omega} \times \mathbb{R}%
^{3}}\int_{\prod_{\ell =1}^{k-1} \mathcal{V}_{\ell}} \mathbf{1}_{ \tilde{t}%
_{k} (t,x,v_{1},v_{2}, \cdots, v_{k-1})>0 } \Pi_{\ell=1}^{k-1} \mathrm{d}
\sigma_{\ell} < \Big\{ \frac{1}{2}\Big\}^{C_{2} T_{0}^{5/4}}.  \label{small2}
\end{equation}
\end{lemma}

\begin{proof}
Since $\tilde{t}_{k} (\e T_{0},x,v,v_{1}, v_{2}, \cdots, v_{k-1}) = \tilde{t}%
_{k} (t,x,v,v_{1}, v_{2}, \cdots, v_{k-1}) + \{\e T_{0} - t\}$, for $0\leq t
\leq \e T_{0}$, 
\begin{equation*}
\mathbf{1}_{\tilde{t}_{k} (t,x,v,v_{1}, v_{2}, \cdots, v_{k-1})>0} \leq 
\mathbf{1}_{\tilde{t}_{k} (\e T_{0},x,v,v_{1}, v_{2}, \cdots, v_{k-1})>0} = 
\mathbf{1}_{ \e T_{0} - \tilde{t}_{k} (\e T_{0},x,v,v_{1}, v_{2}, \cdots,
v_{k-1})< \e T_{0}}.
\end{equation*}

Note that, from Definition \ref{cycle_dyn}, for any $T_{1}, T_{2}>0$ 
\begin{eqnarray*}
&&\big\{T_{1} - \tilde{t}_{k} (T_{1}, x,v , v_{1}, v_{2}, \cdots, v_{k-1})%
\big\} = \tilde{t}_{\mathbf{b}}(x,v) + \tilde{t}_{\mathbf{b}}(x_{1},v_{1} )
+ \cdots + \tilde{t}_{\mathbf{b}}(x_{k-1},v_{k-1}) \\
&&= \e {t}_{\mathbf{b}}(x,v) +\e {t}_{\mathbf{b}}(x_{1},v_{1} ) + \cdots +\e
{t}_{\mathbf{b}}(x_{k-1},v_{k-1}) = \e \big\{ T_{2} - t_{k} (T_{2}
,x,v,v_{1},v_{2} , \cdots, v_{k-1} ) \big\}.
\end{eqnarray*}
Therefore, with $T_{1} =\e T_{0}$ and $T_{2} = T_{0}$, 
\begin{eqnarray*}
\e T_{0} > \e T_{0} - \tilde{t}_{k} (\e T_{0},x,v,v_{1},v_{2}, \cdots,
v_{k-1}) = \e \{ T_{0} - {t}_{k} ( T_{0},x,v,v_{1},v_{2}, \cdots, v_{k-1})
\},
\end{eqnarray*}
and 
\begin{eqnarray*}
\mathbf{1}_{ \e T_{0} - \tilde{t}_{k} (\e T_{0},x,v,v_{1}, v_{2}, \cdots,
v_{k-1})< \e T_{0}} = \mathbf{1}_{ T_{0} - {t}_{k} ( T_{0},x,v,v_{1}, v_{2},
\cdots, v_{k-1})< T_{0} } = \mathbf{1}_{ {t}_{k} ( T_{0},x,v,v_{1}, v_{2},
\cdots, v_{k-1})>0 }.
\end{eqnarray*}
Hence, 
\begin{eqnarray*}
&& \sup_{(t,x,v) \in [0, \e T_{0}] \times \bar{\Omega} \times \mathbb{R}%
^{3}} \int_{\prod_{j=1}^{k-1} \mathcal{V}_{j}} \mathbf{1}_{\tilde{t}_{k}
(t,x,v,v_{1},v_{2}, \cdots, v_{k-1}) >0} \Pi_{j=1}^{k-1}\mathrm{d} \sigma_{j}
\\
& \leq & \sup_{( x,v) \in \bar{\Omega} \times \mathbb{R}^{3}}
\int_{\prod_{j=1}^{k-1} \mathcal{V}_{j}} \mathbf{1}_{ {t}_{k}
(T_{0},x,v,v_{1},v_{2}, \cdots, v_{k-1}) >0} \Pi_{j=1}^{k-1}\mathrm{d}
\sigma_{j}.
\end{eqnarray*}
This proves (\ref{small2}).
\end{proof}

Now we are ready to prove the main result of this section:

\begin{proof}[\textbf{Proof of Proposition \protect\ref{point_dyn}}]
Define, for $w(v)$ and $h$ as (\ref{h_dyn}). 

Then, from (\ref{linear_K}), 
\begin{equation}
\begin{split}  \label{h_eq}
& \big[ \partial_{t} + \e^{-1}v \cdot \nabla_{x} + \e \Phi \cdot \nabla_{v}
+ \e^{-2} C_{0}\langle v\rangle + \frac{\e \Phi \cdot \nabla_{v}w%
}{w} \big] | \e^{\frac 1 2} h| \\
& \leq \e^{-2} \int_{\mathbb{R}^{3}}\mathbf{k}_{\tilde{\beta} }(v,u) | \e%
^{\frac 1 2} h(u) | \mathrm{d} u + \e^{-\frac 1 2}|wg |.
\end{split}%
\end{equation}
We have the boundary condition as (\ref{h_BC}).

\vspace{4pt}

\noindent\textit{Step 1.} We claim, for $t \in [n \e T_{0}, (n+1) \e T_{0}]$
with all $n\in\mathbb{N}$ and $T_{0}$ in Lemma \ref{small_lemma} 
\begin{equation}  \label{claim1}
\begin{split}
&|\e^{\frac 1 2} h (t,x,v)| \\
&\leq \ CT_{0}^{5/2} e^{- \frac{C_{0} (t- n \e T_{0})}{ \e^{2}}} \| \e%
^{\frac 1 2} h (n \e T_{0})\|_{\infty} +
C_{T_{0}}
 \e^{\frac 1 2} \sup_{0 \leq s \leq
t} \| w r(s) \|_{\infty} +C_{T_{0}} T_{0}^{5/2} \e^{{\frac 3 2}} \| \langle
v\rangle^{-1} w g(s) \|_{\infty} \\
& \ \ +C T_{0}^{5/2} \sup_{n \e T_{0} \leq s \leq (n+1) \e T_{0}}\| \mathbf{P%
} f(s)\|_{L^{6}(\Omega )} \\
& \ \ +C T_{0}^{5/2} {\e^{-1}} \sup_{n \e T_{0} \leq s \leq (n+1) \e T_{0}}
\| (\mathbf{I} - \mathbf{P})f(s)\|_{L^{2}(\Omega \times \mathbb{R}^{3})} \\
& \ \ + \big[CT_{0}^{5/4}\Big\{\frac{1}{2} \Big\}^{C_{2} T_{0}^{5/4}} +
o(1)CT_{0}^{5/2} \big] \sup_{n \e T_{0} \leq s \leq (n+1) \e T_{0}} \| \e%
^{\frac 1 2} h (s)\|_{\infty},
\end{split}%
\end{equation}
and 
\begin{equation}  \label{claim1_2}
\begin{split}
&|\e^{\frac 1 2} h (t,x,v)| \\
&\leq \ CT_{0}^{5/2} e^{- \frac{C_{0} (t- n \e T_{0})}{ \e^{2}}} \| \e%
^{\frac 1 2} h (n \e T_{0})\|_{\infty} + \e^{\frac 1 2} \sup_{0 \leq s \leq
t} \| w r(s) \|_{\infty} +CT_{0}^{5/2} \e^{\frac 3 2} \| \langle
v\rangle^{-1} w g(s) \|_{\infty} \\
& \ \ +C T_{0}^{5/2} \frac{1}{\e} \sup_{n \e T_{0} \leq s \leq (n+1) \e %
T_{0}} \| f(s)\|_{L^{2}(\Omega \times \mathbb{R}^{3})} \\
& \ \ + \big[CT_{0}^{5/4}\Big\{\frac{1}{2} \Big\}^{C_{2} T_{0}^{5/4}} +
o(1)CT_{0}^{5/2} \big] \sup_{n \e T_{0} \leq s \leq (n+1) \e T_{0}} \| \e%
^{\frac 1 2} h (s)\|_{\infty}.
\end{split}%
\end{equation}

We first prove (\ref{claim1}). From (\ref{h_eq}), for $\tilde{t}_{1} (t,x,v)
< s \leq t$, 
\begin{eqnarray*}
&&\frac{d}{ds } \Big[ e^{- \int^{t}_{s} \frac{ C_{0} }{\e ^{2}} \langle V(
t- \frac{t-\tau}{\e};t,x,v) \rangle \mathrm{d}\tau } \e^{\frac 1 2}
h^{\ell+1}(s,X _{\mathbf{cl}}(t- \frac{t-s}{\e};t,x,v), V _{\mathbf{cl}}( t- 
\frac{t-s}{\e};t,x,v)) \Big] \\
& \leq& e^{- \int^{t}_{s} \frac{ C_{0} }{\e ^{2}} \langle V( t- \frac{t-\tau%
}{\e};t,x,v) \rangle \mathrm{d}\tau } \frac{1}{\e^{2}} \int_{\mathbb{R}^{3}} 
\mathbf{k}_{\tilde{\beta}} (V _{\mathbf{cl}}( t- \frac{t-s}{\e};t,x,v)
,v^{\prime}) | \e^{\frac 1 2} h( s,X _{\mathbf{cl}}(t- \frac{t-s}{\e}%
;t,x,v), v^{\prime} )| \mathrm{d} v^{\prime} \\
&&+ e^{- \int^{t}_{s} \frac{ C_{0} }{\e ^{2}} \langle V( t- \frac{t-\tau}{\e}%
;t,x,v) \rangle \mathrm{d}\tau } \e^{-\frac 1 2}|w g( s,X _{\mathbf{cl}}(t- 
\frac{t-s}{\e};t,x,v), V _{\mathbf{cl}}( t- \frac{t-s}{\e};t,x,v) )|.
\end{eqnarray*}
Along the stochastic cycles, for $k= C_{1} T_{0}^{5/4}$, we deduce the
following estimate: 
\begin{eqnarray}
&& | \e^{\frac 1 2} h^{\ell+1}(t,x,v) |  \notag \\
&\leq& \mathbf{1}_{\{ \tilde{t}_{1} (t,x,v)< 0\}} e^{ - \int^{t}_{0} \frac{
C_{0} \langle V_{\mathbf{cl}} ( t- \frac{t-\tau}{\e}; t,x,v) \rangle }{\e^{2}%
} \mathrm{d}\tau} | \e^{\frac 1 2} h^{\ell+1}(0, X_{\mathbf{cl}}(t- \frac{t}{%
\e}; t,x,v ), V_{\mathbf{cl}}(t- \frac{t}{\e}; t,x,v )) | \ \ \ \  \ \ 
\label{h1in} \\
&+& \int_{\max{\{0, \tilde{t}_{1} (t,x,v) \}}}^{t} \mathrm{d} s \ \frac{e^{-
\int^{t}_{s} \frac{ C_{0} \langle V_{\mathbf{cl}} ( t- \frac{t-\tau}{\e}%
;t,x,v ) \rangle }{\e^{2}} \mathrm{d} \tau } }{\e^{2}}  \notag \\
&& \ \ \ \times \int_{\mathbb{R}^{3}} \mathrm{d} v^{\prime} \ \mathbf{k}_{%
\tilde{\beta} }( V_{\mathbf{cl}} ( t- \frac{t-s}{\e} ; t,x,v) ,v^{\prime}) %
\big| \e^{\frac 1 2} h^{\ell}(s, X_{\mathbf{cl}}(t- \frac{t-s}{\e}%
;t,x,v),v^{\prime} )\big| \ \ \ \   \label{hk} \\
&+& \int_{\max{\{0, \tilde{t}_{1} (t,x,v) \}}}^{t} \mathrm{d} s \ \frac{e^{-
\int^{t}_{s} \frac{ C_{0} \langle V_{\mathbf{cl}} ( t- \frac{t-\tau}{\e}%
;t,x,v ) \rangle }{\e^{2}} \mathrm{d} \tau } }{\e^{2}}  \notag \\
&& \ \ \ \ \ \ \ \ \ \ \ \ \times {\e^2}\e^{-\frac 1 2}|w g ( s,X _{\mathbf{%
cl}}(t- \frac{t-s}{\e};t,x,v), V _{\mathbf{cl}}( t- \frac{t-s}{\e};t,x,v) )|
\ \ \ \ \ \ \ \ \ \ \ \   \label{hnl} \\
&+&\mathbf{1}_{\{ \tilde{t}_{1}(t,x,v) \geq 0 \}} e^{ - \int^{t}_{\tilde{t}%
_{1} (t,x,v)} \frac{ C_{0} \langle V_{\mathbf{cl}} ( t- \frac{t-\tau}{\e};
t,x,v) \rangle }{\e^{2}} \mathrm{d}\tau} \e^{\frac 1 2} | w {r} ( \tilde{t}%
_{1}(t,x,v), x_{1} (x,v), v_{1} (x,v)) |  \label{h1ber} \\
& +& \mathbf{1}_{\{ \tilde{t}_{1} (t,x,v) \geq 0 \}} \frac{ e^{ - \int^{t}_{%
\tilde{t}_{1} (t,x,v)} \frac{ C_{0} \langle V_{\mathbf{cl}} ( t- \frac{t-\tau%
}{\e}; t,x,v) \rangle }{\e^{2}} \mathrm{d}\tau} }{\tilde{w} (v_{1})}
\int_{\Pi_{j=1}^{k-1} \mathcal{V}_{j}} H,  \notag
\end{eqnarray}
where $H$ is given by 
\begin{eqnarray}
&& \sum_{l =1}^{k-1} \mathbf{1}_{ \tilde{t}_{l+1} \leq 0< \tilde{t}_{l}} %
\big| \e^{\frac 1 2} h (0, X_{\mathbf{cl}}(\tilde{t}_{l} - \frac{\tilde{t}%
_{l}}{\e} ; \tilde{t}_{l}, x_{l}, v_{l}) , V_{\mathbf{cl}} (\tilde{t}_{l} - 
\frac{\tilde{t}_{l}}{\e} ; \tilde{t}_{l}, x_{l}, v_{l}) ) \big|\notag\\
&&
  \ \ \ \ \ \ \ \ \ \ \ \ \ \ \  \ \ \ \ \ \ \ \ \ \ \ \ \ \ \ \ \ \ \ \ \ \ \ \ \ \ \ \   \times \Pi_{m=1}^{l-1} \frac{\tilde{w} (v_{m})}{\tilde{w} (V_{\mathbf{cl}} (\tilde{t}_{m+1}
  - \frac{\tilde{t}_{m+1}}{\e}
  ;  v_{m}))} \mathrm{d}
\Sigma_{l} (0)
  \label{h2in} \\
&+& \sum_{l =1}^{k-1} \int^{ \tilde{t}_{l}}_{\max \{ 0, \tilde{t}_{l+1} \}} \mathrm{d} \tau 
\mathbf{1}_{ \tilde{t}_{l}>0} \frac{1}{\e^{2}} \int_{\mathbb{R}^{3}} \mathbf{%
k}_{\tilde{\beta} } ( V_{\mathbf{cl}} ( \tilde{t}_{l}- \frac{\tilde{t}%
_{l}-\tau}{\e}; \tilde{t}_{l}, x_{l}, v_{l}),u )  \notag \\
&& \ \   \times \big| \e^{\frac 1 2} h
(\tau, X_{\mathbf{cl}}( \tilde{t}_{l}- \frac{\tilde{t}_{l}-\tau}{\e}; \tilde{%
t}_{l}, x_{l}, v_{l}) , u) \big|  
\Pi_{m=1}^{l-1} \frac{\tilde{w} (v_{m})}{\tilde{w} (V_{\mathbf{cl}} (\tilde{t}_{m+1}
  - \frac{\tilde{t}_{m+1}}{\e}
  ;  v_{m}))} 
\mathrm{d} u \mathrm{d} \Sigma_{l} (\tau) 
 \label{h2k} \\
&+& \sum_{l =1}^{k-1} \int^{ \tilde{t}_{l}}_{\max \{ 0, \tilde{t}_{l+1} \}} 
\mathbf{1}_{ \tilde{t}_{l}>0} \Pi_{m=1}^{l-1} \frac{\tilde{w} (v_{m})}{\tilde{w} (V_{\mathbf{cl}} (\tilde{t}_{m+1}
  - \frac{\tilde{t}_{m+1}}{\e}
  ;  v_{m}))}  \notag \\
&& \ \ \ \ \ \times \e^{-\frac 1 2}\big| wg (\tau, X_{\mathbf{cl}}( \tilde{t}%
_{l}- \frac{\tilde{t}_{l}-\tau}{\e}; \tilde{t}_{l}, x_{l}, v_{l}) , V_{%
\mathbf{cl}} ( \tilde{t}_{l}- \frac{\tilde{t}_{l}-\tau}{\e}; \tilde{t}_{l},
x_{l}, v_{l})) \big| \mathrm{d} \Sigma_{l} (\tau) \mathrm{d} \tau \ \ \ \ \
\ \ \ \ \ \   \label{h2g} \\
&+& \sum_{l=1}^{k-1} \mathbf{1}_{\tilde{t}_{l}>0} \e^{\frac 1 2} w(v_{l}) | {%
r} (\tilde{t}_{l}, x_{l+1} , v_{l}) |
\Pi_{m=1}^{l-1} \frac{\tilde{w} (v_{m})}{\tilde{w} (V_{\mathbf{cl}} (\tilde{t}_{m+1}
  - \frac{\tilde{t}_{m+1}}{\e}
  ;  v_{m}))} 
\mathrm{d} \Sigma_{l} (\tilde{t}_{l+1})
\label{h2ber} \\
&+& \mathbf{1}_{ \tilde{t}_{k} >0}| \e^{\frac 1 2} h ( \tilde{t}_{k},x_{k},
v_{k-1})| 
\Pi_{m=1}^{k-2} \frac{\tilde{w} (v_{m})}{\tilde{w} (V_{\mathbf{cl}} (\tilde{t}_{m+1}
  - \frac{\tilde{t}_{m+1}}{\e}
  ;  v_{m}))} 
\mathrm{d} \Sigma_{k-1} ( \tilde{t}_{k})
,  \label{h2er}
\end{eqnarray}
where $\tilde{w} (V_{\mathbf{cl}} (\tilde{t}_{m+1}
  - \frac{\tilde{t}_{m+1}}{\e}
  ;  v_{m}))=\tilde{w} (V_{\mathbf{cl}} (\tilde{t}_{m+1}
  - \frac{\tilde{t}_{m+1}}{\e}
  ; \tilde{t}^{m}, x_{m}, v_{m}))$ and $\mathrm{d} \Sigma_{k-1} ( \tilde{t}_{k})$ is evaluated at $s=\tilde{t}%
_{k}$ of 
\begin{equation}  \label{dsigma_unsteady}
\begin{split}
\mathrm{d}\Sigma _{l} (s) \ := & \ \{\prod_{j=l+1}^{k-1}\mathrm{d} \sigma
_{j}\}\{ e^{ -\int^{\tilde{t}_{l} }_{s} \frac{C_{0} \langle V_{\mathbf{cl}}( 
\tilde{t}_{l} - \frac{ \tilde{t}_{l} -\tau}{\e} ; \tilde{t}_{l} , x_{l},
v_{l}) \rangle }{\e^{2}} \mathrm{d} \tau }\tilde{w}(v_{l})\mathrm{d}\sigma
_{l}\} \\
& \times \prod _{j=1}^{l-1}\{e^{ - \int^{ \tilde{t}_{j} }_{ \tilde{t}_{j+1}
} \frac{ C_{0} \langle V_{\mathbf{cl}} ( \tilde{t}_{j} - \frac{ \tilde{t}%
_{j} -\tau}{\e} ; \tilde{t}_{j} , x_{j}, v_{j})\rangle }{\e^{2}} \mathrm{d}
\tau }\mathrm{d}\sigma _{j}\}.
\end{split}%
\end{equation}

We note that 
\begin{equation*}
|V_{\mathbf{cl}}(\tilde{t}_{m}-\frac{\tilde{t}_{m+1}-t_{m}}{\varepsilon };\tilde{t}_{m},x_{m},v_{m})-v_{m}|%
\leq \frac{\tilde{t}_{m+1}-\tilde{t}_{m}}{\varepsilon }\varepsilon ^{2}
\|\Phi \|_{\infty
}\leq \varepsilon T_{0},
\end{equation*}%
Hence for $\varepsilon T_{0}<1,$ we have 
\begin{equation}\label{w/w_unsteady}
\frac{\tilde{w} (v_{m})}{\tilde{w} (V_{\mathbf{cl}} (\tilde{t}_{m+1}
  - \frac{\tilde{t}_{m+1}}{\e}
  ;  \tilde{t}_{m}, x_{m},v_{m}))} \lesssim 1.
\end{equation}

From our choice $k= C_{1} T_{0}^{5/4}$, 
\begin{equation*}
(\ref{h1in}) + (\ref{h2in}) \lesssim C_{1} T_{0}^{5/4} e^{- \frac{ C_{0} }{ %
\e^{2}} t } \| {\e^{\frac 1 2}}h_{0}\|_{\infty},\ \ (\ref{h1ber}) + (\ref%
{h2ber}) \lesssim C_{1} T_{0}^{5/4} \e \sup_{ 0 \leq s \leq t} \| \e^{\frac
1 2} w{r}(s) \|_{\infty},
\end{equation*}
and 
\begin{eqnarray*}
&&(\ref{hnl})+ (\ref{h2g}) \\
&\lesssim& {\e^2}\e^{-\frac 1 2}\sup_{ 0 \leq s\leq t} \big\| \langle v
\rangle w g (s) \big\|_{\infty} \times \Big\{ \int^{t}_{0} \frac{ \langle V_{%
\mathbf{cl}} (t- \frac{t-s}{\e};t,x,v) \rangle }{\e^{2}} e^{ - \int^{t}_{s} 
\frac{C_{0} \langle V_{\mathbf{cl}} (t- \frac{t-\tau}{\e};t,x,v) \rangle }{\e%
^{2}} \mathrm{d} \tau } \mathrm{d} s \\
&& + C_{1} T_{0}^{5/4} \sup_{l} \int^{\tilde{t}_{l}}_{0} \frac{ \langle V_{%
\mathbf{cl}} ( \tilde{t}_{l}- \frac{\tilde{t}_{l}-\tau}{\e};\tilde{t}%
_{l},x_{l},v_{l}) \rangle }{\e^{2}} e^{ - \int^{\tilde{t}_{l}}_{s} \frac{%
C_{0} \langle V_{\mathbf{cl}} ( \tilde{t}_{l}- \frac{\tilde{t}_{l}-\tau}{\e};%
\tilde{t}_{l},x_{l},v_{l}) \rangle }{\e^{2}} \mathrm{d} \tau } \mathrm{d}
\tau \Big\} \\
&\lesssim & C_{1} T_{0}^{5/4} \e^{\frac 3 2} \sup_{ 0 \leq s\leq t} \big\| %
\langle v \rangle w g (s) \big\|_{\infty} \times \int^{t}_{0} \frac{d}{ds}
e^{ - \int^{t}_{s} \frac{C_{0} \langle V_{\mathbf{cl}} (t- \frac{t-\tau}{\e}%
;t,x,v) \rangle }{\e^{2}} \mathrm{d} \tau } \mathrm{d} s \\
&\lesssim & C_{1} T_{0}^{5/4} \e^{\frac 3 2} \sup_{ 0 \leq s\leq t} \big\| %
\langle v \rangle w g (s) \big\|_{\infty},
\end{eqnarray*}
where we have used the fact that $\mathrm{d} \sigma_{j}$ is a probability
measure of $\mathcal{V}_{j}$.

Now we focus on (\ref{hk}) and (\ref{h2k}). For $N>1$, we can choose $%
m=m(N)\gg 1$ and define $\mathbf{k}_{m} (v,u)$ as (\ref{k_m}). We split $%
\mathbf{k}_{\tilde{\beta} } (v,u) = [ \mathbf{k}_{\tilde{\beta} } (v,u)- 
\mathbf{k}_{m} (v,u) ] + \mathbf{k}_{m} (v,u)$, and the first difference
would lead to a small contribution in (\ref{hk}) and (\ref{h2k}) as, for $N
\gg_{T_{0}} 1$, 
\begin{eqnarray*}
&& \frac{k}{N} \sup_{0 \leq s \leq t}\| \e^{\frac 1 2} h (s) \|_{\infty} = 
\frac{ C_{1} T_{0}^{5/4}}{N} \sup_{0 \leq s \leq t}\| \e^{\frac 1 2} h (s)
\|_{\infty}.
\end{eqnarray*}

We further split the time integrations in (\ref{hk}) and (\ref{h2k}) as $[%
\tilde{t}_{l} - \kappa \e^{2}, \tilde{t}_{l}]$ and $[\max\{ 0, \tilde{t}%
_{l+1} \}, \tilde{t}_{l}-\kappa \e^{2}]$: 
\begin{eqnarray*}
(\ref{hk}) = {\underbrace{ \int^{t}_{ t- \kappa \e^{2} } }}+ \int^{ t-
\kappa \e^{2} }_{\max\{ 0, \tilde{t}_{1} \}}, \ \ \ (\ref{h2k}) = \mathbf{1}%
_{\{ \tilde{t}_{1} \geq 0 \}} \int_{\Pi_{j=1}^{k-1} \mathcal{V}_{j}}
\sum_{l=1}^{k-1} \Big\{ \ \ {\underbrace{\int^{\tilde{t}_{l}}_{ \tilde{t}%
_{l} - \kappa \e^{2} } }}+ \int^{ \tilde{t}_{l} - \kappa \e^{2} }_{\max\{ 0, 
\tilde{t}_{l+1} \}} \Big\}.
\end{eqnarray*}
The first small-in-time contributions of both (\ref{hk}) and (\ref{h2k}),
underbraced terms, are bounded by 
\begin{eqnarray*}
&&\kappa \e^{2} \frac{1}{\e^{2}} \sup_{v} \int_{|v^{\prime}| \leq N} \mathbf{%
k}_{m} (v, v^{\prime}) \mathrm{d} v^{\prime} \sup_{0 \leq s \leq t} \| \e%
^{\frac 1 2} h (s)\|_{\infty} \lesssim \kappa \sup_{0 \leq s \leq t} \| \e%
^{\frac 1 2} h (s)\|_{\infty}, \\
&&C_{1} T_{0}^{5/4} \kappa \e^{2} \frac{1}{\e^{2}} \sup_{v}
\int_{|v^{\prime}| \leq N} \mathbf{k}_{m} (v,v^{\prime} ) \mathrm{d}
v^{\prime} \sup_{0 \leq s \leq t} \| \e^{\frac 1 2} h (s)\|_{\infty}
\lesssim \kappa C_{1} T_{0}^{5/4} \sup_{0 \leq s \leq t} \| \e^{\frac 1 2} h
(s)\|_{\infty},
\end{eqnarray*}
which would be small contribution if $\kappa \ll_{T_{0}} 1$.

For (\ref{h2er}), by Lemma \ref{small_lemma}, 
\begin{equation*}
(\ref{h2er}) \ \lesssim \ \sup_{0 \leq s \leq t}\| \e^{\frac 1 2} h (s)
\|_{\infty} \sup_{(t,x,v) \in [0,\e T_{0}] \times \bar{\Omega} \times 
\mathbb{R}^{3}} \int_{\prod_{j=1}^{k-1} \mathcal{V}_{j}} \mathbf{1}_{\tilde{t%
}_{k} (t,x,v,v_{1}, v_{2}, \cdots, v_{k-1})>0} \Pi_{j=1}^{k-1} \mathrm{d}
\sigma_{j}.
\end{equation*}
Then, by Lemma \ref{small_lemma}, 
\begin{eqnarray*}
&&(\ref{h2er}) \\
& \lesssim& 
\{1+ O(\e)\}^{C_{1}T_{0}^{5/4}}
\sup_{0 \leq s \leq t}\| \e^{\frac 1 2} h (s)
\|_{\infty} \sup_{ (x,v) \in \bar{\Omega} \times \mathbb{R}^{3} }
\int_{\prod_{j=1}^{k-1} \mathcal{V}_{j}} \mathbf{1}_{ t_{k} (
T_{0},x,v,v_{1}, v_{2}, \cdots, v_{k-1}) >0 } \Pi_{j=1}^{k-1} \mathrm{d}
\sigma_{j} \\
& \lesssim& \Big\{\frac{4}{5} \Big\}^{C_{2} T_{0}^{5/4}} \sup_{0 \leq s \leq
t}\| \e^{\frac 1 2} h (s) \|_{\infty} .
\end{eqnarray*}

Overall, for $(t,x,v) \in [0, \e T_{0}] \times \bar{\Omega} \times \mathbb{R}%
^{3}$, 
\begin{equation}  \label{est1_h}
\begin{split}
|\e^{\frac 1 2} h (t,x,v)| & \lesssim \int_{\max{\{0, \tilde{t}_{1} (x,v) \}}%
}^{t- \kappa \e^{2}} \mathrm{d} s \ \frac{ e^{- \frac{C_{0}}{\e^{2}} (t-s)} 
}{\e^{2}} \int_{|v^{\prime}| \leq m} \mathrm{d} v^{\prime} \underbrace{ %
\big| \e^{\frac 1 2} h (s, X_{\mathbf{cl}}(t- \frac{t-s}{\e}%
;t,x,v),v^{\prime} )\big|} \\
&+\mathbf{1}_{\{ \tilde{t}_{1} \geq 0 \}} \frac{ e^{- \frac{C_{0}}{\e^{2}}
(t-\tilde{t}_{1})} }{\tilde{w} (v)} \int_{\Pi_{j=1}^{k-1} \mathcal{V}_{j}}
\sum_{\ell =1}^{k-1} \int^{ \tilde{t}_{\ell} - \kappa \e^{2}}_{\max \{ 0, 
\tilde{t}_{\ell+1} \}} \mathbf{1}_{ \tilde{t}_{\ell}>0} \frac{1}{\e^{2}} \\
& \ \ \ \times \int_{ |v^{\prime\prime}| \leq m} \underbrace{\big| \e^{\frac
1 2} h (\tau, X_{\mathbf{cl}}( \tilde{t}_{\ell}- \frac{\tilde{t}_{\ell}-\tau%
}{\e}; \tilde{t}_{\ell}, x_{\ell}, v_{\ell}) , v^{\prime\prime} ) \big| } 
\mathrm{d} v^{\prime\prime} \mathrm{d} \Sigma_{\ell} (\tau) \mathrm{d} \tau
\\
& + CT_{0}^{5/4} \Big\{ e^{- \frac{C_{0}}{\e^{2}}t} \e^{\frac 1 2}\| h_{0}
\|_{\infty} + \e^{\frac 1 2} \sup_{0 \leq s \leq t} \| w r(s) \|_{\infty} + %
\e^{-\frac 1 2}\sup_{0 \leq s \leq t} \| \langle v\rangle w g(s) \|_{\infty} %
\Big\} \\
& +o(1)CT_{0}^{5/4} \sup_{0 \leq s \leq t} \| \e^{\frac 1 2} h
(s)\|_{\infty} + \Big\{\frac{4}{5} \Big\}^{C_{2} T_{0}^{5/4}} \sup_{0 \leq s
\leq t} \| \e^{\frac 1 2} h (s)\|_{\infty}.
\end{split}%
\end{equation}

Note that the similar estimate holds for the underbraced terms in (\ref%
{est1_h}). We plug these estimates into the underbraced terms of (\ref%
{est1_h}) to conclude 
\begin{equation*}
| \e^{\frac 1 2} h^{\ell+1}(t,x,v)| \leq \mathbf{I}_{1} + \mathbf{I}_{2} + 
\mathbf{I}_{3}.
\end{equation*}
Here, using $w(u)\lesssim_{m} 1$ for $|u| \leq m$, 
\begin{equation*}
\begin{split}
\mathbf{I}_{1} &\lesssim_{m} \int_{\max{\{0, \tilde{t}_{1} \}}}^{t-\kappa \e%
^{2}} \mathrm{d} s \ \frac{ e^{- \frac{C_{0}}{\e^{2}} (t-s)} }{\e^{2}}
\int_{|v^{\prime}| \leq m} \mathrm{d} v^{\prime}\int^{s-\kappa \e^{2}}_{
\max \{ 0, \tilde{t}_{1}^{\prime} \}} \mathrm{d} s^{\prime} \frac{e^{- \frac{%
C_{0} (s-s^{\prime})}{\e^{2}}}}{\e^{2}} \int_{|u| \leq m} \mathrm{d} u \  \\
& \ \ \ \ \ \ \ \ \ \ \ \ \ \ \ \ \ \ \times \big| \e^{\frac 1 2} h(
s^{\prime}, X_{\mathbf{cl}} (s- \frac{s-s^{\prime}}{\e} ; s, X_{\mathbf{cl}%
}(t- \frac{t-s}{\e}; t,x,v), v^{\prime} ), u ) \big| \\
&+ \int_{\max{\{0, \tilde{t}_{1} \}}}^{t-\kappa \e^{2}} \mathrm{d} s \ \frac{
e^{- \frac{C_{0}}{\e^{2}} (t-s)} }{\e^{2}} \int_{|v^{\prime}| \leq m} 
\mathrm{d} v^{\prime} \ \mathbf{1}_{\{ \tilde{t}_{1}^{\prime} \geq 0 \}} 
\frac{ e^{- \frac{C_{0}}{\e^{2}} (s-\tilde{t}_{1}^{\prime})} }{\tilde{w} (v)}
\\
& \ \ \ \times \int_{\Pi_{j=1}^{k-1} \mathcal{V}_{j}^{\prime}}
\sum_{\ell^{\prime} =1}^{k-1} \int^{ \tilde{t}_{\ell^{\prime}}^{\prime}-%
\kappa \e^{2}}_{\max \{ 0, \tilde{t}_{\ell^{\prime}+1}^{\prime} \}} \mathbf{1%
}_{ \tilde{t}_{\ell^{\prime}}^{\prime}>0} \frac{1}{\e^{2}} {\big| \e^{\frac
1 2} h (\tau, X_{\mathbf{cl}}( \tilde{t}_{\ell^{\prime}}^{\prime}- \frac{%
\tilde{t}_{\ell^{\prime}}^{\prime}-\tau}{\e}; \tilde{t}_{\ell^{\prime}}^{%
\prime}, x_{\ell^{\prime}}^{\prime}, v_{\ell^{\prime}}^{\prime}) , u) \big| }
\mathrm{d} u \mathrm{d} \Sigma_{\ell^{\prime}} (\tau) \mathrm{d} \tau,
\end{split}%
\end{equation*}
where 
\begin{eqnarray*}
\tilde{t}_{\ell^{\prime}}^{\prime} &:=& \tilde{t}_{\ell^{\prime}} ( s, X_{%
\mathbf{cl}}(t- \frac{t-s}{\e}; t,x,v), v^{\prime} ) , \\
{x}_{\ell^{\prime}}^{\prime} &:=& {x}_{\ell^{\prime}} ( X_{\mathbf{cl}}(t- 
\frac{t-s}{\e}; t,x,v), v^{\prime} ) , \ \ {v}_{\ell^{\prime}}^{\prime} \ :=
\ {v}_{\ell^{\prime}} ( X_{\mathbf{cl}}(t- \frac{t-s}{\e}; t,x,v),
v^{\prime} ).
\end{eqnarray*}

Moreover 
\begin{equation}
\begin{split}
\mathbf{I}_{2}&\lesssim_{m}\mathbf{1}_{\{ \tilde{t}_{1} \geq 0 \}} \frac{
e^{- \frac{C_{0}}{\e^{2}} (t-\tilde{t}_{1})} }{\tilde{w} (v)}
\int_{\Pi_{j=1}^{k-1} \mathcal{V}_{j}} \sum_{\ell =1}^{k-1} \int^{ \tilde{t}%
_{\ell} -\kappa \e^{2}}_{\max \{ 0, \tilde{t}_{\ell+1} \}} \mathrm{d}
\Sigma_{\ell} (\tau) \mathrm{d} \tau \ \mathbf{1}_{ \tilde{t}_{\ell}>0} 
\frac{1}{\e^{2}} \int_{ |v^{\prime\prime}| \leq m} \mathrm{d}
v^{\prime\prime} \\
& \ \ \ \times \int_{\max{\{0, \tilde{t}_{1}^{\prime\prime} \}}%
}^{\tau-\kappa \e^{2}} \mathrm{d} s^{\prime\prime} \ \frac{ e^{- \frac{C_{0}%
}{\e^{2}} (\tau-s^{\prime\prime})} }{\e^{2}} \int_{|u| \leq m} \mathrm{d} u
\ {\ \big| \e^{\frac 1 2} h \big(s^{\prime\prime}, X_{\mathbf{cl}} (\tau - 
\frac{\tau -s^{\prime\prime}}{\e}; \tau, X_{\mathbf{cl}} (\tilde{t}_{\ell} - 
\frac{\tilde{t}_{\ell} - \tau}{\e} ; \tilde{t}_{\ell}, x_{\ell}, v_{\ell}
),v^{\prime\prime}),u \big)\big|} \\
&+ \mathbf{1}_{\{ \tilde{t}_{1} \geq 0 \}} \frac{ e^{- \frac{C_{0}}{\e^{2}}
(t-\tilde{t}_{1})} }{\tilde{w} (v)} \int_{\Pi_{j=1}^{k-1} \mathcal{V}_{j}}
\sum_{\ell =1}^{k-1} \int^{ \tilde{t}_{\ell}-\kappa \e^{2}}_{\max \{ 0, 
\tilde{t}_{\ell+1} \}} \mathrm{d} \Sigma_{\ell} (\tau) \mathrm{d} \tau \ 
\mathbf{1}_{ \tilde{t}_{\ell}>0} \frac{1}{\e^{2}} \int_{ |v^{\prime\prime}|
\leq m} \mathrm{d} v^{\prime\prime} \\
& \ \ \ \times \mathbf{1}_{\tilde{t}_{1}^{\prime\prime} \geq 0 } \frac{e^{- 
\frac{C_{0}}{\e^{2}}(\tau - \tilde{t}_{1} ^{\prime\prime}) } }{\tilde{w}
(v^{\prime\prime})} \int_{\prod_{j=1}^{k-1} \mathcal{V}_{j}^{\prime\prime}}
\sum_{\ell^{\prime\prime} =1}^{k-1} \int^{\tilde{t}_{\ell^{\prime\prime}}^{%
\prime\prime}-\kappa \e^{2}}_{\max\{ 0, \tilde{t}^{\prime\prime}
_{\ell^{\prime\prime}+1} \}} \mathbf{1}_{\tilde{t}_{\ell^{\prime\prime}}^{%
\prime\prime} >0} \frac{1}{\e^{2}} \\
& \ \ \ \times \int_{|u| \leq m} \big| \e^{\frac 1 2} h \big( %
\tau^{\prime\prime}, X_{\mathbf{cl}}( \tilde{t}_{\ell^{\prime\prime}}^{%
\prime\prime}- \frac{ \tilde{t}_{\ell^{\prime\prime}}^{\prime\prime} -
\tau^{\prime\prime} }{\e}; \tilde{t}^{\prime\prime}_{\ell^{\prime\prime}},
x_{\ell^{\prime\prime}}^{\prime\prime},
v_{\ell^{\prime\prime}}^{\prime\prime} ), u \big) \big| \mathrm{d} u \mathrm{%
d} \Sigma_{\ell^{\prime\prime}}^{\prime\prime} (\tau^{\prime\prime}) \mathrm{%
d} \tau^{\prime\prime},
\end{split}
\notag
\end{equation}
where 
\begin{eqnarray*}
\tilde{t}^{\prime\prime}_{\ell^{\prime\prime}} &: =& \tilde{t}%
_{\ell^{\prime\prime}} ( \tau, X_{\mathbf{cl}}(\tilde{t}_{\ell} - \frac{%
\tilde{t}_{\ell} - \tau}{\e} ; \tilde{t}_{\ell}, x_{\ell} , v_{\ell}),
v^{\prime\prime} ) , \\
{x}^{\prime\prime}_{\ell^{\prime\prime}} &: =& {x}_{\ell^{\prime\prime}} (
X_{\mathbf{cl}}(\tilde{t}_{\ell} - \frac{\tilde{t}_{\ell} - \tau}{\e} ; 
\tilde{t}_{\ell}, x_{\ell} , v_{\ell}), v^{\prime\prime} ), {v}%
^{\prime\prime}_{\ell^{\prime\prime}} \ : = \ {v}_{\ell^{\prime\prime}} ( X_{%
\mathbf{cl}}(\tilde{t}_{\ell} - \frac{\tilde{t}_{\ell} - \tau}{\e} ; \tilde{t%
}_{\ell}, x_{\ell} , v_{\ell}), v^{\prime\prime} ) .
\end{eqnarray*}
Furthermore 
\begin{eqnarray*}
\mathbf{I}_{3} &\lesssim & CT_{0}^{5/2} \Big\{ e^{- \frac{C_{0}}{\e^{2}}t}
\| {\e^{\frac 1 2}} h_{0} \|_{\infty} + \e^{\frac 1 2} \sup_{0 \leq s \leq
t} \| w r(s) \|_{\infty} + \e^{\frac 3 2}\sup_{0 \leq s \leq t} \| \langle
v\rangle w g(s) \|_{\infty} \Big\} \\
&& +o(1)CT_{0}^{5/2} \sup_{0 \leq s \leq t} \| \e^{\frac 1 2} h
(s)\|_{\infty} + T_{0}^{5/4}\Big\{\frac{1}{2} \Big\}^{C_{2} T_{0}^{5/4}}
\sup_{0 \leq s \leq t} \| \e^{\frac 1 2} h (s)\|_{\infty}.
\end{eqnarray*}
This bound of $\mathbf{I}_{3}$ is already included in the RHS of (\ref%
{claim1}) and (\ref{claim1_2}).

Now we focus on $\mathbf{I}_{1}$ and $\mathbf{I}_{2}$. Consider the change
of variables 
\begin{eqnarray*}
v^{\prime}_{\ell^{\prime}} \mapsto X_{\mathbf{cl}}( \tilde{t}%
_{\ell^{\prime}}^{\prime}- \frac{\tilde{t}_{\ell^{\prime}}^{\prime}-\tau}{\e}%
; \tilde{t}_{\ell^{\prime}}^{\prime}, x_{\ell^{\prime}}^{\prime},
v_{\ell^{\prime}}^{\prime}).
\end{eqnarray*}
For $0\leq \tilde{t}^{\prime}_{\ell^{\prime}}\leq \tau - \kappa \e^{2} \leq
\tau \leq \e T_{0}$, 
\begin{eqnarray*}
\frac{\partial X_{i} ( \tilde{t}_{\ell^{\prime}}^{\prime} - \frac{\tilde{t}%
_{\ell^{\prime}}^{\prime}-\tau}{\e} ; \tilde{t}^{\prime}_{\ell^{\prime}})}{%
\partial v_{j}^{\prime}}&=& - \frac{\tilde{t}_{\ell^{\prime}}^{\prime}-\tau}{%
\e} \delta_{ij} + \int^{\tilde{t}^{\prime}_{\ell^{\prime}} - \frac{\tilde{t}%
_{\ell^{\prime}}^{\prime}-\tau}{\e} }_{ \tilde{t}^{\prime}_{\ell^{\prime}} } 
\mathrm{d} \tau^{\prime} \int^{\tau^{\prime}}_{ \tilde{t}^{\prime}_{\ell^{%
\prime}} } \mathrm{d} \tau^{\prime\prime} \e^{2} \sum_{m} \partial_{m}
\Phi_{i} (X(\tau^{\prime\prime} ; \tilde{t}^{\prime}_{\ell^{\prime}})) \frac{%
\partial X_{m}}{\partial v_{j}^{\prime}} (\tau^{\prime\prime} ; \tilde{t}%
^{\prime}_{\ell^{\prime}}) \\
&=& - \frac{\tilde{t}_{\ell^{\prime}}^{\prime}-\tau}{\e} \delta_{ij} +
O(1)\|\Phi \|_{C^{1}} \e^{2}( \frac{\tilde{t}_{\ell^{\prime}}^{\prime}-\tau}{%
\e} )^{3} e^{C_{\Phi} \frac{| {\tilde{t}_{\ell^{\prime}}^{\prime}-\tau} |}{\e%
}} \\
&=& - \frac{\tilde{t}_{\ell^{\prime}}^{\prime}-\tau}{\e} \Big[ \delta_{ij} +
O(1) \|\Phi \|_{C^{1}} \e^{2} T_{0}^{2} e^{C_{\Phi }T_{0}} \Big],
\end{eqnarray*}
and therefore 
\begin{eqnarray*}
&&\det \nabla_{v^{\prime}_{\ell^{\prime}}} X_{\mathbf{cl}}( \tilde{t}%
_{\ell^{\prime}}^{\prime}- \frac{\tilde{t}_{\ell^{\prime}}^{\prime}-\tau}{\e}%
; \tilde{t}_{\ell^{\prime}}^{\prime}, x_{\ell^{\prime}}^{\prime},
v_{\ell^{\prime}}^{\prime})= \Big(\frac{-\tilde{t}_{\ell^{\prime}}^{\prime}+%
\tau}{\e}\Big)^{3} \det \Big( \delta_{ij} + O(1) \|\Phi \|_{C^{1}} \e^{2}
T_{0}^{2} e^{C_{\Phi }T_{0}} \Big) \gtrsim \kappa^{3}\e^{3}.
\end{eqnarray*}
We have similar change of variables for $v^{\prime\prime}_{\ell^{\prime%
\prime}} \mapsto X_{\mathbf{cl}}( \tilde{t}_{\ell^{\prime\prime}}^{\prime%
\prime}- \frac{\tilde{t}_{\ell^{\prime\prime}}^{\prime\prime}-\tau^{\prime%
\prime}}{\e}; \tilde{t}_{\ell^{\prime\prime}}^{\prime\prime},
x_{\ell^{\prime\prime}}^{\prime\prime},
v_{\ell^{\prime\prime}}^{\prime\prime})$, $v^{\prime}\mapsto X_{\mathbf{cl}%
}(s- \frac{s-s^{\prime}}{\e}; s, X_{\mathbf{cl}} ( t-\frac{t-s}{\e};t,x,v),
v^{\prime})$, and $v^{\prime\prime} \mapsto X_{\mathbf{cl}} ( \tau - \frac{%
\tau - s^{\prime\prime}}{\e} ; \tau , X_{\mathbf{cl}}(\tilde{t}_{\ell} - 
\frac{\tilde{t}_{\ell} -\tau}{\e} ; \tilde{t}_{\ell}, x_{\ell}, v_{\ell} ),
v^{\prime\prime} )$.

Hence, 
\begin{equation*}
\mathbf{I}_{1} + \mathbf{I}_{2} \ \lesssim \ T_{0}^{5/2} \frac{1}{\e}
\sup_{0 \leq s \leq t} \| f(s)\|_{L^{2}(\Omega \times \mathbb{R}^{3})},
\end{equation*}
where we applied the above change of variables as 
\begin{eqnarray*}
&& \int_{ \mathbb{R}^{3}} \int_{|u| \leq m } \mathbf{k}_{m} \big( V_{\mathbf{%
cl}} ( \tilde{t}^{\prime\prime}_{\ell^{\prime\prime}} - \frac{ \tilde{t}%
_{\ell^{\prime\prime}}^{\prime\prime} -\tau^{\prime\prime} }{\e}; \tilde{t}%
^{\prime\prime}_{\ell^{\prime\prime}},x^{\prime\prime}_{\ell^{\prime%
\prime}}, v^{\prime\prime}_{\ell^{\prime\prime}} ) ,u\big) \big| f\big( %
\tau^{\prime\prime}, X_{\mathbf{cl}}( \tilde{t}_{\ell^{\prime\prime}}^{%
\prime\prime}- \frac{ \tilde{t}_{\ell^{\prime\prime}}^{\prime\prime} -
\tau^{\prime\prime} }{\e}; \tilde{t}^{\prime\prime}_{\ell^{\prime\prime}},
x_{\ell^{\prime\prime}}^{\prime\prime},
v_{\ell^{\prime\prime}}^{\prime\prime} ), u \big) \big| \mathrm{d} u \mathrm{%
d} v^{\prime\prime}_{\ell^{\prime\prime}} \\
&\leq& \Big[ \iint_{\{|u| \leq m\} \times\mathbb{R}^{3}} \mathbf{1}_{|V_{%
\mathbf{cl}} ( \tilde{t}^{\prime\prime}_{\ell^{\prime\prime}} - \frac{ 
\tilde{t}_{\ell^{\prime\prime}}^{\prime\prime} -\tau^{\prime\prime} }{\e}; 
\tilde{t}^{\prime\prime}_{\ell^{\prime\prime}},x^{\prime\prime}_{\ell^{%
\prime\prime}}, v^{\prime\prime}_{\ell^{\prime\prime}} )| \leq m} \mathrm{d}
u \mathrm{d} v^{\prime\prime}_{\ell^{\prime\prime}} \Big]^{1/2} \\
&&\times \Big[ \int_{v^{\prime\prime}_{\ell^{\prime\prime}}} \int_{u} \big| %
f \big( \tau^{\prime\prime}, X_{\mathbf{cl}}( \tilde{t}_{\ell^{\prime%
\prime}}^{\prime\prime}- \frac{ \tilde{t}_{\ell^{\prime\prime}}^{\prime%
\prime} - \tau^{\prime\prime} }{\e}; \tilde{t}^{\prime\prime}_{\ell^{\prime%
\prime}}, x_{\ell^{\prime\prime}}^{\prime\prime},
v_{\ell^{\prime\prime}}^{\prime\prime} ), u \big) \big|^{2} \mathrm{d} u 
\mathrm{d} v^{\prime\prime}_{\ell^{\prime\prime}}\Big]^{1/2} \\
&\lesssim_{m}& \Big[ \int_{\Omega} \int_{u} \big| f \big( %
\tau^{\prime\prime},y, u \big) \big|^{2} \frac{1}{\kappa^{3}\e^{3}} \mathrm{d%
} u \mathrm{d} y\Big]^{1/2} \lesssim_{m} \frac{1}{\e^{3/2}} \| f(
\tau^{\prime\prime} ) \|_{L^{2}(\Omega \times \mathbb{R}^{3})},
\end{eqnarray*}
where we have used $\mathbf{1}_{|V_{\mathbf{cl}} ( \tilde{t}%
^{\prime\prime}_{\ell^{\prime\prime}} - \frac{ \tilde{t}_{\ell^{\prime%
\prime}}^{\prime\prime} -\tau^{\prime\prime} }{\e}; \tilde{t}%
^{\prime\prime}_{\ell^{\prime\prime}},x^{\prime\prime}_{\ell^{\prime%
\prime}}, v^{\prime\prime}_{\ell^{\prime\prime}} )| \leq m} \leq \mathbf{1}%
_{| v^{\prime\prime}_{\ell^{\prime\prime}}|\leq 2m}$ and 
\begin{eqnarray*}
| v^{\prime\prime}_{\ell^{\prime\prime}} | &\leq& |V_{\mathbf{cl}} ( \tilde{t%
}^{\prime\prime}_{\ell^{\prime\prime}} - \frac{ \tilde{t}_{\ell^{\prime%
\prime}}^{\prime\prime} -\tau^{\prime\prime} }{\e}; \tilde{t}%
^{\prime\prime}_{\ell^{\prime\prime}},x^{\prime\prime}_{\ell^{\prime%
\prime}}, v^{\prime\prime}_{\ell^{\prime\prime}} )| + \e^{2} \| \Phi
\|_{\infty}\frac{ |\tilde{t}_{\ell^{\prime\prime}}^{\prime\prime}
-\tau^{\prime\prime} |}{\e} \leq m + \e^{2} \| \Phi \|_{\infty}T_{0} \leq 2m.
\end{eqnarray*}

Moreover 
\begin{equation*}
\begin{split}
\mathbf{I}_{1} + \mathbf{I}_{2} &\lesssim T_{0}^{5/2} \sup_{0 \leq s \leq t}
\| \mathbf{P}f(s)\|_{L^{6}(\Omega )}+ + T_{0}^{5/2} \frac{1}{\e} \sup_{0
\leq s \leq t} \| (\mathbf{I} - \mathbf{P}) f(s)\|_{L^{2}(\Omega \times 
\mathbb{R}^{3})},
\end{split}%
\end{equation*}
where we have used the above change of variables for $\mathbf{P} f$ as 
\begin{eqnarray*}
&& \int_{v^{\prime\prime}_{\ell^{\prime\prime}}} \int_{u} \big| \mathbf{P}f%
\big( \tau^{\prime\prime}, X_{\mathbf{cl}}( \tilde{t}_{\ell^{\prime%
\prime}}^{\prime\prime}- \frac{ \tilde{t}_{\ell^{\prime\prime}}^{\prime%
\prime} - \tau^{\prime\prime} }{\e}; \tilde{t}^{\prime\prime}_{\ell^{\prime%
\prime}}, x_{\ell^{\prime\prime}}^{\prime\prime},
v_{\ell^{\prime\prime}}^{\prime\prime} ) \big) \psi(u) \big| \mathrm{d} u 
\mathrm{d} v^{\prime\prime}_{\ell^{\prime\prime}} \\
&\lesssim_{m}& \Big[ \int_{v^{\prime\prime}_{\ell^{\prime\prime}}} \int_{u} %
\big| \mathbf{P} f \big( \tau^{\prime\prime}, X_{\mathbf{cl}}( \tilde{t}%
_{\ell^{\prime\prime}}^{\prime\prime}- \frac{ \tilde{t}_{\ell^{\prime%
\prime}}^{\prime\prime} - \tau^{\prime\prime} }{\e}; \tilde{t}%
^{\prime\prime}_{\ell^{\prime\prime}},
x_{\ell^{\prime\prime}}^{\prime\prime},
v_{\ell^{\prime\prime}}^{\prime\prime} ) \big) \big|^{2} \mathrm{d}
v^{\prime\prime}_{\ell^{\prime\prime}}\Big]^{1/2} \\
&\lesssim_{m}& \Big[ \int_{\Omega} \big| \mathbf{P} f \big( %
\tau^{\prime\prime},y \big) \big|^{6} \frac{1}{\kappa^{3}\e^{3}} \mathrm{d} y%
\Big]^{1/6} \lesssim_{m} \frac{1}{\e^{\frac 1 2} } \| \mathbf{S}_{1} f (
\tau^{\prime\prime} ) \|_{L^{6}(\Omega )} .
\end{eqnarray*}

All together we prove our claims (\ref{claim1}) and (\ref{claim1_2}).

\vspace{4pt}

\noindent\textit{Step 2. } Applying (\ref{claim1}) successively, 
\begin{eqnarray*}
&&\|\e^{\frac 1 2} h( n \e T_{0})\|_{\infty} \\
&\leq & CT_{0}^{5/2} e^{- \frac{C_{0} T_{0}}{\e}} \| \e^{\frac 1 2} h((n-1) %
\e T_{0}) \|_{\infty} + \sup_{(n-1) \e T_{0} \leq s \leq n \e T_{0} } D(s) \\
&\leq& \Big[ CT_{0}^{5/2} e^{- \frac{C_{0} T_{0}}{\e}}\Big]^{2}\| \e^{\frac
1 2} h((n-2) \e T_{0}) \|_{\infty} + \sum_{j=0}^{1} \Big[ CT_{0}^{5/2} e^{- 
\frac{C_{0} T_{0}}{\e}}\Big]^{j} \sup_{(n-2) \e T_{0} \leq s \leq n \e T_{0}
} D(s) \\
&\vdots& \\
&\leq & \Big[ CT_{0}^{5/2} e^{- \frac{C_{0} T_{0}}{\e}}\Big]^{n}\| \e^{\frac
1 2} h_{0} \|_{\infty} + \sum_{j=0}^{n-1} \Big[ CT_{0}^{5/2} e^{- \frac{%
C_{0} T_{0}}{\e}}\Big]^{j} \sup_{0 \leq s \leq n \e T_{0}} D(s) ,
\end{eqnarray*}
where 
\begin{eqnarray*}
D(s) &:=& \| \e^{\frac 1 2} w r (s) \|_{\infty} + C T_{0}^{5/2} \| \langle
v\rangle^{-1} \e^{\frac 3 2} w g(s)\|_{\infty} + C T_{0}^{5/2} \| \mathbf{P}%
f(s) \|_{L^{6} (\Omega )} \\
&&+ C T_{0}^{5/2} \frac{1}{\e} \| (\mathbf{I} - \mathbf{P})f(s) \|_{L^{2}
(\Omega\times \mathbb{R}^{3} )} + \big[ CT_{0}^{5/4} \Big\{\frac{1}{2}\Big\}%
^{C_{2} T_{0}^{5/4}} + o(1) CT_{0}^{5/4} \big] \| \e^{\frac 1 2}
h(s)\|_{\infty} .
\end{eqnarray*}
Clearly $\sum_{j} \Big[ CT_{0}^{5/2} e^{- \frac{C_{0} T_{0}}{\e}}\Big]^{j} <
\infty$.

Combining the above estimate with (\ref{claim1}), for $t \in [n \e T_{0},
(n+1) \e T_{0}]$, and absorbing the last term, 
\begin{eqnarray*}
&&CT_{0}^{5/2} e^{- \frac{C_{0} ( t- n \e T_{0})}{\e^{2}}} \sum_{j=0}^{n-1} %
\Big[ CT_{0}^{5/2} e^{- \frac{C_{0} T_{0}}{\e}}\Big]^{j}\big[CT_{0}^{5/4}%
\Big\{\frac{1}{2} \Big\}^{C_{2} T_{0}^{5/4}} + o(1)CT_{0}^{5/4} \big] %
\sup_{0 \leq s \leq t} \| \e^{\frac 1 2} h (s)\|_{\infty} \\
&\lesssim& \frac{ T_{0}^{5/2}}{1- CT_{0}^{5/2} e^{- \frac{C_{0} T_{0}}{\e}}} %
\big[CT_{0}^{5/4}\Big\{\frac{1}{2} \Big\}^{C_{2} T_{0}^{5/4}} +
o(1)CT_{0}^{5/2} \big] \times \sup_{0 \leq s \leq t} \| \e^{\frac 1 2} h
(s)\|_{\infty} \ \lesssim \ o(1) \times \sup_{0 \leq s \leq t} \| \e^{\frac
1 2} h (s)\|_{\infty},
\end{eqnarray*}
where we used 
\begin{equation*}
T_{0}^{5/2}\big[CT_{0}^{5/4}\Big\{\frac{1}{2} \Big\}^{C_{2} T_{0}^{5/4}} +
o(1)CT_{0}^{5/2} \big] \ll 1,
\end{equation*}
to conclude (\ref{point1}). Similarly we can prove (\ref{point2}).
\end{proof}

\subsection{$L^6$ estimate}

\begin{proposition}
\label{p6time} Let $f$ satisfy the assumptions of Proposition \ref{dlinearl2}%
. Then 
\begin{eqnarray*}
&& \displaystyle{\sup_{0\le s\le t}\|e^{\l s}\P f\|_{L^6_x}}\ \lesssim \l %
\|f_0\|_2+\e^{-1}\|(\mathbf{I}-\mathbf{P}) f_0\|_{\nu}+\e^{-\frac 1
2}\|(1-P_\g) f_0\|_{2,\g} \\
&& +\mathcal{D}_\l(t)^{\frac 1 2}+(\l +\e)\mathcal{E}_\l(t)^{\frac 1 2}+\|e^{%
\l t}r\|_{L^\infty_tL^2(\gamma)}+\|\e^{\frac 3 2}e^{\l t}\langle
v\rangle^{-1}wr\|_{L^\infty_tL^\infty(\gamma)} +\|e^{\l t}\nu^{-\frac 12}
g\|_{L^\infty_tL^2_{x,v}} \\
&&+\e^{\frac 3 2}\|e^{\l t}\langle v\rangle^{-1}wg\|_{L^\infty_{t,x,v}}+ \l %
\e^{\frac 3 2}\|e^{\l t}\langle v\rangle^{-1}w f\|_{L^\infty_{t,x,v}}+ \e%
^{\frac 5 2}\|e^{\l t}\langle v\rangle^{-1}wf_t\|_{L^\infty_{t,x,v}}.
\label{P1-P6bistime}
\end{eqnarray*}
\end{proposition}

\medskip

\begin{proof}
By moving $\e\pt_t f$ on the right hand side of (\ref{dlinear}), we can use (%
\ref{P1-P6ter}) to obtain, for any $t>0$ 
\begin{multline}  \label{351}
\|\P e^{\l t}f \| _{L^6_{x,v}}\le \e^{-1}\|(\mathbf{I}-\mathbf{P}) e^{\l %
t}f\|_{\nu}+\e^{-\frac 12}|e^{\l t}(1-P_\g)f|_{L^2(\gamma)} +|e^{\l %
t}r|_{L^2(\gamma)}+|\e^{\frac 1 2}e^{\l t}\langle
v\rangle^{-1}wr|_{L^\infty(\gamma)} \\
+\|e^{\l t}\nu^{-\frac 12}[\l f+ g-\e f_t]\|_{L^2_{xv}} +\e^{\frac 3 2}\|e^{%
\l t}\langle v\rangle^{-1}w[\l f+ g-\e f_t]\|_{L^\infty_{xv}}.
\end{multline}
Since 
\begin{multline}  \label{353}
\sup_{s\in [0,t]} \e^{-2}\|(\mathbf{I}-\mathbf{P}) e^{\l s}f\|_{\nu}^{2}\le %
\e^{-2}\|(\mathbf{I}-\mathbf{P}) f_0\|_{\nu}^{2} +2\e^{-2}\int_0^t ds \|(%
\mathbf{I}-\mathbf{P}) e^{\l s}f\|_{\nu}\|(\mathbf{I}-\mathbf{P}) e^{\l %
s}f_t\|_{\nu} \\
\le \e^{-2}\|(\mathbf{I}-\mathbf{P}) f_0\|^2_{\nu}+ \mathcal{D}_\l(t),
\end{multline}
and, similarly 
\begin{multline}
\sup_{s\in [0,t]} \e^{-1}\| e^{\l t}(1-P_\g)f\|_{2,\gamma}^2\le \e%
^{-1}\|(1-P_\g) f_0\|_{2,\g}^2+2\e^{-1}\int_0^t ds \|e^{\l s}(1-P_\g) f\|_{2,%
\g}\|(\mathbf{I}-\mathbf{P}) e^{\l s}(1-P_\g) f_t\|_{2,\g} \\
\le \e^{-1}|(1-P_\g)f_0|_{2,\g}+ \mathcal{D}_\l(t),
\end{multline}
the first two terms in the right hand side of (\ref{351}) are bounded by 
\begin{equation*}
\e^{-1}\|(\mathbf{I}-\mathbf{P}) f_0\|_{\nu}+\e^{-1/2}|(1-P_\g)f_0|_{2,\g}+%
\sqrt{\mathcal{D}_\l(t)}.
\end{equation*}
The third and fourth terms in the right hand side of (\ref{351}) are bounded
by 
\begin{equation*}
|e^{\l t}r|_{L^\infty_tL^2(\gamma)}+\e^{\frac 1 2}e^{\l t}\langle
v\rangle^{-1}wr|_{L^\infty_tL^\infty(\gamma)}.
\end{equation*}
We have 
\begin{equation*}
\|e^{\l t}\nu^{-\frac 12}[\l f+ g-\e f_t]\|^2_{L^2_{t,x,v}}\le \l %
\|f_0\|_2^2+ (\l +\e)\mathcal{D}(t)+ \|e^{\l t}\nu^{-\frac 12}
g\|^2_{L^\infty_tL^2_{xv}},
\end{equation*}
and 
\begin{equation*}
\e^{\frac 3 2}\|e^{\l t}\langle v\rangle^{-1}w[\l f+ g-\e %
f_t]\|_{L^\infty_{t,x,v}}\le \e^{\frac 3 2}\l \|e^{\l t}\langle
v\rangle^{-1}w f\|_{L^\infty_{t,x,v}}+ \e^{\frac 5 2}\|e^{\l t}\langle
v\rangle^{-1}wf_t\|_{L^\infty_{t,x,v}}+ \e^{\frac 3 2}\|e^{\l t}\langle
v\rangle^{-1}wg\|_{L^\infty_{t,x,v}}.
\end{equation*}
Collecting the bounds we conclude the proof.
\end{proof}

\begin{corollary}
\label{corolla} For $\l $ and $\e$ sufficiently small we have 
\begin{eqnarray*}
&& \displaystyle{\sup_{0\le s\le t}\|e^{\l s}\P f\|_{L^6_x}}\lesssim \\
&& \l \|f_0\|_2 +\l \e\|\e^{\frac 1 2} w f_{0} \|_{\infty}+ \e^2\| \e^{\frac
1 2} w f_{t}(0) \|_{\infty} +\e^{-1}\|(\mathbf{I}-\mathbf{P}) f_0\|_{\nu}+\e%
^{-\frac 1 2}\|(1-P_\g) f_0\|_{2,\g} \\
&& +\sqrt{\mathcal{D}_\l(t)}+(\l +2\e)\sqrt{\mathcal{E}_\l(t)}+\|e^{\l %
t}r\|_{L^\infty_tL^2(\gamma)}+\|\e^{\frac 3 2}e^{\l t}\langle
v\rangle^{-1}wr\|_{L^\infty_tL^\infty(\gamma)} \\
&&+\|e^{\l t}\nu^{-\frac 12} g\|_{L^\infty_tL^2_{x,v}}+\e^{\frac 3 2}\|e^{\l %
t}\langle v\rangle^{-1}wg\|^2_{L^\infty_{t,x,v}}+ \sup_{0 \leq s \leq
\infty}\| \e^{\frac 5 2} w r_t(s) \|_{\infty}+\e^{\frac 7 2} \| \langle
v\rangle^{-1} w g_t(s)\|_{\infty}.  \label{P1-P6tertime}
\end{eqnarray*}
\end{corollary}

\begin{proof}
We use (\ref{point1}) to bound 
\begin{multline}
\l \e^{\frac 3 2}\|e^{\l t}\langle v\rangle^{-1}w
f\|^2_{L^\infty_{t,x,v}}\le \l \e \| \e^{\frac 1 2} w f_{0} \|_{\infty} + \l %
\e \sup_{0 \leq s \leq \infty}\| \e^{\frac 1 2} w r(s) \|_{\infty} + \l \e\e%
^{\frac 3 2} \sup_{0 \leq s \leq \infty} \| \langle v\rangle^{-1} w
g(s)\|_{\infty} \\
+ \l \e\sup_{0\leq s \leq t}\| \P f(s)\|_{L^{6}(\Omega )} + \l \e{\e}^{-1}
\sup_{0 \leq s \leq t} \| (\mathbf{I} - \mathbf{P})f(s)\|_{L^{2}(\Omega
\times \mathbb{R}^{3})}
\end{multline}
and (\ref{353}) to bound the last term with $\sqrt{\mathcal{D}_\l(t)}$.
Moreover, we use (\ref{point2}) written for $f_t$ to bound 
\begin{multline}
\e^{\frac 5 2}\|e^{\l t}\langle v\rangle^{-1}wf_t\|_{L^\infty_{t,x,v}}\le \e%
^2\| \e^{\frac 1 2} w f_{t}(0) \|_{\infty} + \e^2 \sup_{0 \leq s \leq
\infty}\| \e^{\frac 1 2} w r_t(s) \|_{L^\infty_{t,x, v}} + \e^2\e^{\frac 3
2} \| \langle v\rangle^{-1} w g_t(s)\|_{\infty} \\
+ \e \|e^{\l t}f_t\|_{L^\infty_tL^2_{xv}},
\end{multline}
and we bound the last term with $\e \sqrt{\mathcal{E}_\l(t)}$. Combining
previous estimates we conclude the proof.
\end{proof}

\subsection{Estimates of the Collision Operators}

\begin{lemma}
\label{nonlinear} Given $f$ and $g$ in $L^2$, assume that, for $t>0$ 
\begin{equation*}
|a_{i}(f)| \leq \mathbf{S}_1 f (t,x)+\mathbf{S}_2 f (t,x)+\mathbf{S}_3 f
(t,x) , \ \ |a_{i}(g)| \leq \mathbf{S}_1 g (t,x)+\mathbf{S}_2 g (t,x) +%
\mathbf{S}_3 g (t,x) ,
\end{equation*}
where $\mathbf{S}_i f , \mathbf{S}_{i}g \geq 0$ 
and $a_{i}(f)$ are defined as $%
[a_{0},a_{1},a_{2},a_{3},a_{4}]=[a,b_{1},b_{2},b_{3}, c]$ in (\ref{Pabc}).

Then 
\begin{equation}
\begin{split}  \label{est_Gamma1}
&\| \nu^{- \frac{1}{2}} \Gamma_{\pm} ( f,g)\|_{L_{t,x,v}^{2}}\\
& \lesssim \e^{\frac 1 2}[\e^{\frac 1 2} \| w g \|_{L^{\infty}_{t,x,v}}] \{[%
\e ^{-1} \| \nu^{- \frac{1}{2}} ( \mathbf{I}- \mathbf{P})
f\|_{L_{t,x,v}^{2}}]+\e^{-1}\| e^{\lambda t}\S _3 f \|_{L^2_{t,x,v}}\} \\
& + \| \P g\|_{L^\infty_tL^6_x}\| \mathbf{S}_{1} f\|_{L^2_tL^3_x} + \e^{1/4}
[ \e^{1/2} \| g\|_{\infty} ] ^{1/2} \| \P g\|_{L^\infty_tL^{6}_{x,v} }^{1/2}
[\e^{-1/2} \| \S _2 f\|_{L^2_tL^\frac{12}{5}_x}] \\
&+ \e^{\frac 1 2}[\e^{\frac 1 2} \| w f\|_{L^{\infty}_{t,x,v}}] [\e ^{-1} \|
\nu^{- \frac{1}{2}} ( \mathbf{I}- \mathbf{P}) g\|_{L_{t,x,v}^{2}}]
\end{split}%
\end{equation}
and 
\begin{equation}
\begin{split}  \label{Gamma_g_no_t}
&\| \nu^{- \frac{1}{2}} \Gamma_{\pm} ( f,g)\|_{L_{t,x,v}^{2}}  \\
& \lesssim \e^{\frac 1 2}[\e^{\frac 1 2} \| w g \|_{ {\infty} }] \big\{ [\e%
^{-1}\| e^{\lambda t}\S _3 f \|_{L^2_{t,x,v}}] + [\e^{-1} \| (\mathbf{I} - 
\mathbf{P}) f \|_{L^{2}_{t,x,v}}] \big\} + \| \P g\|_{L^\infty_tL^6_{x,v}}\| 
\mathbf{S}_{1} f\|_{L^2_tL^3_x} \\
& + \e^{1/4} [ \e^{1/2} \| e^{\lambda t} g\|_{\infty} ] ^{1/2} \|e^{\lambda
t}\P g\|_{L^\infty_tL^{6}_{x,v} }^{1/2} [\e^{-1/2} \|e^{\lambda t}\S _2
f\|_{L^2_tL^\frac{12}{5}_x}] \\
&+ [\e^{-1} \| (\mathbf{I} - \mathbf{P}) g \|_{L^{\infty}_{t}
L^{2}_{x,v}}]^{1/3} [\e^{\frac{1}{2}} \| w g \|_{\infty}^{2/3} ] \| \mathbf{P%
} f \|_{L^{2}_{t} L^{3}_{x,v}},
\end{split}%
\end{equation}
and \begin{equation}  \label{Gamma_nonsym}
\begin{split}
& \| \nu^{- \frac{1}{2}} \Gamma_{\pm} ( f, g)\|_{L_{t,x,v}^{2}} 
\\
& \lesssim \e^{\frac 1 2}[\e^{\frac 1 2} \| w g\|_{L^{\infty}_{t,x,v}}] \{[%
\e ^{-1} \| \nu^{- \frac{1}{2}} ( \mathbf{I}- \mathbf{P})f%
\|_{L_{t,x,v}^{2}}]+\e^{-1}\| e^{\lambda t}\S _3 f \|_{L^2_{t,x,v}}\} \\
& +\e^{1/4}[\e^{1/2} \| w g\|_{ {\infty} }]^{1/2} \| \mathbf{P} g
\|_{L^{\infty}_{t} L^{6}_{x,v}} ^{1/2} [\e^{-1/2} \| \S _2 f \|_{L^2_tL^%
\frac{12}{5}_x}] \\
& +\Big\{ \| \P g\|_{L^\infty_tL^6_x} + [\e^{1/2} \| w g \|_{\infty}]^{2/3} %
\Big[ \sqrt{\mathcal{D} _{\lambda}[g](\infty) } + \e^{-1} \| (\mathbf{I} - 
\mathbf{P}) g|_{t=0} \|_{\nu} \Big]^{1/3} \Big\} \| \S _1 f \|_{L^2_tL^3_x}.
\end{split}%
\end{equation}

\end{lemma}

\begin{proof}
First we prove (\ref{est_Gamma1}). We decompose 
\begin{equation}  \label{decom_f}
|f(t,x,v)| \leq |\mathbf{P} f(t,x,v)| + |(\mathbf{I} - \mathbf{P}) f(t,x,v)|,
\end{equation}
and $|g(t,x,v)|$ in the same way. We use the same decomposition of (\ref%
{Gamma_decom_s}) replacing the $L^{2}_{x,v}$ norm with $L^{2}_{t,x,v}$ norm.

The first two terms of the RHS of (\ref{Gamma_decom_s}) is bounded by 
\begin{eqnarray*}
 \e \| wg \|_{L^{\infty}_{t,x,v}} \|\nu^{-1/2} \Gamma_{\pm}( \e%
^{-1}|(\mathbf{I}-\mathbf{P})f|, w^{-1} )\|_{L^{2}_{t,x,v}}+ \e \| wf \|_{L^{\infty}_{t,x,v}} 
 \|\nu^{-1/2} \Gamma_{\pm}( \e%
^{-1}|(\mathbf{I}-\mathbf{P})g|, w^{-1} )\|_{L^{2}_{t,x,v}} \\
 \lesssim  \e^{\frac{1}{2}} [\e^{\frac{1}{2}} \| wg \|_{L^{\infty}_{t,x,v}}
] [\e^{-1} \| \nu^{-\frac{1}{2}} (\mathbf{I} - \mathbf{P}) f
\|_{L^{2}_{t,x,v}} ] + \e^{\frac{1}{2}} [\e^{\frac{1}{2}} \| wf
\|_{L^{\infty}_{t,x,v}} ] [\e^{-1} \| \nu^{-\frac{1}{2}} (\mathbf{I} - 
\mathbf{P})g \|_{L^{2}_{t,x,v}} ].
\end{eqnarray*}
To bound the last term of (\ref{Gamma_decom_s}) we note that 
\begin{equation}  \label{decom_Pf}
|\P f|\lesssim \nu^2\sqrt{\mu}\Big[ \sum_{i=1}^3 \mathbf{S}_{i} f \Big].
\end{equation}
From $\| \nu^{-1/2}\Gamma ( \mu^{0+} , \mu^{0+} ) \|_{L^{p}_{ v}}< \infty$,
we get, for any $1\leq p \leq \infty$ 
\begin{equation}
\|\nu^{-1/2} \Gamma_{\pm}(|\mathbf{S}_{i} f|\nu^2\sqrt{\mu}, |\mathbf{P}g|
)\|_{L^2_{t,x,v}}\lesssim \, \big\|\mathbf{S}_i f \| \mathbf{P}
g\|_{L^{p}_{v}} \big\|_{L^2_{t,x}}\, .  \notag
\end{equation}
We estimate $\big\|\mathbf{S}_i f \| \mathbf{P} g\|_{L^{p}_{v}} \big\|%
_{L^2_{t,x}}$ for $i=1,2,3$: 
\begin{eqnarray*}
\| \mathbf{S}_{1}f \| \mathbf{P} g\|_{L^{6}_{v}} \|_{L^2_{t,x}} &\lesssim& %
\big\| \| \mathbf{S}_{1} f \|_{L^{3}_{x}} \| \mathbf{P} g\|_{L^{6}_{x,v}} %
\big\|_{L^{2}_{t}} \lesssim \| \mathbf{S}_{1} f \|_{L^{2}_{t}L^{3}_{x}} \| 
\mathbf{P} g\|_{L^{\infty}_{t}L^{6}_{x,v}}, \\
\big\| \mathbf{S}_{{2}} f \| \mathbf{P}g \|_{L^{6}_{v}} \big\|_{L^2_{t,x}}
&\lesssim& \big\|\| \mathbf{S}_{{2}} f\|_{L^{\frac{12}5}_{x}} \| \mathbf{P}g
\|_{L^{12}_{x}L^{6}_{v}} \big\|_{L^{2}_{t}} \le \| \mathbf{S}_{{2}}
f\|_{L^{2 }_{t}L^{\frac{12}5}_{x} } \| \mathbf{P}g\|_{L^{%
\infty}_{t}L^{12}_{x} L^{6}_{v}} \\
&\lesssim& \| \mathbf{S}_{{2}} f\|_{L^{2}_{t}L^{\frac{12}5}_{x} } \e^{-\frac
1 4}\sqrt{\| \P g\|_{L^{\infty}_{t}L^{6}_{x,v} } \e^{\frac 1 2}\|
g\|_{L^{\infty}_{t,x,v} } }, \\
\big\| \mathbf{S}_{{3}} f \| \mathbf{P}g \|_{L^{\infty}_{v}} \big\|%
_{L^2_{t,x}} &\lesssim& \| \mathbf{S}_{3}f \|_{L^{2}_{t,x}} \| g \|_{\infty}
.
\end{eqnarray*}
By collecting the estimates we obtain the estimate (\ref{est_Gamma1}).

Now we prove (\ref{Gamma_g_no_t}). All the estimates are same as (\ref%
{est_Gamma1}) except 
\begin{equation}  \label{diff_gamma}
\| \nu^{-1/2} \Gamma_{\pm} ( |\mathbf{P} f|,| (\mathbf{I} - \mathbf{P}) g|)
\|_{L^{2}_{t,x,v}}.
\end{equation}
By Holder inequality, 
\begin{eqnarray*}
(\ref{diff_gamma})&\lesssim& \| \mathbf{P} f \|_{L^{2}_{t}L^{3}_{x,v}} \|
\nu^{-\frac{1}{2}} (\mathbf{I} - \mathbf{P}) g\|_{L^{\infty}_{t} L^{6}_{x,v}
} \\
&\lesssim& \| \mathbf{P} f \|_{L^{2}_{t}L^{3}_{x,v}} [\e^{\frac{1}{2}}\| g
\|_{\infty}]^{2/3} [\e^{-1} \| \nu^{-\frac{1}{2}} (\mathbf{I} - \mathbf{P})
g\|_{L^{\infty}_{t} L^{2}_{x,v} }]^{1/3}.
\end{eqnarray*}

The proof of (\ref{Gamma_nonsym}) is same as the proof of (\ref{est_Gamma1})
except (\ref{diff_gamma}). By (\ref{decom_Pf}), 
\begin{eqnarray*}
 (\ref{diff_gamma}) 
&\lesssim& \| \mathbf{S}_{1} f \|_{L^{2}_{t} L^{3}_{x}} \| \nu^{- \frac{1}{2}%
} (\mathbf{I} - \mathbf{P}) g \|_{L_{t}^{\infty}L^{6}_{x,v}} + [\e^{-\frac{1%
}{2}}\| \mathbf{S}_{2} f \|_{L^{2}_{t}L^{\frac{12}{5}}_{x}} ][ \e^{ \frac{1}{%
2}}\| w g \|_{\infty}] \\
&&+ \e^{\frac{1}{2}}[\e^{-1 } \| \mathbf{S}_{3} f \|_{L^{2}_{t,x,v}} ] [\e^{%
\frac{1}{2}} \| g \|_{\infty} ] \\
&\lesssim& \| \mathbf{S}_{1} f \|_{L^{2}_{t} L^{3}_{x}}[\e^{-1}\| \nu^{- 
\frac{1}{2}} (\mathbf{I} - \mathbf{P}) g
\|_{L_{t}^{\infty}L^{2}_{x,v}}]^{1/3} [ \e^{\frac{1}{2}}\| g
\|_{\infty}]^{2/3} \\
&& + [\e^{-\frac{1}{2}}\| \mathbf{S}_{2} f \|_{L^{2}_{t}L^{\frac{12}{5}%
}_{x}} ][ \e^{ \frac{1}{2}}\| w g \|_{\infty}] + \e^{\frac{1}{2}}[\e^{-1 }
\| \mathbf{S}_{3} f \|_{L^{2}_{t,x,v}} ] [\e^{\frac{1}{2}} \| wg \|_{\infty}
].
\end{eqnarray*}
Using (\ref{353}), we deduce 
\begin{equation*}
\e^{-1} \| \nu^{- \frac{1}{2}} (\mathbf{I} - \mathbf{P}) g
\|_{L_{t}^{\infty}L^{2}_{x,v}} \lesssim \e^{-1} \| \nu^{- \frac{1}{2}} (%
\mathbf{I} - \mathbf{P}) g_{0} \|_{ L^{2}_{x,v}} + \sqrt{\mathcal{D}_{0} [g]
(\infty)}.
\end{equation*}
By collecting the terms we prove (\ref{Gamma_nonsym}).
\end{proof}

In order to estimate $\Gamma(f, \pt_t g)$ we will need the following \textit{%
commutation} property:

\begin{lemma}
\label{Spt=ptS}For $t\ge 0$, the 
\begin{equation}
\pt_t (\S _1 f)\le \S _1 (\pt_t f), \quad \pt_t (\S _2 f)\le \S _2 (\pt_t f).
\end{equation}
\end{lemma}

\begin{proof}
From the definition of $\mathbf{S}_{1} f(t,x)$ in (\ref{S1}), 
\begin{equation*}
\partial_{t}[\mathbf{S}_{1} f(t,x)] = 2 \int_{\mathbb{R}^{3}} \text{sgn}%
(f_{\delta} (t,x,v)) \partial_{t} f_{\delta} (t,x,v) \nu^{2} \sqrt{\mu(v)} 
\mathrm{d} v.
\end{equation*}
From the definition of $f_{\delta}$ in (\ref{Z_dyn}), for $t\geq 0$ 
\begin{eqnarray*}
&&\partial_{t} f_{\delta} (t,x,v)|_{t\geq 0} \\
&=& [1-\chi(\frac{n(x) \cdot v}{\delta}) \chi \big(\frac{\xi(x)}{\delta}%
\big) ] \chi(\delta|v|) \big\{ \mathbf{1}_{ t\in[0, \infty)} \partial_{t}
f(t,x,v)+ \mathbf{1}_{t\in(-\infty,0 ]} \chi^{\prime}(t) f_{0}(x,v) \big\}%
\big|_{t\geq 0} \\
&=& [1-\chi(\frac{n(x) \cdot v}{\delta}) \chi \big(\frac{\xi(x)}{\delta}%
\big) ] \chi(\delta|v|) \mathbf{1}_{ t\in[0, \infty)} \partial_{t} f(t,x,v)
\\
&=&[1-\chi(\frac{n(x) \cdot v}{\delta}) \chi \big(\frac{\xi(x)}{\delta}\big) %
] \chi(\delta|v|) \big\{ \mathbf{1}_{ t\in[0, \infty)} \partial_{t}
f(t,x,v)+ \mathbf{1}_{t\in(-\infty,0 ]} \chi (t) \partial_{t} f_{0}(x,v) %
\big\}\big|_{t\geq 0} \\
&=& [\partial_{t} f]_{\delta} (t,x,v) |_{t\geq 0}.
\end{eqnarray*}
Therefore, for $t\geq 0$, 
\begin{eqnarray*}
\partial_{t}[\mathbf{S}_{1} f(t,x)] & \leq& 2 \int_{\mathbb{R}^{3}} |
\partial_{t} f_{\delta} (t,x,v)| \nu^{2} \sqrt{\mu(v)} \mathrm{d} v \leq 2
\int_{\mathbb{R}^{3}} |[\partial_{t} f]_{\delta} (t,x,v)| \nu^{2} \sqrt{%
\mu(v)} \mathrm{d} v \\
&=& \mathbf{S}_{1} \partial_{t} f (t,x).
\end{eqnarray*}
Similarly we show the second inequality.
\end{proof}

\subsection{Global-in-Time Validity}

\begin{proof}[\textbf{Proof of Theorem \protect\ref{energy_nonlinear}}]
\label{3.7} 
For the construction of the solution and the energy estimate, we consider $%
\tilde{f}^{\ell}(t,x,v)$ solving, for $\ell \in \mathbb{N}$, 
\begin{equation}
\begin{split}  \label{Z_eq}
& \partial_{t} [ e^{\lambda t} \tilde{f}^{\ell+1} ]+ \e^{-1}
v\cdot\nabla_{x} [ e^{\lambda t} \tilde{f}^{\ell+1} ] + \e \Phi \cdot
\nabla_{v} [ e^{\lambda t} \tilde{f}^{\ell+1} ] + \e^{-2} L [ e^{\lambda t} 
\tilde{f}^{\ell+1} ] \\
&= \lambda [ e^{\lambda t} \tilde{f}^{\ell+1} ]+ \e^{-1} L_{ {f_w+}f_{s}}
[e^{\lambda t} \tilde{f}^{\ell}] + e^{-\lambda t} \e^{-1} \Gamma( e^{\lambda
t}\tilde{f}^{\ell} , e^{\lambda t} \tilde{f}^{\ell} ) + \e \frac{\Phi \cdot v%
}{2} [e^{\lambda t}\tilde{f}^{\ell+1}], \\
& e^{\lambda t}\tilde{f}^{\ell +1} |_{\gamma_{-}} =P_{\gamma} e^{\lambda t}%
\tilde{f}^{\ell+1 } + \e e^{\lambda t} \mathcal{Q} \tilde{f}^{\ell} ,\ \ \
e^{\lambda t}\tilde{f}^{\ell+1} |_{t=0}= \tilde{f}_{0}.
\end{split}%
\end{equation}
Here we set $\tilde{f}^{0}(t,x,v) := 0$.

Clearly $\tilde{f}^{\ell}_{t}:= \partial_{t} \tilde{f}^{\ell}$, with $\tilde
f_{t,0}=\pt_t \tilde f(0)$ solves 
\begin{equation}  \label{Zt_eq}
\begin{split}
& \partial_{t} [e^{\lambda t} \tilde{f}^{\ell+1}_{t}] + \e^{-1}
v\cdot\nabla_{x} [e^{\lambda t}\tilde{f}^{\ell+1} _{t}] + \e \Phi \cdot
\nabla_{v} [e^{\lambda t} \tilde{f}^{\ell+1}_{t} ] + \e^{-2} L [e^{\lambda t}%
\tilde{f}^{\ell+1}_{t}] \\
&= \lambda [ e^{\lambda t}\tilde{f}^{\ell+1}_{t} ]+\e^{-1}L_{ {f_w+}f_s}[e^{%
\l t} f_t] + e^{-\lambda t}\e^{-1} [\Gamma(e^{\lambda t}\tilde{f}^{\ell}_{t}
, e^{\lambda t}\tilde{f}^{\ell} )+ \Gamma( e^{\lambda t} \tilde{f}^{\ell},
e^{\lambda t}\tilde{f}^{\ell}_{t} )] + \e \frac{\Phi \cdot v}{2} [
e^{\lambda t} \tilde{f}^{\ell+1}_{t}] , \\
&e^{\lambda t}\tilde{f}^{\ell +1}_{t}(t,x,v) |_{\gamma_{-}} = P_{\gamma}
e^{\lambda t} \tilde{f}_{t}^{\ell+1 } + \e e^{\lambda t} \mathcal{Q} \tilde{f%
}^{\ell}_{t},\quad \ \ e^{\lambda t}\tilde{f}_{t}^{\ell+1} |_{t=0} =
\partial_{t} \tilde{f}_{0}.
\end{split}%
\end{equation}

As in the steady case, from (\ref{sym_L1}) and $\int_{n \cdot v \gtrless0}
M_w |n \cdot v| \mathrm{d} v = 1= \int_{n \cdot v \gtrless0} \sqrt{2\pi} \mu
|n \cdot v| \mathrm{d} v $, 
\begin{equation}
\mathbf{P}\big( \e^{-1} L_{f_{s}} \tilde{f} + \e^{-1} \Gamma(\tilde{f}, 
\tilde{f}) \big) =0, \quad \int_{\mathbb{R}^{3}} \int_{n\cdot v<0} \mathcal{Q%
} \tilde{f} \{n \cdot v\} \mathrm{d} v \ = \ 0.  \notag
\end{equation}
Note that Proposition \ref{dlinearl2} guarantees the solvability of such
linear problems (\ref{Z_eq}) and (\ref{Zt_eq}).

Now define the quantity that we want to bound in the iteration scheme of (%
\ref{Z_eq}): 
\begin{equation}  \label{norm_unsteady}
\begin{split}
[\hskip-1pt [\hskip-1pt [ \tilde{f}^{\ell} ]\hskip-1pt ]\hskip-1pt ] : =& 
\sqrt{\mathcal{E}_{\lambda} [\tilde{f}^{\ell}] (\infty)} + \sqrt{\mathcal{D}%
_{\lambda} [\tilde{f}^{\ell}] (\infty) } + \e^{\frac{1}{2}} \| e^{\lambda t} 
\tilde{f}^{\ell} \|_{\infty}+ \e^{\frac{3}{2}} \| \e^{\lambda t}
\partial_{t} \tilde{f}^{\ell} \|_{\infty} + \| e^{\lambda t} \mathbf{P} 
\tilde{f}^{\ell} \|_{L^{\infty}_{t} L^{6}_{x,v}} \\
&+ \| \mathbf{S}_{1} \tilde{f}^{\ell} \| _{L^{2}_{t} L^{3}_{x}}+ \| \mathbf{S%
}_{1} \tilde{f}^{\ell}_{t} \| _{L^{2}_{t} L^{3}_{x}} + \e^{- \frac{1}{2}}\| 
\mathbf{S}_{2} \tilde{f}^{\ell} \| _{L^{2}_{t} L^{\frac{12}{5}}_{x}} + \e^{-%
\frac{1}{2}}\| \mathbf{S}_{2} \tilde{f}^{\ell}_{t} \| _{L^{2}_{t} L^{\frac{12%
}{5}}_{x}} ,
\end{split}%
\end{equation}
where $\mathbf{S}_{1} \tilde{f}^{\ell}$ and $\mathbf{S}_{2} \tilde{f}^{\ell}$
are defined in (\ref{def_S2}). For the sake of simplicity temporally we
denote $\mathcal{E}_{\lambda}^{\ell} = \mathcal{E}_{\lambda} [\tilde{f}%
^{\ell}] (\infty)$ and $\mathcal{D}_{\lambda}^{\ell} = \mathcal{D}_{\lambda}
[\tilde{f}^{\ell}] (\infty)$.

For $0 < \eta_{0} \ll1$ and $0<c_0\ll1$, %
we assume (induction hypothesis) that 
\begin{equation}  \label{induc_hyp_st}
\begin{split}
\sup_{0 \leq j \leq \ell}[\hskip-1pt [\hskip-1pt [ \tilde{f}^{j} ]\hskip-1pt
]\hskip-1pt ] < \eta_0, \ \ & \|\tilde f_0\|_{L^2_{x,v}}+\|\tilde
f_0\|_{L^\infty_{x,v}}+ \|\tilde \partial_{t}f_{0}\|_{L^2_{x,v}}+\|\tilde
\partial_{t}f_{0}\|_{L^\infty_{x,v}} < c_0\eta_{0}, \\
& \| \mathbf{P} f_{s} \|_{L^{6}_{x}} + [\e^{-1} \| (\mathbf{I} - \mathbf{P})
f_{s} \|_{L^{2}_{x,v}}] < \eta_{0}.
\end{split}%
\end{equation}
The condition for the steady solution $f_{s}$ 
can be achieved by choosing further smaller $\| \vartheta_{w} \|_{H^{\frac{1%
}{2}} (\partial\Omega)} + \| \Phi \|_{H^{1}(\Omega)}$ in Theorem 1.2.

In order to show that (\ref{induc_hyp_st}) holds for all $\ell$, it suffice
to show that (\ref{induc_hyp_st}) holds for $j=\ell+1$. Throughout \textit{%
Step 1} to \textit{Step 4} we claim 
\begin{equation}  \label{induc_ell+1}
[\hskip-1pt [\hskip-1pt [ \tilde{f}^{\ell+1} ]\hskip-1pt ]\hskip-1pt ] \leq [%
\hskip-1pt [\hskip-1pt [ \tilde{f}^{\ell } ]\hskip-1pt ]\hskip-1pt ]^{2} +
o(1)[\hskip-1pt [\hskip-1pt [ \tilde{f}^{\ell+1} ]\hskip-1pt ]\hskip-1pt ]
+o(1) .
\end{equation}
This clearly proves (\ref{induc_hyp_st}) for all $\ell$. 

\noindent\textit{Step 1. } We prove the crucial estimates involving operator 
$\Gamma$. We apply (\ref{induc_hyp_st}) repeatedly. Applying (\ref%
{est_Gamma1}) with $f= e^{\lambda t}\tilde{f}^{\ell} =g$, 
\begin{equation}  \label{esgammaff}
\| \nu^{- \frac{1}{2}}e^{-\lambda t} \Gamma (e^{\lambda t}\tilde{f}%
^{\ell},e^{\lambda t}\tilde{f}^{\ell})\|_{L_{t,x,v}^{2}} \lesssim (1+ \e%
^{1/4} + \e^{1/2}) [\hskip-1pt [\hskip-1pt [ \tilde{f}^{\ell}]\hskip-1pt ]%
\hskip-1pt ]^{2}.
\end{equation}
Again applying (\ref{Gamma_nonsym}) with $f= e^{\lambda t}\tilde{f}^{\ell},
g= e^{\lambda t} \tilde{f}^{\ell}_{t}$, 
\begin{eqnarray}
&&\| \nu^{- \frac{1}{2}}e^{-\lambda t} \Gamma (e^{\lambda t}\tilde{f}%
^{\ell},e^{\lambda t}\tilde{f}_t^{\ell})\|_{L_{t,x,v}^{2}} \\
& \lesssim & \e^{\frac 1 2}[\e^{\frac 1 2} \| e^{\lambda t} w \tilde{f}%
^{\ell} \|_{L^{\infty}_{t,x,v}}] \{[\e ^{-1} \| \nu^{- \frac{1}{2}%
}e^{\lambda t} ( \mathbf{I}- \mathbf{P}) \tilde{f}_t^{\ell}%
\|_{L_{t,x,v}^{2}}]+\e^{-1}\| e^{\lambda t}\S _3\tilde f_t^\ell
\|_{L^2_{t,x,v}}\}  \notag  \label{est_Gamma1_t} \\
&& \e^{1/4} [ \e^{1/2} \| e^{\lambda t} \tilde{f}^{\ell}\|_{\infty} ] ^{1/2}
\|e^{\lambda t}\P \tilde f ^\ell\|_{L^\infty_tL^{6}_{x,v} }^{1/2} [\e^{1/2}
\| w e^{\lambda t} \tilde{f}^{\ell} \|_{ {\infty} }] [\e^{-1/2} \|e^{\lambda
t}\S _2 \tilde f_{t}^\ell\|_{L^2_tL^\frac{12}{5}_x}]  \notag \\
&&+\Big\{ \|e^{\lambda t}\P \tilde f^\ell\|_{L^\infty_tL^6_x} + [\e^{1/2} \|
w e^{\lambda t} \tilde{f}^{\ell} \|_{\infty}]^{2/3} \Big[ \sqrt{\mathcal{D}
_{\lambda}[\tilde{f}^{\ell}](\infty) } + \e^{-1} \| (\mathbf{I} - \mathbf{P}%
) f_{0} \|_{\nu} \Big]^{1/3} \Big\} \|e^{\lambda t}\S _1 \tilde
f_t^\ell\|_{L^2_tL^3_x}  \notag \\
& \lesssim & (1+\e^{\frac{1}{2}} ) [\hskip-1pt [\hskip-1pt [ \tilde{f}%
^{\ell} ]\hskip-1pt ]\hskip-1pt ]^{2} + c_{0} \eta_{0} [\hskip-1pt [\hskip%
-1pt [ \tilde{f}^{\ell} ]\hskip-1pt ]\hskip-1pt ] .  \label{esgammafft}
\end{eqnarray}

Recall (\ref{sym_L1}). Applying (\ref{Gamma_g_no_t}) with $g= ( {f_w+}{f}_s)$
and $f= e^{\lambda t} \tilde{f}^{\ell}$, 
\begin{equation}
\begin{split}  \label{esLf}
&\| \nu^{- \frac{1}{2}} L_{f_{w} + f_{s}} e^{\lambda t} \tilde{f}^{\ell}
\|_{L_{t,x,v}^{2}} \lesssim ( 1 + \e^{\frac{1}{4}} + \e^{\frac{1}{2}} )
\eta_{0} [\hskip-1pt [\hskip-1pt [ \tilde{f}^{\ell} ]\hskip-1pt ]\hskip-1pt
]. \\
& \lesssim \e^{\frac 1 2}[\e^{\frac 1 2} \| w g \|_{ {\infty} }] \big\{ [\e%
^{-1}\| e^{\lambda t}\S _3 f \|_{L^2_{t,x,v}}] + [\e^{-1} \| (\mathbf{I} - 
\mathbf{P}) f \|_{L^{2}_{t,x,v}}] \big\} + \| \P g\|_{L^\infty_tL^6_{x,v}}\| 
\mathbf{S}_{1} f\|_{L^2_tL^3_x} \\
& + \e^{1/4} [ \e^{1/2} \| g\|_{\infty} ] ^{1/2} \| \P g\|_{L^%
\infty_tL^{6}_{x,v} }^{1/2} [\e^{-1/2} \|e^{\lambda t}\S _2 f\|_{L^2_tL^%
\frac{12}{5}_x}] \\
&+ [\e^{-1} \| (\mathbf{I} - \mathbf{P}) g \|_{L^{\infty}_{t}
L^{2}_{x,v}}]^{1/3} [\e^{\frac{1}{2}} \| w g \|_{\infty}^{2/3} ] \| \mathbf{P%
} f \|_{L^{2}_{t} L^{3}_{x,v}},
\end{split}%
\end{equation}
Again applying (\ref{Gamma_g_no_t}) with $f= ( {f_w+}{f}_s)$ and $%
g=e^{\lambda t} \tilde{f}^{\ell}_{t}$, 
\begin{equation}
\begin{split}  \label{esLft}
&\| \nu^{- \frac{1}{2}} L_{f_{w} + f_{s}} e^{\lambda t} \tilde{f}_{t}^{\ell}
\|_{L_{t,x,v}^{2}} \lesssim ( 1 + \e^{\frac{1}{4}} + \e^{\frac{1}{2}} )
\eta_{0} [\hskip-1pt [\hskip-1pt [ \tilde{f}^{\ell} ]\hskip-1pt ]\hskip-1pt
].
\end{split}%
\end{equation}

\noindent\textit{Step 2. } 
From (\ref{completes_dyn}), (\ref{boundQ}), (\ref{esgammaff}), (\ref{esLf}),
(\ref{esgammafft}), and (\ref{esLft})%
\begin{eqnarray}
&&\| e^{\lambda t} \tilde{f}^{\ell+1}(t) \|_{2}^{2} + \frac{1}{\e}
\int^{t}_{0} | e^{\lambda s} {(1-P_\g)} \tilde{f}^{\ell+1} |_{2 }^{2} + 
\frac{1}{\e^{2}} \int^{t}_{0} \| e^{\lambda s}(\mathbf{I}- \mathbf{P}) 
\tilde{f}^{\ell+1} \|_{\nu} ^{2} + \int^{t}_{0} \| e^{\lambda s} \mathbf{P} 
\tilde{f}^{\ell+1} \|_{2}^{2}  \notag \\
&\lesssim& \| \tilde{f}^{\ell+1}(0)\|_{2}^{2} + \e^{-1} \int^{t}_{0}
|e^{\lambda s} \e \mathcal{Q} \tilde{f}^{\ell}|_{2,-}^{2} + \int^{t}_{0} {%
e^{-\lambda s}} \|\nu^{-\frac{1}{2}} \Gamma( e^{\lambda s}\tilde{f}^{\ell },
e^{\lambda s}\tilde{f}^{\ell }) \|_{ 2}^{2} + \int^{t}_{0} \| \nu^{-\frac{1}{%
2}} L_{ {f_w+}{f}_s}e^{\lambda s} \tilde{f}^{\ell} \|_{2}^{2}  \notag \\
&\lesssim& (c_0\eta_0)^{2}+ \e \|\t_w\|_\infty [\hskip-1pt [\hskip-1pt [ 
\tilde{f}^{\ell} ]\hskip-1pt ]\hskip-1pt ]^{2} + [\hskip-1pt [\hskip-1pt [ 
\tilde{f}^{\ell} ]\hskip-1pt ]\hskip-1pt ]^{4} + \eta_{0}^{2} [\hskip-1pt [%
\hskip-1pt [ \tilde{f}^{\ell} ]\hskip-1pt ]\hskip-1pt ]^{4} ,  \notag
\end{eqnarray}
and 
\begin{eqnarray}
&&\| e^{\lambda t} \tilde{f}_t^{\ell+1}(t) \|_{2}^{2} + \frac{1}{\e}
\int^{t}_{0} | e^{\lambda s} {(1-P_\g)} \tilde{f}_t^{\ell+1} |_{2 }^{2} + 
\frac{1}{\e^{2}} \int^{t}_{0} \| e^{\lambda s}(\mathbf{I}- \mathbf{P}) 
\tilde{f}_t^{\ell+1} \|_{\nu} ^{2} + \int^{t}_{0} \| e^{\lambda s} \mathbf{P}
\tilde{f}_t^{\ell+1} \|_{2}^{2}  \notag \\
&\lesssim& \| \tilde{f}_t^{\ell+1}(0)\|_{2}^{2} + \e^{-1} \int^{t}_{0}
|e^{\lambda s} \e \mathcal{Q} \tilde{f}_t^{\ell}|_{2,-}^{2} + \int^{t}_{0} {%
e^{-\lambda s}} \|\nu^{-\frac{1}{2}} \Gamma( e^{\lambda s}\tilde{f}_t^{\ell
}, e^{\lambda s}\tilde{f}^{\ell }) \|_{ 2}^{2} + \int^{t}_{0} \| \nu^{-\frac{%
1}{2}} L_{ {f_w+}{f}_s}e^{\lambda s} \tilde{f}_t^{\ell} \|_{2}^{2}  \notag \\
&\lesssim& (c_0\eta_0)^{2}+ \e \|\t_w\|_\infty [\hskip-1pt [\hskip-1pt [ 
\tilde{f}^{\ell} ]\hskip-1pt ]\hskip-1pt ]^{2} + [\hskip-1pt [\hskip-1pt [ 
\tilde{f}^{\ell} ]\hskip-1pt ]\hskip-1pt ]^{4} + \eta_{0}^{2} [\hskip-1pt [%
\hskip-1pt [ \tilde{f}^{\ell} ]\hskip-1pt ]\hskip-1pt ]^{4} .  \notag
\end{eqnarray}

Therefore we conclude 
\begin{equation}  \label{bound_ED}
\mathcal{E}^{\ell+1}_{\lambda} + \mathcal{D}^{\ell+1}_{\lambda} \lesssim %
\big\{[\hskip-1pt [\hskip-1pt [ \tilde{f}^{\ell}]\hskip-1pt ]\hskip-1pt
]^{2} + o(1) [\hskip-1pt [\hskip-1pt [ \tilde{f}^{\ell}]\hskip-1pt ]\hskip%
-1pt ] +o(1)(\eta_{0})\big\}^{2} .
\end{equation}

\vspace{12pt}

\noindent\textit{Step 3. } We apply Proposition \ref{prop_3} to (\ref{Z_eq}%
): Set 
\begin{eqnarray*}
f&=& e^{\lambda t} \tilde{f}^{\ell+1} , \\
g&=& -\e^{-1} L [e^{\lambda t} \tilde{f}^{\ell+1}] + \e \lambda [ e^{\lambda
t} \tilde{f}^{\ell+1} ]+ L_{ {f_w+}{f}_s} [ e^{\lambda t} \tilde{f}^{\ell} ]
+ e^{-\lambda t} \Gamma( e^{\lambda t}\tilde{f}^{\ell} , e^{\lambda t} 
\tilde{f}^{\ell} ) + \e^{2} \frac{\Phi \cdot v}{2} [e^{\lambda t}\tilde{f}%
^{\ell+1}].
\end{eqnarray*}

Then, from (\ref{esgammaff}) and (\ref{esLf}), 
\begin{eqnarray}
&&\|\mathbf{S}_{1} e^{\lambda t} \tilde{f}^{\ell+1} \|_{ L^{2}_{t}L^{3}_{x}}
+ \e^{-\frac{1}{2}}\|\mathbf{S}_{2} e^{\lambda t} \tilde{f}^{\ell+1} \|_{
L^{2}_{t}L^{\frac{12}{5}}_{x}}  \notag \\
&\lesssim& \big\| w^{-1} \Big[ -\e^{-1} L [e^{\lambda t} \tilde{f}^{\ell+1}]
+ \e \lambda [ e^{\lambda t} \tilde{f}^{\ell+1} ]+ L_{ {f_w+}{f}_s} [
e^{\lambda t}\tilde{f}^{\ell}] + e^{-\lambda t} \Gamma( e^{\lambda t}\tilde{f%
}^{\ell} , e^{\lambda t} \tilde{f}^{\ell} ) + \e^{2} \frac{\Phi \cdot v}{2}
[e^{\lambda t}\tilde{f}^{\ell+1}] \Big] \big\|_{L^{2}_{t,x,v}}  \notag \\
& &+ \| e^{\lambda t} \tilde{f}^{\ell+1} \|_{L^{2}_{t,x,v}} + \| e^{\lambda
t} \tilde{f}^{\ell+1}\|_{L^{2}_{t}L^{2}_{\gamma } } + \| f_{0}
\|_{L^{2}_{x,v}} + \| [ v\cdot \nabla_{x} + \e^{2} \Phi \cdot \nabla_{v}]
f_{0} \|_{L^{2}_{x,v}} + \| f_{0} \|_{L^{2}(\gamma)}  \notag \\
&\lesssim & (1+ \e \lambda + \e^{2} \| \Phi \|_{\infty} ) \mathcal{E}%
^{\ell+1}_{\lambda} + (c_{0} \eta_{0}+ [\hskip-1pt [\hskip-1pt [ \tilde{f}%
^{\ell } ]\hskip-1pt ]\hskip-1pt ]) [\hskip-1pt [\hskip-1pt [ \tilde{f}%
^{\ell } ]\hskip-1pt ]\hskip-1pt ] + c_{0} \eta_{0} .  \notag
\end{eqnarray}
From (\ref{bound_ED}) 
\begin{equation}
\|\mathbf{S}_{1} e^{\lambda t} \tilde{f}^{\ell+1} \|_{ L^{2}_{t}L^{3}_{x}} + %
\e^{-\frac{1}{2}}\|\mathbf{S}_{2} e^{\lambda t} \tilde{f}^{\ell+1} \|_{
L^{2}_{t}L^{\frac{12}{5}}_{x}} \lesssim [\hskip-1pt [\hskip-1pt [ \tilde{f}%
^{\ell } ]\hskip-1pt ]\hskip-1pt ]^{2} + o(1) \eta_{0}.  \label{est_S1f}
\end{equation}

Similarly, we apply Proposition \ref{prop_3} to (\ref{Zt_eq}): Set 
\begin{eqnarray*}
f&=& e^{\lambda t} \tilde{f}^{\ell+1}_{t}, \\
g&=& - \e^{-1} L [e^{\lambda t}\tilde{f}^{\ell+1}_{t}] + \e \lambda [
e^{\lambda t}\tilde{f}^{\ell+1}_{t} ] + L_{ {f_w+}{f}_s}[e^{\lambda t} 
\tilde{f}^{\ell} _{t}] + e^{-\lambda t} [\Gamma(e^{\lambda t}\tilde{f}%
^{\ell}_{t} , e^{\lambda t}\tilde{f}^{\ell} )+ \Gamma( e^{\lambda t} \tilde{f%
}^{\ell}, e^{\lambda t}\tilde{f}^{\ell}_{t} )] \\
\\
&& \ \ +\e^{2} \frac{\Phi \cdot v}{2}[ e^{\lambda t} \tilde{f}^{\ell+1}_{t}].
\end{eqnarray*}
Then 
\begin{equation}  \label{est_S1ft}
\| \mathbf{S}_{1} e^{\lambda t} \tilde{f}^{\ell+1}_{t} \|_{
L^{2}_{t}L^{3}_{x}} +\e^{-\frac{1}{2}} \| \mathbf{S}_{2} e^{\lambda t} 
\tilde{f}^{\ell+1}_{t} \|_{ L^{2}_{t}L^{3}_{x}} \lesssim [\hskip-1pt [\hskip%
-1pt [ \tilde{f}^{\ell } ]\hskip-1pt ]\hskip-1pt ]^{2} + o(1) \eta_{0} .
\end{equation}
\vspace{8pt}

\noindent\textit{Step 4. } We apply Proposition \ref{point_dyn} to (\ref%
{Z_eq}): Set $f = e^{\lambda t} \tilde{f}^{\ell+1}$. Note $\e^{-1}L
e^{\lambda t} \tilde{f}^{\ell+1 } = \e^{-1}\nu(v)e^{\lambda t} \tilde{f}%
^{\ell+1 } - \e^{-1}\int_{\mathbb{R}^{3}} \mathbf{k} (v,u)e^{\lambda t} 
\tilde{f}^{\ell+1 }(u) \mathrm{d} u$ with $\nu(v) \sim \langle v\rangle$ and 
$|\mathbf{k}(v,u)| \lesssim \mathbf{k}_{\beta}(v,u)$. Moreover, 
\begin{equation*}
\e^{-1}\nu(v) - \e \lambda - \e \frac{\Phi \cdot v}{2} \ \gtrsim \ \e^{-1} C
\langle v\rangle - \e \lambda- \e \| \Phi \|_{\infty} |v| \ \gtrsim \ \e%
^{-1}C_{0 } \langle v\rangle.
\end{equation*}
Therefore (\ref{linear_K}), the condition of Proposition \ref{point_dyn}, is
satisfied with the following setting 
\begin{eqnarray*}
g&=& L_{f_s} [ e^{\lambda t} \tilde{f}^{\ell } ] + e^{-\lambda t} \Gamma(
e^{\lambda t}\tilde{f}^{\ell} , e^{\lambda t} \tilde{f}^{\ell} ) , \quad r =%
\e e^{\l t}\mathcal{Q}\tilde{f}^{\ell} \quad f_{0} \ = \ \tilde{f}_{0}.
\end{eqnarray*}
By Proposition \ref{point_dyn}, from (\ref{point1}), 
\begin{eqnarray*}
&&\e^{\frac 12}\|e^{\lambda t} w\tilde{f}^{\ell+1} \|_{L^{\infty}_{t,x,v}} \\
&\lesssim& \e^{\frac 12}\| w \tilde{f}^{\ell+1} (0) \|_{\infty} + \e^{\frac
12} \max_{0 \leq j \leq \ell}\sup_{0 \leq t \leq \infty} \e \| e^{\lambda t}
w\tilde{f}^{j}\|_{\infty} + \e\sup_{0 \leq t \leq \infty} \e^{1/2} \|
e^{\lambda t} w \mathcal{Q}\tilde f^\ell \|_{\infty} \\
& +& \e^{\frac 3 2} \sup_{0 \leq t \leq \infty} \big\| w \nu^{-1} \Big[ L_{ {%
f_w+}{f}_s} [ e^{\lambda t} \tilde{f}^{\ell} ] + \Gamma( e^{\lambda t}\tilde{%
f}^{\ell} , e^{\lambda t} \tilde{f}^{\ell} ) \Big] \big\|_{\infty} \\
& +& \| e^{\lambda s} \mathbf{P}\tilde{f}^{\ell+1} (s)\|_{L^{\infty}_{t}L^{ {%
6}}_{x}} + \frac{1}{\e} \| e^{\lambda s} (\mathbf{I} - \mathbf{P}) \tilde{f}%
^{\ell+1} (s)\|_{L^{\infty}_{t}L^{2}_{x,v} }.
\end{eqnarray*}

Using $|w \Gamma_{\pm} (w^{-1},w^{-1})| \lesssim \langle v\rangle \lesssim
\nu$, we obtain 
\begin{eqnarray*}
&& \e^{\frac32} \sup_{0 \leq t\leq \infty} \big\| w \nu^{-1} e^{-\lambda t}
\Gamma(e^{\lambda t} \tilde{f}^{\ell} ,e^{\lambda t} \tilde{f}^{\ell} ) %
\big\|_{\infty} \lesssim \e^{1/2} [ \sup_{0 \leq t \leq \infty} \| \e^{\frac
1 2} w e^{\lambda t} \tilde{f}^{\ell} \|_{\infty} ]^{2} | \nu^{-1}w
\Gamma(w^{-1}, w^{-1})| \\
&& \ \ \ \ \lesssim \e^{1/2} [ \sup_{0 \leq t \leq \infty} \| \e^{\frac 1 2}
w e^{\lambda t} \tilde{f}^{\ell} \|_{\infty} ]^{2} \lesssim \e^{1/2} [\hskip%
-1pt [\hskip-1pt [ \tilde{f}^{\ell}]\hskip-1pt ]\hskip-1pt ] ^{2} , \\
&&| \e^{\frac 3 2} w \nu^{-1} L_{ {f_w+}{f}_s } ( e^{\lambda t} \tilde{f}%
^{\ell})| \lesssim \e^{\frac 1 2} | w \nu^{-1} \Gamma_{\pm } ( \e^{\frac 1
2} ( {f_w+}{f}_s) , e^{\lambda t} \e^{\frac 1 2} \tilde f ^{\ell})| \\
&& \ \ \ \ \lesssim \e^{\frac 1 2} \| w \e^{\frac 1 2} ( {f_w+}{f}_s)
\|_{\infty} \| w e^{\lambda t} [\e^{\frac 1 2} \tilde{f}^{\ell}] \|_{\infty}
| \nu^{-1} w \Gamma_{\pm } (w^{-1}, w^{-1})| \lesssim \e^{\frac 1 2} \| w \e%
^{\frac 1 2} ( {f_w+}{f}_s) \|_{\infty} \| w \e^{\frac 1 2} \tilde f
^{\ell}\|_{\infty} \\
&& \ \ \ \ \lesssim \e^{1/2} \eta_{0} [\hskip-1pt [\hskip-1pt [ \tilde{f}%
^{\ell} ]\hskip-1pt ]\hskip-1pt ] .
\end{eqnarray*}
Using Corollary \ref{corolla} with $f=e^{\l t}\tilde f^{\ell+1}$, $g=\big[%
e^{-\l t}\Gamma(e^{\l t}\tilde f^\ell,e^{\l t}\tilde f^\ell)+e^{-\l %
t}L_{f_s}e^{\l t}\tilde f_t^\ell\big]$, $r= e^{\l t}\e\mathcal{Q}\tilde
f^{\ell}$ and $g_t=\e^{-\l t}\Gamma(e^{\l t}\tilde f^\ell,e^{\l t}\tilde
f_t^\ell)+e^{-\l t}\Gamma(e^{\l t}\tilde f_t^\ell,e^{\l t}\tilde f^\ell)+e^{-%
\l t}L_{f_s}e^{\l t}\tilde f_t^\ell$, $r_t= e^{\l t}\e\mathcal{Q}\tilde
f_t^{\ell}$, we have 
\begin{eqnarray*}
&&\|\P \tilde f^{\ell+1}e^{\l t}\|_{L^\infty_tL^6_x}\le \l \|f_0\|_2+\e%
^{-1}\|(\mathbf{I}-\mathbf{P}) f_0\|_{\nu}+\e^{-\frac 1 2}\|(1-P_\g)
f_0\|_{2,\g} \\
&& +\mathcal{D}^{\ell+1}_\l(t)^{\frac 1 2}+(\l +2\e)\mathcal{E}%
^{\ell+1}_\l(t)^{\frac 1 2}+\|e^{\l t}\e e^{\l t}\mathcal{Q}\tilde
f^\ell\|_{L^\infty_tL^2(\gamma)}+\|\e^{\frac 3 2}e^{\l t}\langle
v\rangle^{-1}w\e e^{\l t}\mathcal{Q} \tilde
f^\ell\|_{L^\infty_tL^\infty(\gamma)} \\
&& +\|e^{\l t}\nu^{-\frac 12}\big[e^{-\l t}\Gamma(e^{\l t}\tilde f^\ell,e^{%
\l t}\tilde f^\ell)+e^{-\l t}L_{{f_w+}{f}_s}e^{\l t}\tilde f^\ell\big] %
\|_{L^\infty_tL^2_{x,v}} \\
&&+\e^{\frac 3 2}\|e^{\l t}\langle v\rangle^{-1}w\big[ e^{-\l t}\Gamma(e^{\l %
t}\tilde f^\ell,e^{\l t}\tilde f^\ell)+e^{-\l t}L_{{f_w+}{f}_s}e^{\l %
t}\tilde f^\ell\big]\|_{L^\infty_{t,x,v}}+ \e^{\frac 5 2}\|e^{\l t}\langle
v\rangle^{-1}w\e\mathcal{Q}\tilde f^\ell\|_{L^\infty_{t,\gamma}} \\
&&+ \e^{\frac 7 2}\|\langle v\rangle w \big[e^{-\l t}\Gamma(e^{\l t}\tilde
f^\ell,e^{\l t}\tilde f_t^\ell)+e^{-\l t}\Gamma(e^{\l t}\tilde f_t^\ell,e^{%
\l t}\tilde f^\ell)+e^{-\l t}L_{{f_w+}{f}_s}e^{\l t}\tilde f_t^\ell\big]%
\|_{L^\infty_t,x,v}.
\end{eqnarray*}

By (\ref{esgammaff})$-$(\ref{esLft}), and (\ref{bound_ED}), 
\begin{equation}  \label{infty_iteration}
\e^{\frac 12} \| e^{\lambda t} w \tilde f^{\ell+1} \|_{L^{\infty}_{t,x,v}}+\|%
\P \tilde f^{\ell+1}e^{\l t}\|^2_{L^\infty_tL^6_x}\ \lesssim \ [\hskip-1pt [%
\hskip-1pt [ \tilde{f}^{\ell} ]\hskip-1pt ]\hskip-1pt ]^{2} +o(1)\eta_{0}.
\end{equation}

\vspace{4pt}

Now we consider $\partial_{t} \tilde{f}^{\ell+1}$. Apply Proposition \ref%
{point_dyn} to (\ref{Zt_eq}): Set $f= e^{\lambda t} \tilde{f}^{\ell+1}_{t}$.
Note $\e^{-1} L e^{\lambda t} \tilde{f}^{\ell+1}_{t} = \e^{-1} \nu(v)
e^{\lambda t} \tilde{f}^{\ell+1}_{t} - \e^{-1} \int_{\mathbb{R}^{3}} \mathbf{%
k} (v,u) e^{\lambda t} \tilde{f}^{\ell+1}_{t}(u) \mathrm{d} u$ with $\nu(v)
\sim \langle v\rangle$ and $|\mathbf{k} (v,u)| \lesssim \mathbf{k}_{\beta}
(v,u)$. For Proposition \ref{point_dyn} we set 
\begin{eqnarray*}
g&=& L_{ {f_w+}{f}_s} [ e^{\lambda t} \tilde{f}^{\ell}_{t} ] + e^{-\lambda
t} \Gamma( e^{\lambda t}\tilde{f}^{\ell}_{t} , e^{\lambda t} \tilde{f}%
^{\ell} )+ e^{-\lambda t} \e^{ 1/2} \Gamma( e^{\lambda t}\tilde{f}^{\ell} ,
e^{\lambda t} \tilde{f}^{\ell}_{t} ), \\
r &=& \e e^{\lambda t} \partial_{t} \mathcal{Q}\tilde{f}_{t}^{\ell}, \ \ \
f_{0} \ = \ \partial_{t}\tilde{f}_{0}.
\end{eqnarray*}

From (\ref{point2}), 
\begin{eqnarray*}
&&\e^{\frac 1 2}\|e^{\lambda t} w \tilde{f}^{\ell+1}_{t}
\|_{L^{\infty}_{t,x,v}} \\
&\lesssim& \e^{\frac 1 2} \| w \tilde{f}^{\ell+1}_{t} (0) \|_{\infty} +\e%
^{\frac 1 2} \e \| w e^{\lambda t} \mathcal{Q}\tilde f_t \|_{L^{\infty}_{t,\g%
}} + \frac{1}{\e} \| e^{\lambda t} \tilde{f}^{\ell+1}_{t}
(t)\|_{L^{\infty}_{t}L^{2}_{x,v}} \\
&&+ \e^{\frac 3 2} \sup_{0 \leq t \leq \infty} \big\| w \nu^{-1} \Big[ L_{{%
f_w+}{f}_s} [ e^{\lambda t} \tilde{f}^{\ell}_{t} ] + e^{-\lambda t} \Gamma(
e^{\lambda t}\tilde{f}^{\ell}_{t} , e^{\lambda t} \tilde{f}^{\ell} )+
e^{-\lambda t} \Gamma( e^{\lambda t}\tilde{f}^{\ell} , e^{\lambda t} \tilde{f%
}^{\ell}_{t} ) \Big] \big\|_{\infty}.
\end{eqnarray*}
From $|w \Gamma_{\pm} (w^{-1},w^{-1})| \lesssim \langle v\rangle \lesssim
\nu $, 
\begin{eqnarray*}
&& \e^{\frac 3 2} \sup_{0 \leq t\leq \infty} \big\| w \nu^{-1} e^{-\lambda
t} \Gamma(e^{\lambda t} \tilde{f}^{\ell}_{t} ,e^{\lambda t} \tilde{f}^{\ell}
) \big\|_{\infty} + \e^{\frac 3 2} \sup_{0 \leq t\leq \infty} \big\| w
\nu^{-1} e^{-\lambda t} \Gamma(e^{\lambda t} \tilde{f}^{\ell} ,e^{\lambda t} 
\tilde{f}^{\ell}_{t} ) \big\|_{\infty} \\
&& \ \ \ \ \lesssim \e^{1/2} [ \| \e^{\frac 1 2} w e^{\lambda t} \tilde{f}%
^{\ell} \|_{\infty} ] [ \| \e^{\frac 3 2} w e^{\lambda t} \tilde{f}%
^{\ell}_{t} \|_{\infty} ] | \nu^{-1}w \Gamma(w^{-1}, w^{-1})| \\
&& \ \ \ \ \lesssim \e^{1/2} [ \|\e^{\frac 1 2} w e^{\lambda t} \tilde{f}%
^{\ell} \|_{\infty} ] [ \| \e^{\frac 1 2} w e^{\lambda t} \tilde{f}%
^{\ell}_{t} \|_{\infty} ], \\
&&|\e^{\frac 3 2} w \nu^{-1} L_{{f_w+}{f}_s} e^{\lambda t} \tilde{f}%
^{\ell}_{t}|\lesssim \e^{\frac 1 2} \| w ({f_w+}{f}_s)\|_{\infty} \| w
e^{\lambda t} \e^{\frac 12} \tilde{f}^{\ell}_{t} \|_{\infty} .
\end{eqnarray*}
Altogether 
\begin{eqnarray*}
&&\e^{\frac 3 2}\|e^{\lambda t} w\tilde{f}^{\ell+1}_{t}
\|_{L^{\infty}_{t,x,v}} \lesssim \e^{\frac 3 2} \| w \tilde{f}^{\ell+1}_{t}
(0) \|_{\infty} + \e^{\frac 3 2} \|\t_w\|_\infty\|e^{\l t} w \tilde{f}%
^{\ell}_{t} \|_{\infty}+\|\e^{\l t} \tilde
f_t^{\ell+1}\|_{L^\infty_tL^2_{x,v}} \\
&&+\e^{\frac 3 2}[ \|\e^{\frac 1 2} w e^{\lambda t} \tilde{f}^{\ell}
\|_{\infty} ] [ \| \e^{\frac 1 2} w e^{\lambda t} \tilde{f}^{\ell}_{t}
\|_{\infty} + \| w ({f_w+}{f}_s)\|_{\infty} \| w e^{\lambda t} \e^{\frac 12} 
\tilde{f}^{\ell}_{t} \|_{\infty}],
\end{eqnarray*}
and therefore 
\begin{equation}  \label{infty_t}
\e^{\frac 3 2}\|e^{\lambda t} w\tilde{f}^{\ell+1}_{t}
\|_{L^{\infty}_{t,x,v}} \ < [\hskip-1pt [\hskip-1pt [ \tilde{f}^{\ell} ]%
\hskip-1pt ]\hskip-1pt ]^{2}+ o(1) \eta_{0} .
\end{equation}

Collecting (\ref{bound_ED}), (\ref{est_S1f}), (\ref{est_S1ft}), (\ref%
{infty_iteration}), and (\ref{infty_t}), we prove the claim (\ref%
{induc_ell+1}).

\vspace{4pt}

\noindent\textit{Step 5.} We repeat $\mathit{Step 1}\sim \mathit{Step 4}$
for $\tilde{R}^{\ell+1}- \tilde{R}^{\ell}$ to show that $\tilde{R}^{\ell}$
is Cauchy sequence in $L^{\infty} \cap L^{2}$ for fixed $\e$. Now we pass a
limit $\ell \rightarrow \infty$ in $L^{\infty} \cap L^{2}$ to conclude the
existence. The proof of uniqueness is standard. (See \cite{EGKM} for details)

\noindent\textit{Step 6.} The proof that the components of the limit $\tilde
f_1$ satisfy the unsteady INSF system \ref{tilde_u_theta} is achieved
similarly to the steady case and will not be repeated here.
\end{proof}

\subsection{Positivity of Solutions}

In this section, we prove the non-negativity of $F_{s}$ in the main theorem.
The proof is based on the asymptotical stability of $F_{s}$ (Proposition \ref%
{prop_3}) and the non-negativity of unsteady solution.

\begin{proof}[\textbf{Proof of the non-negativity of $F(t,x,v)$ in Theorem 
\protect\ref{energy_nonlinear} and $F_{s}(x,v)$ in Theorem \protect\ref%
{mainth}}]
We use the positivity-preserving sequence as in \cite{Guo08,EGKM}. Set $%
F^{0} (t,x,v) = F_{0} (x,v)\geq 0$ and for $\ell\geq 0$ 
\begin{eqnarray*}
\partial_{t} F^{\ell+1} + \frac{1}{\e} v\cdot \nabla_{x} F^{\ell+1} + \e %
\Phi \cdot \nabla_{v} F^{\ell+1} + \frac{1}{\e^{2}} \nu(F^{\ell}) F^{\ell+1}
= \frac{1}{\e^{2}} Q_{{+}} (F^{\ell},F^{\ell}), \\
F^{\ell+1}(x,v) |_{\gamma_{-}} = M_w \int_{n(x) \cdot u>0} F^{\ell+1 } (x,u)
\{n(x) \cdot u\} \mathrm{d} u, \ \ \ F^{\ell+1} (t,x,v)|_{t=0} = F_{0}(x,v),
\end{eqnarray*}
where $\nu(F)(v)=\int_{\mathbb{R}^3}du \int_{\mathbb{S}^2}d\o B(v-u,\o %
)F(v_*)$.

Note that $F^{\ell+1}\geq0$ for all $\ell$ by proof of Theorem 4 in page 807 of \cite{Guo08}.

\vspace{4pt}

\noindent\textit{Step 1. } We set $F^{\ell} = \mu + \e  f^\ell\sqrt{\mu}$ 
and let $F^{0}(t,x,v):= F_{0}(x,v)$. We claim that, there exists $0<T = T (
\| \e
w f^{\ell} (t_{0} )\|_{\infty}) \ll 1$ and $C_{1}=C_{1}
(T)\gg 1$ for any $t_{0}\geq 0$ such that 
\begin{equation}  \label{claim_pp_bounded}
\sup_{t_{0} \leq t \leq t_{0} + \e^{2} T}\| \e w f^{\ell+1}(t) \|_{\infty}
\leq C_{1}\big\{  \| \e w f^{\ell+1}(t_{0}) \|_{\infty} 
+ \big(\sup_{t_{0} \leq t \leq t_{0} + \e^{2} T}\| \e w f^{\ell}(t) \|_{\infty}\big)^{2}
+O(\e^{2}) \| A\|_{\infty}\big\}.
\end{equation}
It suffices to show (\ref{claim_pp_bounded}) for $t_{0}=0$. Clearly (\ref%
{claim_pp_bounded}) holds for $\ell=0$. Now we assume (\ref{claim_pp_bounded}%
) for $0 \leq l \leq \ell.$

Clearly, $f^{\ell+1}$ solves 
\begin{equation}  \label{eq_R_pp}
\begin{split}
& \pt_t f^{\ell+1} +\e^{-1}v\cdot \nabla_x f^{\ell+1} + \e\Phi\cdot \nabla_v
f^{\ell+1} + \e \frac{\Phi \cdot v}{2} f^{\ell+1} +\e^{-2}\nu f^{\ell+1}-\e%
^{-2} K f^{\ell} \\
& \ \ =\ \ \ \ \ \e^{-1}\Gamma_{{+}}(f^{\ell},f^{\ell}) - \e^{-1}
\nu(f^{\ell} \sqrt{\mu}) f^{\ell+1}+A, \\
&f^{\ell+1} |_{t=0} = f_{0},
\end{split}%
\end{equation}
where $A=2 \Phi\cdot v\sqrt{\mu}$. {The boundary condition is given by 
\begin{equation}  \label{R_bdry_pp}
f^{\ell+1}|_{\gamma_{-}} = P_{\gamma} f^{\ell+1} + \e \mathcal{Q}f^{\ell+1}.
\end{equation}%
} Define $h^{\ell} (t,x,v): = w(v)^{-1} f^{\ell} (t,x,v)$. Note that 
\begin{eqnarray*}
\check{\nu} (v) := \nu(v) - \e  \frac{\| \Phi \|_{\infty} |v|}{2} -\e %
\nu( \sqrt{\mu}) \| w f^{\ell}\|_{L^{\infty}_{t,x,v}} \geq \frac{\nu(v)}{2}
\gtrsim \langle v\rangle.
\end{eqnarray*}
We define $\check{\mathbf{k}}$ such that $\int_{\mathbb{R}^{3}} \check{%
\mathbf{k}} (v,u) \frac{w(v)}{w(u)} f^{\ell}(u)\mathrm{d} u= K f^{\ell} + \e %
\nu(f^{\ell} \sqrt{\mu}) [f_{w} ] \sqrt{\mu} - 
\e  [  \Gamma_{{+}} ( f_{w} , f^{\ell} ) + \Gamma_{{+}} (  f^{\ell},f_{w} ) 
]$. Then $\check{\mathbf{k}}(v,u) \lesssim \mathbf{k}_{\beta}(v,u)$.

Then, for $\tilde{t}_{1}\leq s \leq t$ 
\begin{eqnarray*}
&&\frac{d}{ds} \big[ | \e {h}^{\ell+1} (s, Y_{\mathbf{cl}} (s; t,x,v), W_{%
\mathbf{cl}}(s;t,x,v))| e^{- \int^{t}_{s} \e^{-2}\check{\nu} (W_{\mathbf{cl}%
}(\tau;t,x,v) ) \mathrm{d} \tau} \big] \\
&\lesssim & \Big\{ \e^{-2} \int_{\mathbb{R}^{3}} {\mathbf{k}_{\beta}}(W_{%
\mathbf{cl}} (s;t,x,v), u) | \e h^{\ell} (s, Y_{\mathbf{cl}} (s;t,x,v), u)| 
\mathrm{d} u \\
&& \ \ + \e^{-2} \langle W_{\mathbf{cl}} (s;t,x,v) \rangle \| \e h^{\ell}(s)
\|_{\infty}^{2} +  |A| \Big\} e^{- \int^{t}_{s}\e^{-2}\check{\nu} (W_{%
\mathbf{cl}}(\tau;t,x,v) ) \mathrm{d} \tau} \\
&\lesssim&\big\{ 1+ \langle W_{\mathbf{cl}} (s;t,x,v) \rangle \| \e %
h^{\ell}(s) \|_{\infty} \big\} \| \e h^{\ell} (s) \|_{\infty} \e^{-2} e^{-
\int^{t}_{s}\e^{-2}\check{\nu} (W_{\mathbf{cl}}(\tau;t,x,v) ) \mathrm{d}
\tau} \\
&& + \| A\|_{\infty}\e^2 \e^{-2} e^{- \int^{t}_{s}\e^{-2}\check{\nu} (W_{%
\mathbf{cl}}(\tau;t,x,v) ) \mathrm{d} \tau},
\end{eqnarray*}
where we used the fact $\int_{\mathbb{R}^{3}} {\mathbf{k}_{\beta}}(W_{%
\mathbf{cl}} (s;t,x,v), u) \mathrm{d} u \lesssim 1$ and $w\Gamma(\frac{\e %
h^{\ell}}{w}, \frac{\e h^{\ell}}{w} )(v) \lesssim \langle v \rangle \| \e %
h^{\ell}\|_{\infty}^{2}$.

Then, for $t \in [0, \e^{2} T]$, 
\begin{eqnarray*}
&& | \e h^{\ell+1}(t,x,v) |  \notag \\
&\lesssim& \mathbf{1}_{\{ \tilde{t}_{1} < 0\}} e^{ - CT} | | \e %
h^{\ell+1}(0)\| _{\infty} \\
&+& \int_{\max{\{0, \tilde{t}_{1} (x,v) \}}}^{t} \mathrm{d} s \ \frac{e^{-
\int^{t}_{s} \frac{\check{\nu} ( V_{\mathbf{cl}} ( t- \frac{t-\tau}{\e}%
;t,x,v) )}{\e^{2}} \mathrm{d} \tau } }{\e^{2}} \\
&& \ \ \ \ \ \ \ \times \Big\{ (1+ \langle V_{\mathbf{cl}} ( t- \frac{t-s}{\e%
};t,x,v)\rangle ) \| \e h^{\ell} (s) \|_{\infty} )\| \e h^{\ell} (s)
\|_{\infty} + \e^2 \| A\|_{\infty} \Big\} \\
&+&\mathbf{1}_{\{ \tilde{t}_{1} \geq 0 \}} e^{- \int^{t}_{\tilde{t}_{1}} 
\frac{\tilde{\nu} ( V_{\mathbf{cl}}(t- \frac{t-\tau}{\e} ; t,x,v ))}{\e^{2}} 
\mathrm{d} \tau } O(\e) \mu( \tilde{v}_{1} )^{\frac{1}{2}-} + \mathbf{1%
}_{\{ \tilde{t}_{1} \geq 0 \}} \frac{ e^{- \int^{t}_{ \tilde{t}_{1}}\frac{%
\tilde{\nu} ( t- \frac{t-\tau}{\e};t,x,v )}{\e^{2}} \mathrm{d} \tau} }{%
\tilde{w} (v_{1})} \int_{\Pi_{j=1}^{k-1} \mathcal{V}_{j}} H,  \notag
\end{eqnarray*}
where $H$ is given by 
\begin{eqnarray*}
&& \sum_{l =1}^{k-1} \mathbf{1}_{ \tilde{t}_{l+1} \leq 0< \tilde{t}_{l}} \| %
\e h^{\ell +1}(0 )\|_{\infty}
\Pi_{m=1}^{l-1} \frac{\tilde{w} (v_{m})}{\tilde{w} (V_{\mathbf{cl}} (\tilde{t}_{m+1}
  - \frac{\tilde{t}_{m+1}}{\e}
  ;  v_{m}))} 
 \mathrm{d} \Sigma_{l} (0) \\
&+& \sum_{l =1}^{k-1} \int^{ \tilde{t}_{l}}_{\max \{ 0, \tilde{t}_{l+1} \}} 
\mathbf{1}_{ \tilde{t}_{l}>0}\Big\{ (1+ \langle V_{\mathbf{cl}} ( t- \frac{%
t-\tau}{\e};t,x,v)\rangle ) \| \e h^{\ell } (\tau) \|_{\infty} \| \e %
h^{\ell } (\tau) \|_{\infty}
+ \e^2 \|
A\|_{\infty} \Big\}\\
&& \ \ \ \ \ \  \ \ \ \ \ \ \ \ \ \ \ \ \ \ \ \ \ \ \ \ \ \ \ \ \ \ \ \  \times
\Pi_{m=1}^{l-1} \frac{\tilde{w} (v_{m})}{\tilde{w} (V_{\mathbf{cl}} (\tilde{t}_{m+1}
  - \frac{\tilde{t}_{m+1}}{\e}
  ;  v_{m}))} 
 \mathrm{d} \Sigma_{l} (\tau) \mathrm{d} \tau \\
&+& \sum_{l=1}^{k-1} \mathbf{1}_{\tilde{t}_{l}>0} O( \e) \mu (
v_{l})^{\frac{1}{2}-} 
\Pi_{m=1}^{l-1} \frac{\tilde{w} (v_{m})}{\tilde{w} (V_{\mathbf{cl}} (\tilde{t}_{m+1}
  - \frac{\tilde{t}_{m+1}}{\e}
  ;  v_{m}))} 
\mathrm{d} \Sigma_{l} (\tilde{t}_{l+1})\\
& +& \mathbf{1}_{ 
\tilde{t}_{k} >0}\| \e h^{\ell+1 }( \tilde{t}_{k} )\|_{\infty}
\Pi_{m=1}^{k-2} \frac{\tilde{w} (v_{m})}{\tilde{w} (V_{\mathbf{cl}} (\tilde{t}_{m+1}
  - \frac{\tilde{t}_{m+1}}{\e}
  ;  v_{m}))} 
 \mathrm{d}
\Sigma_{k-1} ( \tilde{t}_{k}),
\end{eqnarray*}
and $\mathrm{d} \Sigma_{k-1} ( \tilde{t}_{k})$ is evaluated at $s=\tilde{t}%
_{k}$ of 
\begin{equation*}
\mathrm{d}\Sigma _{l} (s):=\{\Pi _{j=l+1}^{k-1}\mathrm{d} \sigma _{j}\}\{
e^{ -\int^{\tilde{t}_{l} }_{s} \frac{\check{\nu} ( V_{\mathbf{cl}} ( \tilde{t%
}_{l} - \frac{ \tilde{t}_{l} -\tau}{\e} ; \tilde{t}_{l} , x_{l}, v_{l}))}{\e%
^{2}} \mathrm{d} \tau }\tilde{w}(v_{l})\mathrm{d}\sigma _{l}\}\Pi
_{j=1}^{l-1}\{e^{ - \int^{ \tilde{t}_{j} }_{ \tilde{t}_{j+1} } \frac{ \tilde{%
\nu} ( V_{\mathbf{cl}}( \tilde{t}_{j} - \frac{ \tilde{t}_{j} -\tau}{\e} ; 
\tilde{t}_{j} , x_{j}, v_{j}))}{\e^{2}} \mathrm{d} \tau }\mathrm{d}\sigma
_{j}\}.
\end{equation*}
Recall (\ref{w/w_unsteady}). With the choice of $k= C_{1} T_{0}^{5/4}$ (clearly $0 \leq t \leq \e^{2} T
\ll \e T_{0}$), for $t \in [0, \e^{2} T]$, 
\begin{eqnarray*}
&&| \e h^{\ell+1} (t,x,v)| \\
&\lesssim& C_{1} T_{0}^{5/4}\Big\{ e^{- CT}  \| %
\e h^{\ell+1} (0) \|_\infty + T  \sup_{0 \leq s\leq t}
\| \e h^{ \ell} (s) \|_{\infty} + \e^{2} \| A \|_{\infty} + O(\e^{2}) \\
&& \ \ \ \ \ \ \ \ \ \ \ \ + \underbrace{\int_{0}^{t} \frac{ \check{\nu} (
V_{\mathbf{cl}}(t- \frac{t-s}{\e} ; t,x,v) )}{\e^{2}} e^{- \int^{t}_{s} 
\frac{\check{\nu} ( V_{\mathbf{cl}} ( t- \frac{t-\tau}{\e};t,x,v) )}{\e^{2}} 
\mathrm{d} \tau } \mathrm{d} s}_{\lesssim 1} \times  \sup_{0 \leq s\leq t} \| \e h^{\ell} (s) \|_{\infty}^{2} \Big\} \\
&&+ T_{0}^{5/4} \Big\{ \frac{1}{2}\Big\}^{C_{2} T_{0}^{5/4}}   \sup_{0 \leq s\leq t} \| \e h^{\ell+1} (s) \|_{\infty}.
\end{eqnarray*}
For $T_{0} \gg 1, \ 0<T \ll 1$, and $0< \e \ll 1$, we then established (\ref%
{claim_pp_bounded}).
 
Using (\ref%
{claim_pp_bounded}) with a small initial datum and an induction in $\ell$, we deduce
\begin{eqnarray*}
 \sup_{0 \leq t \leq \e^{2} T}\| \e h^{\ell+ 1 } (t) \|_{\infty} 
\lesssim
    \| \e h_{0} \|_{\infty}+ \e^{2}\|
A\|_{\infty}+ O(\e^{2})  ,
\end{eqnarray*}
for $C_{1} \geq 10 C_{T_{0}}$.

\vspace{4pt}

\noindent\textit{Step 2. } From \textit{Step 2}, $w f^{\ell} \rightarrow wf$
weak-$*$ in $L^{\infty}([0, \e^{2} T] \times \Omega \times \mathbb{R}^{3})$
up to subsequence. Clearly $f$ satisfies the bound (\ref{claim_pp_bounded}).
On the other hand, applyig the argument of previous step to $f^{\ell+1}-f^{\ell}$, from (\ref{claim_pp_bounded}) we can prove that $wf^{\ell}$ is a Cauchy
sequence in $L^{\infty} ([0, \e^{2} T] \times \Omega \times \mathbb{R}^{3})$%
. It is standard to show that $f$ solves (\ref{eq_R_pp}) and (\ref{R_bdry_pp}%
) with $f^{\ell+1}=f=f^{\ell}$. Therefore $F= \mu + \e \f  \sqrt{\mu}$
solves the Boltzmann equation with diffuse BC. Since the unique solution $f$
has a uniform-in-time bound from Proposition \ref{energy_nonlinear}, we can
continue the \textit{Step 2} for $[\e^{2} T, 2 \e^{2} T],[2\e^{2} T,3 \e^{2}
T] , \cdots$, to conclude $w r^{\ell} \rightarrow wr$ in $L^{\infty}(\mathbb{%
R}_{+} \times \Omega \times \mathbb{R}^{3})$. Therefore $F^{\ell}
\rightarrow F\geq 0$ a.e.

\vspace{4pt}

\noindent\textit{Step 3. } Let, for sufficiently large $m$, 
\begin{equation*}
F_0(x,v)=\mu+ \sqrt{\mu}\e (f_{w} + f_{s}+\tilde f(0))+\mu^{ {\frac 12+{\beta^-}}}%
\mathbf{1}_{|v|>m| \log \e|}.
\end{equation*}
Clearly, by the $L^\infty$ estimate of $f_s$ we have $F(0)\ge 0$. Moreover, $%
F_0$ satisfies the assumptions of Theorem 1.2. By Theorem 1.2, we have $%
\|F(t) - F_{s}\|_{L^{2}} \lesssim e^{-\lambda t}$. Then, as $t \rightarrow
\infty$, for any non negative test function $\psi(x,v)$, 
\begin{eqnarray*}
&&\iint_{\Omega \times \mathbb{R}^{3}} F_{s} (x,v) \psi(x,v) \mathrm{d} x 
\mathrm{d} v \\
&=& \iint_{\Omega \times \mathbb{R}^{3}} F(t,x,v) \psi(x,v)\mathrm{d} x 
\mathrm{d} v + O(1) \iint_{\Omega \times \mathbb{R}^{3}}|F_{s}(x,v) -
F(t,x,v)|\psi(x,v) \mathrm{d} x \mathrm{d} v \\
&\geq & 0 - O(1) \| F(t) - F_{s} \|_{L^{2}(\Omega \times \mathbb{R}^{3})} \\
&\geq & 0.
\end{eqnarray*}
This proves $F_{s} (x,v) \geq 0$ a.e.
\end{proof}

\subsection{Local-in-Time Validity}

\begin{proof}[Proof of Theorem \protect\ref{remb}]
We fix a time interval $[0,T]$ and a $L^\infty_tC^2_{x,v}$ solution to the
INSF system in this interval with Dirichlet boundary conditions on $\pt\O $.
Under the assumptions of Theorem \ref{remb}, are well defined the functions \begin{eqnarray}  \label{f1f2}
&&f_1=\sqrt{\mu}[\rho+u\cdot v+\th \frac{|v|^2-3}2],  \notag \\
\\
&&f_2= \frac 1 2\sum_{i,j=1}^3 \mathcal{A}_{ij}[\pt_{x_i}u_{j,s}+\pt _{x_j}
u_{i,s}]+\sum_{i=1}^3 \mathcal{B}_i \pt_{x_i}\th -L^{-1}[\Gamma(f_1,f_1)] +%
\frac{|v|^2-3}{2} \th _2\sqrt{\mu},  \notag
\end{eqnarray}
with $\th _2= p-\fint p- \th \rho$. We write the solution to the Boltzmann
equation as 
\begin{equation*}
F=\mu+\e\sqrt{\mu}[f_1+\e^2 f_2+\e^{\frac 1 2} R],
\end{equation*}
so that $R$ satisfies the equation 
\begin{equation}  \label{tild_RRR}
\begin{split}
& \partial_{t} {R} + \e^{-1} v\cdot\nabla_{x} {R} + \e \Phi \cdot \nabla_{v} 
{R} + \e^{-2} L {R} \\
= & \ \e^{-1} L_{f_1+\e f_2} {R} + \e^{-1/2} \Gamma({R}, {R}) + \e \frac{%
\Phi \cdot v}{2}{R}+ \e^{-1/2} {A},
\end{split}%
\end{equation}
where 
\begin{equation}  \label{Astaaa}
\begin{split}
A\ &= -(\mathbf{I}-\mathbf{P}) [v\cdot \nabla_x (f_2 +\tilde f_2)]-2\Gamma(
f_1,f_2) \\
&-\e\big\{\pt_t f_2+\Phi\cdot\frac 1 {\sqrt{\mu}}\nabla_v \big[\sqrt{\mu}%
(f_1+\e f_2)) \big]-\Gamma(f_2,f_2)\big\},
\end{split}%
\end{equation}
and the boundary condition 
\begin{equation}  \label{bdry_tildeRRR}
\begin{split}
{R} |_{\gamma_{-} } = P_{\gamma} {R} +\e \mathcal{Q} {R} + \e^{1/2}{r},
\end{split}%
\end{equation}
where ${r}:= \e^{-1} [\mu^{- \frac{1}{2}} \mathcal{P}^w_{\gamma} ({f}_{1} 
\sqrt{\mu}) - \tilde{f}_{1} ] + [ \mu^{- \frac{1}{2}} \mathcal{P}%
_{\gamma}^{w} ({f}_{2} \sqrt{\mu}) - {f}_{2} ]$.

We use the iteration scheme 
\begin{equation}
\begin{split}  \label{Z_eqR}
& \partial_{t} {R}^{\ell+1} + \e^{-1} v\cdot\nabla_{x} {R}^{\ell+1} + \e %
\Phi \cdot \nabla_{v} {R}^{\ell+1} + \e^{-2} L {R}^{\ell+1} \\
&= \e^{-1} L_{f_1+\e f_2} {R}^{\ell+1} +\e^{-1/2} \Gamma( {R}^{\ell} , {R}%
^{\ell} ) + \e \frac{\Phi \cdot v}{2} {R}^{\ell+1}+ \e^{-1/2} {A}, \\
& {R}^{\ell +1} |_{\gamma_{-}} =P_{\gamma} {R}^{\ell+1 } + \e \mathcal{Q} {R}%
^{\ell} + \e^{1/2}{r},\ \ \ {R}^{\ell+1} |_{t=0}= {R}_{0}.
\end{split}%
\end{equation}
Here we set ${R}^{0}(t,x,v) := {R}_{0} (x,v)$.

Clearly ${R}^{\ell}_{t}:= \partial_{t} \tilde{R}^{\ell}$ solves 
\begin{equation}  \label{Zt_eqRR}
\begin{split}
& \partial_{t} {R}^{\ell+1}_{t} + \e^{-1} v\cdot\nabla_{x} {R}^{\ell+1} _{t}
+ \e \Phi \cdot \nabla_{v} {R}^{\ell+1}_{t} + \e^{-2} L {R}^{\ell+1}_{t} \\
&= \e^{-1} L_{ f_1+\e f_2}{R}^{\ell+1} _{t} + \e^{-1} L_{\partial_{t} {f}%
_{1} + \e \partial_{t} {f}_{2}} {R}^{\ell+1} + \e^{-1/2} [\Gamma({R}%
^{\ell}_{t} , {R}^{\ell} )+ \Gamma( {R}^{\ell}, {R}^{\ell}_{t} )] \\
& \ \ \ \ \ \ + \e \frac{\Phi \cdot v}{2} {R}^{\ell+1}_{t}+ \e^{-1/2}
\partial_{t}{A}, \\
&{R}^{\ell +1}_{t}(t,x,v) |_{\gamma_{-}} = P_{\gamma} {R}_{t}^{\ell+1 } + \e 
\mathcal{Q} {R}^{\ell}_{t}+ \e^{1/2}\partial_{t}{r}, \ \ {R}_{t}^{\ell+1}
|_{t=0} = \partial_{t} {R}_{0}.
\end{split}%
\end{equation}
We use 
\begin{eqnarray*}
\| \langle v\rangle^{-1/2} L_{f_{1} + \e f_{2}} {R}^{\ell}
\|_{L^{2}_{t}([0,t]) L^{2}_{x,v}} &\lesssim& \| w[ f_{1} + \e f_{2} ]
\|_{L^{\infty}_{t,x,v}} \| {R}^{\ell} \|_{L^{2}_{t}([0,t]) L^{2}_{x,v}} \\
\| \langle v\rangle^{-1/2} L_{f_{1} + \e f_{2}} {R}^{\ell}_{t}
\|_{L^{2}_{t}([0,t]) L^{2}_{x,v}} &\lesssim& \| w[ f_{1} + \e f_{2} ]
\|_{L^{\infty}_{t,x,v}} \| {R}_t^{\ell} \|_{L^{2}_{t}([0,t]) L^{2}_{x,v}}, \\
\| \langle v\rangle^{-1/2} L_{\partial_{t} {f}_{1} + \e \partial_{t}{f}_{2}} 
{R}^{\ell} \|_{L^{2}_{t}([0,t]) L^{2}_{x,v}} &\lesssim& \| w[ \partial_{t} {f%
}_{1} + \e \partial_{t} {f}_{2} ] \|_{L^{\infty}_{t,x,v}} \| {R}^{\ell}
\|_{L^{2}_{t}([0,t]) L^{2}_{x,v}}.
\end{eqnarray*}
In all the estimates the time interval is restricted to $[0,T]$. Let 
\[
\mathcal{K}= \max\{\| w[ f_{1} + \e f_{2} ] \|_{L^{\infty}_{t,x,v}}, \| w[
\partial_{t} {f}_{1} + \e \partial_{t} {f}_{2} ] \|_{L^{\infty}_{t,x,v}}
\}.\] We use Corollary \ref{corolla} and Lemma \ref{nonlinear} and repeat the
Step 1 in Subsection \ref{3.7}. Step 2 is replaced by the inequality 
\begin{multline}
\|R^{\ell+1}(t)\|_{L^2_{x,v}}^2+\e^{-1}
\|(1-P_\g)R^{\ell+1}\|_{L^2_tL^2(\gamma)}^2+\e^{-2} \int_0^t \|(\mathbf{I}-%
\mathbf{P}) R^{\ell+1}\|_\nu^2 \\
\lesssim\|R_0\|_{L^2_{x,v}}^2+\mathcal{K}\| R^{\ell+1}\|_{L^2_{t,x,v}} + \e%
^{\frac 1 2}\|\Gamma(R^\ell,R^\ell)\|^2_{L^2_{t,x,v}} +\e^{\frac 1
2}\|A\|_{L^2_{t,x,v}}^2+\|r\|^2_{L^2_tL^2_\gamma}.
\end{multline}
Using Gronwall's inequality, we thus obtain for $t\in [0,T]$,
\begin{equation}
\|R^{\ell+1}(t)\|_{L^2_{x,v}}^2\lesssim Te^{\mathcal{K}t}\e^{\frac 1
2}\|\Gamma(R^\ell,R^\ell)\|^2_{L^\infty_{t}L^2_{x,v}} + e^{\mathcal{K}t}%
\Big\{T\big[\e^{\frac 1
2}\|A\|_{L^\infty_{t}L^2_{x,v}}^2+\|r\|^2_{L^\infty_tL^2_\gamma}\big]%
+\|R_0\|_{L^2_{x,v}}^2\Big\}.
\end{equation}

A similar estimate holds for $R_t$. Steps 3 and 4 are the same as in
Subsection \ref{3.7}, but for the fact that the non linear term is
multiplied by an extra factor $\e^{\frac 1 2}$. We conclude that, for any
fixed $T$ there is an $\e(T)$ such that $Te^{\mathcal{K}T}\e^{\frac 1 2}\ll
1 $ and the inductive hypothesis holds for $R^{\ell+1}$ for any $\e\le \e(T)$%
.

We repeat this process with $\tilde{R}^{\ell+1} - \tilde{R}^{\ell}$ to show
that $\tilde{R}^{\ell}$ is a Cauchy sequence in $L^{2} \cap L^{\infty}$.
Then it is standard to conclude the existence and uniqueness.
\end{proof}

\appendix 

\section{Extensions and Compactness}

\numberwithin{equation}{subsection} \numberwithin{theorem}{section} %
\setcounter{equation}{0}

\subsection{Extension}

\begin{proof}[Proof of Lemma \protect\ref{extension_dyn}]
\label{A1} \ 

\noindent\textit{Step 1. } In the sense of distributions on $[0,
\infty)\times \Omega \times \mathbb{R}^{3}$, 
\begin{equation}  \label{eq_Z_dyn}
\begin{split}
&\e \partial_{t} f_{\delta}+ v\cdot \nabla_{x} f_{\delta} + \e^{2} \Phi
\cdot \nabla_{v} f_{\delta} \\
& \ = \Big[1-\chi(\frac{n(x) \cdot v}{\delta}) \chi \big( \frac{ \xi(x) }{%
\delta}\big)\Big] \chi(\delta|v|) \\
& \ \ \ \ \ \ \ \ \times \big[ \mathbf{1}_{t \in [0,\infty)} g + \mathbf{1}%
_{t \in ( - \infty, 0 ]} \chi (t) \{ \e \frac{\chi^{\prime} (t)}{ \chi(t)} +
v \cdot \nabla_{x} + \e^{2} \Phi \cdot \nabla_{v} \} f_{0} (x,v)\big] \\
& \ \ \ + \big[\mathbf{1}_{t \in [0,\infty)} f + \mathbf{1}_{t \in (-
\infty, 0 ]} \chi(t) f_{0} (x,v) \big] \{v\cdot \nabla_{x} + \e^{2} \Phi
\cdot \nabla_{v}\} \Big( [1-\chi(\frac{n(x) \cdot v}{\delta}) \chi \big( 
\frac{\xi(x)}{\delta}\big) ] \chi(\delta|v|)\Big).
\end{split}%
\end{equation}
Note that, 
\begin{eqnarray}
&&\Big|\{v\cdot \nabla_{x} + \e^{2} \Phi \cdot \nabla_{v}\} [1-\chi(\frac{%
n(x) \cdot v}{\delta}) \chi \big( \frac{ \xi(x )}{\delta}\big) ]
\chi(\delta|v|)\Big|  \label{der_chi} \\
&=&\Big| - \frac{1}{\delta} \{v\cdot \nabla_{x} n(x) \cdot v + \e^{2} \Phi
\cdot n(x) \} \chi^{\prime} \big(\frac{n(x) \cdot v}{\delta} \big) \chi %
\big( \frac{ \xi(x )}{\delta}\big) \chi(\delta|v|)  \notag \\
&& - \ \frac{1}{\delta} v\cdot \nabla_{x} \xi(x) \chi^{\prime} \big( \frac{%
\xi(x)}{\delta}\big) \chi (\frac{n(x) \cdot v}{\delta}) \chi(\delta|v|) + \e%
^{2}\delta \Phi \cdot \frac{v}{|v|} \chi^{\prime} (\delta|v|)[1-\chi(\frac{%
n(x) \cdot v}{\delta}) \chi \big( \frac{\xi(x)}{\delta}\big) ] \Big|  \notag
\\
&\leq& \frac{4}{\delta}( |v|^{2}\|\xi\|_{C^2} + \e^{2}\|\Phi\|_\infty )
\chi(\delta|v|) + \frac{C_{\Omega}}{\delta} |v|\chi(\delta|v|) + \e^{2}
\delta\|\Phi\|_\infty \mathbf{1}_{|v| \leq {2}{\delta}^{-1}}  \notag \\
&\lesssim & {\delta^{-3}} \mathbf{1}_{|v| \leq 2 \delta^{-1}}.  \notag
\end{eqnarray}

This proves the second line of (\ref{force_Z_dyn}). Since $\big[1-\chi(\frac{%
n(x) \cdot v}{\delta}) \chi ( \frac{\xi(x)}{\delta}) \big] \chi(\delta|v|)
\leq 1$, we prove the first line of (\ref{force_Z_dyn}) directly. 

\vspace{4pt}

\noindent\textit{Step 2. } We claim that if $0 \leq \xi (x) \leq \tilde{C}
\delta^{4}, \ |n(x) \cdot v|> \delta$ and $|v| \leq \frac{1}{\delta}$ then
either $\xi(\tilde{x}_{\mathbf{f}} (x,v))=\tilde{C}\delta^{4}$ or $\xi(%
\tilde{x}_{\mathbf{b}}(x,v)) =\tilde{C}\delta^{4}$.

To show this, if $v \cdot n(x) \geq \delta$, we take $s>0$, while if $v\cdot
n(x) \leq - \delta$ then we take $s<0$. From (\ref{intfor}), 
\begin{eqnarray}
&&\xi(X(s; 0, x, v) ) = \xi(x) + \int_{0}^{s} V( \tau; 0, x, v) \cdot
\nabla_{x} \xi( X( \tau; 0, x, v) ) \mathrm{d}\tau  \notag \\
&=&\xi(x) + \int^{s}_{0} \{ v+ O(1)\e^{2} \|\Phi\|_{\infty} \tau \} \cdot \{
\nabla_{x} \xi(x) + O(1)\|\xi\|_{C^{2}} (|v| + \e^{2} \|\Phi\|_{\infty}
\tau) \tau \}  \notag \\
&=&\xi(x) + v \cdot \nabla_{x} \xi(x) s +O(1) \|\xi\|_{C^{2}} \big\{ |v|^{2}
s^{2} + \e^{2} \|\Phi \|_{\infty} s ^{2} + \e^{2} \|\Phi\|_{\infty}|v|s^{3}
+ \e^{4} \| \Phi \|_{\infty}^{2} s^{4} \big\}.  \notag
\end{eqnarray}
From $\xi(x)\ge 0$, 
\begin{eqnarray}
\xi(X(s; 0, x, v) ) &\geq & \delta |s| \Big\{ 1 - \frac{ \|\xi\|_{C^{2}} }{%
\delta}\Big[ |v|^{2} |s| + \e^{2} \|\Phi \|_{\infty} |s| + \e^{2}
\|\Phi\|_{\infty}|v\|s|^{2} + \e^{4} \| \Phi \|_{\infty}^{2} |s|^{3} \Big] %
\Big\}  \notag \\
&\ge & \delta |s| \Big\{ 1 - \frac{ \|\xi\|_{C^{2}} }{\delta}\Big[\frac{1}{%
\delta^{2}} |s| + \e^{2} \|\Phi \|_{\infty} | s| + \e^{2} \|\Phi\|_{\infty} 
\frac{1}{\delta}|s|^{2} + \e^{4} \| \Phi \|_{\infty}^{2} |s|^{3} \Big] \Big\}
\notag \\
&\geq& \delta |s| \Big\{ 1- \Big[ \frac{1}{4} + \frac{ \e^{2} \delta^{2} \|
\Phi \|_{\infty}}{4 } + \frac{\e^{2} \delta^{4} \| \Phi \|_{\infty}}{16} + 
\frac{\e^{4}\delta^{8} \| \Phi \|_{\infty}^{2}}{64} \Big] \Big\} \ \geq \ 
\frac{\delta |s| }{2},  \label{xi_est}
\end{eqnarray}
for $0 \leq |s| \leq \frac{\delta^{3}}{ 4(1+ \|\xi \|_{C^{2}})}$ and $0< \e %
\ll 1$. Therefore 
\begin{eqnarray*}
\xi(Y(s;0,x,v)) = \xi(X(\frac{s}{\e}; 0,x,v )) \geq \delta \frac{|s|}{2\e} ,
\end{eqnarray*}
for all $0\leq |s| \leq \frac{ \e\delta^{3}}{ 4(1+ \|\xi \|_{C^{2}})}$ with $%
0< \varepsilon \ll1$. Especially with $\e s_{*} = +\frac{ \varepsilon%
\delta^{3}}{4(1+ \| \xi \|_{C^{2}})}$ for $n(x) \cdot v >\delta$ and $\e %
s_{*} = -\frac{\varepsilon\delta^{3} }{4(1+ \| \xi \|_{C^{2}})}$ for $n(x)
\cdot v <\delta$, 
\begin{equation*}
\xi (Y(s_{*};0,x,v)) > \tilde{C} \delta^{4}.
\end{equation*}
Therefore, by the intermediate value theorem, we prove our claim.

\vspace{10pt}

\noindent\textit{Step 3. } We define $f_{E}(t,x,v)$ for $(x,v) \in [ \mathbb{%
R}^{3} \backslash \bar{\Omega}] \times \mathbb{R}^{3}$: 
\begin{equation}  \label{ZE_dyn}
\begin{split}
f_{E}(t,x,v) & \ := \ f_{\delta}(t- \e t_{\mathbf{b}}^{*}(x,v),x_{\mathbf{b}%
}^{*}(x,v), v_{\mathbf{b}}^{*}(x,v)) \chi\big( \frac{\xi(x)}{\tilde{C}
\delta^{4}} \big) \chi \big( {t_{\mathbf{b}}^{*} (x,v)} \big) \ \ \ \text{if}
\ \ x_{\mathbf{b}}^{*}(x,v) \in \partial\Omega, \\
& \ := \ f_{\delta}(t+ \e t_{\mathbf{f}}^{*}(x,v),x_{\mathbf{f}}^{*}(x,v),
v_{\mathbf{f}}^{*}(x,v)) \chi\big( \frac{\xi(x)}{\tilde{C} \delta^{4}} \big) %
\chi \big( {t_{\mathbf{f}}^{*} (x,v)} \big) \ \ \ \ \text{if} \ \ x_{\mathbf{%
f}}^{*}(x,v) \in \partial\Omega, \\
& \ := \ 0 \ \ \ \ \ \ \ \ \ \ \ \ \ \ \ \ \ \ \ \ \ \ \ \ \ \ \ \ \ \ \ 
\text{if} \ \ x_{\mathbf{b}}^{*}(x,v) \notin \partial\Omega \ \text{and} \
x_{\mathbf{f}}^{*}(x,v) \notin \partial\Omega.
\end{split}%
\end{equation}

We check that $f_{E}$ is well-defined. It suffices to prove the following: 
\begin{equation*}
\begin{split}
& \text{If} \ \ x_{\mathbf{b}}^{*}(x,v) \in \partial\Omega \text{ and } x_{%
\mathbf{f}}^{*}(x,v) \in \partial\Omega \\
& \ \text{then} \ \ f_{\delta}( t- \e t_{\mathbf{b}}^{*} (x,v), x_{\mathbf{b}%
}^{*}(x,v), v_{\mathbf{b}}^{*}(x,v)) \chi\big( \frac{\xi(x)}{\tilde{C}
\delta^{4}} \big ) = 0 = f_{\delta}( t+ \e t_{\mathbf{f}}^{*} (x,v), x_{%
\mathbf{f}}^{*}(x,v), v_{\mathbf{f}}^{*}(x,v)) \chi\big( \frac{\xi(x)}{%
\tilde{C} \delta^{4}} \big) .
\end{split}%
\end{equation*}
If $|n(x_{\mathbf{b}}^{*}(x,v)) \cdot v_{\mathbf{b}}^{*}(x,v)| \leq \delta$
or $|v_{\mathbf{b}}^{*}(x,v)| \geq \frac{1}{\delta}$ then $f_{\delta}( t- \e %
t^{*}_{\mathbf{b}}(x,v), x_{\mathbf{b}}^{*}(x,v), v_{\mathbf{b}}^{*}(x,v))=0$
due to (\ref{Z_support_dyn}). If $n(x_{\mathbf{b}}^{*}(x,v)) \cdot v_{%
\mathbf{b}}^{*}(x,v) > \delta$ and $|v_{\mathbf{b}}^{*}(x,v)|\leq \frac{1}{%
\delta}$ then, due to \textit{Step 2}, $\xi(x_{\mathbf{f}}^{*}(x,v)) = \xi(
x_{\mathbf{f}}^{*}( x_{\mathbf{b}}^{*}(x,v), v_{\mathbf{b}}^{*}(x,v) ) ) = 
\tilde{C}\delta^{4}$ so that $x_{\mathbf{f}}^{*}(x,v) \notin \partial\Omega$.

On the other hand, if $|n(x_{\mathbf{f}}^{*}(x,v)) \cdot v_{\mathbf{f}%
}^{*}(x,v)| \leq \delta$ or $|v_{\mathbf{f}}^{*}(x,v)|\geq \frac{1}{\delta}$
then $f_{\delta}( t+ \e t^{*}_{\mathbf{f}}(x,v) ,x_{\mathbf{f}}^{*}(x,v), v_{%
\mathbf{f}}^{*}(x,v))=0$ due to (\ref{Z_support_dyn}). If $n(x_{\mathbf{f}%
}^{*}(x,v)) \cdot v_{\mathbf{f}}^{*}(x,v) <-\delta$ and $|v_{\mathbf{f}%
}^{*}(x,v)|\leq \frac{1}{\delta}$ then, due to \textit{Step 2}, $\xi(x_{%
\mathbf{b}}^{*}(x,v)) = \xi( x_{\mathbf{b}}^{*}( x_{\mathbf{f}}^{*}(x,v), v_{%
\mathbf{f}}^{*}(x,v) ) ) = \tilde{C}\delta^{4}$ so that $x_{\mathbf{b}%
}^{*}(x,v) \notin \partial\Omega$.

Note that 
\begin{equation}  \label{fE=fd_dyn}
f_{E}(t,x,v) = f_{\delta} (t,x,v) \ \ \ \text{for all } x \in \partial\Omega.
\end{equation}
If $x \in \partial\Omega$ and $n(x) \cdot v >\delta$ then $(x_{\mathbf{b}%
}^{*}(x,v), v_{\mathbf{b}}^{*}(x,v)) = (x ,v)$. From the definition (\ref%
{ZE_dyn}), for those $(x,v)$, we have $f_{E}(t,x,v) = f_{\delta} (t,x,v)$.
If $x \in \partial\Omega$ and $n(x) \cdot v< -\delta$ then $(x_{\mathbf{f}%
}^{*}(x,v), v_{\mathbf{f}}^{*}(x.v)) = (x,v)$. From the definition (\ref%
{ZE_dyn}), we conclude (\ref{fE=fd_dyn}) again. Otherwise, if $-\delta< n(x)
\cdot v< \delta$ then $f_{E}|_{\partial\Omega} \equiv 0 \equiv
f_{\delta}|_{\partial\Omega}$.

\vspace{10pt}

\noindent\textit{Step 4. } We claim that $f_{E}(x,v) \in L^{2}([\mathbb{R}%
^{3} \backslash \bar{\Omega}] \times \mathbb{R}^{3})$.

From the definition of (\ref{ZE_dyn}), we have $f_{E}(x,v) \equiv 0$ if $x_{%
\mathbf{b}}^{*} (x,v) \notin \partial\Omega$ and $x_{\mathbf{f}}^{*} (x,v)
\notin \partial\Omega$. Therefore, from (\ref{ZE_dyn}), 
\begin{eqnarray}
&&\int_{-\infty}^{\infty} \iint_{[ \mathbb{R}^{3 } \backslash \Omega ]\times 
\mathbb{R}^{3}} |f_{E} (t,x,v)|^{2} \mathrm{d} x \mathrm{d} v \mathrm{d} t 
\notag \\
&=&\int_{-\infty}^{\infty} \iint_{ [ \mathbb{R}^{3 } \backslash \Omega ]
\times \mathbb{R}^{3}} \mathbf{1}_{x_{\mathbf{b}}^{*} (x,v) \in
\partial\Omega} |f_{E} |^{2} + \int_{-\infty}^{\infty} \iint _{ [ \mathbb{R}%
^{3 } \backslash \Omega ] \times \mathbb{R}^{3}} \mathbf{1}_{x_{\mathbf{f}%
}^{*} \in \partial\Omega} |f_{E} |^{2}  \notag \\
&=& \int_{-\infty}^{\infty} \iint_{ [ \mathbb{R}^{3 } \backslash \Omega
]\times \mathbb{R}^{3}} \mathbf{1}_{x_{\mathbf{b}}^{*} (x,v) \in
\partial\Omega} |f_{\delta} ( t- \e t^{*}_{\mathbf{b}} , x_{\mathbf{b}}^{*}
, v_{\mathbf{b}}^{*} )|^{2} \big| \chi \big( \frac{\xi(x)}{\tilde{C}
\delta^{4}} \big) \big|^{2} |\chi ( {t_{\mathbf{b}}^{*} } ) |^{2}\mathrm{d}
x \mathrm{d} v \mathrm{d} t  \label{fE1_dyn} \\
&&+ \int_{-\infty}^{\infty} \iint_{ [ \mathbb{R}^{3 } \backslash \Omega ]
\times \mathbb{R}^{3}} \mathbf{1}_{x_{\mathbf{f}}^{*} (x,v) \in
\partial\Omega} |f_{\delta} ( t+ \e t^{*}_{\mathbf{f}} , x_{\mathbf{f}}^{*}
, v_{\mathbf{f}}^{*} )|^{2} \big| \chi \big( \frac{\xi(x)}{\tilde{C}
\delta^{4}} \big) \big|^{2} |\chi ( {t_{\mathbf{f}}^{*} } ) |^{2} \mathrm{d}
x \mathrm{d} v \mathrm{d} t ,  \label{fE2_dyn}
\end{eqnarray}
where $(t^{*}_{\mathbf{b}},x^{*}_{\mathbf{b}},v^{*}_{\mathbf{b}} )$ and $%
(t^{*}_{\mathbf{f}},x^{*}_{\mathbf{f}},v^{*}_{\mathbf{f}} )$ are evaluated
at $(x,v)$.

By (\ref{bdry_int2}), 
\begin{eqnarray*}
(\ref{fE1_dyn}) &\leq& \int^{\infty}_{- \infty} \mathrm{d} t
\int_{\partial\Omega} \int_{n(x) \cdot v >0} \int_{0}^{ \min\{ t_{\mathbf{f}%
}^{*}(x,v), 1 \}}\mathrm{d} S_{x} \mathrm{d} v \mathrm{d} s \{|n(x) \cdot v|
+ O(\e)(1+ |v|)s \} \\
\\
&& \ \ \ \ \ \times \big| f_{\delta} \big( t- \e s , x_{\mathbf{b}}^{*}
(X(s;0,x,v), V(s;0,x,v)), v_{\mathbf{b}}^{*} (X(s;0,x,v), V(s;0,x,v)) \big) %
\big|^{2} \\
&\leq&\int^{\infty}_{-\infty}\mathrm{d} t \int_{\partial\Omega} \int_{n(x)
\cdot v >0} \int_{0}^{ 1} \big| f_{\delta} (t,x,v ) \big|^{2} \{|n(x) \cdot
v| + O(\e)(1+ |v|)s \} \mathrm{d} s \mathrm{d} v \mathrm{d} S_{x} \\
&\lesssim& \int_{-\infty}^{\infty} \mathrm{d} t \int_{\partial\Omega}
\int_{n(x) \cdot v >0} \big| f_{\delta} (t,x,v ) \big|^{2} |n(x) \cdot v| 
\mathrm{d} v \mathrm{d} S_{x} \lesssim \| f_{\delta} \|_{L^{2} (\mathbb{R}
\times \partial\Omega \times \mathbb{R}^{3})}^{2},
\end{eqnarray*}
where we have used the fact, from (\ref{Z_dyn}), $O(\e)(1+|v|) |s| \leq O(\e%
) (1+ \frac{1}{\delta})\lesssim \delta \lesssim |n(x) \cdot v|$ for $(x,v)
\in \text{supp} ( f_{\delta} )$, and, for $n(x) \cdot v>0$, $%
x\in\partial\Omega$, and $0 \leq s \leq t_{\mathbf{f}}^{*}(x,v)$, 
\begin{equation*}
(x_{\mathbf{b}}^{*} (X(s;0,x,v), V(s;0,x,v)), v_{\mathbf{b}}^{*}
(X(s;0,x,v), V(s;0,x,v)) ) = (x_{\mathbf{b}}^{*} (x,v), v_{\mathbf{b}}^{*}
(x,v) ) = (x,v),
\end{equation*}
and $t_{\mathbf{b}}^{*} (X(s;0,x,v), V(s;0,x,v)) =s$ and the change of
variables $t- \e s \mapsto t$. Similarly we can show $(\ref{fE2_dyn})
\lesssim \| f_{\delta} \|_{L^{2} (\mathbb{R} \times \partial\Omega \times 
\mathbb{R}^{3})}^{2}$.

\vspace{10pt}

\noindent\textit{Step 5. } We claim that, in the sense of distributions on $%
\mathbb{R} \times [\Omega_{\tilde{C} \delta^{4}}\backslash \bar{\Omega}]
\times \mathbb{R}^{3}$, 
\begin{eqnarray}
&& \e \partial_{t } f_{E}+ v\cdot \nabla_{x} f_{E} + \e^{2} \Phi \cdot
\nabla_{v} f_{E}  \label{eq_ZE_dyn} \\
& = & \ \frac{1}{\tilde{C} \delta^{4}}v \cdot \nabla_{x} \xi(x)
\chi^{\prime} \big( \frac{\xi(x)}{\tilde{C} \delta^{4}} \big) \Big[ %
f_{\delta}( t- \e t_{\mathbf{b}}^{*} (x,v) , x_{\mathbf{b}}^{*}(x,v), v_{%
\mathbf{b}}^{*}(x,v) ) \chi(t_{\mathbf{b}}^{*} (x,v))\mathbf{1}_{x_{\mathbf{b%
}}^{*}(x,v) \in \partial\Omega}  \notag \\
&& \ \ \ \ \ \ \ \ \ \ \ \ \ \ \ \ \ \ \ \ \ \ \ \ \ \ \ \ \ \ \ \ \ +
f_{\delta}( t+ \e t_{\mathbf{f}}^{*} (x,v) , x_{\mathbf{f}}^{*}(x,v), v_{%
\mathbf{f}}^{*}(x,v) ) \chi(t_{\mathbf{f}}^{*} (x,v)) \mathbf{1}_{x_{\mathbf{%
f}}^{*}(x,v) \in \partial\Omega} \Big]  \notag \\
&& + f_{\delta}( t- \e t_{\mathbf{b}}^{*}(x,v), x_{\mathbf{b}}^{*} (x,v), v_{%
\mathbf{b}}^{*}(x,v)) \chi \big( \frac{\xi(x)}{\tilde{C} \delta^{4}} \big) %
\chi^{\prime} (t_{\mathbf{b}}^{*}(x,v))\mathbf{1}_{x_{\mathbf{b}}^{*}(x,v)
\in \partial\Omega}  \notag \\
&& - f_{\delta}( t+ \e t_{\mathbf{f}}^{*} (x,v), x_{\mathbf{b}}^{*} (x,v),
v_{\mathbf{b}}^{*}(x,v)) \chi \big( \frac{\xi(x)}{\tilde{C} \delta^{4}} %
\big) \chi^{\prime} (t_{\mathbf{f}}^{*}(x,v))\mathbf{1}_{x_{\mathbf{f}%
}^{*}(x,v) \in \partial\Omega} .  \notag
\end{eqnarray}

Note that 
\begin{eqnarray*}
&&[ \e \partial_{t} + v\cdot\nabla_{x} + \e^{2} \Phi \cdot \nabla_{v} ]f(t- %
\e t_{\mathbf{b}}^{*} (x,v), x_{\mathbf{b}}^{*} (x,v) , v_{\mathbf{b}}^{*}
(x,v)) \\
& &= \underbrace{ [ \e \partial_{t} + v\cdot\nabla_{x} + \e^{2} \Phi \cdot
\nabla_{v} ] (t-\e t_{\mathbf{b}}^{*} (x,v)) } \times \partial_{t} f(t- \e %
t_{\mathbf{b}}^{*} (x,v), x_{\mathbf{b}}^{*} (x,v) , v_{\mathbf{b}}^{*}
(x,v)) \\
& & \ \ \ + [ v\cdot\nabla_{x} + \e^{2} \Phi \cdot \nabla_{v} ] f(s, x_{%
\mathbf{b}}^{*} (x,v) , v_{\mathbf{b}}^{*} (x,v)) |_{s= t- \e t_{\mathbf{b}%
}^{*} (x,v)}, \\
&&[ \e \partial_{t} + v\cdot\nabla_{x} + \e^{2} \Phi \cdot \nabla_{v} ]f(t+ %
\e t_{\mathbf{f}}^{*} (x,v), x_{\mathbf{f}}^{*} (x,v) , v_{\mathbf{f}}^{*}
(x,v)) \\
&&= \underbrace{ [ \e \partial_{t} + v\cdot\nabla_{x} + \e^{2} \Phi \cdot
\nabla_{v} ] (t+\e t_{\mathbf{f}}^{*} (x,v)) } \times \partial_{t} f(t+ \e %
t_{\mathbf{f}}^{*} (x,v), x_{\mathbf{f}}^{*} (x,v) , v_{\mathbf{f}}^{*}
(x,v)) \\
& & \ \ \ + [ v\cdot\nabla_{x} + \e^{2} \Phi \cdot \nabla_{v} ] f(s, x_{%
\mathbf{f}}^{*} (x,v) , v_{\mathbf{f}}^{*} (x,v)) |_{s= t + \e t_{\mathbf{f}%
}^{*} (x,v)}.
\end{eqnarray*}
The underbraced terms vanish 
because $[ v\cdot \nabla_{x} + \e^{2} \Phi \cdot \nabla_{v} ] (t- \e t_{%
\mathbf{b}}^{*} (x.v)) = \frac{d}{ds} \Big|_{s=0} (t- \e t_{\mathbf{b}}^{*}
(X(s;0,x,v) , V(s;0,x,v) )) = \frac{d}{ds} \Big|_{s=0} (t- \e s ) = - \e  $,
and $[ v\cdot \nabla_{x} + \e^{2} \Phi \cdot \nabla_{v} ] (t+ \e t_{\mathbf{f%
}}^{*} (x.v)) = \frac{d}{ds} \Big|_{s=0} (t+ \e t_{\mathbf{f}}^{*}
(X(s;0,x,v) , V(s;0,x,v) )) = \frac{d}{ds} \Big|_{s=0} (t- \e s + \e t_{%
\mathbf{f}}^{*} (x,v)) = - \e  $. Moreover, in the sense of distributions on 
$[\Omega_{\tilde{C} \delta^{4}}\backslash \bar{\Omega}] \times \mathbb{R}^{3}
$, 
\begin{eqnarray}
&&v\cdot \nabla_{x} f_{E} + \e^{2} \Phi \cdot \nabla_{v} f_{E}  \notag \\
& = & \ \frac{1}{\tilde{C} \delta^{4}}v \cdot \nabla_{x} \xi(x)
\chi^{\prime} \big( \frac{\xi(x)}{\tilde{C} \delta^{4}} \big) \Big[ %
f_{\delta}( x_{\mathbf{b}}^{*}(x,v), v_{\mathbf{b}}^{*}(x,v) ) \chi(t_{%
\mathbf{b}}^{*} (x,v))\mathbf{1}_{x_{\mathbf{b}}^{*}(x,v) \in \partial\Omega}
\notag \\
& & \ \ \ \ \ \ \ \ \ \ \ \ \ \ \ \ \ \ \ \ \ \ \ \ \ \ \ \ \ \ \ \ \ +
f_{\delta}( x_{\mathbf{f}}^{*}(x,v), v_{\mathbf{f}}^{*}(x,v) ) \chi(t_{%
\mathbf{f}}^{*} (x,v)) \mathbf{1}_{x_{\mathbf{f}}^{*}(x,v) \in
\partial\Omega} \Big]  \label{eq_ZE} \\
&& + f_{\delta}(x_{\mathbf{b}}^{*} (x,v), v_{\mathbf{b}}^{*}(x,v)) \chi %
\big( \frac{\xi(x)}{\tilde{C} \delta^{4}} \big) \chi^{\prime} (t_{\mathbf{b}%
}^{*}(x,v))\mathbf{1}_{x_{\mathbf{b}}^{*}(x,v) \in \partial\Omega}  \notag \\
&& - f_{\delta}(x_{\mathbf{f}}^{*} (x,v), v_{\mathbf{f}}^{*}(x,v)) \chi %
\big( \frac{\xi(x)}{\tilde{C} \delta^{4}} \big) \chi^{\prime} (t_{\mathbf{f}%
}^{*}(x,v))\mathbf{1}_{x_{\mathbf{f}}^{*}(x,v) \in \partial\Omega} .  \notag
\end{eqnarray}
For $\phi \in C_{c}^{\infty}( [\Omega_{\tilde{C} \delta^{4}}\backslash \bar{%
\Omega}] \times \mathbb{R}^{3} )$, we choose small $t>0$ such that $X(s; 0,
x,v) \in \Omega_{\tilde{C} \delta^{4}}\backslash \bar{\Omega}$ for all $|s|
\leq t$ and all $(x,v) \in \text{supp}(\phi)$. Then, from (\ref{ZE_dyn}),
for $(X(s),V(s))= (X(s;0,x,v), V(s;0,x,v))$, 
\begin{eqnarray*}
&&\frac{d}{ds}f_{E}(X(s), V(s)) \\
&=& \frac{d}{ds} \Big[ f_{\delta}( x_{\mathbf{b}}^{*}(X(s),V(s)), v_{\mathbf{%
b}}^{*} (X(s),V(s)) ) \chi (t_{\mathbf{b}}^{*} (X(s),V(s))) \mathbf{1}_{x_{%
\mathbf{b}}^{*} (X(s),V(s)) \in \partial\Omega} \\
&& \ \ \ \ + f_{\delta}( x_{\mathbf{f}}^{*} (X(s),V(s)), v_{\mathbf{f}}^{*}
(X(s),V(s)) ) \chi (t_{\mathbf{f}}^{*} (X(s),V(s))) \mathbf{1}_{x_{\mathbf{f}%
}^{*} (X(s),V(s)) \in \partial\Omega} \Big] \times \chi \big( \frac{\xi(X(s
))}{\tilde{C} \delta^{4}} \big) \\
&&+ \Big[ f_{\delta}( x_{\mathbf{b}}^{*}(X(s),V(s)), v_{\mathbf{b}}^{*}
(X(s),V(s)) ) \chi (t_{\mathbf{b}}^{*} (X(s),V(s))) \mathbf{1}_{x_{\mathbf{b}%
}^{*} (X(s),V(s)) \in \partial\Omega} \\
&& \ \ \ \ + f_{\delta}( x_{\mathbf{f}}^{*} (X(s),V(s)), v_{\mathbf{f}}^{*}
(X(s),V(s)) ) \chi (t_{\mathbf{f}}^{*} (X(s),V(s))) \mathbf{1}_{x_{\mathbf{f}%
}^{*} (X(s),V(s)) \in \partial\Omega} \Big] \times \frac{d}{ds}\chi \big( 
\frac{\xi(X(s ))}{\tilde{C} \delta^{4}} \big).
\end{eqnarray*}
From $(x_{\mathbf{b}}^{*}(X(s;0,x,v),V(s;0,x,v)), v_{\mathbf{b}%
}^{*}(X(s;0,x,v),V(s;0,x,v)))= (x_{\mathbf{b}}^{*}(x,v), v_{\mathbf{b}%
}^{*}(x,v))$ and \newline
$(x_{\mathbf{f}}^{*}(X(s;0,x,v),V(s;0,x,v)), v_{\mathbf{f}%
}^{*}(X(s;0,x,v),V(s;0,x,v)))= (x_{\mathbf{f}}^{*}(x,v), v_{\mathbf{f}%
}^{*}(x,v))$ and \newline
$t_{\mathbf{f}}^{*}(X(s;0,x,v),V(s;0,x,v) )=t_{\mathbf{f}}^{*}(x,v ) -s$ and 
$t_{\mathbf{b}}^{*}(X(s;0,x,v),V(s;0,x,v) )=t_{\mathbf{b}}^{*}(x,v )+s$, 
\begin{eqnarray}
&& \frac{d}{ds}f_{E}(X(s), V(s))  \notag \\
&=&\Big[ f_{\delta}( x_{\mathbf{b}}^{*}(x,v), v_{\mathbf{b}}^{*}(x,v) )
\chi^{\prime}( t_{\mathbf{b}}^{*} (X(s),V(s))) \mathbf{1}_{x_{\mathbf{b}%
}^{*}(x,v) \in \partial\Omega}  \notag \\
&& \ \ \ - f_{\delta}( x_{\mathbf{f}}^{*}(x,v), v_{\mathbf{f}}^{*}(x,v) )
\chi^{\prime}( t_{\mathbf{f}}^{*} (X(s),V(s))) \mathbf{1}_{x_{\mathbf{f}%
}^{*}(x,v) \in \partial\Omega} \Big] \chi \big( \frac{\xi(X(s ))}{\tilde{C}
\delta^{4}} \big)  \notag \\
& & + \Big[ f_{\delta}( x_{\mathbf{b}}^{*}(x,v), v_{\mathbf{b}}^{*}(x,v) )
\chi (t_{\mathbf{b}}^{*}(X(s),V(s)) ) \mathbf{1}_{x_{\mathbf{b}}^{*}(x,v)
\in \partial\Omega}  \label{derfE} \\
&& \ \ \ + f_{\delta}( x_{\mathbf{f}}^{*}(x,v), v_{\mathbf{f}}^{*}(x,v) )
\chi (t_{\mathbf{f}}^{*}(X(s),V(s)) )\mathbf{1}_{x_{\mathbf{f}}^{*}(x,v) \in
\partial\Omega} \Big] \frac{1}{\tilde{C} \delta^{4}}V(s ) \cdot \nabla_{x}
\xi(X(s )) \chi^{\prime} \big( \frac{\xi(X(s ))}{\tilde{C} \delta^{4}} \big).
\notag
\end{eqnarray}

By the change of variables $(x,v) \mapsto (X(s; 0, x,v), V(s;0,x,v))$, for
sufficiently small $s$, 
\begin{eqnarray}
&&-\iint_{[\Omega_{\tilde{C} \delta^{4}} \backslash \bar{\Omega}] \times 
\mathbb{R}^{3} } f_{E} (x,v) \{ v\cdot \nabla_{x} + \e^{2} \Phi \cdot
\nabla_{v} \}\phi(x,v) \mathrm{d} x \mathrm{d} v  \label{iniz} \\
&=&- \iint_{[\Omega_{\tilde{C} \delta^{4}} \backslash \bar{\Omega}] \times 
\mathbb{R}^{3} } f_{E} (X(s), V(s)) \{ V(s)\cdot \nabla_{X} + \e^{2} \Phi
\cdot \nabla_{V} \} \phi(X(s), V(s)) \mathrm{d} x \mathrm{d} v  \notag \\
&=& - \iint_{[\Omega_{\tilde{C} \delta^{4}} \backslash \bar{\Omega}] \times 
\mathbb{R}^{3} } f_{E} (X(s), V(s)) \frac{d}{ds} \phi(X(s), V(s)) \mathrm{d}
x \mathrm{d} v.  \notag
\end{eqnarray}

Since the change of variables $(x,v) \mapsto (X(s; 0, x,v), V(s;0,x,v))$ has
unit Jacobian, it follows that, for $s$ sufficiently small, 
\begin{equation*}
\iint_{[\Omega_{\tilde{C} \delta^{4}} \backslash \bar{\Omega}] \times 
\mathbb{R}^{3} } f_{E} ((X(s), V(s)) \phi(X(s), V(s)))=\iint_{[\Omega_{%
\tilde{C} \delta^{4}} \backslash \bar{\Omega}] \times \mathbb{R}^{3} } f_{E}
(x,v) \phi(x,v) ,
\end{equation*}
and hence 
\begin{equation*}
\frac{d}{ds} \iint_{[\Omega_{\tilde{C} \delta^{4}} \backslash \bar{\Omega}]
\times \mathbb{R}^{3} } f_{E} ((X(s), V(s)) \phi(X(s), V(s)) =0.
\end{equation*}
Therefore we can move the $s$-derivative on $f_E$: By (\ref{derfE}), 
\begin{eqnarray*}
&&(\ref{iniz}) \\
&=& \iint_{[\Omega_{\tilde{C} \delta^{4}} \backslash \bar{\Omega}] \times 
\mathbb{R}^{3} } \frac{d}{ds} f_{E} (X(s), V(s)) \phi(X(s), V(s)) \mathrm{d}
x \mathrm{d} v \\
&=& \iint_{[\Omega_{\tilde{C} \delta^{4}} \backslash \bar{\Omega}] \times 
\mathbb{R}^{3} } \Big[ f_{\delta}( x_{\mathbf{b}}^{*}(x,v), v_{\mathbf{b}%
}^{*}(x,v) ) \chi^{\prime}( t_{\mathbf{b}}^{*} (X(s),V(s))) \mathbf{1}_{x_{%
\mathbf{b}}^{*}(x,v) \in \partial\Omega} \\
&& \ \ \ \ \ \ \ \ \ \ \ \ \ \ - f_{\delta}( x_{\mathbf{f}}^{*}(x,v), v_{%
\mathbf{f}}^{*}(x,v) ) \chi^{\prime}( t_{\mathbf{f}}^{*} (X(s),V(s))) 
\mathbf{1}_{x_{\mathbf{f}}^{*}(x,v) \in \partial\Omega} \Big] \chi \big( 
\frac{\xi(X(s ))}{\tilde{C} \delta^{4}} \big) \phi(X(s ), V(s )) \\
& +& \iint_{[\Omega_{\tilde{C} \delta^{4}} \backslash \bar{\Omega}] \times 
\mathbb{R}^{3} } \Big[ f_{\delta}( x_{\mathbf{b}}^{*}(x,v), v_{\mathbf{b}%
}^{*}(x,v) ) \chi (t_{\mathbf{b}}^{*}(X(s),V(s)) ) \mathbf{1}_{x_{\mathbf{b}%
}^{*}(x,v) \in \partial\Omega} \\
&& + f_{\delta}( x_{\mathbf{f}}^{*}(x,v), v_{\mathbf{f}}^{*}(x,v) ) \chi (t_{%
\mathbf{f}}^{*}(X(s),V(s)) )\mathbf{1}_{x_{\mathbf{f}}^{*}(x,v) \in
\partial\Omega} \Big] \\
&&\hskip 5cm \times \frac{1}{\tilde{C} \delta^{4}}V(s ) \cdot \nabla_{x}
\xi(X(s ))\chi^{\prime} \big( \frac{\xi(X(s ))}{\tilde{C} \delta^{4}} \big)%
\phi(X(s ), V(s )) .
\end{eqnarray*}
From the change of variable $(X(s;0,x,v), V(s;0,x,v)) \mapsto (x,v)$, 
\begin{eqnarray*}
(\ref{iniz}) &=& \iint_{[\Omega_{\tilde{C} \delta^{4}} \backslash \bar{\Omega%
}] \times \mathbb{R}^{3} } \big[ f_{\delta}(x_{\mathbf{b}}^{*} (x,v), v_{%
\mathbf{b}}^{*}(x,v)) \chi^{\prime} (t_{\mathbf{b}}^{*}(x,v)) \mathbf{1}_{x_{%
\mathbf{b}}^{*}(x,v) \in \partial\Omega} \\
&&- f_{\delta}(x_{\mathbf{f}}^{*} (x,v), v_{\mathbf{f}}^{*}(x,v))
\chi^{\prime} (t_{\mathbf{f}}^{*}(x,v)) \mathbf{1}_{x_{\mathbf{f}}^{*}(x,v)
\in \partial\Omega} \big] \chi \big(\frac{\xi(x)}{\tilde{C} \delta^{4}} %
\big) \phi(x,v) \\
& +& \iint_{[\Omega_{\tilde{C} \delta^{4}} \backslash \bar{\Omega}] \times 
\mathbb{R}^{3} } \Big[ f_{\delta}( x_{\mathbf{b}}^{*}(x,v), v_{\mathbf{b}%
}^{*}(x,v) ) \chi (t_{\mathbf{b}}^{*}(x,v) ) \mathbf{1}_{x_{\mathbf{b}%
}^{*}(x,v) \in \partial\Omega} \\
&& \ \ \ \ \ \ + f_{\delta}( x_{\mathbf{f}}^{*}(x,v), v_{\mathbf{f}%
}^{*}(x,v) ) \chi (t_{\mathbf{f}}^{*}(x,v) )\mathbf{1}_{x_{\mathbf{f}%
}^{*}(x,v) \in \partial\Omega} \Big] \frac{1}{\tilde{C} \delta^{4}}v \cdot
\nabla_{x} \xi(x) \chi^{\prime} \big( \frac{\xi(x)}{\tilde{C} \delta^{4}} %
\big)\phi(x,v).
\end{eqnarray*}
Hence (\ref{eq_ZE_dyn}) is proved.

On the other hand, following the bounds of (\ref{fE1_dyn}) and (\ref{fE2_dyn}%
) in \textit{Step 4} we prove the third line of (\ref{force_Z_dyn}).

\vspace{8pt}

\noindent\textit{Step 6. } We define $\bar{f}(t,x,v)$ for $(t,x,v) \in 
\mathbb{R} \times \mathbb{R}^{3} \times\mathbb{R}^{3}$: 
\begin{equation}  \label{bar_Z_dyn}
\begin{split}
\bar{f}(t,x,v) & \ : = \ f_{\delta}(t,x,v) \mathbf{1}_{(x,v) \in \bar{\Omega}
\times \mathbb{R}^{3}} + f_{E}(t,x,v) \mathbf{1}_{(x,v) \in [\mathbb{R}^{3}
\backslash \bar{\Omega}] \times \mathbb{R}^{3}}.
\end{split}%
\end{equation}
For $\phi \in C^{\infty}_{c}(\mathbb{R} \times \mathbb{R}^{3} \times \mathbb{%
R}^{3})$, by Lemma \ref{timedepgreen}, 
\begin{eqnarray*}
&&- \int_{- \infty}^{\infty} \mathrm{d} t \iint_{\mathbb{R}^{3} \times 
\mathbb{R}^{3}} \bar{f}(t,x,v) \{\e \partial_{t} + v\cdot \nabla_{x} + \e%
^{2} \Phi \cdot \nabla_{v} \} \phi(t,x,v) \mathrm{d} x \mathrm{d} v \\
&=&- \int_{- \infty}^{\infty} \mathrm{d} t\iint_{\Omega \times \mathbb{R}%
^{3}} f_{\delta}(t,x,v) \{ \e \partial_{t} + v\cdot \nabla_{x} + \e^{2} \Phi
\cdot \nabla_{v} \} \phi(t,x,v) \mathrm{d} x \mathrm{d} v \\
&&- \int_{- \infty}^{\infty} \mathrm{d} t \iint_{ [\mathbb{R}^{3} \backslash 
\bar{\Omega}] \times \mathbb{R}^{3}} f_{E}(t,x,v) \{\e \partial_{t} + v\cdot
\nabla_{x} + \e^{2} \Phi \cdot \nabla_{v} \} \phi(t,x,v) \mathrm{d} x 
\mathrm{d} v \\
&=& \int_{- \infty}^{\infty} \mathrm{d} t \int_{\gamma} f_{\delta} (t,x,v)
\phi(t,x,v) \{n(x) \cdot v\} \mathrm{d} S_{x} \mathrm{d} v + \int_{-
\infty}^{\infty} \mathrm{d} t \int_{\gamma} f_{E}(t,x,v) \phi(t,x,v) \{-n(x)
\cdot v\} \mathrm{d} S_{x} \mathrm{d} v \\
&&+ \int_{- \infty}^{\infty} \mathrm{d} t \iint_{\Omega\times \mathbb{R}%
^{3}} \{\e \partial_{t} + v\cdot \nabla_{x} + \e^{2} \Phi \cdot \nabla_{v}
\} f_{\delta}(t,x,v) \phi(t,x,v) \mathrm{d} x \mathrm{d} v \\
&&+ \int_{- \infty}^{\infty} \mathrm{d} t \iint_{ [\Omega_{\tilde{C}
\delta^{4}} \backslash \bar{\Omega}]\times \mathbb{R}^{3}} \{ \e %
\partial_{t} + v\cdot \nabla_{x} + \e^{2} \Phi \cdot \nabla_{v} \} {f}%
_{E}(x,v) \phi(x,v) \mathrm{d} x \mathrm{d} v,
\end{eqnarray*}
where the contributions of $\{t= \infty\}$ and $\{t= -\infty\}$ vanish since 
$\phi(t) \in C_{c}^{\infty} (\mathbb{R})$.

From (\ref{fE=fd_dyn}), the boundary contributions are cancelled: 
\begin{equation*}
\int_{- \infty}^{\infty} \int_{\gamma} f_{\delta} (t,x,v) \phi(t,x,v) 
\mathrm{d} \gamma \mathrm{d} t - \int_{- \infty}^{\infty} \int_{\gamma}
f_{E}(t,x,v) \phi(t,x,v) \mathrm{d} \gamma \mathrm{d} t =0.
\end{equation*}

Further from (\ref{eq_Z_dyn}) and (\ref{eq_ZE_dyn}), we prove that $\bar{f}$
solves (\ref{eq_barf_dyn}) in the sense of distributions on $\mathbb{R}
\times \mathbb{R}^{3} \times \mathbb{R}^{3}$.
\end{proof}

\subsection{Compactness of $K\mathcal{L}^{-1}$}

\label{A3}

\begin{proof}[Proof of Lemma \protect\ref{compact1}]
From Proposition \ref{calL}, $\mathcal{L}^{-1}$ maps $L^{2}$ to $L^{2}$ so
that $\sup_{n} \| f^{n}\|_{L^{2}} < +\infty$.

\noindent\textit{Step 1}. We approximate $K$ by a compactly supported smooth 
$K_{N}$.

\vspace{4pt}

For any $N\gg 1$, by the H\"{o}lder inequality, 
\begin{eqnarray*}
&& \| K(\mathbf{1}_{\{|u|> N \text{ or } |v|>N \text{ or } |v-u|< \frac{1}{N}
\}} f) \|_{2} \\
&=& \Big\| \int_{\mathbb{R}^{3}} \mathbf{k}(v,u) \mathbf{1}_{\{|u|>N \text{
or } |v|> {N} \text{ or } |v-u|< \frac{1}{N} \}} f(u) \mathrm{d} u \Big\|_{2}
\\
&\leq& \sup_{v} \sqrt{\int_{|u| >N} | \mathbf{k}(v,u)| \mathrm{d} u } \times
\sup_{u} \sqrt{\int_{\mathbb{R}^{3}} | \mathbf{k}(v,u)| \mathrm{d} v }
\times \| f \|_{2} \\
& +&\sup_{v} \sqrt{\int_{\mathbb{R}^{3}} |\mathbf{k}(v,u)| \mathrm{d} u }
\times \sup_{u} \sqrt{\int_{|v|>N}| \mathbf{k}(v,u)| \mathrm{d} v } \times
\| f \|_{2} \\
& +& \sup_{v} \sqrt{\int_{\mathbb{R}^{3}} |\mathbf{k}(v,u)| \mathrm{d} u }
\times \sup_{u} \sqrt{\int_{\mathbb{R}^{3}} \frac{e^{-\theta |v-u|^{2}}}{%
|v-u|} \mathbf{1}_{|v-u| < \frac{1}{N}} \mathrm{d} v} \times \| f\|_{2} \
\lesssim \ o(1) \| f\|_{2},
\end{eqnarray*}
where we have used the fact $\int_{\mathbb{R}^{3}} \mathbf{k}(v,u) \mathrm{d}
u \lesssim 1$, $\int_{\mathbb{R}^{3}} \mathbf{k}(v,u) \mathrm{d} v \lesssim
1 $, and $\int_{\mathbb{R}^3}\frac{e^{-\beta|v|^{2}}}{|v|} \mathrm{1}_{|v| < 
\frac{1}{N}}=o(1)$.

Note that 
\begin{equation*}
\mathbf{k}(v,u) = \mathbf{k}_{1} (v,u) - \mathbf{k}_{2} (v,u), \ \ \ \text{%
with} \ \ \ \mathbf{k}_{1}, \mathbf{k}_{2} >0,
\end{equation*}
and therefore $\mathbf{k}_{i}(v,u) \mathbf{1}_{\{|u|\leq N \ \& \ |v|\leq N
\ \& \ |v-u| \geq \frac{1}{N}\}}> \delta$ for $0< \delta\ll1$.

We now define $K_{N} f:= \int \mathbf{k}_{N} f$ with 
\begin{equation}  \label{k_N}
\begin{split}
\mathbf{k}_{N}(v,u) =& \ \mathbf{k}_{1,N} (v,u) - \mathbf{k}_{2,N} (v,u) \\
:=& \ \mathbf{k}_{1}(v,u)\mathbf{1}_{\{|u|\leq N \ \& \ |v|\leq N \ \& \
|v-u| \geq \frac{1}{N}\}} * \phi_{ \frac{1}{N}}(v)\phi_{ \frac{1}{N}}(u) \\
& -\mathbf{k}_{2}(v,u)\mathbf{1}_{\{|u|\leq N \ \& \ |v|\leq N \ \& \ |v-u|
\geq \frac{1}{N}\}} * \phi_{ \frac{1}{N}}(v)\phi_{ \frac{1}{N}}(u) \\
=& \ \int_{\mathbb{R}^{3}} \mathrm{d} u^{\prime} \phi_{ \frac{1}{N}}(u-
u^{\prime})\int_{\mathbb{R}^{3}} \mathrm{d} v^{\prime} \phi_{ \frac{1}{N}%
}(v-v^{\prime}) \mathbf{k}_{1}(v,u)\mathbf{1}_{\{|u|\leq N \ \& \ |v|\leq N
\ \& \ |v-u| \leq \frac{1}{N}\}} \\
&- \int_{\mathbb{R}^{3}} \mathrm{d} u^{\prime} \phi_{ \frac{1}{N}}(u-
u^{\prime})\int_{\mathbb{R}^{3}} \mathrm{d} v^{\prime} \phi_{ \frac{1}{N}%
}(v-v^{\prime}) \mathbf{k}_{2}(v,u)\mathbf{1}_{\{|u|\leq N \ \& \ |v|\leq N
\ \& \ |v-u| \leq \frac{1}{N}\}},
\end{split}%
\end{equation}
where $\phi_{ \frac{1}{N}}$ is the standard mollifiers. In particular, for $%
i=1,2$, 
\begin{eqnarray*}
&&\sup_{v}\int_{\mathbb{R}^{3}} \big| \mathbf{k}_{i,N}(v,u) - \mathbf{k}_{i}
(v,u) \mathbf{1}_{\{|u|\leq N \ \& \ |v|\leq N \ \& \ |v-u| \leq \frac{1}{N}%
\}} \big| \mathrm{d} u =o(1), \\
&&\sup_{u}\int_{\mathbb{R}^{3}} \big| \mathbf{k}_{i,N}(v,u) - \mathbf{k}_{i}
(v,u) \mathbf{1}_{\{|u|\leq N \ \& \ |v|\leq N \ \& \ |v-u| \leq \frac{1}{N}%
\}} \big| \mathrm{d} v =o(1).
\end{eqnarray*}
Consequently, $\| K_{N } f- K f\|_{2} \lesssim o(1) \| f\|_{2}$.

We denote 
\begin{equation*}
\bar{\mathbf{k}}_{N} (v,u) := \mathbf{k}_{1,N} (v,u) + \mathbf{k}_{2,N}
(v,u), \ \ \ \bar{K}_{N} f (v) : = \int_{\mathbb{R}^{3}} \bar{\mathbf{k}}%
_{N} (v,u) f(u) \mathrm{d} u.
\end{equation*}
Note that $\mathbf{k}_{1,N}, \mathbf{k}_{2,N} \geq 0$ and hence $\bar{%
\mathbf{k}}_{N} \geq 0$ and $\bar{\mathbf{k}}_{N} \geq | \mathbf{k}_{N}|$.

\vspace{8pt}

\noindent\textit{Step 2. }

We fix $\delta \ll 1$. Given $f^n$, we define $f_{\delta}^{n}$ as 
\begin{equation}  \label{Z}
f^n_{\delta}(x,v) : = \Big[1-\chi(\frac{n(x) \cdot v}{\delta}) \chi (\frac{%
\xi(x)}{\delta}) \Big] \chi(\delta|v|) f^(x,v),
\end{equation}
and extend it to the whole space $\mathbb{R}^{3} \times \mathbb{R}^{3}$. We
follow the process in the proof of Lemma \ref{extension_dyn} and we only
pinpoint the difference.

Similarly to \textit{Step 3} of the proof of Lemma \ref{extension_dyn}, we
define $f_{E}^{n}(x,v)$ for $(x,v) \in [\mathbb{R}^{3} \backslash \bar{\Omega%
}] \times \mathbb{R}^{3}$ 
\begin{eqnarray}
&& f_{E}^{n}(x,v)  \label{ZE_compact} \\
&& \ := \ e^{- \lambda t_{\mathbf{b}}^{*} (x,v) + \int^{- t_{\mathbf{b}%
}^{*}(x,v)}_{0} \big[\frac{ {\nu}(X(\tau ),V(\tau ))}{\e} - \frac{1}{2} \e%
^{2} \Phi(X(\tau)) \cdot V(\tau) \big] \mathrm{d} \tau }  \notag \\
&& \ \ \ \ \ \ \ \ \ \ \ \ \ \ \ \ \ \ \ \ \ \ \times f_{\delta}^{n}(x_{%
\mathbf{b}}^{*}(x,v), v_{\mathbf{b}}^{*}(x,v)) \chi\big( \frac{\xi(x)}{%
\tilde{C} \delta^{4}} \big) \chi ( t_{\mathbf{b}}^{*} (x,v) ) \ \ \ \ \text{%
for} \ \ x_{\mathbf{b}}^{*}(x,v) \in \partial\Omega,  \notag \\
&& \ := \ e^{ \lambda t_{\mathbf{f}}^{*} (x,v) + \int^{ t_{\mathbf{f}%
}^{*}(x,v)}_{0} \big[\frac{ {\nu}(X(\tau ),V(\tau ))}{\e} - \frac{1}{2} \e%
^{2} \Phi(X(\tau)) \cdot V(\tau) \big] \mathrm{d} \tau }  \notag \\
&& \ \ \ \ \ \ \ \ \ \ \ \ \ \ \ \ \ \ \ \ \ \ \ \times f_{\delta}^{n}(x_{%
\mathbf{f}}^{*}(x,v), v_{\mathbf{f}}^{*}(x,v)) \chi\big( \frac{\xi(x)}{%
\tilde{C} \delta^{4}} \big) \chi ( t_{\mathbf{f}}^{*} (x,v) ) \ \ \ \ \text{%
for} \ \ x_{\mathbf{f}}^{*}(x,v) \in \partial\Omega,  \notag \\
&& \ := \ 0 \ \ \ \ \ \ \ \ \ \ \ \ \ \ \ \ \ \ \ \ \ \ \ \ \ \ \ \ \ \ \ \
\ \ \ \ \ \ \ \ \ \ \ \ \ \ \ \ \ \text{for} \ \ x_{\mathbf{b}}^{*}(x,v)
\notin \partial\Omega \ \text{and} \ x_{\mathbf{f}}^{*}(x,v) \notin
\partial\Omega,  \notag
\end{eqnarray}
where $(X(\tau), V(\tau)) = (X(\tau; 0,x,v), V(\tau;0,x,v))$.

Then 
\begin{equation}  \label{fE=fd}
f^n_{E}(x,v) = f^n_{\delta} (x,v) \ \ \ \text{for all } x \in \partial\Omega,
\end{equation}
because, if $x \in \partial\Omega$ and $n(x) \cdot v>\delta$, then $t_{%
\mathbf{b}}^{*}(x,v)=0$. If $x \in \partial\Omega$ and $n(x) \cdot v<\delta$
then $t_{\mathbf{f}}^{*}(x,v)=0$.

We define $\bar{f}^{n}(x,v)$ as 
\begin{equation}  \label{bar_Z}
\begin{split}
\bar{f}(x,v) & \ : = \ f_{\delta}(x,v) \mathbf{1}_{(x,v) \in \bar{\Omega}
\times \mathbb{R}^{3}} + f_{E}(x,v) \mathbf{1}_{(x,v) \in [\mathbb{R}^{3}
\backslash \bar{\Omega}] \times \mathbb{R}^{3}}.
\end{split}%
\end{equation}
Note that $\bar{f}^{n}$ solves 
\begin{equation}
\lambda \bar{f}^{n} + v\cdot \nabla_{x} \bar{f} ^{n} + \frac{1}{\e} \nu \bar{%
f} ^{n} +\e^{2} \Phi \cdot \nabla_{v} \bar{f}^{n} - \frac 1 2\e^{2} \Phi
\cdot v \bar{f}^{n} \ = \ h^{n} \ = \ h_{1}^{n}+ h_{2}^{n} + h_{3}^{n} +
h_{4}^{n},  \notag
\end{equation}
where 
\begin{eqnarray*}
h_{1}^{n} (x,v)&=& \mathbf{1}_{(x,v) \in \Omega \times \mathbb{R}^{3}}
[1-\chi(\frac{n(x) \cdot v}{\delta}) \chi( \frac{\xi(x)}{\delta})\chi( \frac{%
\xi(x)}{\delta}) ] \chi(\delta|v|) g^{n}, \\
h_{2}^{n} (x,v) &=& \mathbf{1}_{(x,v) \in \Omega \times \mathbb{R}^{3}}
f^{n} \{ v\cdot \nabla_{x} + \e^{2} \Phi \cdot \nabla_{v} \} \Big( [1-\chi(%
\frac{n(x) \cdot v}{\delta}) \chi( \frac{\xi(x)}{\delta}) ] \chi ( \frac{%
\xi(x)}{\delta} ) \chi(\delta|v|)\Big), \\
h_{3}^{n} (x,v) &=& \mathbf{1}_{{\ (x,v) \in [\Omega_{\tilde{C} \delta^{4}}
\backslash \bar{\Omega}]\times \mathbb{R}^{3}}} \ \frac{1}{\tilde{C}
\delta^{4}}v \cdot \nabla_{x} \xi(x) \chi^{\prime} \big( \frac{\xi(x)}{%
\tilde{C} \delta^{4}} \big)  \notag \\
&& \times \Big[ f_{\delta}^{n}( x_{\mathbf{b}}^{*}(x,v), v_{\mathbf{b}%
}^{*}(x,v) ) e^{- \lambda t_{\mathbf{b}}^{*} (x,v) + \int^{- t_{\mathbf{b}%
}^{*}(x,v)}_{0} \big[\frac{ {\nu}(X(\tau ),V(\tau ))}{\e} - \frac{1}{2} \e%
^{2} \Phi(X(\tau)) \cdot V(\tau) \big] \mathrm{d} \tau } \mathbf{1}_{x_{%
\mathbf{b}}^{*}(x,v) \in \partial\Omega} \\
&& \ \ + f_{\delta}^{n}( x_{\mathbf{f}}^{*}(x,v), v_{\mathbf{f}}^{*}(x,v) ) 
\mathbf{1}_{x_{\mathbf{f}}^{*}(x,v) \in \partial\Omega} e^{ \lambda t_{%
\mathbf{f}}^{*} (x,v) + \int^{ t_{\mathbf{f}}^{*}(x,v)}_{0} \big[\frac{ {\nu}%
(X(\tau ),V(\tau ))}{\e} - \frac{1}{2} \e^{2} \Phi(X(\tau)) \cdot V(\tau) %
\big] \mathrm{d} \tau } \Big] , \\
h_{4}^{n} (x,v) &=& \mathbf{1}_{{\ (x,v) \in [\Omega_{\tilde{C} \delta^{4}}
\backslash \bar{\Omega}]\times \mathbb{R}^{3}}} f^{n}_{\delta}(x_{\mathbf{b}%
}^{*}(x,v), v_{\mathbf{b}}^{*}(x,v)) \\
&& \ \ \times e^{- \lambda t_{\mathbf{b}}^{*} (x,v) + \int^{- t_{\mathbf{b}%
}^{*}(x,v)}_{0} \big[\frac{ {\nu}(X(\tau ),V(\tau ))}{\e} - \frac{1}{2} \e%
^{2} \Phi(X(\tau)) \cdot V(\tau) \big] \mathrm{d} \tau } \chi \big( \frac{%
\xi(x)}{\tilde{C} \delta^{4}} \big) \chi^{\prime}(t_{\mathbf{b}}^{*}(x,v)) 
\mathbf{1}_{ x_{\mathbf{b}}^{*}(x,v) \in \partial\Omega} \\
&&+ \mathbf{1}_{{\ (x,v) \in [\Omega_{\tilde{C} \delta^{4}} \backslash \bar{%
\Omega}]\times \mathbb{R}^{3}}} f^{n}_{\delta}(x_{\mathbf{f}}^{*}(x,v), v_{%
\mathbf{f}}^{*}(x,v)) \\
&& \ \ \times e^{ \lambda t_{\mathbf{f}}^{*} (x,v) + \int^{ t_{\mathbf{f}%
}^{*}(x,v)}_{0} \big[\frac{ {\nu}(X(\tau ),V(\tau ))}{\e} - \frac{1}{2} \e%
^{2} \Phi(X(\tau)) \cdot V(\tau) \big] \mathrm{d} \tau } \chi \big( \frac{%
\xi(x)}{\tilde{C} \delta^{4}} \big) \chi^{\prime}(t_{\mathbf{f}}^{*}(x,v)) 
\mathbf{1}_{ x_{\mathbf{f}}^{*}(x,v) \in \partial\Omega},
\end{eqnarray*}
and 
\begin{eqnarray*}
&&\| h_{1}^{n} \|_{L^{2}(\mathbb{R}^{3} \times \mathbb{R}^{3})} \ \lesssim \
\| g^{n}\|_{L^{2}(\Omega \times \mathbb{R}^{3})}, \ \ \| h_{2}^{n} \|_{L^{2}
(\mathbb{R}^{3} \times \mathbb{R}^{3})} \ \lesssim_{\delta} \ \|f^{n}
\|_{L^{2}(\Omega \times \mathbb{R}^{3})}, \\
&&\| h_{3}^{n} \|_{L^{2} (\mathbb{R}^{3} \times \mathbb{R}^{3})} +\|
h_{4}^{n} \|_{L^{2} (\mathbb{R}^{3} \times \mathbb{R}^{3})} \
\lesssim_{\delta} \ \| f_{\delta}^{n} \|_{L^{2} (\gamma)}.
\end{eqnarray*}

\vspace{4pt}

\noindent\textit{Step 3}. For $(x,v) \in \text{supp}(\bar{f})$, we can
choose a fixed $T>0$ such that 
\begin{equation}  \label{large_time}
X(T; 0,x,v) \notin \text{supp}(\bar{f} ) \ \ \text{and} \ \ X(T; 0,x,v)
\notin \text{supp}(h),
\end{equation}
so that 
\begin{equation*}
\bar{f} (X(T;0,x,v), V( T;0,x,v))=0.
\end{equation*}
Directly, 
\begin{equation*}
| X( T;0,x,v) -x| = | v T+ O(\e^{2}) \| \Phi \|_{\infty}T^{2}| \geq \delta
T- O(\e^{2}) \| \Phi \|_{\infty }T^{2}.
\end{equation*}
We choose $T= \frac{C}{\delta}$ for large but fixed $C\gg 1$ such that $| X(
T;0,x,v) -x|\geq C - O(\frac{\e^{2}}{\delta^{2}}) \geq \frac{C}{2} \gg 1$.
This proves our claim (\ref{large_time}).

With this choice $T$, 
\begin{equation}
\bar{f}^{n}(x,v) = - \int^{T}_{0} h^{n}(X(s;0,x,v), V(s;0,x,v)) e^{\lambda s
+ \int^{s}_{0} \big[\frac{\nu( V(\tau;0,x,v))}{\e} - \frac{1}{2} \e^{2}
\Phi(X(\tau;0,x,v) ) \big] \mathrm{d} \tau } \mathrm{d} s.  \notag
\end{equation}
By the averaging lemma \cite{GLPS}, for any given $v$, 
\begin{equation*}
\int \mathbf{k}_{N}(v,u)\bar{f}^{n}(x,u) \mathrm{d} u\in H^{1/4}(\mathbb{R}%
^{3}).
\end{equation*}
Since $\bar{f}^{n}$'s support is bounded unformly, by a diaganolization
argument, it follows that there exists a weak limit $\bar{f}\in L^{2}$ of $%
\bar{f}^{n}$ such that for any rational point $v$ 
\begin{equation}  \label{conv_kf}
\int \mathbf{k}_{N}(v,u)\bar{f}^{n}(x,u)\mathrm{d} u\rightarrow \int \mathbf{%
k}_{N}(v,u)\bar{f}(x,u) \mathrm{d} u \ \ \ \text{strongly in} \ L_{x}^{2}.
\end{equation}
Since $\mathbf{k}_{N}(v,u)$ is smooth in $v$ with compact support, we deduce
(\ref{conv_kf}) for all $v \in \mathbb{R}^{3}.$

\vspace{4pt}

\noindent\textit{Step 4}. Finally, 
\begin{eqnarray*}
K f^{n} &=& K_{N} f^{n} + (Kf^{n}-K_{N}f^{n}) = {K_{N} f^{n}_{\delta}} \ + \ 
{\ K_{N}( f^{n}- f^{n}_{\delta}) + (Kf^{n}-K_{N}f^{n})} .
\end{eqnarray*}
From \textit{Step 3}, 
\begin{eqnarray*}
K_{N} {f}^{n}_{\delta} = K_{N} \bar{f}^{n}|_{\Omega} \rightarrow K_{N} \bar{f%
}|_{\Omega} \ \ \ \text{strongly in } L^{2}.
\end{eqnarray*}
Note that 
\begin{equation}  \label{small_chi}
1-\Big[1-\chi(\frac{n(x) \cdot v}{\delta}) \chi (\frac{ \xi(x )}{\delta}) %
\Big] \chi(\delta|v|) \leq \mathbf{1}_{|v| \geq \frac{1}{\delta}} + \mathbf{1%
}_{|v| \leq \frac{1}{\delta}, \text{dist} (x, \partial\Omega) < \frac{\delta%
}{2}, |n(x) \cdot v|< \delta}.
\end{equation}
Hence, from (\ref{k_N}), 
\begin{eqnarray*}
\| K_{N} ( f^{n} - f^{n}_{\delta})\|_{L^{2}_{x,v}} &\lesssim& \| f^{n}
-f^{n}_{\delta}\|_{L^{2}_{x,v}} \lesssim \big\| \chi_{\delta}^{c} \big\|%
_{\infty} \| f^{n}\|_{L^{2}_{x,v}} \lesssim O(\delta), \\
\| Kf^{n} - K_{N} f^{n} \|_{L^{2}_{x,v}} &\lesssim& o(1) \| f^{n}
\|_{L^{2}_{x,v}} \lesssim o(1).
\end{eqnarray*}
We conclude the proof by choosing $\delta \ll 1, \ 1 \ll N $ then letting $%
n\rightarrow \infty $.
\end{proof}

\noindent\textbf{Acknowledgements.} We thank Prof.  Fujun Zhou for pointing out some mistakes in previous versions of this paper. Y. Guo's research is supported in part
by NSFC grant 10828103, NSF grant DMS-0905255, and BICMR. C. Kim's research
is supported in part by NSF DMS-1501031, the Herchel Smith Foundation and
the University of Wisconsin-Madison Graduate School with funding from the
Wisconsin Alumni Research Foundation. R. Marra is partially supported by
MIUR-Prin.

\end{document}